\documentclass[10pt, leqno]{article}
\usepackage{amsmath,amsfonts,amsthm,amscd,amssymb}
\numberwithin{equation}{section}
\theoremstyle{plain}
\newtheorem{theo}{Theorem}[section]
\newtheorem{prop}[theo]{Proposition}
\newtheorem{coro}[theo]{Corollary} 
\newtheorem{lemm}[theo]{Lemma}
\theoremstyle{definition}
\newtheorem{defi}[theo]{Definition}
\newtheorem{rema}[theo]{Remark}
\newtheorem{theo-defi}[theo]{Theorem-Definition}
\newtheorem{prop-defi}[theo]{Proposition-Definition}
\newtheorem{rema-defi}[theo]{Remark-Definition}
\newtheorem{exem-defi} [theo]{Example-Definiton}

\newtheorem{exem}[theo]{Example}
\newtheorem{conj}[theo]{Conjecture}
\newtheorem{prob}[theo]{Problem}

\def \al{\alpha}
\def \bet{\beta}
\def \bul{\bullet}
\def \col{\colon}
\def \Del{\Delta}

\def \del{\delta}

\def \Gam{\Gamma}

\def \kap{\kappa}
\def \Lam{\Lambda}
\def \lam{\lambda}

\def \Lo{\Longrightarrow}
\def \lo{\longrightarrow}
\def \lom{\longmapsto}
\def \mab{\mathbb}

\def \Om{\Omega}
\def \om{\omega}
\def \ol{\overline}
\def \os{\overset}
\def \parno{\par\noindent}

\def \sig{\sigma}

\def \sus{\subset}

\def \ul{\underline}
\def \us{\underset}

\def \vil{\varinjlim}

\def \wh{\widehat}
\def \wt{\widetilde}
\bigskip

\newcommand{\getsfrom}
{\ensuremath{\longleftarrow\kern-.
52em\lower-.1ex\hbox%
{$\shortmid\,$}}}

\begin{document}

\title{Degenerations of log Hodge de Rham spectral sequences, 
log Kodaira vanishing theorem in characteristic $p>0$  
and log weak Lefschetz conjecture for log crystalline cohomologies}
\author{Yukiyoshi Nakkajima, Fuetaro Yobuko
\date{}\thanks{ 
The first named author is supported from JSPS
Grant-in-Aid for Scientific Research (C)
(Grant No.~80287440, 18K03224). 
The second named author is supported by 
JSPS Fellow (Grant No.~15J05073).\endgraf}}
\maketitle

$${\bf Abstract}$$
In this article we prove that the 
log Hodge de Rham spectral sequences of 
certain proper log smooth schemes of Cartier type 
in characteristic $p>0$ degenerate at $E_1$.  
We also prove that the log Kodaira vanishings for them hold 
when they are projective. 
We formulate the log weak Lefschetz conjecture 
for log crystalline cohomologies 
and prove that it is true in certain cases.

$${\bf Contents}$$

%\parno 
%{\bf Part I}
\medskip 
\parno
\S\ref{sec:int}. Introduction
\parno 
\S\ref{sec:amf}. The heights of Artin-Mazur formal groups 
of certain schemes
\parno 
\S\ref{sec:dc}. The dimensions of cohomologies of closed differential forms 
\parno 
\S\ref{sec:latv}. Log deformation theory vs log deformation theory 
with abrelative and relative Frobenus morphisms
%\parno 
%\S\ref{ldtv}. Log deformation theory vs log deformation theory 
%with relative Frobenius morphisms
%\parno
%\S\ref{sec:lst}. Recall on Tsuji's results on log schemes
\parno 
\S\ref{sec:app}. Applications of log deformation theory with relative Frobenius morphisms
\parno 
\S\ref{sec:lw}. Log weak Lefschetz conjecture 
\parno 
\S\ref{sec:qfs}. Quasi-$F$-split log schemes
\parno
%\S\ref{sec:mix}. 
%Infinitesimal strict semistable families in mixed characteristics and equal positive 
%characteristic
\parno 
\S\ref{sec:lif}. Lifts of certain log schemes over ${\cal W}_2$

\medskip 
\parno 
{\bf Appendix}
\medskip 
\parno 
Yukiyoshi Nakkajima
\medskip 
\parno 
\S\ref{sec:wlt}. 
Weak Lefschetz theorem for isocrystalline cohomologies 
\bigskip
\parno

\section{Introduction}\label{sec:int} 
Let $\kap$ be a perfect field of characteristic $p>0$.  
Let ${\cal W}$ (resp.~${\cal W}_n$ $(n\in {\mab Z}_{>0})$) be the Witt ring of $\kap$  
(resp.~the Witt ring of $\kap$ of length $n$).  
\par 
In \cite{mup} Mumford has shown that the $E_1$-degeneration of 
the Hodge de Rham spectral sequence of a proper smooth scheme
over $\kap$ does not hold in general unlike the case of characteristic 0 in \cite{dl}. 
In \cite{raykv} Raynaud has shown that 
the Kodaira vanishing for a projective smooth scheme
over $\kap$ does not hold in general unlike the case of characteristic 0 in \cite{ko}. 
%They do not hold for certain projective smooth surfaces over $\kap$. 
However, in their famous article \cite{di},  
Deligne and Illusie have given a sufficient condition 
for the $E_1$-degeneration and the Kodaira vanishing theorem: 
if a proper (resp.~projective) smooth scheme $X$ 
over $\kap$ has a smooth lift over ${\cal W}_2$, 
then the $E_1$-degeneration (resp.~the vanishing theorem) 
holds in characteristic $p$ in a restricted sense. 
However there is no concretely calculable criterion for the existence of 
a smooth lift over ${\cal W}_2$ of a given $X/\kap$ in general.  
%(except Mehta's remark mentioned soon later).  
Consequently one does not know whether 
the $E_1$-degeneration and 
the vanishing theorem for $X/\kap$ hold a priori. 
On the other hand, in \cite{az}
Achinger and Zdanowicz have constructed 
projective smooth schemes over $\kap$ 
which do not have smooth lifts over ${\cal W}_2$ 
and for which the $E_1$-degenerations hold.  
In \cite{eknl} Ekedahl has shown that 
the Hirokado variety in \cite{hi} 
does not have a smooth lift over 
${\cal W}_2$ when $p=3$. However, 
in \cite{tay} Takayama has proved that 
(a part of) Kodaira vanishing theorem holds for it. 
\par 
On the other hand, Deligne has proved the hard Lefschetz theorem 
for the $l$-adic \'{e}tale cohomologies of $X/\kap$ 
in \cite{dw2}  as in the case of characteristic $0$. 
Using this result and Berthelot's 
weak Lefschetz theorem for crystalline cohomologies of $X/\kap$ 
for any hypersurface sections of high degrees in \cite{bwl}, 
Katz and Messing have proved 
the hard Lefschetz theorem and the weak Lefschetz theorem 
for isocrystalline cohomologies of $X/\kap$ (\cite{kme}). 
However we would like to point out that 
there is a gap in the proof of 
Berthelot's weak Lefschetz theorem in \cite{bwl} 
and we fill this gap in the text. 
\par   
Let $X$ be a proper (smooth) scheme over $\kap$ of 
pure dimension $d\geq 1$. 
Let $q$ be a nonnegative integer. 
Let $\Phi^q_{X/\kap}$ be the Artin-Mazur functor of $X/\kap$ in degree $q$: 
$\Phi^q_{X/\kap}$ is the functor defined by the following: 
%\begin{align*} 
$$\Phi^q_{X/\kap}(A):={\rm Ker}(H^q_{\rm et}
(X\otimes_{\kap}A,{\mab G}_m)\lo H^q_{\rm et}(X,{\mab G}_m)) \in ({\rm Ab}).$$
%\end{align*}
Here $A$ is an artinian local $\kap$-algebra with residue field $\kap$. 
If $\Phi^{q-1}_{X/\kap}$ is formally smooth, 
then $\Phi^q_{X/\kap}$ is pro-represented by 
a commutative formal group over $\kap$ (\cite{am}). 
%(In the case $d=1$, 
%$\Phi_{X/\kap}$ is also pro-represented by 
%a commutative formal Lie group over $\kap$ (\cite{sch}).) 
Let $h^q(X/\kap)$ be the height of $\Phi^q_{X/\kap}$ if it is 
pro-representable. 
We call $h^q(X/\kap)$  
the $q$-th {\it Artin-Mazur height} of $X/\kap$. 
\par 
Let $X$ be a geometrically irreducible proper smooth scheme over $\kap$ 
of dimension $d\geq 1$. 
We say that $X$ is a Calabi-Yau variety over $\kap$ of dimension $d$ if    
$H^q(X,{\cal O}_X)=0$ $(0< q< d)$ and 
$\Omega^d_{X/\kap}\simeq {\cal O}_X$. 
\par 
%(Though the main result in \cite{y} has been 
%stated for a Calabi-Yau variety over $\kap$, 
%that is, for a proper smooth scheme over $\kap$ such that 
%$H^0(Y,{\cal O}_Y)=\kap$,  
%$H^q(Y,{\cal O}_Y)=0$ $(0< q< d)$ and 
%$\Omega^d_{Y/\kap}\simeq {\cal O}_Y$,  
%the four conditions above is enough for the conclusion of  (\ref{theo:y}).) 
%In \cite{y} the second named author in this paper has recently 
%given a sufficient condition for 
%the existence of a smooth lift over ${\cal W}_2$ of 
%very interesting classes of proper smooth varieties including a class of 
%Calabi-Yau varieties over $\kap$ of any dimension 
%with finite Artin-Mazur height. 
%Let us recall this result briefly. 

In \cite{y} the second named author of this article  
has recently proved the following:

\begin{theo}[{\bf \cite{y}}]\label{theo:y}
Let $X$ be a Calabi-Yau variety over $\kap$ of dimension $d\geq 1$. 
If $h^d(X/\kap)< \infty$,    
%Assume that the following four conditions hold$:$
%\par 
%$({\rm a})$ $H^{d-1}(Y,{\cal O}_Y)=0$ if $d\geq 2$, 
%\par 
%$({\rm b})$ $H^{d-2}(Y,{\cal O}_Y)=0$ if $d\geq 3$, 
%\par 
%$({\rm c})$ $\Omega^d_{Y/\kap}\simeq {\cal O}_Y$, 
%\par 
%$({\rm d})$ $h(Y/\kap)< \infty$. 
%\parno 
then there exists a proper smooth scheme ${\cal X}$ over ${\cal W}_2$ 
such that ${\cal X}\otimes_{{\cal W}_2}\kap=X$. 
\end{theo}   
Using Deligne-Illusie's theorem and (\ref{theo:y}), 
we see that 
the Hodge de Rham spectral sequence 
\begin{align*} 
E_1^{ij}=H^j(X,\Om^i_{X/\kap})\Lo H^{i+j}_{\rm dR}(X/\kap)
\tag{1.1.1}\label{ali:oedxs} 
\end{align*}  
of $X/\kap$ degenerates at $E_1$ if $d\leq p$. 
In particular, if $p\not=2$ and $h^3(X/\kap)<\infty$, 
then (\ref{ali:oedxs}) degenerates at $E_1$ 
for a 3-dimensional Calabi-Yau variety $X/\kap$. 
%As far as we know, such a  general result for 
%the $E_1$-degeneration for $X/\kap$ 
%in the case $d>2$ was not known before his result
%(see the introduction of \cite{sc}).  
Using Joshi's theorem (\cite{j}), we easily see that
the slope spectral sequence 
\begin{align*} 
E_1^{ij}=H^j(X,{\cal W}\Om^i_{X/\kap})\Lo H^{i+j}_{\rm crys}(X/{\cal W})
\tag{1.1.2}\label{ali:oesdxs} 
\end{align*}  
degenerates at $E_1$ for 
a 3-dimensional Calabi-Yau variety $X/\kap$ such that 
$h^3(X/\kap)<\infty$. 
Furthermore we see that it is of Hodge-Witt type 
by a fundamental theorem in \cite{ir}: 
$H^q_{\rm crys}(X/{\cal W})
=\bigoplus_{i+j=q}H^j(X,{\cal W}\Om^i_{X/\kap})$ $(q\in {\mab N})$. 
(This is a $3$-dimensional analogue of the Hodge-Witt decomposition 
of a $K3$-surface over $\kap$ with finite second Artin-Mazur height (\cite{idw}).)
Using Ekedahl's remark in \cite{ir}, we see that 
the following spectral sequence 
\begin{align*} 
E_1^{ij}=H^j(X,{\cal W}_n\Om^i_{X/\kap})\Lo H^{i+j}_{\rm crys}(X/{\cal W}_n)
\quad (n\in {\mab N})
\tag{1.1.3}\label{ali:oednxs} 
\end{align*}
degenerates at $E_2$. 
%By virtue of this result, it may be not too optimistic to expect that 
%(\ref{ali:oedxs}) degenerates at $E_1$ without the assumption $d\leq p$ 
%(for any $d$) at least in the case where 
%$h^d(X/\kap)$ is finite. 
%The main result of \cite{y} is the following theorem: 
%\begin{theo}[{\bf \cite{y}}]\label{theo:y}
%Assume that the following four conditions hold$:$
%\par 
%$({\rm a})$ $H^{d-1}(Y,{\cal O}_Y)=0$ if $d\geq 2$, 
%\par 
%$({\rm b})$ $H^{d-2}(Y,{\cal O}_Y)=0$ if $d\geq 3$, 
%\par 
%$({\rm c})$ $\Omega^d_{Y/\kap}\simeq {\cal O}_Y$, 
%\par 
%$({\rm d})$ $h(Y/\kap)< \infty$. 
%\parno 
%Then there exists a proper smooth scheme ${\cal Y}$ over ${\cal W}_2$ 
%such that ${\cal Y}\otimes_{{\cal W}_2}\kap=Y$. 
%\end{theo}   
%(Though the main result in \cite{y} has been 
%stated for a Calabi-Yau variety over $\kap$, 
%that is, for a proper smooth scheme over $\kap$ such that 
%$H^0(Y,{\cal O}_Y)=\kap$,  $H^q(Y,{\cal O}_Y)=0$ $(0< q< d)$ and 
%$\Omega^d_{Y/\kap}\simeq {\cal O}_Y$,  
%the four conditions above is enough for the conclusion of  (\ref{theo:y}).) 
%By virtue of Deligne and Illusie's theorem mentioned above,  
%(\ref{theo:y}) tells us that the Hodge de Rham spectral sequence 
%\begin{align*} 
%E_1^{ij}=H^j(Y,\Om^i_{Y/\kap})\Lo H^{i+j}_{\rm dR}(Y/\kap)
%\tag{1.1.1}\label{ali:oedxs} 
%\end{align*}  
%of $Y/\kap$ degenerates at $E_1$ if $\dim Y\leq p$. 
\par 
For an ample line bundle ${\cal L}$ on $X$, 
in \cite{yoh} the second named author  
has also proved that $H^j(X,{\cal L})=0$ $(j>0)$ 
without any assumption on $d$ and $p$.  
\par 
%He has proved (\ref{theo:y}) as follows. 
To prove (\ref{theo:y}), he has introduced 
a new invariant $h_F(X)$ of $X$ as follows. 
\par 
Let $Y$ be a (proper smooth) scheme over $\kap$. 
Let $F_Y$ be the Frobenius endomorphism of $Y$. 
Set $F:={\cal W}_n(F_Y^*)\col 
{\cal W}_n({\cal O}_Y)\lo F_{Y*}({\cal W}_n({\cal O}_Y))$. 
This is a morphism of ${\cal W}_n({\cal O}_Y)$-modules. 
In \cite{y} he has introduced the notion of the 
{\it quasi-Frobenius splitting height} $h_F(Y)$ 
for any (proper smooth) scheme $Y$ over $\kap$. 
(In [loc.~cit.] he has denoted it by ${\rm ht}^S(Y)$.) 
It is the minimum of positive integers  $n$'s such that 
there exists a morphism 
$\rho \col F_{Y*}({\cal W}_n({\cal O}_Y))\lo {\cal O}_Y$ 
of ${\cal W}_n({\cal O}_Y)$-modules 
such that 
$\rho \circ F\col {\cal W}_n({\cal O}_Y)\lo {\cal O}_Y$ 
is the natural projection. 
(If there does not exist such $n$, then  we set $h_F(Y)=\infty$.)  
(Because the ``quasi-Frobenius splitting height'' is too long, 
we call this the {\it quasi-$F$-split height} simply.) 
This is a nontrivial generalization of the notion of the Frobenius splitting 
by Mehta and Ramanathan in \cite{mr} 
because they have said that, for a scheme $Z$ of characteristic $p>0$, 
$Z$ is a Frobenius splitting(=$F$-split) scheme if 
$F\col {\cal O}_Z\lo F_{Z*}({\cal O}_Z)$ 
has a section of ${\cal O}_Z$-modules. 
%(The scheme $Z$ is $F$-split if and only if $h_F(Z)=1$.) 
Mehta has already remarked that any proper smooth $F$-split scheme 
over $\kap$ has a proper smooth lift over ${\cal W}_2$ (\cite{j}) 
as a corollary of Nori and Srinivas' beautiful deformation theory with 
absolute Frobenius endomorphisms  
in \cite{ns} and \cite{sr}. 
By using their theory, 
the second named author has proved that 
any proper smooth scheme over $\kap$ 
has a proper smooth lift over ${\cal W}_2$ if $h_F(Y)<\infty$ (\cite{y}). 
%the second named author's result above 
%is a generalization of Mehta's remark to the case of any finite $h_F(Y)$.  
Furthermore he has proved a fundamental equality $h_F(X)=h^d(X/\kap)$ 
by using Serre's exact sequence in \cite{smxco}, 
the calculation of the dimensions of the cohomologies of sheaves 
of closed differential forms of degree $1$ due to 
Katsura and Van der Geer (\cite{vgk}) and Serre's duality (\cite{y}).   
As a result, he has obtained (\ref{theo:y}). 
Recently Achinger has proved that, if $Z/\kap$ is a (proper smooth) scheme 
over $\kap$ with finite quasi-split height, then $Z/\kap$ has a (proper smooth) lift 
over ${\cal W}_2$ by a method in  the Appendix of \cite{az}. 

\par
This article is a continuation of \cite{y} in an expanded form. 
The results in this article are the log versions of \cite{y}, a part of \cite{yoh}, 
\cite{ns}, \cite{sr}, \cite{bwl} and more. 
\par  
The philosophy of log geometry of 
Fontaine-Illusie-Kato (\cite{klog1}, \cite{klog2}) tells us that 
one can give statements and prove them for certain non-smooth schemes 
by similar methods for smooth schemes 
if one can endow them with fine or fs(=fine and saturated) log structures and 
if one makes multiplicative calculations of local sections of log structures in addition to 
multiplicative and additive calculations of local sections 
of structure sheaves of schemes with 
the use of various cohomologies of various sheaves. 
Supported by this philosophy, we give the log versions of results 
in the articles in the previous paragraph. 
Though the proofs of a lot of results in this article are not 
psychologically  extremely seriously difficult (after giving nontrivial formulations), 
the results themselves 
are nontrivial generalizations of the results in the articles above. 
(Of course there are often technically hard points in the proofs.)
%The powerful theory of Fontaine-Illusie-Kato enables 
%us to obtain the statements and the proofs 
%in this paper. 
This is the typical merit of the log geometry of Fontaine-Illusie-Kato: 
it gives us appropriate languages. 
\par 
Next let us recall Kawamata-Namikawa's  result briefly. 
This gives a not a little influence to this article. 
\par 
%\par 
%Mehta has remarked that any proper smooth Frobenius splitting scheme 
%over $\kap$ has a proper smooth lift over ${\cal W}_2$ (\cite{j}). 
%We can regard the Theorem (\ref{theo:y}) as a generalization of Mehta's remark 
%to the case of finite heights under the assumptions of (a), (b) and (c) in (1.1).  
\par 
Let $\kap$ be a field of any characteristic.   
Let $s$ be an fs log scheme whose underlying scheme is 
${\rm Spec}(\kap)$ 
and whose log structure is associated to a morphism 
${\mab N}\owns 1\lom a\in \kap$ for some $a\in \kap$ 
(see \cite{klog1} and \cite{klog2} for fundamental terminologies of log schemes).  
If $a=0$, then $s$ is called the log point of $\kap$; 
if $a\not=0$, then $s=({\rm Spec}(\kap),\kap^*)$. 
For a log scheme $Z$, we denote 
by $\os{\circ}{Z}$ the underlying scheme of $Z$. 
For a relative log scheme $Z/s$, we denote the log de Rham complex of 
$Z/s$ by $\Om^{\bul}_{Z/s}$ and we set  
$H^q_{\rm dR}(Z/s):=H^q(Z,\Om^{\bul}_{Z/s})$ $(q\in {\mab N})$. 

When $\kap={\mab C}$, Kawamata and Namikawa 
have proved the following theorem in \cite {kwn}: 

\begin{theo}[{\bf \cite[(4.2)]{kwn}}]\label{theo:kn}
%Assume that $\kap={\mab C}$. 
Let $s$ be the log point of ${\mab C}$.  
Let $X$ be a proper SNCL$($=simple normal crossing log$)$ scheme
over $s$ of pure dimension $d$. Assume that $d\geq 3$. 
%Let $\Omega^d_{X/s}$ be the $d$-th wedge product of the sheaf of 
%log differential forms on $X/s$. 
Let $S$ be a small disk with canonical log structure. 
Let $\os{\circ}{X}{}^{(0)}$ be the disjoint union of 
the irreducible components of $\os{\circ}{X}$. 
Assume that the following three conditions hold$:$
\par 
$({\rm a})$ $H^{d-1}(X,{\cal O}_X)=0$, 
\par 
$({\rm b})$ $H^{d-2}(\os{\circ}{X}{}^{(0)},{\cal O}_{\os{\circ}{X}{}^{(0)}})=0$, 
\par 
$({\rm c})$ $\Om^d_{X/s}\simeq {\cal O}_X$. 
\parno 
Then there exists an analytically strict semistable family ${\cal X}$ over $S$ 
such that ${\cal X}\times_Ss=X^{\rm an}$, where $X^{\rm an}$ is the 
associated log analytic space to $X/s$ $($cf.~{\rm \cite{ktn}}$)$. 
\end{theo}

%\parno 
%Using this theorem, 
%they have produced analytically strict semistable degenerations of 
%Calabi-Yau 3-folds over ${\mab C}$ with 26 classes 
%of different Euler numbers starting from log Calabi-Yau varieties over $s$ 
%(see the next paragraph for the definition of 
%a log Calabi-Yau variety over a log  point). 
\par   
Let us go back to the case 
where $\kap$ is a perfect field of characteristic $p>0$. 
Let $s$ be an fs log scheme before (\ref{theo:kn}).    
%We say that $X$ is a {\it log Calabi-Yau variety} over $s$ if 
%$\os{\circ}{X}$ is a geometrically connected proper scheme 
%of pure dimension $d\geq 1$, if 
%$X$ is an SNCL scheme over $s$, if  
%$H^q(X,{\cal O}_X)=0$ $(0< q< d)$ and 
%$\Om^d_{X/s}\simeq {\cal O}_X$.  
%This notion is a generalization of 
%the log $K3$ surface defined in \cite{nlk3}.
%In this article we prove the following: 
%\begin{theo}\label{theo:lcv} 
%Let $X$ be a log Calabi-Yau variety 
%over $s$ of pure dimension $d$ 
%with finite $d$-th Artin-Mazur height. 
%Then the following hold$:$
%\par 
%$(1)$ The log Hodge-de Rham spectral sequence  
%\begin{align*} 
%E_1^{ij}=H^j(X,\Om^i_{X/s})\Lo H^{i+j}_{\rm dR}(X/s)
%\tag{1.3.1}\label{ali:oegdxs} 
%\end{align*}  
%degenerates at $E_1$  if $d\leq p$. 
%\par 
%$(2)$ Let ${\cal L}$ be an ample line bundle on 
%$\os{\circ}{X}$. Then $H^j(X,{\cal L})=0$ for $j>0$ and $H^j(X,{\cal L}^{-1})=0$ 
%for $j<d$.   
%\end{theo}  
%(Note that $\os{\circ}{X}$ is not necessarily integral. It is only reduced.) 
%We can apply the Theorem (\ref{theo:ny1}) and the Corollaries  
%(\ref{coro:np}), (\ref{coro:v}) and (\ref{coro:stnk}) 
%for a log Calabi-Yau variety over $s$ with $h(\os{\circ}{X}/\kap)< \infty$, 
%which are nontrivial.
%, then 
%$H^j(X,\Om^i_{X/s}\otimes{\cal L})=0$ for it and $H^j(X,{\cal L})=0$ for $j>0$. 
Let $X$ be a proper log smooth log scheme over $s$ of Cartier type. 
Let ${\cal I}_{X/s}$ be Tsuji's ideal sheaf of the log structure $M_X$ of $X$ 
defined in \cite{tsp} and denoted by $I_f$, where $f\col X\lo s$ is 
the structural morphism.  
%(In \S\ref{sec:lst} below we recall the definition of ${\cal I}_{X/s}$.) 
Here ${\cal I}_{X/s}$ stems from the ``horizontal'' log structure on $X$; 
in the text we shall recall the definition of ${\cal I}_{X/s}$.
We say that $X/s$ is {\it of vertical type} 
if ${\cal I}_{X/s}{\cal O}_X={\cal O}_X$.  
If $X/s$ is an SNCL scheme (\cite{nlk3}, \cite{nlw}), more generally, 
if $X/s$ is locally a product of SNCL schemes, 
then $X/s$ is of vertical type. 
One of the main results in this article is the following theorem:    

\begin{theo}\label{theo:ny1} 
Let $X$ be a proper log smooth log scheme over $s$ of Cartier type. 
%Let $S$ be an fs log scheme with exact closed immersion 
%$s\os{\sus}{\lo} S$. 
Let ${\cal W}_2(s)$ be a log scheme whose underlying scheme 
is ${\rm Spec}({\cal W}_2)$ and whose log structure 
is associated to a morphism 
${\mab N}\owns 1\lom (a,0)\in {\cal W}_2$. 
Then the following hold$:$
\par 
$(1)$ If $h_F(\os{\circ}{X})<\infty$, then 
there exists a proper log smooth log scheme
${\cal X}$ over ${\cal W}_2(s)$ such that ${\cal X}\times_{{\cal W}_2(s)}s=X$.
\par 
$(2)$ 
Furthermore, assume that $\os{\circ}{X}$ is of pure dimension $d$ and 
that $X/s$ is of vertical type and that 
the following three conditions hold$:$
\par 
$({\rm a})$ $H^{d-1}(X,{\cal O}_X)=0$ if $d\geq 2$, 
\par 
$({\rm b})$ $H^{d-2}(X,{\cal O}_X)=0$ if $d\geq 3$, 
\par 
$({\rm c})$ $\Om^d_{X/s}\simeq {\cal O}_X$, 
%\par 
%$({\rm d})$ $h^d(\os{\circ}{X}/\kap)< \infty$. 
\parno 
Then $h_F(\os{\circ}{X})=h^d(\os{\circ}{X}/\kap)$. 
\end{theo}

\parno 
By using K.~Kato's theorem in 
\cite{klog1}(=the log version of Deligne-Illusie's theorem) 
and (\ref{theo:ny1}) (1), we obtain the following: 
%the log version 
%Let $F_s\col s\lo s$ be the Frobenius endomorphism of $s$. 
%Let us recall K.~Kato's theorem in 
%\cite{klog1}(=the log version of Deligne-Illusie's theorem already stated), 
%though (3) and (4) in (\ref{theo:np}) below 
%have not been stated in [loc.~cit.]; 
%(3) follows from (2) in (\ref{theo:np}) 
%and Tsuji's log Serre duality in \cite{tsp}; 
%(4) follows from the log version of the proof in \cite{di}:    

\begin{theo}\label{theo:np}
Let $Y\lo s$ be a proper log smooth morphism of Cartier type of dimension $d$.   
Assume that $h_F(\os{\circ}{Y})<\infty$. 
%Set $Y':=Y\times_{s,F_s}s$ and 
%let $F\col Y\lo Y'$ be the relative Frobenius morphism of $Y/s$. 
%Assume that $Y'/s$ has a log smooth lift over ${\cal W}_2(s)$. 
%Then the following hold$:$ 
%\par 
%$(1)$ There exists an isomorphism 
%\begin{align*} 
%\bigoplus_{i<p}\Om^i_{Y'/s}[-i]
%\os{\sim}{\lo} \tau_{<p}F_*(\Om^{\bul}_{Y/s}). 
%\tag{1.5.1}\label{ali:ode1xs} 
%\end{align*} 
%in the derived category of bounded above complexes of ${\cal O}_{Y'}$-modules. 
%\par 
%$(2)$ 
Then the log Hodge de Rham spectral sequence   
\begin{align*} 
E_1^{ij}=H^j(Y,\Om^i_{Y/s})\Lo H^{i+j}_{\rm dR}(Y/s)
\tag{1.4.1}\label{ali:oemxs} 
\end{align*}  
degenerates at $E_1$ if $d< p$.
If $F_{Y*}({\cal O}_Y)$ is 
a locally free ${\cal O}_Y$-modules $($of finite rank$)$ and if $d\leq p$, 
then this spectral sequence degenerates at $E_1$. 
Here $F_Y\col Y\lo Y$ is the absolute Frobenius endomorphism of $Y$. 
\end{theo}

\parno
We also give another short proof of 
Kato's theorem by using our log deformation theory 
with absolute Frobenius endomorphisms explained soon later. 
This is the log version of a generalization of 
Srinivas' another short proof of Deligne-Illusie's theorem (\cite{sr}). 
\par 
\par 
Let $Y/s$ be a log smooth log scheme of Cartier type. 
%By abuse of notation, denote the absolute Frobenius endomorphism 
%${\cal W}_n(Y)\lo {\cal W}_n(Y)$ also by $F_Y$. 
%The first key new ingredient for the proof of (\ref{theo:ny1}) 
%is the log version of Serre's exact sequence in \cite{smxco}
%\begin{align*} 
%0\lo {\cal W}_n({\cal O}_Y)\os{F}{\lo} 
%F_{Y*}({\cal W}_n({\cal O}_Y))
%\lo B_n\Om^1_{Y/s}\lo 0,
%\tag{1.4.2}\label{ali:oofo}
%\end{align*}  
%where $F=F_Y^*$ and 
%the sheaf $B_n\Om^1_{Y/s}\subset F^n_{Y*}(\Om^1_{Y/s})$ 
%of ${\cal O}_Y$-modules will be recalled in \S\ref{sec:dc} below  
%($B_1\Om^1_{Y/s}:=F_{Y*}(B\Om^1_{Y/s})
%\simeq F_{Y*}({\cal O}_Y)/{\cal O}_Y$, 
%where $B\Om^1_{Y/s}:={\rm Im}(d\col {\cal O}_Y\lo \Om^1_{Y/s})$). 
%Though Serre has considered (\ref{ali:oofo}) in the trivial logarithmic case 
%with the assumption of the normality of $\os{\circ}{Y}$ 
%only as an exact sequence of  {\it abelian sheaves} on $\os{\circ}{Y}$, 
%we consider (\ref{ali:oofo}) as an exact sequence of 
%sheaves of ${\cal W}_n({\cal O}_Y)$-{\it modules}. 
%Our point of view plays a 
%definitive role in this article. 
\par 
One of the new key ingredient for the proof of (\ref{theo:ny1}) is  
our log deformation theory with absolute Frobenius endomorphisms. 
This is the log version of Nori and Srinivas' deformation theory with 
absolute Frobenius endomorphisms in \cite{ns} and \cite{sr}.  
In this theory, the sheaf 
\begin{align*}
B_1\Om^1_{Y/s}:=
F_{Y*}(B\Om^1_{Y/s}):=F_{Y*}({\rm Im}(d\col {\cal O}_Y\lo \Om^1_{Y/s}))
(\simeq F_{Y*}({\cal O}_Y)/{\cal O}_Y))
\end{align*} 
plays an important role as follows 
(In the trivial log case, $B_1\Om^1_{Y/s}$ 
in this article is equal to $B\Om^1_{Y/s}$ in [loc.~cit.].): 

\begin{theo}\label{theo:td}
Let $F_{{\cal W}_2(s)}\col {\cal W}_2(s)\lo {\cal W}_2(s)$ 
be the Frobenius~endomorphism of ${\cal W}_2(s)$. 
Let ${\rm Lift}_{(Y,F_Y)/({\cal W}_2(s),F_{{\cal W}_2(s)})}$ 
be the following sheaf 
\begin{align*} 
{\rm Lift}_{(Y,F_Y)/({\cal W}_2(s),F_{{\cal W}_2(s)})}&(U):=
\{{\rm isomorphism~classes~of}~(\wt{U},\wt{F})\,\vert \,
\wt{U}~{\rm is~a~log~smooth~lift}\\
&~{\rm of}~U~{\rm over}~{\cal W}_2(s)~{\rm and}~
\wt{F}\col \wt{U}\lo \wt{U}~{\rm~is~a~lift~of~}F_U~
{\rm over}~F_{{\cal W}_2(s)}\}
\end{align*}
for each log open subscheme $U$ of $Y$, 
where $F_U$ is the absolute Frobenius endomorphism of $U$. 
Then ${\rm Lift}_{(Y,F_Y)/({\cal W}_2(s),F_{{\cal W}_2(s)})}$ on 
$\os{\circ}{Y}$ is a torsor under 
${\cal H}{\it om}_{{\cal O}_{Y}}(\Om^1_{Y/s},B_1\Om^1_{Y/s})$. 
In particular, the obstruction class of a log smooth lift of $(Y,F_Y)/s$ 
over ${\cal W}_2(s)$ is a canonical element of 
${\rm Ext}^1_Y(\Om^1_{Y/s},B_1\Om^1_{Y/s})$ if $\os{\circ}{Y}$ is separated.  
This obstruction class is the extension class of 
the following exact sequence of 
${\cal O}_Y$-modules$:$
\begin{align*} 
0\lo B_1\Om^1_{Y/s}\lo Z_1\Om^1_{Y/s}\os{C}{\lo} \Om^1_{Y/s}\lo 0,  
\end{align*} 
where 
%$Z\Om^1_{Y/s}:={\rm Ker}(d\col {\cal O}_Y\lo \Om^1_{Y/s})$, 
$Z_1\Om^1_{Y/s}:=F_{Y*}({\rm Ker}(d\col \Om^1_{Y/s}\lo \Om^2_{Y/s}))$ 
and $C$ is the log Cartier operator$\,:$
$C\col Z_1\Om^1_{Y/s}\os{\rm proj.}{\lo} Z_1\Om^1_{Y/s}/B_1\Om^1_{Y/s}
\os{C^{-1},\sim}{\longleftarrow} \Om^1_{Y/s}$. 
Here $C^{-1}$ is the log Cartier isomorphism defined in {\rm \cite{klog1}}.  
\end{theo}
\parno 
This is a special case of the main result in \S\ref{sec:latv} below. 
Note that, because the log structure of ${\cal W}_2(s)$ has a chart 
${\mab N}\lo {\cal W}_2$, the structural morphism $\wt{U}\lo {\cal W}_2(s)$ 
is automatically integral (\cite{klog1}). 
In the case where a base log scheme is more general, 
we have to consider log smooth {\it integral} lifts instead of log smooth lifts; 
the integrality is an essential condition in log deformation theory: 
deformation theory for log smooth schemes in \cite{klog1} (and \cite{kaf}) 
has a serious defect to be corrected in general. 
\par  
To construct the log deformation theory with 
absolute Frobenius endomorphisms itself is our aim in this article.  
To give the correct proof of (\ref{theo:td}) is very involved. 
Indeed, even in the trivial logarithmic case in \cite{ns}, 
we need a new additional quite extraordinary argument. 
More generally, we construct the log deformation theory with 
two kinds of {\it relative} Frobenius morphisms 
instead of {\it absolute} Frobenius endomorphisms in [loc.~cit.] 
because relative Frobenius morphisms 
go well with (log) inverse Cartier isomorphisms 
when we consider log deformation theory 
with Frobenius morphisms over a more general fine log base scheme 
of characteristic $p>0$. 
Our log deformation theory with Frobenius morphisms
also has an application for 
the canonical lift of 
a log ordinary projective log smooth log scheme over $s$ 
with trivial log cotangent bundle 
over the canonical lift ${\cal W}(s)$ of $s$ over ${\rm Spec}({\cal W})$  (\cite{nclw}). 
(This is the log version of theory of a canonical lift in \cite{ns}.)
\par 
Other necessary new ingredient for the proof of (\ref{theo:ny1}) 
is the caluculation of  
dimension of $H^q(X,B_n\Om^1_{X/s})$ $(d-2\leq q\leq d)$ 
by following the method of Katsura and Van der Geer in \cite{vgk}.

\par 
As a corollary of (\ref{theo:np}), 
we also prove the log version of 
Raynaud's vanishing theorem(=an analogue in characteristic $p$ of 
Kodaira-Akizuki-Nakano's vanishing theorem in characteristic $0$) as in 
\cite{di}: 

\begin{theo}[{\bf Log Kodaira-Akizuki-Nakano-Raynaud Vanishing theorem}]\label{theo:ray} 
Let the notations and the assumption be as in {\rm (\ref{theo:np})}. 
Furthermore, assume that $Y$ is fs, that the structural morphism 
$\os{\circ}{Y}\lo \os{\circ}{s}$ of schemes 
is projective and that 
%log smooth morphism of Cartier type of fs log schemes 
%which has a log smooth lift over ${\cal W}_2(s)$. Assume that 
$\os{\circ}{Y}$ is of pure dimension $d$.  
Let ${\cal L}$ be an ample invertible ${\cal O}_Y$-module. 
Then 
$H^j(Y,{\cal I}_{Y/s}\Om^i_{Y/s}\otimes_{{\cal O}_Y}{\cal L})=0$ for 
$i+j>\max \{d,2d-p\}$.  
\end{theo}

\parno 
In the most important case $i=d$ in (\ref{theo:ray}), 
we prove a stronger theorem than this theorem 
(this stronger theorem is also one of the main results in this article): 
%with $h_F(\os{\circ}{X})<\infty$ and ${\cal L}$:  

\begin{theo}[{\bf Log Kodaira Vanishing theorem I}]\label{theo:stk} 
%Let $Y\lo s$ be a projective log smooth morphism 
%of Cartier type of fs log schemes. 
%Assume that $\os{\circ}{Y}$ is of pure dimension $d$.  
Let the notations and the assumptions be as in {\rm (\ref{theo:ray})}. 
%Let ${\cal L}$ be an ample invertible sheaf on $\os{\circ}{Y}$. 
%Assume also that $h_F(\os{\circ}{Y})<\infty$. 
Then $H^j(Y,{\cal I}_{Y/s}\Om^d_{Y/s}\otimes_{{\cal O}_Y}{\cal L})=0$ for $j>0$. 
\end{theo}

\parno   
This theorem is the log version of a nontrivial generalization of 
Mehta and Ramanathan's vanishing theorem in \cite{mr}; 
the proof of this theorem is more nontrivial than that of their theorem. 
This theorem is important because 
we can obtain the new class of log schemes such that 
Kodaira vanishing theorem holds in characteristic $p>0$. 
%We hope that this theorem will 
%be useful for the base point freeness 
%in characteristic $p>0$ if $Y$ is of vertical type 
%as in the characteristic $0$ case in \cite{kwv}.    
The theorem (\ref{theo:stk}) has  
an interesting application for congruences of the cardinalities of 
rational points of log Fano varieties  
with finite quasi-$F$-split heights over the log point of a finite field 
(\cite{nlfano}). 
This is a generalization of Esnault's theorem  
in \cite{es} (under the (mild) assumption ``the finiteness of the quasi-$F$-split height'').
We hope that (\ref{theo:stk}) will have more important applications for 
algebraic geometry in characteristic $p$. 
\par 
%As a corollary of (\ref{theo:stk}), we obtain the following:   
%\begin{coro}[{\bf Log Kodaira Vanishing theorem II}]\label{coro:stnk} 
%Let the notations and the assumptions be as in {\rm (\ref{theo:ray})}. 
%Then $H^j(Y,{\cal L}^{-1})=0$ for $j<d$.  
%\end{coro}
%\parno  
\par 
As a corollary of the vanishing theorem (\ref{theo:ray}), 
we prove an analogous vanishing theorem in characteristic $0$.

\par
Lastly in this introduction, 
we formulate the log weak Lefschetz conjecture for 
log crystalline cohomologies and 
we give an affirmative result for this conjecture. 
\par 
%Let us go back to the case ${\rm ch}(\kap)=p>0$.  
Let $Y$ be a projective SNCL scheme over the log point $s$. 
Let $E$ be a horizontal smooth divisor on $Y$ 
which will be defined in the text;  
roughly speaking, $E$ is locally defined by a local coordinate which 
has ``no relation with a nontrivial local section of $M_Y/{\cal O}_Y^*$''. 
Let $q$ be a nonnegative integer. 
For a proper log smooth scheme $Y/s$, 
let $H^q_{\rm crys}(Y/{\cal W}(s))$ 
be the log crystalline cohomology of $Y/{\cal W}(s)$ (\cite{klog1}). 
By the works in \cite{msemi} and \cite{ndw} (cf.~\cite{nlw}), 
$H^q_{\rm crys}(Y/{\cal W}(s))$ and 
$H^q_{\rm crys}(E/{\cal W}(s))$ have the weight filtrations $P$'s.
Set $K_0:={\rm Frac}({\cal W})$. 
For a module $M$ over ${\cal W}$, 
set $M_{K_0}:=M\otimes_{\cal W}K_0$. 
Let $\iota \col E\os{\sus}{\lo} Y$ be the closed immersion. 
By a general theorem in \cite{nlw}, the pull-back  of $\iota$ 
\begin{align*} 
\iota^*_{\rm crys}  \col H^q_{\rm crys}(Y/{\cal W}(s))_{K_0}\lo 
H^q_{\rm crys}(E/{\cal W}(s))_{K_0} \quad (q\in {\mab Z})
\tag{1.7.1}\label{ali:crixw} 
\end{align*} 
is strictly compatible with $P$'s. 
In this article we conjecture the following: 

\begin{conj}[{\bf Log weak Lefschetz conjecture for log isocrystalline cohomologies}]\label{conj:lwl}
Assume that ${\cal O}_Y(E)$ is ample. Then the morphism 
{\rm (\ref{ali:crixw})} is a filtered isomorphism with respect to $P$'s 
if $q\leq d-2$ and strictly injective for $q=d-1$. 
\end{conj} 

In the text we give affirmative results for this conjecture. 
For example, we prove the following: 

\begin{theo}\label{theo:qwl}   
Assume that $Y$ and $E$ have log smooth lifts over ${\cal W}_2(s)$. 
Assume also that $\dim \os{\circ}{Y}\leq p$. 
Then the following pull-back 
\begin{align*} 
\iota^*_{\rm crys}  \col H^q_{\rm crys}(Y/{\cal W}(s))\lo 
H^q_{\rm crys}(E/{\cal W}(s)) \quad (q\in {\mab Z})
\tag{1.9.1}\label{ali:crxlw} 
\end{align*} 
is an isomorphism if $q< d-1$ and 
injective for $q= d-1$ with torsion free cokernel. 
In particular, {\rm (\ref{conj:lwl})} is true under the assumptions above.  
\end{theo} 
We prove this theorem by 
following but nontrivially correcting the method of Berthelot in \cite{bwl}. 
In the future we would like to prove that this conjecture is true in general. 
%in \cite{lkan}. 
%The results for log weak Lefschetz conjecture in this paper 
%are not covered by this corollary. 
Note that because in \cite{nlpi} and \cite{nlw} 
we have proved that 
the log hard Lefschetz conjecture is true in the strict semistable 
cases in mixed characteristics and equal characteristic $p>0$,  
we can prove that the log weak Lefschetz conjecture is true 
in these important cases as a corollary. 
\par 
In \cite{nlpi} we have proved that 
the log hard Lefschetz conjecture is true 
in characteristic $0$ by using M.~Saito's result (\cite{sm}). 
%(His result uses theory of mixed Hodge structures.) 
As a corollary, we can prove that 
the log weak Lefschetz conjecture in characteristic $0$ is true. 
In this article we prove this theorem by an algebraic method 
as Deligne and Illusie have proved 
the $E_1$-degeneration of the Hodge-de Rham spectral sequence 
of a proper smooth scheme   
in characteristic 0 in \cite{di} by an algebraic method.

\par 
The contents of this article are as follows.  
\par 
Let $Z$ be a proper scheme over $\kap$.  
Let $q$ be a nonnegative integer. 
Assume that $H^q(Z,{\cal O}_Z)\simeq \kap$, 
that $H^{q+1}(Z,{\cal O}_Z)=0$ and 
that the $q$-th Artin-Mazur functor 
$\Phi^q_{Z/\kap}:=\Phi^q_{Z/\kap}({\mab G}_m)$ is pro-representable. 
%that $H^{d-1}(Z,{\cal O}_Z)=0$ if $d\geq 2$.
Assume also 
that the Bockstein operator 
\begin{align*} 
\bet \col H^{q-1}(Z,{\cal O}_Z)\lo 
H^q(Z,{\cal W}_{n-1}({\cal O}_Z))
\end{align*} 
arising from the following exact sequence 
\begin{align*} 
0\lo {\cal W}_{n-1}({\cal O}_Z)\os{V}{\lo} {\cal W}_n({\cal O}_Z)
{\lo} {\cal O}_Z\lo 0
\end{align*} 
is zero for any $n\in {\mab Z}_{\geq 2}$. 
In \S\ref{sec:amf} we prove that the $q$-th Artin-Mazur height 
$h^q(Z/\kap)$ of $Z/\kap$ is equal to 
the minimum of positive integers $n$'s of the non-vanishing of 
the Frobenius endomorphism $F\col H^q(Z,{\cal W}_n({\cal O}_Z))
\lo H^q(Z,{\cal W}_n({\cal O}_Z))$ by 
imitating the proof in \cite{vgk} completely. 
(However we have needed a work to give this generalized statement.) 
Recently it has turned out that this characterization of $h^q(Z/\kap)$ also 
has two applications for the congruences of the cardinalities of 
rational points of (log) Calabi-Yau varieties 
over the log point of a finite field
(\cite{nlfano}) and for the fundamental inequality between Aritn-Mazur heights  
and a quasi-$F$-split height (\cite{nlfc}). 
\par 
In \S\ref{sec:dc} we prove that there exists 
the log version of Serre's exact sequence in \cite{smxco}
in an elementary but elegant way and 
calculate ${\rm dim}_{\kap}H^q(X,B_n\Om^1_{X/s})$ 
$(d-2\leq q \leq d)$ for $X/s$ in (\ref{theo:ny1}). 
\par
In \S\ref{sec:latv}, following the methods in \cite{ns} and \cite{sr} but 
modifying and generalizing them,  
we construct log deformation theory with relative Frobenius morphisms. 
Our new theory is a geometric key part for the proof of (\ref{theo:ny1}).   
This is the most complicated part in this article. 
In addition, we give an additional result for 
the deformation theory for log smooth schemes in \cite{klog1} (and \cite{kaf}), 
which is an important correction of the theory in [loc.~cit.], 
and we establish a relationship between these two deformation theories. 
%\par 
%In \S\ref{sec:lst}, for later sections, 
%we recall Tsuji's two results on log schemes: 
%Tsuji's log Serre duality and properties of saturated morphisms. 
\par 
In \S\ref{sec:app}, as applications of our deformation theory, 
we give another short proof of Kato's theorem (cf.~(\ref{theo:np})). 
We also prove (\ref{theo:ray}) by using Tsuji's log Serre duality in \cite{tsp}. 
%As a corollary of (\ref{theo:ray}), 
%we prove analogous vanishing theorem in characteristic $0$. 
As in \cite{di} we prove the log versions of the weak Lefschetz theorems for 
log de Rham cohomologies in characteristics $p>0$ and $0$. 
The proof of the weak Lefschetz theorem in the case characteristics $p>0$ 
includes an immediate correction of an elementary error in \cite{di}. 
Using this theorem in characteristics $p>0$, we prove (\ref{theo:qwl}). 
We also prove the log version of Berthelot's weak Lefschetz theorem 
and we fill a gap in the proof in \cite{bwl}. 
%By virtue of our generalization, we have an application for 
%the log deformation of the proper (projective) scheme in equal characteristic $p$ and 
%we can prove the $l$-adic and $p$-adic monodromy-weight conjecture 
%for a projective SNCL log scheme over the log point $s$ 
%satisfying the condition (a), (b), (c) and (d) in 
%(\ref{theo:ny1}) (see \cite{nlw}). 
\par 
In \S\ref{sec:qfs} we give the notion of quasi-$F$-split schemes, which is 
the relative version of the notion of 
quasi-$F$-split varieties in \cite{y}. 
In this section we prove two fundamental theorems 
for quasi-$F$-split log schemes as in [loc.~cit.]: 
a lifting theorem and two vanishing theorems for them. 
The lifting theorem and one of the vanishing theorems 
are the relative and log versions of theorems in \cite{y} and \cite{yoh}. 
This vanishing theorem is a generalization of one of 
Mehta and Ramanathan's vanishing theorems in \cite{mr}.
We also prove another vanishing theorem(=(\ref{theo:stk})),  
which is a generalization of their another vanishing theorem in [loc.~cit.].
%This vanishing theorem is much stronger than 
%the vanishing theorem in \S\ref{sec:app} in an important case as 
%already mentioned. 
\par 
%In \S\ref{sec:mix} we give a relationship between 
%infinitesimal semistable families over ${\cal W}_2$  
%and infinitesimal semistable families over $\kap[t]/(t^2)$. 
%This is a partial answer for the log version of 
%Schr\"{o}er's question stated in the introduction of \cite{sc}: 
%``A natural question to ask is what is the relation 
%between obstructions for deformations over Artin $\kap$-algebras 
%and deformations over Artin ${\cal W}$-algebras? 
%Unfortunately, there are very few examples that could shed some light on the situation.''
%\par  
In \S\ref{sec:lif} we prove (\ref{theo:ny1}) by following the method in \cite{y} 
and by using results in \S\ref{sec:dc}$\sim$\S\ref{sec:qfs}.  
\par 
In \S\ref{sec:wlt} we give a short proof of 
the weak Lefschetz theorem for crystalline cohomologies of 
proper smooth schemes over $\kap$ due to 
Berthelot-Katz-Messing (\cite{kme}) by 
using theory of rigid cohomologies of Berthelot 
(\cite{brc}, \cite{bfi}, \cite{bd}). 

\bigskip
\parno
{\bf Acknowledgment.}
We greatly appreciate the referee 
for reading every part of this article unbelievably carefully  
and for pointing out a lot of notational errors 
in the previous version of this article with great patience. 
By his/her very help we can improve the proof of (\ref{theo:nex}) and 
we can give the precisely correct proof of (\ref{theo:ts}) (3). 
We also thank P.~Achinger very much 
for explaining the log version of the method in \cite{az} to us.

\bigskip
\par\noindent
{\bf Notations.} 
(1) For a commutative ring $A$ with unit element and 
two $A$-modules $M$ ($M$ has two distinct $A$-module structures)
and for $f\in {\rm Hom}_A(M,M)$, ${}_fM$ (resp.~$M/f$) denotes ${\rm Ker}(f\col M\lo M)$ 
(resp.~${\rm Coker}(f\col M\lo M)$). 
We use the same notation for an endomorphism of 
two ${\cal A}$-modules on a topological space, 
where ${\cal A}$ is a sheaf of commutative rings with unit elements 
on the topological space.

\par 
(2) For a log scheme $Z$ in the sense of 
Fontaine-Illusie-Kato (\cite{klog1}, \cite{klog2}), 
we denote by $\os{\circ}{Z}$ (resp.~$M_Z:=(M_Z,\al_Z)$)
the underlying scheme (resp.~the log structure) of $Z$. 
In this article we consider the log structure 
on the Zariski site on $\os{\circ}{Z}$.  
\par 
(3) For a morphism $f\col Z\lo T$ of log schemes, 
we denote by $\os{\circ}{f}\col \os{\circ}{Z}\lo \os{\circ}{T}$ 
the underlying morphism of schemes of $f$. 
\par
(4) For a morphism $Z\lo T$ of log schemes, 
we denote by $\Om^{\bul}_{Z/T}$ the log de Rham complex 
of $Z/T$ which was denoted by $\om^{\bul}_{Z/T}$ 
in \cite{klog1}. 
%\par 
%(5) For a morphism $Z\lo s$ of fine log smooth schemes, 
%we denote by ${\cal W}_{\star}\Om^{\bul}_{Z/s}$ 
%$(\star$ is a positive integer or nothing) the log de Rham-Witt complex 
%of $Z/s$ which was denoted by $W_{\star}\om^{\bul}_{Z}$ in \cite{hk}. 
\bigskip
\par\noindent
{\bf Convention.} 
We omit the second ``log'' in the terminology a ``log smooth (integral) log scheme''. 

\section{The heights of Artin-Mazur formal groups 
of certain schemes}\label{sec:amf}
Let $\kap$ be a perfect field of characteristic $p>0$. 
Let ${\cal W}$ (resp.~${\cal W}_n$) be the Witt ring of $\kap$ 
(resp.~the Witt ring of $\kap$ of length $n>0$).  
Let $Y$ be a proper scheme over $\kap$. 
%of pure dimension $d\geq 1$. 
Let $q$ be a nonnegative integer. 
Assume that $H^q(Y,{\cal O}_Y)\simeq \kap$, $H^{q+1}(Y,{\cal O}_Y)=0$  
and 
%that $H^{d-1}(Z,{\cal O}_Z)=0$ if $d\geq 2$.
that the Bockstein operator 
\begin{align*} 
\bet \col H^{q-1}(Y,{\cal O}_Y)\lo 
H^q(Y,{\cal W}_{n-1}({\cal O}_Y))
\end{align*} 
arising from the following exact sequence 
\begin{align*} 
0\lo {\cal W}_{n-1}({\cal O}_Y)\os{V}{\lo} {\cal W}_n({\cal O}_Y)
{\lo} {\cal O}_Y\lo 0
\end{align*} 
is zero for any $n\in {\mab Z}_{\geq 2}$. 
%Assume also that $H^{d-2}(Y,{\cal O}_Y)=0$ if $d\geq 3$. 
In this section we characterize the height of 
the $q$-th Artin-Mazur formal group of $Y/\kap$ (if it is pro-representable) 
by using the operator 
\begin{align*} 
F\col H^q(Y,{\cal W}_n({\cal O}_Y))\lo H^q(Y,{\cal W}_n({\cal O}_Y)) 
\quad (n\in {\mab Z}_{\geq 1}).
\end{align*}  
This is a generalization of a result of Katsura and Van der Geer (\cite{vgk}). 
Though they have proved this characterization for 
a Calabi-Yau variety over $\kap$, 
it is not necessary to assume this strong condition 
nor to assume even that $Y$ is smooth over $\kap$. 
Though the proof of our generalization is essentially the same as that of 
their result,  
we 
%state a generalization of their result and 
reprove our generalization 
because 
%we think that the conditions for a Calabi-Yau variety is too strong and 
we would like to clarify  
how the assumptions above 
are necessary for the characterization. 
\par 
The following is easy to prove.

\begin{prop}\label{prop:ehe}
Let $g\col Z\lo S_0$ be a proper morphism of schemes of characteristic $p>0$.  
%of pure relative dimension $d\geq 1$. 
Let ${\cal W}_n({\cal O}_Z)$ $(n\in {\mab Z}_{\geq 1})$ be the sheaf of Witt rings of 
${\cal O}_Z$ of length $n$. 
Let $V\col {\cal W}_n({\cal O}_Z)\lo {\cal W}_{n+1}({\cal O}_Z)$ 
be the Verschiebung and 
let $F\col  {\cal W}_n({\cal O}_Z)\lo {\cal W}_n({\cal O}_Z)$ 
be the Frobenius operator. 
Let $R\col {\cal W}_n({\cal O}_Z)\lo {\cal W}_{n-1}({\cal O}_Z)$ 
be the projection. 
Let $q$ be a nonnegative integer. 
Assume that 
the Bockstein operator 
\begin{align*} 
\bet \col R^{q-1}g_*({\cal O}_Z)\lo 
R^qg_*({\cal W}_{n-1}({\cal O}_Z))
\tag{2.1.1}\label{ali:bos}
\end{align*} 
arising from the following exact sequence 
\begin{align*} 
0\lo {\cal W}_{n-1}({\cal O}_Z)\os{V}{\lo} {\cal W}_n({\cal O}_Z)
\os{R^{n-1}}{\lo} {\cal O}_Z\lo 0
\tag{2.1.2}\label{ali:fes}
\end{align*} 
of abelian sheaves on $Z$ is zero for any $n\in {\mab Z}_{\geq 2}$. 
%$R^{d-1}g_*({\cal O}_Z)=0$ if $d\geq 2$.   
Assume that $R^{q+1}g_*({\cal O}_Z)=0$. 
Then the following hold$:$
\par 
$(1)$ 
%$R^{d-1}g_*({\cal W}_n({\cal O}_Z))=0$ if $d\geq 2$. 
%\par 
%$(2)$ 
%Assume that $S_0$ is reduced. 
The following sequence 
\begin{align*} 
0\lo R^qg_*({\cal W}_{n-1}({\cal O}_Z))\os{V}{\lo} 
R^qg_*({\cal W}_n({\cal O}_Z))\os{R^{n-1}}{\lo} 
R^qg_*({\cal O}_Z)\lo 0
\tag{2.1.3}\label{ali:nov}
\end{align*} 
of abelian sheaves on $S_0$ is exact. 
Consequently, if the projective system 
$\{R^qg_*({\cal W}_n({\cal O}_Z))\}_{n=1}^{\infty}$ 
satisfies the Mittag-Leffler condition, 
then the following sequence 
\begin{align*} 
0\lo R^qg_*({\cal W}({\cal O}_Z))\os{V}{\lo} 
R^qg_*({\cal W}({\cal O}_Z)){\lo} 
R^qg_*({\cal O}_Z)\lo 0
\tag{2.1.4}\label{ali:fyes}
\end{align*} 
of abelian sheaves on $S_0$ 
is exact. 
%if $S_0$ is the spectrum of a perfect field of characteristic $p>0$. 
\par 
$(2)$ 
Assume that $Z$ is reduced and that $S_0$ is perfect.  
Assume also that $g_*({\cal O}_Z)={\cal O}_{S_0}^{\oplus c}$ 
for some positive integer $c$. 
Then 
$g_*({\cal W}_n({\cal O}_Z)/F)$ is a 
subsheaf of 
$R^1g_*({\cal W}_n({\cal O}_Z))$ of 
${\cal W}_n({\cal O}_{S_0})$-modules. 
\par 
$(3)$ Let the notations be as in {\rm (2)}. 
%Assume that the morphism $g\col Z\lo S_0$ 
%is of pure relative dimension $d\geq 1$. 
Assume that $R^{q-1}g_*({\cal O}_Z)=0$ if $q\geq 2$ 
and that $R^{q-2}g_*({\cal O}_Z)=0$ if $q\geq 3$. 
If $q=2$, assume also that $g_*({\cal O}_Z)={\cal O}_{S_0}^{\oplus c}$ 
for some positive integer $c$. 
Then $R^{q-2}g_*({\cal W}_n({\cal O}_Z)/F)=0$. 
\par 
$(4)$ Let the assumptions be as in {\rm (2)} and {\rm (3)}. 
If $R^ig_*({\cal O}_Z)=0$ $(0<i<q)$, 
then $R^ig_*({\cal W}_n({\cal O}_Z)/F)=0$ 
for $0\leq i\leq q-2$. 
\end{prop}
\begin{proof} 
(1): 
Taking the long exact sequence of the exact sequence (\ref{ali:fes}), 
we have the following exact sequence 
\begin{align*} 
\cdots &\lo R^qg_*({\cal W}_{n-1}({\cal O}_Z))
\os{V}{\lo} R^qg_*({\cal W}_n({\cal O}_Z))\os{R^{n-1}}{\lo} 
R^qg_*({\cal O}_Z) \tag{2.1.5}\label{ali:fels}\\
&\lo R^{q+1}g_*({\cal W}_{n-1}({\cal O}_Z))\lo  \cdots.  
\end{align*} 
Hence $R^{q+1}g_*({\cal W}_n({\cal O}_Z))=0$.  
By the assumption, the morphism 
$V\col R^qg_*({\cal W}_{n-1}({\cal O}_Z))
\lo R^qg_*({\cal W}_n({\cal O}_Z))$ is injective. 
Hence we obtain the exact sequence (\ref{ali:nov}). 
Taking the projective limit of (\ref{ali:nov}), we obtain 
the exact sequence (\ref{ali:fyes}).
\par 
(2): Because $Z$ is reduced, the following sequence 
\begin{align*} 
0\lo {\cal W}_n({\cal O}_Z)\os{F}{\lo} {\cal W}_n({\cal O}_Z)\lo 
{\cal W}_n({\cal O}_Z)/F\lo 0
\end{align*} 
is exact. Taking the long exact sequence of this exact sequence, 
we have the following exact sequence 
\begin{align*} 
\cdots & \lo R^qg_*({\cal W}_n({\cal O}_Z))\os{F}{\lo} 
R^qg_*({\cal W}_n({\cal O}_Z))\lo 
R^qg_*({\cal W}_n({\cal O}_Z)/F)
\tag{2.1.6}\label{ali:fqels}\\
&\lo R^{q+1}g_*({\cal W}_n({\cal O}_Z))\lo \cdots.  
\end{align*} 
We claim that the following natural morphism 
\begin{align*} 
{\cal W}_n({\cal O}_{S_0})^{\oplus c}
={\cal W}_n({\cal O}_{S_0}^{\oplus c})\lo 
{\cal W}_n(g_*({\cal O}_Z))=
g_*({\cal W}_n({\cal O}_Z))
\end{align*} 
is an isomorphism. 
Indeed, assume that 
$g_*({\cal W}_{n-1}({\cal O}_Z))={\cal W}_{n-1}({\cal O}_{S_0})^{\oplus c}$. 
Then, by the following commutative diagram 
\begin{equation*} 
\begin{CD} 
0@>>> {\cal W}_{n-1}({\cal O}_{S_0})^{\oplus c} @>{V^{\oplus c}}>> 
{\cal W}_n({\cal O}_{S_0})^{\oplus c}@>{R^{n-1}}>> {\cal O}_{S_0}^{\oplus c}@>>> 0 \\ 
@. @| @VVV @| @. \\
0@>>> g_*({\cal W}_{n-1}({\cal O}_Z))@>{g_*(V)}>> 
g_*({\cal W}_n({\cal O}_Z))@>{g_*(R^{n-1})}>> g_*({\cal O}_Z)@>>> 0
\end{CD} 
\end{equation*}  
of exact sequences, we see that 
${\cal W}_n({\cal O}_{S_0})^{\oplus c}=g_*({\cal W}_n({\cal O}_Z))$. 
Because $F\col {\cal W}_n({\cal O}_{S_0})
\lo {\cal W}_n({\cal O}_{S_0})$ is bijective by the assumption, 
the morphism 
$g_*({\cal W}_n({\cal O}_Z)/F) \lo R^1g_*({\cal W}_n({\cal O}_Z))$ is injective. 
\par 
(3): By (\ref{ali:fels}) we easily see that 
$R^{q-1}g_*({\cal W}_n({\cal O}_Z))=0$ $(n\in {\mab Z}_{\geq 1})$ 
if $q\geq 2$. 
Hence we have the following exact sequence 
\begin{align*} 
R^{q-2}g_*({\cal W}_n({\cal O}_Z))\os{F}{\lo} 
R^{q-2}g_*({\cal W}_n({\cal O}_Z))\lo 
R^{q-2}g_*({\cal W}_n({\cal O}_Z)/F)\lo 0 \quad (q\geq 2). 
\end{align*} 
\par 
First assume that $q\geq 3$. 
Then $R^{q-2}g_*({\cal W}_n({\cal O}_Z))=0$.  
Hence $R^{q-2}g_*({\cal W}_n({\cal O}_Z)/F)=0$. 
\par 
Assume that $q=2$.  Then $R^1g_*({\cal W}_n({\cal O}_Z))=0$. 
Hence $g_*({\cal W}_n({\cal O}_Z)/F)=0$ by (2).  
\par 
If $q=1$, then $R^{q-2}g_*({\cal W}_n({\cal O}_Z)/F)$ 
obviously vanishes. 
\par 
(4): Because $R^ig_*({\cal O}_Z)=0$ $(0< i<q)$, 
$R^ig_*({\cal W}_n({\cal O}_Z))=0$ $(0< i<q)$ by (\ref{ali:fels}). 
Hence  $R^ig_*({\cal W}_n({\cal O}_Z)/F)=0$ 
$(0< i<q-1)$ by (\ref{ali:fqels}). By (2),  
$g_*({\cal W}_n({\cal O}_Z)/F)=0$. 
%By the argument in (1) and (\ref{ali:fqels}), 
%$R^qg_*({\cal W}_n({\cal O}_Z)/F)=0$ $(q>d)$.
%\par 
%(5): 
\end{proof}

\begin{coro}\label{coro:esa} 
Let the assumptions be as in {\rm (\ref{prop:ehe}) (1)}.  
Furthermore, assume 
that $R^qg_*({\cal O}_Z)$ is equal to a line bundle ${\cal L}$ on $S_0$. 
Then $R^qg_*({\cal W}({\cal O}_Z))/V={\cal L}$.
\end{coro}
\begin{proof} 
Obvious. 
\end{proof}

\par 
%Let $g\col Z\lo S$ be a morphism of schemes.
%Let $E$ be the sheaf of abelian groups 
%on the big \'{e}tale site of $Z$. 
%Let $({\rm et}/Z)$ be the small \'{e}tale site of $Z$. 
%Let $\wt{E}_S\col ({\rm et}/Z) \lo ({\rm Ab})$ 
%be a functor defined by the following for $Z'\in ({\rm et}/Z)$: 
%\begin{align*} 
%\wt{E}_S(Z'):={\rm Ker}(E(Z')\lo E(Z'\times_SS_{\rm red})). 
%\end{align*}
%Set 
%$$\Phi^q_{Z/S}(E):=R^qg_*(\wt{E}_S) \quad (q\in {\mab N}).$$ 
%This is a sheaf of abelian groups on the small \'{e}tale 
%site $({\rm et}/S)$ on $S$. 
%We call $\Phi^q_{Z/S}(E)$ the {\it Artin-Mazur formal group sheaf} 
%of $E$ on $S$.  This extends naturally to a sheaf of abelian groups 
%on the big \'{e}tale site $({\rm Et}/S)$ on $S$. 
%Consider a sheaf $\wh{E}$ defined by the formula 
%$\wh{E}(W):={\rm Ker}(E(W)\lo E(W_{\rm red}))$. 
%Then, by \cite[II (1.7) (i)]{am}, 
%$\Phi^q_{Z/S}(E)=\Phi^q_{Z/S}(\wh{E})$. 
%When $E={\mab G}_m$, then we denote 
%$\Phi^q_{Z/S}(E)$ simply by $\Phi^q_{Z/S}$. 
%By \cite[II (1.8)]{am}, if $g$ is proper and flat and 
%if $\wh{E}$ is a formal group scheme over $Z$ 
%and if $R^{q-1}g_*(E)$ is formally smooth, 
%then $\Phi^q_{Z/S}(E)$ is represented by 
%a formal group scheme over $S$.

\par 
%More specially, 
Let ${\rm Art}_{\kap}$ be the category of artinian local $\kap$-algebras 
with residue fields $\kap$. 
Let us go back to the beginning of this section. 
Let $q$ be a nonnegative integer. 
Let $\Phi^q_{Y/{\kap}}\col {\rm Art}_{\kap}\lo ({\rm Ab})$ 
be the following  functor:  
for $A\in {\rm Art}_{\kap}$, set 
\begin{align*} 
\Phi^q_{Y/{\kap}}(A):={\rm Ker}
(H^q_{\rm et}(Y\otimes_{\kap}A,{\mab G}_m)
\lo H^q_{\rm et}(Y,{\mab G}_m))\in ({\rm Ab}). 
\end{align*}
%Denote $\Phi^d_{Y/{\kap}}$ simply by $\Phi_{Y/{\kap}}$. 
By \cite[II (2.11)]{am}, $\Phi^q_{Y/{\kap}}$ 
is pro-represented by a formal group over $\kap$ 
if $\Phi^{q-1}_{Y/{\kap}}$ is formally smooth. 
%(In the case $\dim Y=1$, 
($\Phi^1_{Y/{\kap}}$ is pro-represented by 
a formal group over $\kap$ \cite[(3.2)]{sch}.)
By \cite[II (4.3)]{am} the covariant Dieudnonn\'{e} module 
$D(\Phi^q_{Y/\kap})$ of $\Phi^q_{Y/\kap}$ is equal to $H^q(Y,{\cal W}({\cal O}_Y))$. 
Let $h^q(Y/\kap)$ be the height of $\Phi^q_{Y/\kap}$. 
If $H^{q+1}(Y,{\cal O}_Y)=0$, 
then $\Phi^q_{Y/\kap}$ is formally smooth over $\kap$. 
Moreover, if $H^q(Y,{\cal O}_Y)\simeq \kap$, then 
$\Phi^q_{Y/\kap}$ is a formal Lie group over $\kap$ of dimension 1 
and  $D(\Phi^q_{Y/\kap})$ is a free ${\cal W}$-module of rank 
$h^q(Y/\kap)$ if $h^q(Y/\kap)< \infty$ (\cite[V (28.3.10)]{ha}).

The following is a generalization of 
Katsura and Van der Geer's theorem 
(\cite[(5.1), (5.2), (16.4)]{vgk}).  

\begin{theo}[{\bf cf.~\cite[(5.1), (5.2), (16.4)]{vgk}}]\label{theo:nex} 
Let $Y$ be a proper scheme over $\kap$.
$($We do not assume that $Y$ is smooth over $\kap$.$)$ 
Let $q$ be a nonnegative integer. 
Assume that $H^q(Y,{\cal O}_Y)\simeq \kap$, that 
$H^{q+1}(Y,{\cal O}_Y)=0$ and 
that $\Phi^q_{Y/\kap}$ is pro-representable. 
Assume also that the Bockstein operator 
\begin{align*} 
\bet \col H^{q-1}(Y,{\cal O}_Y)\lo 
H^q(Y,{\cal W}_{n-1}({\cal O}_Y))
\tag{2.3.1}\label{ali:fbs}
\end{align*} 
arising from the following exact sequence 
\begin{align*} 
0\lo {\cal W}_{n-1}({\cal O}_Y)\os{V}{\lo} {\cal W}_n({\cal O}_Y)
\os{R^{n-1}}{\lo} {\cal O}_Y\lo 0
\end{align*} 
is zero for any $n\in {\mab Z}_{\geq 2}$. 
%that $H^{d-1}(Y,{\cal O}_Y)=0$ if $d\geq 2$. 
Let $n^q(Y)$ be the minimum of positive integers $n$'s  
such that 
$$F\col H^q(Y,{\cal W}_n({\cal O}_Y))\lo H^q(Y,{\cal W}_n({\cal O}_Y))$$ 
is not zero. 
$($If $F=0$ for all $n$, then set $n^q(Y):=\infty.)$ 
Then $h^q(Y/\kap)=n^q(Y)$. 
\end{theo}
\begin{proof} 
(Though the proof is essentially the same as that of \cite[(5.1)]{vgk} 
as stated in the beginning of this section, 
we reproduce the proof because the setting of 
(\ref{theo:nex}) is considerably more general than that in [loc.~cit.].)
%for the completeness of this article.) 
%First note that $H^d(X,{\cal O}_X)=H^d(X,\Om^d_{X/s})\simeq \kap$ 
%by Tsuji's log Serre duality. 
Set $h:=h^q(Y/\kap)$, $M:=H^q(Y,{\cal W}({\cal O}_Y))$ and 
$M_n:=H^q(Y,{\cal W}_n({\cal O}_Y))$. 
It suffices to prove that $h-1\geq n$ if and only if the morphism 
$F\col M_n\lo M_n$ is zero. 
\par 
By (\ref{ali:nov}) we see that ${\rm length}_{{\cal W}_n}(M_n)=n$. 
First we prove the implication ``if''-part. 
If $h=\infty$, then the implication is obvious. 
Hence we may assume that $h<\infty$. 
Since $M=D(\Phi^q_{Y/\kap})$ 
is $p$-torsion free, the following sequence 
\begin{align*} 
0\lo M/F\os{V}{\lo} M/p\lo M/V\lo 0
\end{align*}
of abelian groups is exact. 
Let $\sig \col \kap \lo \kap$ be the $p$-th power map. 
Since $V(\sig(a)\cdot x)=aV(x)$ $(a\in \kap, x\in M/F)$ and $\sig\in {\rm Aut}(\kap)$, 
we have the following exact sequence 
\begin{align*} 
0\lo \sig_*(M/F)\os{V}{\lo} M/p\lo M/V\lo 0
\end{align*} 
of $\kap$-vector spaces. 
Hence 
\begin{align*} 
{\rm dim}_{\kap}(M/F)={\rm dim}_{\kap}(\sig_*(M/F))
={\rm dim}_{\kap}(M/p)-{\rm dim}_{\kap}(M/V)=h-1
\end{align*} 
by (\ref{coro:esa}). The surjective morphism 
$H^q(Y,{\cal W}({\cal O}_Y))\lo H^q(Y,{\cal W}_n({\cal O}_Y))$ 
((\ref{ali:fyes}))
induces a surjective morphism 
$M/F\lo M_n/F=M_n$. 
Because ${\rm dim}_{\kap}M_n=n$ by (\ref{ali:nov}), 
we obtain the inequality $h-1\geq n$. 
\par 
Next we prove the converse implication. 
(In \cite{vgk} $\kap$ is assumed to be algebraically closed; 
it is not necessary to assume this.) 
\par 
Let $\kap \lo \kap'$ be a morphism of perfect fields.  
Set $Y':=Y\otimes_{\kap}\kap'$. 
In the proof of \cite[I (1.9.2)]{idw}, Illusie has proved 
that ${\cal W}_n({\cal O}_{Y'})
={\cal W}_n({\cal O}_Y)\otimes_{{\cal W}_n}{\cal W}_n(\kap')$. 
Since the morphism 
${\cal W}_n\lo {\cal W}_n(\kap')$ is flat, 
$H^i(Y,{\cal W}_n({\cal O}_{Y'}))=
H^i(Y,{\cal W}_n({\cal O}_Y))\otimes_{{\cal W}_n}{\cal W}_n(\kap')$. 
Let $\ol{\kap}$ be an algebraic closure of $\kap$. 
Since the morphism ${\cal W}_n\lo {\cal W}_n(\ol{\kap})$ is faithfully flat, 
we may assume that $\kap$ is algebraically closed. 
If $h=\infty$, then $F=0$ on $M=D(\Phi_{Y/\kap})=D(\wh{\mab G}_a)$. 
Hence $F=0$ on $M_n$ for all $n$. 
We may assume that $h< \infty$.  
%Consider the following filtration 
%\begin{align*} 
%V^{h-1}M\subset \cdots \subset VM \subset M.
%\tag{2.3.2}\label{amv}
%\end{align*}
Let $D(\kap)$ be the Cartier-Dieudonn\'{e} algebra over $\kap$.  
As explained in \cite[p.~266]{vgk}, 
$M=D(\Phi^q_{Y/\kap})\simeq D(\kap)/D(\kap)(F-V^{h-1})$. 
%{\cal W}[F,V]/({\cal W}[F,V](F-V^{h-1},FV-p))$.  
(In [loc.~cit.] $D(\kap)$ has been denoted by ${\cal W}[F,V]$; 
this is misleading.) 
%(In [loc.~cit.] the second relation $FV-p$ is forgotten.)
%By the calculation in [loc.~cit.], 
\par 
(The following argument is due to the referee.)
It is easy to see that $V^{h-1}M=FM$ as in [loc.~cit.]. 
Consider the composite morphism 
$V^{h-1}M=FM\os{\subset}{\lo} M\os{{\rm proj}.}{\lo} M_{h-1}$. 
Obviously this composite morphism is a zero morphism, while 
the image of this morphism is equal to $FM_{h-1}$ since the following diagram 
\begin{equation*} 
\begin{CD}
M@>{F}>>M\\
@VVV @VVV \\
M_{h-1}@>{F}>>M_{h-1}
\end{CD}
\end{equation*} 
is commutative and since the morphism $M\lo M_{h-1}$ is surjective. 
%The filtration (\ref{amv}) is mapped surjectively to the following filtration 
%\begin{align*} 
%FM_{h-1}=V^{h-1}M_{h-1}\subset \cdots \subset VM_{h-1} \subset M_{h-1}.
%\tag{2.3.3}\label{amsv}
%\end{align*}
%By (\ref{ali:nov}) the inclusions in (\ref{amsv}) are strict. 
%Since ${\rm length}_{{\cal W}_{h-1}}(M_{h-1})=h-1$, 
Hence $FM_{h-1}=0$ and $F=0$ on $M_n$ for all $n\leq h-1$. 
\end{proof}

The following is a generalization of \cite[(5.6)]{vgk}: 
\begin{coro}[{\bf cf.~\cite[(5.6)]{vgk}}]\label{coro:dim}
Set ${}_FH^q(Y,{\cal W}_n({\cal O}_Y))
:={\rm Ker}(F\col H^q(Y,{\cal W}_n({\cal O}_Y))
\lo H^q(Y,{\cal W}_n({\cal O}_Y)))$. 
Then 
\begin{align*} 
{\rm dim}_{\kap}({}_FH^q(Y,{\cal W}_n({\cal O}_Y)))=
{\rm min}\{n, h^q(Y/\kap)-1\}.
\tag{2.4.1}\label{ali:fdwy}
\end{align*} 
Consequently 
\begin{align*} 
{\rm dim}_{\kap}(H^q(Y,{\cal W}_n({\cal O}_Y))/F)
={\rm min}\{n, h^q(Y/\kap)-1\}. 
\tag{2.4.2}\label{ali:fdwhy}
\end{align*} 
\end{coro}
\begin{proof}
Let the notations be as in the proof of (\ref{theo:nex}).  
As in the proof of (\ref{theo:nex}), 
we may assume that $\kap$ is algebraically closed. 
We may assume that $h<\infty$. 
If $n\leq h-1$, then $F=0$ on $M_n$. 
Hence ${\rm Ker}(F\col M_n\lo M_n)=M_n$ and 
this is an $n$-dimensional vector space over $\kap$. 
\par 
Assume that $n\geq h$.
Because $M=D(\Phi^q_{Y/\kap})\simeq D(\kap)/D(\kap)(F-V^{h-1})$, 
there exists an element $\om \in M$ such that 
$\{\om, V(\om),\ldots, V^{h-1}(\om)\}$ is 
a basis of  $M$ over ${\cal W}$. 
Let $\ol{\om}$ be the image of $\om$ in $M_n$. 
Let $R\col M_m\lo M_{m-1}$ 
$(m\geq 2)$ be the induced morphism by the projection 
$R\col {\cal W}_m({\cal O}_Y)\lo {\cal W}_{m-1}({\cal O}_Y)$. 
Then we claim that 
$$\{V^{n-1}R^{n-1}(\ol{\om}), \ldots, V^{n-(h-1)}R^{n-(h-1)}(\ol{\om})\}$$  
is a basis of ${}_FM_n$.  
Indeed, this follows from the consideration in the case $n=h$ and induction on $n$ 
(by using the injectivity of the morphism $V\col M_n\lo M_{n+1}$) 
and the relation $FV=VF$. 
The claim implies (\ref{ali:fdwy}). 
\par 
The equality (\ref{ali:fdwhy}) follows from the following exact sequence 
\begin{align*} 
0\lo {\rm Ker}(F)\lo M_n\os{F}{\lo} \sig_*(M_n)\lo {\rm Coker}(F)\lo 0. 
\end{align*} 
(Note that, since $\kap$ is perfect, 
${\rm length}_{{\cal W}_n}(M_n)={\rm length}_{{\cal W}_n}(\sig_*(M_n))$.
\end{proof} 

%\begin{coro}\label{coro:ftc} 
%The following three conditions are equivalent$:$ 
%\par 
%$({\rm a})$ $h^q(Y/\kap)=1$.  
%\par 
%$({\rm b})$ The morphism 
%$F\col H^q(Y,{\cal W}_n({\cal O}_Y))\lo H^q(Y,{\cal W}_n({\cal O}_Y))$ 
%is injective for a positive integer $n$. 
%\par 
%$({\rm c})$ The morphism 
%$F\col H^q(Y,{\cal W}_n({\cal O}_Y))\lo H^q(Y,{\cal W}_n({\cal O}_Y))$ 
%is surjective for a positive integer $n$. 
%\end{coro}
%\begin{proof} 
%Obvious. 
%\end{proof}

\begin{rema}\label{rema:gend} 
We can generalize a part of (\ref{theo:nex}) as follows. 
\par 
Assume that $H^d(Y,{\cal O}_Y)\simeq \kap^m$ for a positive integer $m$ instead of the assumption 
$H^d(Y,{\cal O}_Y)\simeq \kap$ 
and the operator $F\col 
H^d(Y,{\cal W}_n({\cal O}_Y))\lo 
H^d(Y,{\cal W}_n({\cal O}_Y))$ is zero. 
Then $n\leq m^{-1}h^q(Y/\kap)-1$. 
The proof of this fact is the same as that of 
a part of the proof of (\ref{theo:nex}). 
\end{rema}

\section{The dimensions of cohomologies of closed differential forms}\label{sec:dc} 
Let $S_0$ be a fine log scheme of characteristic $p>0$. 
Let $F_{S_0}\col S_0\lo S_0$ be the Frobenius endomorphism of $S_0$.  
Let $Y$ be a log smooth scheme of Cartier type over $S_0$. 
Let $g\col Y\lo S_0$ be the structural morphism. 
Set $Y':=Y\times_{S_0,F_{S_0}}S_0$. 
Let $W\col Y'\lo Y$ be the projection and 
let $F\col Y\lo Y'$ be the relative Frobenius morphism over $S_0$. 
First recall the log inverse Cartier isomorphism due to Kato 
(\cite[(4.12) (1)]{klog1}). 
It is the following isomorphism 
of sheaves of ${\cal O}_{Y'}$-modules: 
\begin{align*} 
C^{-1}\col \Om^i_{Y'/S_0}
\os{\sim}{\lo} F_*({\cal H}^i(\Om^{\bul}_{Y/S_0})). 
\tag{3.0.1}\label{ali:oxs}
\end{align*} 
Consider the case $i=0$ in (\ref{ali:oxs}). 
Then $C^{-1}\col {\cal O}_{Y'}\os{\sim}{\lo} F_*({\cal H}^0(\Om^{\bul}_{Y/S_0}))$ 
is the following isomorphism 
\begin{align*} 
{\cal O}_{Y'}\owns  a \lom F^*(a)\in F_*({\cal H}^0(\Om^{\bul}_{Y/S_0})).
\tag{3.0.2}\label{ali:cyfis}
\end{align*} 
In particular, the following composite morphism 
\begin{align*} 
{\cal O}_{Y'}\os{\sim}{\lo} F_*({\cal H}^0(\Om^{\bul}_{Y/S_0}))\os{\sus}{\lo} F_*({\cal O}_Y)
\tag{3.0.3}\label{ali:cis}
\end{align*} 
is injective.

\begin{rema}\label{rema:rds}  
Assume that $\os{\circ}{S}_0$ is reduced. 
Then $F_{S_0}$ induces an injective morphism 
$F_{S_0}^*\col {\cal O}_{S_0}\lo F_{S_0*}({\cal O}_{S_0})$. 
By \cite[(4.5)]{klog1} the structural morphism $\os{\circ}{Y}\lo \os{\circ}{S}_0$ is flat. 
Hence the natural morphism 
${\cal O}_{Y}\lo W_*({\cal O}_{Y'})$ is injective. 
Because the composite morphism of this morphism and 
$W_*((\ref{ali:cis}))$ is the $p$-th power endomorphism of ${\cal O}_Y$,  
$\os{\circ}{Y}$ is reduced (cf.~\cite[(2.3.2)]{s1}, 
Tsuji's result (\ref{prop:ktj}) below). 
\end{rema} 

The following is Tsuji's result (\cite{tsa}), which will be used in later sections. 

\begin{prop}[{\bf \cite[II (2.11) (1), (2.11) (2), (2.13) (1), (2.13) (2), (2.14), (4.2)]{tsa}}]\label{prop:ktj} 
The following hold$:$
\par 
$(1)$ The composite morphism of two saturated morphisms of integral log schemes is saturated. 
\par 
$(2)$ The saturated morphisms of 
integral log schemes are stable 
under the base change of integral log schemes. 
\par 
$(3)$ Let $g\col Y\lo Z$ be an integral morphism of 
$($fine$)$ saturated log schemes. 
Then $g$ is saturated if and only if the base change 
$Y'$ of $Y$ with respect to any morphism $Z'\lo Z$ from 
any $($fine$)$ saturated log scheme are saturated. 
\par 
$(4)$ Let $g\col Y\lo Z$ be a morphism of integral  
%$($=fine and saturated$)$ 
log schemes in characteristic $p>0$. 
Then $g$ is $p$-saturated 
if and only if $g$ is of Cartier type. 
\par 
$(5)$ Let $g\col Y\lo Z$ be a log smooth integral morphism of fs 
log schemes.
Then $g$ is saturated 
if and only if every fiber of $\os{\circ}{g}$ is reduced. 
\end{prop}

\begin{prop}\label{prop:ee}
Set $B\Om^1_{Y/S_0}:={\rm Im}(d\col {\cal O}_Y\lo \Om^1_{Y/S_0})$. 
Then the following sequence 
\begin{align*} 
0\lo {\cal O}_{Y'}\os{F^*}{\lo} F_*({\cal O}_Y)\os{F_*(d)}{\lo} 
F_*(B\Om^1_{Y/S_0})\lo 0
\tag{3.3.1}\label{ali:cifes}
\end{align*} 
of ${\cal O}_{Y'}$-modules is exact.
\end{prop}
\begin{proof} 
Except the surjecitvity of $F_*(d)$, 
this is nothing but a reformulation of (\ref{ali:cyfis}). 
Since $\os{\circ}{F}$ is a homeomorphism  (\cite[XV Proposition 2 a)]{sga5-2}), 
$R^qF_*({\cal E})=0$ $(q>0)$ 
for an abelian sheaf ${\cal E}$ on $Y$. 
Hence $F_*(d)$ is surjective. 
\end{proof}

%\par 
%Let $s$ be a fine log scheme whose underlying scheme is the spectrum 
%of a perfect field $\kap$ of characteristic $p>0$. 
%Let $F_s$ (resp.~$F_{\kap}$) be the Frobenius endomorphism of $s$ 
%(resp.~$\kap$). 
%Let $X$ be a log smooth scheme of Cartier type over $s$. 
%Set $X':=X\times_{s,F_s}s$. Let $f\col X'\lo X$ be the projection 
%and let $F\col X\lo X'$ be the relative Frobenius morphism. 
%Because 
%$\os{\circ}{F}$ is a homeomorphism and because 
%$\Om^i_{X'/s}=f^*(\Om^i_{X/s})=
%\Om^i_{X/s}\otimes_{\kap,F_{\kap}}\kap\simeq \Om^i_{X/s}$, 
%the log inverse Cartier isomorphism (\ref{ali:oxs}) in the case $S=s$ 
%can be rewritten as the following isomorphism
%\begin{align*} 
%C^{-1}\col \Om^i_{X/s}
%\os{\sim}{\lo} {\cal H}^i(\Om^{\bul}_{X/s}) 
%\tag{3.2.2}\label{ali:oxho}
%\end{align*} 
%of abelian sheaves on $\os{\circ}{X}$. 
%However we would not like to use this isomorphism as in \cite[0 (2.2)]{ida}. 
%Even in this special case, we would like to use 
%the isomorphism 
%\begin{align*} 
%C^{-1}\col \Om^i_{X'/s}
%\os{\sim}{\lo} {\cal H}^i(\Om^{\bul}_{X/s}) 
%\tag{3.2.3}\label{ali:oxrho}
%\end{align*} 

%\begin{rema}\label{rema:rds}  
%If $X/s$ is of Cartier type, then 
%$\os{\circ}{X}$ is reduced (\cite[(2.3.2)]{s1}). 
%\end{rema} 

Let us recall well-known sheaves 
$B_n\Om^i_{Y/S_0}$ and $Z_n\Om^i_{Y/S_0}$ $(n\geq 1)$ 
of $g^{-1}({\cal O}_{S_0})$-modules 
on $\os{\circ}{Y}$ as in \cite[0 (2.2)]{idw} 
and \cite[(4.3)]{hk} defined by induction on $n$. 
\par 
Because 
$\os{\circ}{F}$ is a homeomorphism, 
we can identify an abelian sheaf on $\os{\circ}{Y}$ with 
an abelian sheaf on $\os{\circ}{Y}{}'$. 
Under this identification, we can express (\ref{ali:oxs}) as the equality 
\begin{align*} 
C^{-1}\col \Om^i_{Y'/S_0}
= {\cal H}^i(\Om^{\bul}_{Y/S_0}) 
\tag{3.3.2}\label{ali:oxfws}
\end{align*} 
of abelian sheaves. 
Set $B_0\Om^i_{Y/S_0}:=0$
and 
$Z_0\Om^i_{Y/S_0}:=\Om^i_{Y/S_0}$.  
We define $B_n\Om^i_{Y/S_0}$ and $Z_n\Om^i_{Y/S_0}$ 
by the following equalities $(n\geq 1)$ : 
$$C^{-1}\col B_{n-1}\Om^i_{Y'/S_0}=
B_n\Om^i_{Y/S_0}/B\Om^i_{Y/S_0}, \quad 
C^{-1}\col Z_{n-1}\Om^i_{Y'/S_0} =
Z_n\Om^i_{Y/S_0}/B\Om^i_{Y/S_0}.$$ 
%(Note that the log smooth morphism of Cartier type is stable 
%under the base change of fine log schemes of characteristic $p>0$. 
%Hence, if one has defined 
%$B_n\Om^i_{Y/S_0}$ and $Z_n\Om^i_{Y/S_0}$, then 
%$B_n\Om^i_{Y'/S_0}$ and $Z_n\Om^i_{Y'/S_0}$ have already been defined.)
Then we  have the following inclusions: 
\begin{align*}
0\subset B_1\Om^i_{Y/S_0}\subset 
\cdots \subset &B_n\Om^i_{Y/S_0} \subset \cdots 
\subset Z_n\Om^i_{Y/S_0} \subset  \cdots 
\subset  Z_1\Om^i_{Y/S_0} \subset \Om^i_{Y/S_0}.
\end{align*}
Set $Y^{\{p\}}:=Y'$ and $Y^{\{p^n\}}:=(Y^{\{p^{n-1}\}})'$. 
We consider $Z_n\Om^i_{Y/S_0}$ and $B_n\Om^i_{Y/S_0}$ as 
${\cal O}_{Y^{\{p^n\}}}$-submodules of 
$F^n_*(\Om^i_{Y/S_0})$. 
%(These notations $Z_n\Om^i_{Y/S_0}$ and 
%$B_n\Om^i_{Y/S_0}$ are the same as Illusie's notations in \cite[0 (2.2.2)]{idw} 
%in the trivial logarithmic case.)  
We recall the following result: 

\begin{lemm}[{\bf \cite[0 (2.2.8)]{idw}, \cite[(1.13)]{lodw}}]\label{lemm:locf}
The sheaves $B_n\Om^i_{Y/S_0}$ and  $Z_n\Om^i_{Y/S_0}$ 
$(n\in {\mab N}$, $i\in {\mab N})$ are locally free sheaves of 
${\cal O}_{Y^{\{p^n\}}}$-modules of finite rank. 
They commute with the base changes of $S_0$. 
\end{lemm}

\par 
If $\os{\circ}{S}_0$ is perfect, then 
$\os{\circ}{Y}{}'\os{\sim}{\lo}\os{\circ}{Y}$. 
Hence the equality (\ref{ali:oxfws}) induces the following isomorphism: 
\begin{align*} 
C^{-1}\col \Om^i_{Y/S_0}
\os{\sim}{\lo} {\cal H}^i(\Om^{\bul}_{Y/S_0}). 
\tag{3.4.1}\label{ali:oxyfws}
\end{align*}
Then we have the following Cartier morphisms $C$'s: 
\begin{align*} 
C\col B_{n+1}\Om^i_{Y/S_0} \os{{\rm proj}.}{\lo} 
B_{n+1}\Om^i_{Y/S_0}/B\Om^i_{Y/S_0}
\os{C^{-1},\sim}{\longleftarrow} B_n\Om^i_{Y'/S_0}
\os{\sim}{\longleftarrow}B_n\Om^i_{Y/S_0}
\end{align*}
and 
\begin{align*} 
C\col Z_{n+1}\Om^i_{Y/S_0} \os{{\rm proj}.}{\lo} 
Z_{n+1}\Om^i_{Y/S_0}/B\Om^i_{Y/S_0}
\os{C^{-1},\sim}{\longleftarrow} Z_n\Om^i_{Y'/S_0}
\os{\sim}{\longleftarrow}Z_n\Om^i_{Y/S_0}. 
\end{align*}
These morphisms are only morphisms of 
abelian sheaves on $\os{\circ}{Y}$. 
%structures of $B_{n+1}\Om^i_{Y/S_0}$ and 
%$Z_{n+1}\Om^i_{Y/S_0}$ and 
%${\cal O}_{Y^{\{p^{n+1\}}}}$-module 
%structures of $B_{n}\Om^i_{Y/S_0}$ and 
%$Z_{n}\Om^i_{Y/S_0}$. 

%These morphisms are compatible with ${\cal O}_{Y^{\{p^{n+1\}}}}$-module 
%structures of $B_{n+1}\Om^i_{Y/S_0}$ and 
%$Z_{n+1}\Om^i_{Y/S_0}$ and 
%${\cal O}_{Y^{\{p^{n+1\}}}}$-module 
%structures of $B_{n}\Om^i_{Y/S_0}$ and 
%$Z_{n}\Om^i_{Y/S_0}$. 
\par 
Until the end of this section except the remark (\ref{rema:rfm}) below, 
assume that $\os{\circ}{S}_0$ is perfect.  
Let $F\col Y\lo Y$ be the absolute Frobenius endomorphism of $Y$. 
In \cite[\S7 (18)]{smxco} Serre has defined 
the following morphism of abelian sheaves
\begin{align*} 
d_n \col F_*({\cal W}_n({\cal O}_Y))
\lo F_*^n(\Om^1_{\os{\circ}{Y}/\os{\circ}{S}_0}) 
\end{align*} 
defined by the following formula$:$
\begin{align*} 
d_n((a_0,\ldots,a_{n-1}))
=\sum_{i=0}^{n-1}a^{p^{n-1-i}-1}_ida_i 
\quad (a_i\in {\cal O}_Y).  
\tag{3.4.2}\label{ali:sed}
\end{align*} 
(In [loc.~cit.] he has denoted $d_n$ by $D_n$ 
and he has considered $D_n$ 
only in the case $\os{\circ}{S}_0={\rm Spec}(\kap)$.)
He has remarked that the following formula holds:
%\begin{align*} 
%d_n((a_0,\ldots,a_{n-1})+(b_0,\ldots,b_{n-1}))=
%d_n((a_0,\ldots,a_{n-1}))+d_n((b_0,\ldots,b_{n-1})), 
%\end{align*} 
%and 
\begin{align*} 
d_n((a_0,\ldots,a_{n-1})(b_0,\ldots,b_{n-1}))=
b_0^{p^{n-1}}d_n((a_0,\ldots,a_{n-1}))+
a_0^{p^{n-1}}d_n((b_0,\ldots,b_{n-1})). 
\tag{3.4.3}\label{ali:dna} 
\end{align*} 
By (\ref{ali:dna}) it is easy to check that 
$d_n\col F_*({\cal W}_n({\cal O}_Y))
\lo F_*^n(\Om^1_{\os{\circ}{Y}/\os{\circ}{S}_0})$ 
is a morphism of ${\cal W}_n({\cal O}_Y)$-modules.  
%(This remark was not given in [loc.~cit.].)

\begin{rema}\label{rema:rfm}
Let $F\col Y\lo Y'$ be the relative Frobenius morphism 
as in the beginning of this section. 
Then the morphism 
$$d_n\col F_*({\cal W}_n({\cal O}_Y))
\lo F_*^n(\Om^1_{\os{\circ}{Y}/\os{\circ}{S}_0})$$ 
{\it cannot}  be a morphism of 
${\cal W}_n({\cal O}_{Y'})$-modules in general except the case $n=1$. 
%Indeed, let $(a_0,\ldots,a_{n-1})$ (resp.~$(b_0,\ldots,b_{n-1})$) be a local section of 
%${\cal W}_n({\cal O}_{Y'})$ (resp.~$F_*({\cal W}_n({\cal O}_Y))$). 
%Then 
%\begin{align*} 
%d_n((a_0,\ldots,a_{n-1})(b_0,\ldots,b_{n-1}))
%&=d_n((F^*(a_0),\ldots,F^*(a_{n-1}))(b_0,\ldots,b_{n-1}))
%\tag{3.4.1}\label{ali:abfa}\\
%&=F^*(a_0)^{p^{n-1}}d_n((b_0,\ldots,b_{n-1})). 
%\end{align*} 
%This formula tells us that $d_n$ is {\it not} 
%a morphism of ${\cal O}_{S_0}$-modules 
%in general except the case $n=1$. 
\end{rema}

\par 
The following (\ref{prop:wus}) 
is the log version of a generalization of Serre's result in 
\cite[\S7 Lemme 2]{smxco}. 
%The proof of the following is an alternative proof of his proof. 
Our proof of (\ref{prop:wus}) is more elementary 
and more elegant than his proof. 

\begin{prop}\label{prop:wus}
%Let $s$ be a fine log scheme whose underlying scheme is 
%the spectrum of a perfect field $\kap$ of characteristic $p>0$. 
%Let $X$ be a log smooth scheme over $s$ of Cartier type. 
Assume that $\os{\circ}{S}_0$ is perfect. 
Let $F\col Y\lo Y$ be the absolute Frobenius endomorphism of $Y$. 
Let $n$ be a positive integer. 
Denote the following composite morphism 
\begin{align*} 
F_*({\cal W}_n({\cal O}_Y))\os{d_n}{\lo} 
F^n_*(\Om^1_{\os{\circ}{Y}/\os{\circ}{S}_0})\lo F^n_*(\Om^1_{Y/S_0})
\end{align*} 
by $d_n$ again. Then $d_n$ factors through $B_n\Om^1_{Y/S_0}$ 
and the following sequence 
\begin{align*} 
0\lo {\cal W}_n({\cal O}_Y)\os{F}{\lo} 
F_*({\cal W}_n({\cal O}_Y))
\os{d_n}{\lo} B_n\Om^1_{Y/S_0}\lo 0
\tag{3.6.1;$n$}\label{ali:wnox} 
\end{align*} 
is exact. 
Here we denote the morphism 
${\cal W}_n(F^*)$ $($resp.~$F_*({\cal W}_n({\cal O}_{Y}))\lo B_n\Om^1_{Y/S_0})$ 
by $F$ $($resp.~$d_n)$ again by abuse of notation. 
Consequently $d_n$ induces the following isomorphism of 
${\cal W}_n({\cal O}_Y)$-modules$:$ 
\begin{align*} 
F_*({\cal W}_n({\cal O}_Y))/F({\cal W}_n({\cal O}_{Y}))
\os{\sim}{\lo}
B_n\Om^1_{Y/S_0}. 
\tag{3.6.2}\label{ali:fwo} 
\end{align*} 
\end{prop}
\begin{proof} 
First consider the case $n=1$. In this case, 
(\ref{prop:wus}) is obtained by (\ref{prop:ee}). 
%By (\ref{rema:rds}) the morphism 
%$F\col {\cal O}_{Y'} \lo F_*({\cal O}_Y)$ is injective. 
%Obviously $d_1 \col F_*({\cal O}_X)\lo \Om^1_{X/s}$ 
%induces the surjective morphism 
%$d_1\col  F_*({\cal O}_Y)\lo F_*(B\Om^1_{Y/S_0})$ 
%(Note that $\os{\circ}{F}$ is finite.).  
%We claim that ${\rm Ker}(d_1)={\cal O}_{Y'}$. 
%Indeed, ${\rm Ker}(d_1)={\cal H}^0(\Om^1_{Y/S_0})$.  
%Consider the case $i=0$ in (\ref{ali:oxho}). 
%Then $C^{-1}\col {\cal O}_X\os{\sim}{\lo} {\cal H}^0(\Om^{\bul}_{X/s})$ 
%is the following isomorphism 
%\begin{align*} 
%{\cal O}_X\owns a \lom {\rm class~of~}a^p\in {\cal H}^0(\Om^{\bul}_{X/s}).
%\end{align*} 
%This shows that the sequence (\ref{ali:wnox}) is exact for the case $n=1$.   
\par 
We proceed by induction on $n$. 
Assume that $(3.6.1;n-1)$ is exact. 
\par 
By the definition of $B_n\Om^1_{Y/S_0}$, 
the isomorphism $C^{-1}$ in (\ref{ali:oxyfws}) induces the isomorphism 
$C^{-1}\col B_{n-1}\Om^1_{Y/S_0}\os{\sim}{\lo} 
B_n\Om^1_{Y/S_0}/B\Om^1_{Y/S_0}$.  
Let $R\col {\cal W}_n({\cal O}_Y)\lo {\cal W}_{n-1}({\cal O}_Y)$ 
be the projection. 
%Using the following exact sequence 
%\begin{align*} 0\lo B_1\Om^1_{Y/S_0}\lo B_n\Om^1_{Y/S_0}
%\os{C}{\lo}B_{n-1}\Om^1_{Y/S_0}\lo 0 
%\end{align*} 
%of abelian sheaves on $Y$, 
By the inductive definiton of $B_n\Om^1_{Y/S_0}$, it is easy to check 
that ${\rm Im}(d_n)\subset B_n\Om^1_{Y/S_0}$. 
Consider the following diagram 
\begin{equation*} 
\begin{CD} 
@.0@. 0@. 0\\
@. @VVV @VVV @VVV \\
0@>>> {\cal O}_{Y}@>{V^{n-1}}>> 
{\cal W}_n({\cal O}_{Y})@>{R}>>
{\cal W}_{n-1}({\cal O}_{Y})@>>> 0\\
@. @V{F}VV @V{F}VV @VV{F}V \\
0@>>> F_*({\cal O}_Y)@>{F_*(V^{n-1})}>> 
F_*({\cal W}_n({\cal O}_Y))
@>{F_*(R)}>>F_*({\cal W}_{n-1}({\cal O}_Y))@>>> 0\\
@. @V{d}VV @V{d_n}VV @VV{d_{n-1}}V \\
0@>>> B_1\Om^1_{Y/S_0}
@>{\subset}>> B_n\Om^1_{Y/S_0}
@>{C}>>B_{n-1}\Om^1_{Y/S_0}@>>> 0\\
@. @VVV @VVV @VVV \\
@.0@. 0@. 0. 
\end{CD}
\tag{3.6.3}\label{cd:dbx}
\end{equation*} 
The three rows above are exact sequences of abelian sheaves on $Y$. 
By (\ref{ali:cis}) the morphism 
$F\col {\cal W}_n({\cal O}_{Y})\lo F_*({\cal W}_n({\cal O}_Y))$ 
is injective. It is clear that ${\rm Im}(F)\subset {\rm Ker}(d_n)$. 
It is also clear that the upper two diagrams are commutative. 
%We claim that the lower two diagrams are commutative. 
The commutativity of the left square of the lower diagram follows from 
the obvious relation 
\begin{align*} 
d_n V=d_{n-1}. 
\tag{3.6.4}\label{ali:dvd} 
\end{align*}
%Indeed, $d_n((a_0\ldots, a_{n-1}))=
%da_{n-1}+a_{n-2}^{p-1}da_{n-2}+\cdots +a^{p^{n-1}-1}_0da_0$ 
%and 
%$d_{n-1}((a_0\ldots, a_{n-2}))=
%da_{n-2}+a_{n-3}^{p-1}da_{n-3}+\cdots +a^{p^{n-2}-1}_0da_0$. 
Since 
\begin{align*} 
C^{-1}(a_i^{p^{n-2-i}-1}da_i)=
(a_i^{p^{n-2-i}-1})^pa^{p-1}_ida_i=a_i^{p^{n-1-i}-1}da_i 
\quad (i\in {\mab N}, a_i\in {\cal O}_Y),
\end{align*}
we see that the right square of the lower diagram is commutative. 
%Since $d\col F_*({\cal O}_Y)\lo B\Om^1_{Y/s}$ is surjective, 
Induction on $n$ and the snake lemma show that the middle column is exact. 
%the morphism $d_n\col {\cal W}_n({\cal O}_Y)\lo B_n\Om^1_{Y/s}$ is surjective. 
\end{proof} 

\begin{rema}\label{rema:xkl} 
%(1) 
In \cite{smxco} Serre has considered (\ref{ali:wnox}) 
in the trivial logarithmic case 
with the assumption of the normality of $\os{\circ}{Y}$ 
only as an exact sequence of 
sheaves of {\it abelian sheaves}. 
In this article
we have to consider (\ref{ali:wnox}) as an exact sequence of 
sheaves of ${\cal W}_n({\cal O}_Y)$-{\it modules}. 
%\par 
%(2) In (\ref{prop:fkervcoker}), (\ref{coro:locdw}), 
%(\ref{prop:seg}) and (\ref{coro:nif}) below, 
%we shall generalize (\ref{prop:wus}). 
\end{rema} 

\begin{defi}
We call the exact sequence (\ref{ali:wnox}) of 
${\cal W}_n({\cal O}_Y)$-modules 
the {\it log Serre exact sequence of $Y/S_0$ in level $n$}. 
\end{defi} 

It is worth stating the following 
(this has been used in \cite{nlfc} 
in a key point): 

\begin{coro}\label{coro:wsta}
The following diagram 
\begin{equation*} 
\begin{CD} 
F_*({\cal W}_n({\cal O}_Y))
@>{F_*(R)}>>F_*({\cal W}_{n-1}({\cal O}_Y))\\
@V{d_n}VV @VV{d_{n-1}}V \\
B_n\Om^1_{Y/S_0}
@>{C}>>B_{n-1}\Om^1_{Y/S_0}
\end{CD}
\tag{3.9.1}\label{cd:dowm}
\end{equation*} 
is commutative. 
\end{coro}

\begin{coro}\label{coro:chob} 
Consider the case $\os{\circ}{S}_0={\rm Spec}(\kap)$ 
as in the beginning of the previous section. 
Denote $S_0$ by $s$ in this case. 
Let the notations and the assumptions be as in {\rm (\ref{theo:nex})}. 
%Assume that $\os{\circ}{Y}$ is proper over $\kap$. 
%Assume that $H^q(Y,{\cal O}_Y)\simeq \kap$ and 
%that the Bockstein operator {\rm (\ref{ali:fbs})} is zero. 
%Assume also that $H^{q+1}(Y,{\cal O}_Y)=0$ and 
%that $\Phi^q_{Y/\kap}$ is pro-representable.  
%Let $F\col H^q(Y,{\cal W}_n({\cal O}_Y))
%\lo H^q(Y,{\cal W}_n({\cal O}_Y))$ 
%$(q\in {\mab N})$ 
%be the Frobenius operator.  
Then the following hold$:$
\par 
$(1)$ $H^q(Y,{\cal W}_n({\cal O}_Y))/F=H^q(Y,B_n\Om^1_{Y/s})$. 
Consequently 
\begin{equation*} 
{\rm dim}_{\kap}H^q(Y,B_n\Om^1_{Y/s})=
{\rm min}\{n, h^q(\os{\circ}{Y}/\kap)-1\}. 
\tag{3.10.1}\label{eqn:pkdefpw}
\end{equation*}   
%\begin{equation*} 
%{\rm dim}_{\kap}H^q(Y,B_n\Om^1_{Y/s}) =
%\begin{cases} 
%0 & ({\rm if}~q\not= d-1,d), \\
%{\rm min}\{n, h(\os{\circ}{Y}/\kap)-1\} & ({\rm if}~q= d-1,d). 
%\end{cases}
%\tag{4.0.1}\label{eqn:pkdefpw}
%\end{equation*}  
\par 
$(2)$ Assume that $H^{q-1}(Y,{\cal O}_Y)=0$ if $q\geq 2$. 
%if $d\geq 2$. 
Then 
${}_FH^q(Y,{\cal W}_n({\cal O}_Y))=H^{q-1}(Y,B_n\Om^1_{Y/s})$. 
Consequently 
\begin{equation*} 
{\rm dim}_{\kap}H^{q-1}(Y,B_n\Om^1_{Y/s})=
{\rm min}\{n, h^q(\os{\circ}{Y}/\kap)-1\}. 
\tag{3.10.2}\label{eqn:pkdppw}
\end{equation*}   
\par 
$(3)$ Assume that $H^{q-1}(Y,{\cal O}_Y)=0$ if $q\geq 2$ and 
that $H^{q-2}(Y,{\cal O}_Y)=0$ if $q\geq 3$.  
%without assuming that $h(\os{\circ}{Y}/\kap)<\infty$. 
Then $H^{q-2}(Y,B_n\Om^1_{Y/s})=0$. 
\end{coro} 
\begin{proof} 
%By (\ref{prop:wus}) 
%$H^q(Y,B_n\Om^1_{Y/s})=H^q(Y,{\cal W}_n({\cal O}_Y)/F)$. 
%\par 
%Consider the following exact sequence 
%\begin{align*}
%0\lo {\cal W}_n({\cal O}_Y)\os{F}{\lo} F_*({\cal W}_n({\cal O}_Y))
%\lo F_*({\cal W}_n({\cal O}_Y))/F\lo 0. 
%\tag{2.7.2}\label{ali:fwx}
%\end{align*} 
(1): Taking the long exact sequence of (\ref{ali:wnox}),   
we have the following exact sequence of ${\cal W}_n$-modules:  
\begin{align*}
&H^{q-1}(Y,{\cal W}_n({\cal O}_Y))\os{F}{\lo}
H^{q-1}(Y,F_*({\cal W}_n({\cal O}_Y))) \lo H^{q-1}(Y,B_n\Om^1_{Y/s})
\lo H^q(Y,{\cal W}_n({\cal O}_{Y}))\tag{3.10.3}\label{ali:box}\\
&\os{F}{\lo}H^q(Y,F_*({\cal W}_n({\cal O}_Y)))
\lo H^q(Y,B_n\Om^1_{Y/s})\lo 0. 
\end{align*} 
Since $\os{\circ}{F}$ is finite, 
$H^q(Y,F_*({\cal W}_n({\cal O}_Y)))=\sig_*H^q(Y,{\cal W}_n({\cal O}_Y))$,  
where  $\sig$ is the Frobenius automorphism of ${\cal W}_n$.  
Hence we have the following exact sequence of ${\cal W}_n$-modules: 
\begin{align*}
&H^{q-1}(Y,{\cal W}_n({\cal O}_Y))\os{F}{\lo}
\sig_*H^{q-1}(Y,{\cal W}_n({\cal O}_Y)) \lo H^{q-1}(Y,B_n\Om^1_{Y/s})
\lo H^q(Y,{\cal W}_n({\cal O}_{Y}))\tag{3.10.4}\label{ali:wnlof}\\
&\os{F}{\lo} \sig_*H^q(Y,{\cal W}_n({\cal O}_Y))
\lo H^q(Y,B_n\Om^1_{Y/s})\lo 0. 
\end{align*} 
Hence $H^q(Y,{\cal W}_n({\cal O}_Y))/F=H^q(Y,B_n\Om^1_{Y/s})$. 
By (\ref{coro:dim}) the dimension of this vector space over $\kap$ 
is ${\rm min}\{n, h^q(\os{\circ}{Y}/\kap)-1\}$. 
%Hence (1) follows. 
\par 
(2):  If $q\geq 2$, it is easy to see that $H^{q-1}(Y,{\cal W}_n({\cal O}_Y))=0$. 
Hence we have the following exact sequence of ${\cal W}_n$-modules:  
\begin{align*}
&0 \lo H^{q-1}(Y,B_n\Om^1_{Y/s})
\lo H^q(Y,{\cal W}_n({\cal O}_{Y}))\os{F}{\lo}
\sig_*H^q(Y,{\cal W}_n({\cal O}_Y))
\tag{3.10.5}\label{ali:wolof}\\
&\lo H^q(Y,B_n\Om^1_{Y/s})\lo 0. 
\end{align*} 
In the case $q=1$, we see that (\ref{ali:wolof}) is also exact 
by the proof of (\ref{prop:ehe}) (2).  
%Hence we have the following exact sequence 
%\begin{align*}
%0& \lo H^{q-1}(Y,B_n\Om^1_{Y/s})\lo 
%H^q(Y,{\cal W}_n({\cal O}_Y))\os{F}{\lo}
%H^q(Y,{\cal W}_n({\cal O}_Y))
%\tag{3.9.6}\label{ali:wnylof}\\
%&\lo H^q(Y,B_n\Om^1_{Y/s})\lo 0. 
%\end{align*}  
This tells us that 
$H^{q-1}(Y,B_n\Om^1_{Y/s})={}_FH^q(Y,{\cal W}_n({\cal O}_Y))$. 
By (\ref{coro:dim}) the dimensions of these vector spaces over $\kap$ 
are ${\rm min}\{n, h^q(\os{\circ}{Y}/\kap)-1\}$. 
\par 
(3): 
By (\ref{ali:fwo}) 
$F_*({\cal W}_n({\cal O}_Y))/F({\cal W}_n({\cal O}_Y))=B_n\Om^1_{Y/s}$. 
Hence 
\begin{align*} 
H^{q'}(Y,B_n\Om^1_{Y/s})=
H^{q'}(Y,F_*({\cal W}_n({\cal O}_Y))/F({\cal W}_n({\cal O}_Y)))
=H^{q'}(Y,{\cal W}_n({\cal O}_Y)/F) \quad (q'\in {\mab N})
\end{align*} 
since $\os{\circ}{F}\col {\cal W}_n(\os{\circ}{Y})\lo {\cal W}_n(\os{\circ}{Y})$ 
is a homeomorphism. 
Because $\os{\circ}{Y}$ is reduced by (\ref{rema:rds}), 
(3) is nothing but a special case of (\ref{prop:ehe}) (3). 
\end{proof}

%\begin{coro}\label{coro:od} 
%Let the assumptions be before {\rm (\ref{coro:chob}) (1)}. 
%Then the following hold$:$
%\par 
%$(1)$ The following conditions are equivalent$:$
%\par 
%$({\rm a})$ $h^q(\os{\circ}{Y}/\kap)=1$.  
%\par 
%$({\rm b})$ $H^q(Y,B_n\Om^1_{Y/s})=0$ for a positive integer $n$.
%\par 
%$({\rm c})$ $H^q(Y,B_n\Om^1_{Y/s})=0$ for any positive integer $n$. 
%\par 
%\par 
%If $H^{q-1}(Y,{\cal O}_Y)=0$ for $q\geq 2$, 
%then these are equivalent to the following conditions$:$
%\par 
%$({\rm d})$
%$H^{q-1}(Y,B_n\Om^1_{Y/s})=0$ for a positive integer $n$.
%\par 
%$({\rm e})$ $H^{q-1}(Y,B_n\Om^1_{Y/s})=0$ for any positive integer $n$. 
%\par 
%$(2)$  
%Assume that $H^i(Y,{\cal O}_Y)=0$ $(0< i< q)$. 
%Then $h^q(\os{\circ}{Y}/\kap)=1$ if and only if 
%$H^i(Y,B_n\Om^1_{Y/s})=0$ for all $0\leq i\leq q$ and for any positive integer $n$. 
%\end{coro}
%\begin{proof} 
%(1): This is a restatement of (\ref{coro:ftc}).  
%\par 
%(2): By (1) it suffices to prove the implication ``only if''-part. 
%It suffices to prove that 
%$H^i(Y,{\cal W}_n({\cal O}_Y)/F)=0$ for $i<q-1$. 
%However we have already proved this in (\ref{prop:ehe}) (4). 
%\end{proof} 

\section{Log deformation theory vs log deformation theory 
with abrelative and relative Frobenus morphisms}\label{sec:latv}
In this section we give the log versions of two relative versions  
of Nori and Srinivas' deformation theory in \cite{ns} and \cite{sr}. 
In [loc.~cit.] they have considered the deformation theory with 
the absolute Frobenius endomorphisms over the spectrum of 
the Witt ring of finite length of a perfect field of characteristic $p>0$. 
In this section we construct the theory of log deformations with 
non well-known relative Frobenius morphisms 
instead of the absolute Frobenius endomorphisms  
over a more general base fine log scheme; 
we also remark that we can construct the theory of 
log deformations with well-known relative Frobenius morphisms. 
%By using this relative Frobenius morphisms, 
%we can recover the log deformation theory with absolute Frobenius endomorphisms
%when the underlying scheme of the base log scheme is perfect of characteristic $p>0$.   
\par 
In (\ref{theo:loc}) below 
we give an important correction of K.~Kato's deformation theory 
for log smooth schemes in \cite{klog1} (and \cite{kaf}) and we 
establish a relationship between our log deformation theories and 
the correction of his theory.  
First let us recall the following proposition due to K.~Kato. 

\begin{prop}[{\bf \cite[(3.9)]{klog1}}]\label{prop:eq}
Let 
\begin{equation*} 
\begin{CD} 
T_0@>{\subset}>>T\\
@V{s}VV @VVV \\
Z@>>> S
\end{CD} 
\tag{4.1.1}\label{cd:ttzs}
\end{equation*} 
be a commutative diagram of fine log schemes such that 
the upper horizontal morphism is an exact closed immersion 
defined by a square zero ideal sheaf ${\cal I}$ of ${\cal O}_T$. 
Let $P(s)$ be a Zariski sheaf on $T$ such that, for a log open subscheme 
$U$ of $T$, $P(s)(U)$ is the set of morphisms $\wt{s} \col U\lo Z$'s making 
the resulting two triangles commutative in {\rm (\ref{cd:ttzs})}, 
where we replace $T$, $T_0$ and $s$ by 
$U$, $U_0:=T_0\cap U$ and $s\vert_{U_0}$, respectively.  
Then $P(s)$ is a torsor under 
${\cal  H}{\it om}_{{\cal O}_{T_0}}(s^*(\Om^1_{Z/S}),{\cal I})$ on $T$. 
That is, for a morphism $g\col U\lo Z$ 
making the resulting two triangles commutative,  
there exists a bijection between 
the set of morphisms $h \col U\lo Z$'s making 
the resulting two triangles commutative in {\rm (\ref{cd:ttzs})} 
and the set $H^0(U,{\cal H}{\it om}_{{\cal O}_{T_0}}(s^*(\Om^1_{Z/S}),{\cal I}))
=H^0(Z,{\cal H}{\it om}_{{\cal O}_Z}(\Om^1_{Z/S},(s\vert_{U_0})_*({\cal I}\vert_U)))$. 
\end{prop}

\parno 
%Let $g\col U\lo Z$ be a morphism 
%making the resulting two triangles commutative,  
The bijection in (\ref{prop:eq}) is obtained by the following two maps 
$$h\lom (da\lom h^*(a)-g^*(a)\in {\cal I})\quad (a\in {\cal O}_Z)$$
and 
$$h\lom (d\log m\lom u_{h,g}(m)-1\in {\cal I}) \quad (m\in M_Z),$$
where $u_{h,g}(m)\in {\cal O}^*_T$ is a unique local section such 
that $h^*(m)=g^*(m)u_{h,g}(m)$. 
We denote the corresponding element to $h$ in 
${\rm Hom}_{{\cal O}_Z}(\Om^1_{Z/S},(s\vert_{U_0})_*({\cal I}\vert_U))$ by $h^*-g^*$. 
%(From the viewpoint of the multiplicative notation of 
%log structures, it may be better to denote this 
%element by $h^*g^*{}^{-1}$.) 

\begin{prop}\label{prop:ep}
Let the notations be as in {\rm (\ref{prop:eq})}. Then the following hold$:$ 
\par 
$(1)$ {\rm (cf.~\cite[III (5.6)]{sga1}, \cite[(2.11)]{ifh})}
In 
$$H^1(T_0,{\cal H}{\it om}_{{\cal O}_{T_0}}(s^*(\Om^1_{Z/S}),{\cal I})),$$  
there exists a canonical obstruction class of the existence of a morphism 
$T\lo Z$ making the diagrams of 
the two resulting triangles commutative in {\rm (\ref{cd:ttzs})}.  
If $s^*(\Om^1_{Z/S})$ is a locally free ${\cal O}_{T_0}$-module, 
then this group is equal to ${\rm Ext}^1_{T_0}(s^*(\Om^1_{Z/S}),{\cal I})$. 
\par 
$(2)$ 
In $$H^1(Z,{\cal H}{\it om}_{{\cal O}_Z}(\Om^1_{Z/S},s_*({\cal I}))),$$  
there exists a canonical obstruction class of the existence of a morphism 
$T\lo Z$ making the diagrams of 
the two resulting triangles commutative in {\rm (\ref{cd:ttzs})}.  
If $\Om^1_{Z/S}$ is a locally free ${\cal O}_Z$-module, 
then this group is equal to ${\rm Ext}^1_{Z}(\Om^1_{Z/S},s_*({\cal I}))$. 
\end{prop}
\begin{proof} 
(1): By (\ref{prop:eq}) this is only a special case of a general well-known result 
(see \cite[p.~70--71]{sga1}, \cite{gir}).  We can also give the proof of (1) which is similar
to the proof of (2) below. 
%We give the proof for convenience. 
%Because the proof of (2) is similar to that of (1), we prove only (1). 
\par 
(2): Let ${\cal U}:=\{Z_i\}_i$ be a log affine open covering of $Z$.
Let $\wt{s}_i\col T\lo Z_i$ be a local lift of a restriction of $s\col T_0\lo Z$ 
obtained by shrinking $T$. It is easy to see that 
$\{\wt{s}{}^*_j-\wt{s}{}_i^*\}_{ij}$ is an element of 
$Z^1({\cal U},{\cal H}{\it om}_{{\cal O}_Z}(\Om^1_{Z/S},s_*({\cal I})))$. 
Then we claim that the obstruction class stated in (\ref{prop:ep}) is the class 
$$\{\wt{s}{}^*_j-\wt{s}{}_i^*\}_{ij}\in 
\us{\cal U}{\vil}~H^1({\cal U},{\cal H}{\it om}_{{\cal O}_Z}(\Om^1_{Z/S},s_*({\cal I})))
=
H^1(Z,{\cal H}{\it om}_{{\cal O}_Z}(\Om^1_{Z/S},s_*({\cal I}))).$$ 
Indeed, if $\wt{s}_i$ is the restriction of a global lift $\wt{s}$ of $s$, 
then $\{\wt{s}{}^*_j-\wt{s}{}_i^*\}_{ij}=0$. 
Conversely, if it is the coboundary, then there exists a class 
$\{t_i\}$ $(t_i\in {\rm Hom}_{{\cal O}_{Z_i}}(\Om^1_{Z_i/S},s_*({\cal I})))$ 
such that $\wt{s}{}^*_j-\wt{s}{}_i^*=t_j-t_i$. Hence 
$\wt{s}{}_i^*-t_i=\wt{s}{}^*_j-t_j$ and $\wt{s}{}_i^*-t_i$'s patch together. 
These sections define a global morphism $T\lo Z$ over $S$. 
It is clear that this morphism is a lift of $s\col T_0\lo Z$ over $S$ since 
${\rm Im}(t_i)\subset s_*({\cal I})$.
We have to prove that the class  
$\{\wt{s}{}^*_j-\wt{s}{}_i^*\}_{ij}$ in 
$H^1(Z,{\cal H}{\it om}_{{\cal O}_Z}(\Om^1_{Z/S},s_*({\cal I})))$ 
is independent of the choice of ${\cal U}$. 
Assume that we are given another covering 
${\cal V}:=\{Z'_{i'}\}_{i'}$ and another local lift 
$\wt{s}{}'_{i'}\col T\lo Z'_{i'}$. Then, by considering the 
refinement ${\cal U}\cap {\cal V}:=\{Z_i\cap Z_{i'}\}_{ii'}$
of ${\cal U}$ and ${\cal V}$,  
$$\{\wt{s}{}^*_{i'}-\wt{s}{}^*_{i}\}_{ii'}
\in Z^1({\cal U}\cap {\cal V},
{\cal H}{\it om}_{{\cal O}_Z}(\Om^1_{Z/S},s_*({\cal I})))$$ 
gives us a 1-coboundary.  
This implies the desired independence. 
\par 
%By \cite[(3.10)]{klog1}, 
Assume that $\Om^1_{Z/S}$ is a locally free ${\cal O}_Z$-module. 
Then we obtain the equality 
$$H^1(Z,{\cal H}{\it om}_{{\cal O}_Z}(\Om^1_{Z/S},s_*({\cal I})))
={\rm Ext}^1_{Z}(\Om^1_{Z/S},s_*({\cal I}))$$ 
by the following spectral sequence: 
%for a scheme $Z$ and a flat ${\cal O}_Z$-module ${\cal F}$, 
$$E_2^{ij}=H^j(Z,{\cal E}{\it xt}^i_{{\cal O}_Z}({\cal F},{\cal G}))\Lo 
{\rm Ext}^{i+j}_Z({\cal F},{\cal G})\quad (i,j\in {\mab N})$$
for ${\cal O}_Z$-modules ${\cal F}$ and ${\cal G}$.  
\end{proof}

Let $S$ be a fine log scheme. 
Let $S_0\os{\sus}{\lo} S$ be an exact closed immersion of fine log schemes 
defined by a square zero ideal sheaf ${\cal I}$ of ${\cal O}_S$.  
Let $Y$ be a log smooth scheme over $S_0$. 
Recall that $\wt{Y}/S$ is called a log smooth lift 
%(or simply a lift) 
of $Y/S_0$ 
if $\wt{Y}$ is a log smooth scheme over $S$ such that 
$\wt{Y}\times_SS_0=Y$.

\par

\par 
Let $\wt{Y}/S$ be a log smooth lift of $Y/S_0$. 
As an immediate corollary of (\ref{prop:eq}), we obtain 
%By this corollary we can consider 
the $\del$ in (\ref{coro:lys}) below as an element of 
${\rm Hom}_{{\cal O}_Y}(\Om^1_{Y/S},{\cal I}{\cal O}_{\wt{Y}})$: 
%(indeed, to prove (\ref{prop:eq}), we have to prove the following, 
%which is left to the reader because it is easy): 

\begin{coro}\label{coro:lys}
Let $\wt{Y}/S$ be a log smooth lift of $Y/S_0$. 
Then the following hold$:$
\par 
$(1)$ Let $g$ be an automorphism of $\wt{Y}/S$ 
such that $g\vert_Y={\rm id}_Y$. 
Express $g^*(a)=a+\del(\ol{a})$ $(a\in {\cal O}_{\wt{Y}})$ with 
$\del(\ol{a})\in {\cal I}{\cal O}_{\wt{Y}}$. 
Here $\ol{a}$ is the image of $a$ in ${\cal O}_Y$.  
Then $\del \col {\cal O}_Y\lo {\cal I}{\cal O}_{\wt{Y}}$ is a derivation over ${\cal O}_S$. 
\par 
$(2)$ 
Express $g^*(m)=m(1+\del(\ol{m}))$ $(m\in M_{\wt{Y}})$ with 
$\del(\ol{m})\in {\cal I}{\cal O}_{\wt{Y}}$. 
Here $\ol{m}$ is the image of $m$ in $M_Y$.  
%Then $\del(mm')=\del(m)+\del(m')$. Indeed, since $g(m)=g(m)g(n)$, 
%$1+\pi \del(mm')=(1+\pi \del(m))(1+\pi \del(m'))$. 
Then $\del(\ol{m}\ol{m'})=\del(\ol{m})+\del(\ol{m'})$ $(m,m'\in M_{\wt{Y}})$. 
\par 
$(3)$ Let $\al \col M_Y\lo {\cal O}_Y$ 
be the structural morphism. 
Then $\al(m)\del(m)=\del(\al(m))$ $(m\in  M_Y)$. 
%\par 
%$(4)$ Assume that $S_{0,{\rm red}}$ is of characteristic $p>0$. 
%and that there exists an endomorphism $F$ of $Y$ which is a lift 
%of $F_{Y_{\rm red}}$.  Let $\wt{F}_i$ $(i=1,2)$ be a lift of $F$. 
\end{coro}

We also recall the following result due to K.~Kato:

\begin{prop}[{\bf \cite[(3.14) (2), (3)]{klog1}}]\label{prop:npp} 
%Then the following hold$:$
\par 
$(1)$ Let $\wt{Y}/S$ be a log smooth lift of $Y/S_0$. 
Let ${\rm Aut}_S(\wt{Y},Y)$ be the group of 
automorphisms $g\col \wt{Y}\os{\sim}{\lo}\wt{Y}$ over $S$  
such that $g\vert_Y={\rm id}_Y$. 
%Assume that ${\cal I}=0$. 
Then the morphism 
\begin{align*} 
{\rm Aut}_S(\wt{Y},Y)\owns g \lom \del \in 
{\rm Hom}_{{\cal O}_Y}(\Om^1_{Y/S_0},{\cal I}{\cal O}_{\wt{Y}}). 
\tag{4.4.1}\label{ali:omys}
\end{align*}
obtained by {\rm (\ref{coro:lys})} 
gives the following isomorphism of groups$:$ 
\begin{align*} 
{\rm Aut}_S(\wt{Y},Y)\os{\sim}{\lo}
H^0(Y,{\cal H}{\it om}_{{\cal O}_Y}(\Om^1_{Y/S_0},{\cal I}{\cal O}_{\wt{Y}}))
= 
{\rm Hom}_{{\cal O}_Y}(\Om^1_{Y/S_0},{\cal I}{\cal O}_{\wt{Y}}). 
\tag{4.4.2}\label{ali:oo0mys}
\end{align*}
\par 
$(2)$ 
Let ${\rm Lift}'_{Y/(S_0\subset S)}$ be the following sheaf 
\begin{align*} 
{\rm Lift}'_{Y/(S_0\subset S)}(U):=
\{{\rm isomorphism~classes~of~log~smooth~lifts~of}~U/S_0~{\rm over}~S\}
\end{align*}
for each log open subscheme $U$ of $Y$. 
%Then ${\rm Lift}_{Y/(S_0\subset S)}$ is a torsor under 
%${\cal H}{\it om}_{{\cal O}_Y}(\Om^1_{Y/S_0},{\cal I}{\cal O}_{\wt{Y}})$. 
If $Y/S_0$ has a lift $\wt{Y}/S$, 
then there exists the following $($natural$)$ bijection of sets$:$
\begin{align*} 
{\rm Lift}'_{Y/(S_0\subset S)}(Y)\os{\sim}{\lo}
H^1(Y,{\cal H}{\it om}_{{\cal O}_Y}(\Om^1_{Y/S_0},{\cal I}{\cal O}_{\wt{Y}}))
= 
{\rm Ext}_Y^1(\Om^1_{Y/S_0},{\cal I}{\cal O}_{\wt{Y}}). 
\tag{4.4.3}\label{ali:oom1ys}
\end{align*} 
\end{prop} 

\par 
Let us recall the map (\ref{ali:oom1ys}). 
%and the element ${\rm obs}_{Y/(S_0\subset S)}$ 
%as follows.
%\par 
%Let $g$ be an element of ${\rm Aut}_S(\wt{Y},Y)$. 
%The corresponding element to $g$ 
%in ${\rm Hom}_{{\cal O}_Y}(\Om^1_{Y/S},{\cal I}{\cal O}_{\wt{Y}})$ 
%is defined by the following maps 
%$$\Om^1_{Y/S}  \owns da\lom g^*(a)-a\in {\cal I}{\cal O}_{\wt{Y}} \quad (a\in {\cal O}_Y)$$  
%\parno 
%and 
%$$\Om^1_{Y/S}\owns d\log m\lom u_g(m)-1\in {\cal I}{\cal O}_{\wt{Y}}\quad (m\in M_Y),$$ 
%where $u_g(m)\in {\cal O}_{\wt{Y}}^*$ is a unique section such that $g^*(m)=mu_g(m)$. 
%\smallskip 
%To recall the equality we have to recall the following proposition: 

\par 
Let $\wt{Z}/S$ be a log smooth lift of $Y/S_0$. 
Let $\{\wt{U}_i\}_{i\in I}$ be a log affine open covering of $\wt{Z}$ 
%such that $U_i$ lifts to a log smooth lift $\wt{U}_i$ over $S$ 
such that there exists a morphism $g_i\col \wt{U}_i\lo \wt{Y}$ making  
the resulting two triangles  in 
\begin{equation*} 
\begin{CD} 
U_i@>{\subset}>>\wt{U}_i\\
@VVV @VVV \\
\wt{Y}@>>> S
\end{CD} 
\end{equation*} 
commutative. Here $U_i:=\wt{U}_i\cap Y$. 
Set ${\cal U}:=\{U_i\}_{i\in I}$. 
Then we have a section 
$$g_{ij}^*:=g^*_j-g^*_i\in 
{\cal H}{\it om}_{{\cal O}_Y}(\Om^1_{Y/S_0},{\cal I}{\cal O}_{\wt{Y}})(U_{ij}),$$
where $U_{ij}:=U_i\cap U_j$. These sections define an element of 
$\check{H}{}^1({\cal U},{\cal H}{\it om}_{{\cal O}_Y}(\Om^1_{Y/S_0},{\cal I}{\cal O}_{\wt{Y}}))$. 
Consequently we have an element of 
$H^1(Y,{\cal H}{\it om}_{{\cal O}_Y}(\Om^1_{Y/S_0},{\cal I}{\cal O}_{\wt{Y}}))$. 

\begin{rema}\label{rema:ktm}
Let the notations be as in \cite[(3.14) (4)]{klog1}. 
There is a mistake in [loc.~cit.]. 
The statement \cite[(3.14) (4)]{klog1} has no sense since 
a lift $\wt{X}$ of $(X,M,f)$ appears 
in the sufficient condition 
$$H^2(X,{\cal H}{\it om}_{{\cal O}_X}(\om^1_{X/Y},I{\cal O}_{\wt{X}}))=0$$ 
for an existence of a lift $\wt{X}$ of $(X,M,f)$. 
(If one claims that \cite[(3.14) (4)]{klog1} has a sense, one has to prove that 
the sheaf $I{\cal O}_{\wt{X}}$ on $X_{\rm et}$ is independent of the choice of $\wt{X}$.)
\end{rema} 

\par 
To make \cite[(3.14) (3)]{klog1}(=(4.4) (2)) better and correct \cite[(3.14) (4)]{klog1}, 
more generally to define an obstruction class of 
a lift of $Y/S_0$ over $S$, we moreover assume that $Y/S_0$ is {\it integral}.  
In the following we always assume this.  
That is, $Y$ is assumed to be a log smooth integral scheme over $S_0$. 
We say that $\wt{Y}/S$ is a {\it log smooth integral lift} (or simply a {\it lift}) of $Y/S_0$ 
if $\wt{Y}$ is a log smooth integral scheme over $S$ such that 
$\wt{Y}\times_SS_0=Y$. 

\begin{rema}\label{rema:mgd}
The obvious analogues of (\ref{prop:npp}) (1) and (2) hold for 
a log smooth integral scheme $Y/S_0$ by the proof of \cite[(3.14)]{klog1}. 
\end{rema} 

\par 
The following includes an important correction of \cite[(3.14) (4)]{klog1} and 
\cite[(8.6)]{kaf}. 
%precisely:
This is a log version of \cite[III (6.3)]{sga1} and 
a generalization of \cite[(2.2)]{kwn}.

\begin{theo}\label{theo:loc}
Let the notations be as above. 
%Furthermore assume that the morphism $Y\lo S_0$ is log smooth and integral. 
For a log scheme $Z$ over $S_0$, set 
${\cal T}_{Z/S_0}:={\cal H}{\it om}_{{\cal O}_Z}(\Om^1_{Z/S_0},{\cal O}_Z)$. 
Then the following hold$:$
\par 
$(1)$  Let $U$ be a log open subscheme of $Y$ and 
let $\wt{U}$ be a log smooth integral lift of $U$ over $S$. 
Then 
\begin{align*} 
{\cal H}{\it om}_{{\cal O}_U}
(\Om^1_{U/S_0},{\cal I}{\cal O}_{\wt{U}})
={\cal T}_{U/S_0}\otimes_{{\cal O}_{S_0}}{\cal I}.  
\tag{4.7.1}\label{ali:deuel} 
\end{align*} 
\par
$(2)$ Assume that $\os{\circ}{Y}$ is separated.  
Then, in  
\begin{align*} 
H^2(Y,{\cal T}_{Y/S_0})\otimes_{{\cal O}_{S_0}}{\cal I},
\tag{4.7.2}\label{ali:deuomel} 
\end{align*} 
there exists a canonical obstruction class 
${\rm obs}_{Y/(S_0\subset S)}$ of a lift of $Y/S_0$ over $S$. 
\par 
Let ${\rm Lift}_{Y/(S_0\subset S)}$ be the following sheaf 
\begin{align*} 
{\rm Lift}_{Y/(S_0\subset S)}(U):=
\{{\rm isomorphism~classes~of~log~smooth~integral~lifts~of}~U/S_0~{\rm over}~S\}
\end{align*}
for each log open subscheme $U$ of $Y$. 
If the obstruction class vanishes, 
then there exists the following $($natural$)$ bijection of sets:  
\begin{align*} 
{\rm Lift}_{Y/(S_0\subset S)}(Y)
\os{\sim}{\lo} H^1(Y,{\cal T}_{Y/S_0})\otimes_{{\cal O}_{S_0}}{\cal I}.
\end{align*}  
\par 
$(3)$ Let the assumption and the notations be as in {\rm (2)}. 
If $\os{\circ}{S}$ is an affine scheme, say, ${\rm Spec}(A)$, if we set 
$I:=\Gam(\os{\circ}{S},{\cal I})$ and $A_0:=A/I$ and 
if $I$ is a flat $A_0$-module, 
then the cohomology {\rm (\ref{ali:deuomel})} is equal to 
$H^2(Y,{\cal T}_{Y/S_0})\otimes_{A_0}I
= {\rm Ext}_Y^2(\Om^1_{Y/S_0}, {\cal O}_{Y})\otimes_{A_0}I$.
\end{theo} 
\begin{proof} 
(1): (Because we have to give a comment in (\ref{rems:lifk}) (2) below, 
we have to give the following very easy proof of (1).)
%By (\ref{rema:mgd}) (cf.~(\ref{ali:omys})), 
%it suffices to prove that the following equality holds:  
%\begin{align*} 
%{\cal H}{\it om}_{{\cal O}_U}
%(\Om^1_{U/S_0},{\cal I}{\cal O}_{\wt{U}})=
%{\cal H}{\it om}_{{\cal O}_U}
%(\Om^1_{U/S_0},{\cal O}_U)\otimes_{{\cal O}_{S_0}}{\cal I}. 
%\tag{4.7.1}\label{ali:deuel} 
%\end{align*} 
Consider the following obvious exact sequence: 
\begin{align*} 
0\lo {\cal I}\lo {\cal O}_S\lo {\cal O}_{S_0}\lo 0. 
\end{align*} 
Taking the tensorization ${\cal O}_{\wt{U}}\otimes_{{\cal O}_S}$ of this exact sequence 
and noting that $\os{\circ}{\wt{U}}\lo \os{\circ}{S}$ is flat (\cite[(4.5)]{klog1}), 
we see that 
${\cal I}{\cal O}_{\wt{U}}=
{\cal I}\otimes_{{\cal O}_S}{\cal O}_{\wt{U}}=
{\cal I}\otimes_{{\cal O}_{S_0}}{\cal O}_{U}$. 
Hence we obtain the following equalities: 
\begin{align*} 
{\cal H}{\it om}_{{\cal O}_U}
(\Om^1_{U/S_0},{\cal I}{\cal O}_{\wt{U}})&=
{\cal H}{\it om}_{{\cal O}_U}
(\Om^1_{U/S_0},{\cal O}_U\otimes_{{\cal O}_{S_0}}{\cal I})\\
&={\cal H}{\it om}_{{\cal O}_U}(\Om^1_{U/S_0},{\cal O}_U)
\otimes_{{\cal O}_{S_0}}{\cal I} 
\end{align*}  
since $\Om^1_{U/S_0}$  is a finite locally free ${\cal O}_U$-module 
(\cite[(3.10)]{klog1}). 
\par 
(2): We construct the obstruction class 
${\rm obs}_{Y/(S_0\subset S)}$ as in \cite[p.~79]{sga1}. 
Though the statement \cite[III (6.3)]{sga1} is well-known, 
the proof of it is not well-understood at all. (See (\ref{rems:lifk}) (1) below.)
Because we cannot find a detailed proof of (2) using cocycles 
in references, e.~g., 
%we cannot be satisfied with the explanations in the proofs of 
\cite[III (6.3)]{sga1}, \cite[(2.12)]{ifh}, 
\cite[(2.2)]{kwn} nor \cite[(8.6)]{kaf} unfortunately, 
%and because we cannot find a detailed proof of (\ref{theo:loc}) using cocycles 
%in references, 
we have to give the detailed proof of (\ref{theo:loc}) as follows. 
%Because it seems to me that 
%it is not easy to understand the proof in \cite[p.~79]{sga1} for a non-specialist, 
%we give the perfect proof. 
\par 
Let ${\cal U}:=\{U_i\}_{i\in I}$ be a log affine open covering of $Y$ such that 
$U_i$ has a log smooth integral lift $\wt{U}_i$ over $S$. 
Set $\wt{U}_{ij}
:=\wt{U}_i\vert_{U_{ij}}$ 
%(This convention of the index of $\wt{U}_{ij}$ is different from that in \cite[III (6.3)]{sga1}.)
(Note that we cannot use \cite[(3.14) (1)]{klog1} for {\it any} log affine open subscheme of $Y$ 
because we cannot use \cite[(3.14) (4)]{klog1}. 
However, by the proof of \cite[(3.14)]{klog1} and \cite[(4.1) (ii)]{klog1}, 
we have the $U_i$ over $S_0$ and the $\wt{U}_i$ over $S$.)
Because $\os{\circ}{Y}$ is separated, 
$\os{\circ}{U}_{ij}:=\os{\circ}{U}_i\cap \os{\circ}{U}_j$ is affine. 
Since $\wt{U}_{ji}$ (resp.~$\wt{U}_{ij}$) is log smooth over $S$, 
there exists a morphism 
$g_{ij}\col \wt{U}_{ij}\lo \wt{U}_{ji}$ 
(resp.~$h_{ij}\col \wt{U}_{ij} \lo \wt{U}_{ji}$) over $S$ 
which is an extension of ${\rm id}_{U_{ij}}\col U_{ij}\os{\sim}{\lo} U_{ji}$. 
Since $g_{ij}\circ h_{ij}\in {\rm End}_S(\wt{U}_{ij})$ 
which is an extension of ${\rm id}_{U_{ij}}$, 
$g_{ij}\circ h_{ij}\in {\rm Aut}_S(\wt{U}_{ij},U_{ij})$. 
In particular, $g_{ij}$ is an isomorphism. 
(This convention of the index of $g_{ij}$ is different from that in 
\cite[III (6.3)]{sga1} but the same as that in \cite[p.~113]{ifh}.
For the existence of the isomorphism $g_{ij}$, 
we do not have to use the vanishing of 
$H^1(U_{ij},{\cal H}{\it om}_{{\cal O}_{U_{ij}}}(\Om^1_{U_{ij}/S_0},{\cal O}_{U_{ij}})
\otimes_{{\cal O}_{S_0}}{\cal I})$, 
though it has been used in the proof of \cite[III (6.3)]{sga1}.)
If one wants, one can assume that $g_{ij}=g_{ji}^{-1}$ as in \cite{ns} because 
we can endow $I$ with a total order. 
(We shall use this equality in the proof of (\ref{theo:ts}) (4) below.) 
Set $U_{ijk}:=U_i\cap U_j \cap U_k$ and 
$\wt{U}_{ijk}:=\wt{U}_i\vert_{U_{ijk}}$. 
Set $g_{ijk}:=
%g_{ij}\circ g_{jk}\circ g^{-1}_{ik}
g_{ik}^{-1}g_{jk}g_{ij}\in {\rm Aut}_S(\wt{U}_{ijk},U_{ijk})$. 
%(This convention of the index of $g_{ijk}$ is different from that in \cite[III (6.3)]{sga1}.)
Consider a section 
\begin{align*}
{\mathfrak g}_{ijk}:=g^*_{ijk}-{\rm id}^*_{\wt{U}_{ijk}}
\in 
{\rm Hom}_{{\cal O}_{U_{ijk}}}
(\Om^1_{U_{ijk}/S_0},{\cal I}{\cal O}_{\wt{U}_{ijk}}). 
\tag{4.7.3}\label{ali:delrel} 
\end{align*} 
%where $\del_{ij}$ is the $\del$ for $g_{ij}$. 
%(g_{ki}^{-1}g_{kj}g_{ji})^*(a)-a=
%(g_{kj}g_{ji})^*(a-\del_{ki}(a))-a
%=g_{ji}^*(a-\del_{ki}(a)+\del_{kj}(a-\del_{ki}))-a=
%g_{ji}^*(a-\del_{ki}(a)+\del_{kj}(a))-a
%=\del_{ji}+\del_{kj}-\del_{ki}(a)=\del_{ij}+\del_{jk}-\del_{ik}
%It is easy to prove that the following equality holds:  
By (\ref{ali:deuel}) we obtain the following equality: 
\begin{align*} 
{\cal H}{\it om}_{{\cal O}_{U_{ijk}}}
(\Om^1_{U_{ijk}/S_0},{\cal I}{\cal O}_{\wt{U}_{ijk}})=
{\cal H}{\it om}_{{\cal O}_{U_{ijk}}}
(\Om^1_{U_{ijk}/S_0},{\cal O}_{U_{ijk}})\otimes_{{\cal O}_{S_0}}{\cal I}. 
\end{align*} 
%Indeed, consider the following obvious exact sequence: 
%\begin{align*} 
%0\lo {\cal I}\lo {\cal O}_S\lo {\cal O}_{S_0}\lo 0. 
%\end{align*} 
%Taking the tensorization ${\cal O}_{\wt{U}_{ijk}}\otimes_{{\cal O}_S}$ 
%of this exact sequence 
%and noting that $\os{\circ}{\wt{U}}_{ijk}\lo \os{\circ}{S}$ is flat (\cite[(4.5)]{klog1}), 
%we see that 
%${\cal I}{\cal O}_{{\wt{U}}_{ijk}}=
%{\cal I}\otimes_{{\cal O}_S}{\cal O}_{{\wt{U}}_{ijk}}=
%{\cal I}\otimes_{{\cal O}_{S_0}}{\cal O}_{U_{ijk}}$. 
%Hence 
%\begin{align*} 
%{\cal H}{\it om}_{{\cal O}_{U_{ijk}}}
%(\Om^1_{U_{ijk}/S_0},{\cal I}{\cal O}_{\wt{U}_{ijk}})&=
%{\cal H}{\it om}_{{\cal O}_{U_{ijk}}}
%(\Om^1_{U_{ijk}/S_0},{\cal O}_{U_{ijk}}\otimes_{{\cal O}_{S_0}}{\cal I})\\
%&={\cal H}{\it om}_{{\cal O}_{U_{ijk}}}
%(\Om^1_{U_{ijk}/S_0},{\cal O}_{U_{ijk}})\otimes_{{\cal O}_{S_0}}{\cal I} 
%\end{align*} 
%by using the locally freeness of $\Om^1_{U_{ijk}/S_0}$ (\cite[(3.10)]{klog1}). 
Consequently 
$${\mathfrak g}_{ijk}\in 
\Gam(U_{ijk},{\cal H}{\it om}_{{\cal O}_Y}(\Om^1_{Y/S_0},{\cal O}_{Y})
\otimes_{{\cal O}_{S_0}}{\cal I}).$$ 
In fact, we can check 
$${\mathfrak g}:=({\mathfrak g}_{ijk})\in 
Z^2({\cal U},{\cal H}{\it om}_{{\cal O}_Y}(\Om^1_{Y/S_0},{\cal O}_{Y})
\otimes_{{\cal O}_{S_0}}{\cal I})$$ 
%by the formula (\ref{ali:delrel}) 
(cf.~\cite[p.~79]{sga1}). 
Indeed, because 
$(\partial ({\mathfrak g}))_{ijkl}=
{\mathfrak g}_{jkl}-{\mathfrak g}_{ikl}+{\mathfrak g}_{ijl}-{\mathfrak g}_{ijk}$, 
it suffices to prove that 
${\mathfrak g}_{jkl}={\mathfrak g}_{ikl}+{\mathfrak g}_{ijk}-{\mathfrak g}_{ijl}$. 
The element of ${\rm Aut}_S(\wt{U}_{jkl}\vert_{U_{ijkl}})$ corresponding to 
the right hand side of the equality above is equal to 
\begin{align*} 
(g_{il}^{-1}g_{kl}g_{ik})(g_{ik}^{-1}g_{jk}g_{ij})(g_{ij}^{-1}g_{jl}^{-1}g_{il})
& =(g_{il}^{-1}g_{kl}g_{jk})(g_{jl}^{-1}g_{il})
\tag{4.7.4}\label{ali:dgggel} \\
&=(g^{-1}_{il}g_{jl})(g_{jl}^{-1}g_{kl}g_{jk})(g^{-1}_{il}g_{jl})^{-1}. \\
%&=g_{jl}^{-1}g_{kl}g_{jk}=g_{jkl}.
\end{align*}  
Hence 
\begin{align*} 
&\{(g_{il}^{-1}g_{kl}g_{ik})(g_{ik}^{-1}g_{jk}g_{ij})(g_{ij}^{-1}g_{jl}^{-1}g_{il})\}^*-{\rm id}_{\wt{U}_{jkli}}^*\tag{4.7.5}\label{ali:dgidrel} \\
&=((g^{-1}_{il}g_{jl})^{-1})^*((g_{jl}^{-1}g_{kl}g_{jk})^*-{\rm id}_{\wt{U}_{jkli}}^*)(g^{-1}_{il}g_{jl})^*
=(g_{jl}^{-1}g_{kl}g_{jk})^*-{\rm id}_{\wt{U}_{jkli}}^*.
\end{align*} 
Here, to obtain the second equality above, 
we have used the lemma (\ref{lemm:elin}) below. 
%because ${\rm Aut}_S(\wt{U}_{ij},U_{ij})$ is an abelian group. 
Now we have the desired element 
$${\rm obs}_{Y/(S_0\subset S)}:={\rm the}~{\rm class}~{\rm of}~{\mathfrak g}$$ 
in 
\begin{align*} 
\check{H}{}^2({\cal U},{\cal H}{\it om}_{{\cal O}_Y}(\Om^1_{Y/S_0},{\cal O}_{Y})
\otimes_{{\cal O}_{S_0}}{\cal I})
=H^2(Y,{\cal H}{\it om}_{{\cal O}_Y}(\Om^1_{Y/S_0},{\cal O}_{Y})
\otimes_{{\cal O}_{S_0}}{\cal I}).
\tag{4.7.6}\label{ali:ooys} 
\end{align*} 
Here we have used the assumption on the separatedness of $\os{\circ}{Y}$ 
to obtain the equality above. 
\par 
We claim that ${\mathfrak g}$ is independent of the choice of $g_{ij}$'s. 
Let $g'_{ij}\col \wt{U}_{ij}\os{\sim}{\lo} \wt{U}_{ji}$ be another isomorphism 
which is a lift of ${\rm id}_{U_{ij}}$. 
Then $g'_{ij}g^{-1}_{ij}$ is an element of ${\rm Aut}_S(\wt{U}_{ji},U_{ji})$. 
Let $\del_{ij}$ be the $\del$ corresponding to $g'_{ij}g^{-1}_{ij}\in  
{\cal H}{\it om}_{{\cal O}_{U_{ji}}}(\Om^1_{U_{ji}/S_0},{\cal O}_{U_{ji}})
\otimes_{{\cal O}_{S_0}}{\cal I}$: 
$g'{}^*_{\! ij}(a)=g_{ij}^*(a+\del_{ij}(\ol{a}))$ $(a\in {\cal O}_{\wt{U}_{ji}})$, 
$g'{}^*_{\! ij}(m)=g_{ij}^*(m(1+\del_{ij}(\ol{m})))$ $(m\in M_{\wt{U}_{ji}})$. 
Since $g_{ij}\vert_{U_{ij}}={\rm id}_{U_{ij}}$ and since $g_{ij}$ is a morphism over $S$, 
$g'{}^*_{\! ij}(a)=g_{ij}^*(a)+\del_{ij}(\ol{a})$ and  
$g'{}^*_{\! ij}(m)=g_{ij}^*(m)(1+\del_{ij}(\ol{m}))$. 
Using these relations and (\ref{lemm:elin}) below 
and making simple calculations, we obtain an equality 
${\mathfrak g}':=(g'_{ijk})={\mathfrak g}+\partial ((\del_{ij}))$. 
Indeed we have the following equations:  
\begin{align*} 
g'{}^*_{\! ij}g'{}^*_{\! jk}g'{}^{*-1}_{\! ik}(a)&=g'{}^*_{\! ij}g'{}^*_{\! jk}(g_{ik}^{*-1}(a)-\del_{ik}(\ol{a}))
=\cdots \tag{4.7.7}\label{ali:delael} \\
&=g^*_{\! ij}g^*_{\! jk}(g_{ik}^{*-1}(a))+\del_{ij}(\ol{a})+\del_{jk}(\ol{a})-\del_{ik}(\ol{a})
\quad (a\in {\cal O}_{\wt{U}_{ijk}}). 
\end{align*} 
Similarly we have the following equation:  
\begin{align*} 
g'{}^*_{\! ij}g'{}^*_{\! jk}g'{}^{*-1}_{\! ik}(m)=
g^*_{\! ij}g^*_{\! jk}(g_{ik}^{*-1}(m))(1+\del_{ij}(\ol{m})+\del_{jk}(\ol{m})-\del_{ik}(\ol{m})) 
\quad (m\in M_{\wt{U}_{ijk}}). 
\tag{4.7.8}\label{ali:delrijkel} 
\end{align*} 
Hence ${\mathfrak g}'={\mathfrak g}+\partial ((\del_{ij}))$. 
This shows that our claim holds. 
%Here, to obtain the equality above, we have used 
%(\ref{lemm:elin})  (2) and (3) below. 
%the equality ${\cal I}^2=0$.  
\par 
We have to show that ${\mathfrak g}$ 
is independent of the choice of the lift $\wt{U}_i$ of $U_i$ over $S$. 
Let $\wt{V}_i$ be another lift of $U_i$ over $S$. 
Then, as shown in the second paragraph in the proof of (2), 
there exists an isomorphism $g_i\col \wt{U}_i\os{\sim}{\lo} \wt{V}_i$ over $S$ 
such that $g_i\vert_{U_i}={\rm id}_{U_i}$. 
Set $g'_{ij}:=(g_j\vert_{\wt{V}_{ji}})g_{ij}(g_i^{-1}\vert_{\wt{V}_{ij}})
\col \wt{V}_{ij}\os{\sim}{\lo} \wt{V}_{ji}$. 
Then it is easy to check that 
\begin{align*} 
g'_{ijk}=g_ig_{ijk}g_i^{-1}. 
\tag{4.7.9}\label{ali:dgrijkel} 
\end{align*} 
Let ${\mathfrak g}'_{ijk}$ be the analogue of ${\mathfrak g}_{ijk}$ 
for $g'_{ij}$. 
Because 
${\cal I}{\cal O}_{{\wt{U}}_{ijk}}=
{\cal I}\otimes_{{\cal O}_S}{\cal O}_{{\wt{U}}_{ijk}}=
{\cal I}\otimes_{{\cal O}_{S_0}}{\cal O}_{U_{ijk}}=
{\cal I}\otimes_{{\cal O}_S}{\cal O}_{{\wt{V}}_{ijk}}
={\cal I}{\cal O}_{{\wt{V}}_{ijk}}$, 
we have an equality ${\mathfrak g}_{ijk}={\mathfrak g}'_{ijk}$ by 
(\ref{ali:dgrijkel}) and (\ref{lemm:elin}) below. 
This implies that ${\mathfrak g}$ 
is independent of the choice of the lift $\wt{U}_i$.
\par 
Next we claim that the class ${\mathfrak g}$ is independent of the choice of 
the open covering ${\cal U}$. 
Since two log affine open coverings of ${\cal U}$ has a log affine refinement, 
we consider a refinement ${\cal V}:=\{V_{i'}\}$ of ${\cal U}$ with 
a morphism $\tau \col \{i'\}\lo \{i\}$ such that $\os{\circ}{V}_{i'}$ 
is affine and such that $V_{i'}\subset U_{\tau(i')}$. 
Set $\wt{V}_{i'}:=\wt{U}_{\tau(i')}\vert_{V_{i'}}$ over $S$. 
It suffices to prove that 
${\mathfrak g}$ in 
$\check{H}{}^2({\cal U},{\cal H}{\it om}_{{\cal O}_Y}(\Om^1_{Y/S_0},{\cal O}_{Y})
\otimes_{{\cal O}_{S_0}}{\cal I})$ is mapped to 
${\mathfrak g}$ in 
$\check{H}{}^2({\cal V},{\cal H}{\it om}_{{\cal O}_Y}(\Om^1_{Y/S_0},{\cal O}_{Y})
\otimes_{{\cal O}_{S_0}}{\cal I})$.  
This is clear since the isomorphism 
$g_{\tau(i')\tau(j')}\col \wt{U}_{\tau(i')\tau(j')}\os{\sim}{\lo}  
\wt{U}_{\tau(j')\tau(i')}$ induces an isomorphism 
$g_{i'j'}\col \wt{V}_{i'j'}\os{\sim}{\lo} \wt{V}_{j'i'}$. 
%By considering a refinement of two open coverings, 
%it is easy to see that this obstruction class is independent of the choice of 
%the open covering above. 
\par 
If ${\mathfrak g}$ is coboundary, then there exists an element 
$h_{ij}$ of ${\rm Aut}_S(\wt{U}_{ij},U_{ij})$ 
such that $\{g'_{ij}\}=\{g_{ij}h_{ij}\}$ satisfies the transitivity condition 
$g'_{ik}=g'_{jk}g'_{ij}$ as in \cite[p.~79 (2)]{sga1}. 
Consequently we have a lift $\wt{Y}/S$ of $Y/S_0$. 
\par 
The last statement in (2) follows from (1) and (\ref{prop:npp}) (2). 
\par 
(3): (3) immediately follows from (\ref{ali:ooys}) and 
the assumption of the flatness of $I$ (cf.~\cite[p.~75]{sga1}). 
\end{proof} 

\begin{lemm}\label{lemm:elin}
Let $F\in {\rm Hom}_{{\cal O}_{U_{jkli}}}(\Om^1_{U_{jkli}/S_0},
{\cal O}_{U_{jkli}})\otimes_{{\cal O}_{S_0}}{\cal I}$ 
be the element corresponding to an element 
$g\in {\rm Aut}_S(\wt{U}_{jkli},U_{jkli})$.  
Let $h\col \wt{U}_{jil}\os{\sim}{\lo} \wt{U}_{ijl}$ 
%$($resp.~$k\col \wt{U}_{ijl}\os{\sim}{\lo} \wt{U}_{jil})$ 
be an isomorphism over $S$ 
such that $h\vert_{U_{jilk}}={\rm id}_{U_{jilk}}$ 
%$($resp.~k\vert_{U_{ijl}}={\rm id}_{U_{ijl}})$. 
Then 
%the following hold$:$ \par $(1)$ 
$(h^{-1}\vert_{\wt{U}_{jilk}})^*F(h\vert_{\wt{U}_{jilk}})^*=F$.  
%\par 
%$(2)$ $h^*(\del_{ij}(a))=\del_{ij}(a)$ $(a\in {\cal O}_{U_{ijl}})$, $h^*(\del_{ij}(m))=\del_{ij}(m)$ 
%$(m\in M_{U_{ijl}})$. 
%\par 
%$(3)$ $\del_{ij}(k^*(a))=\del_{ij}(a)$ $(a\in {\cal O}_{U_{jil}})$, 
%$\del_{ij}(k^*(m))=\del_{ij}(m)$ $(m\in M_{U_{jil}})$.
\end{lemm}
\begin{proof} 
This is obvious since 
$h\vert_{U_{ijl}}={\rm id}_{U_{ijl}}$ 
%(resp.~$k\vert_{U_{ijl}}={\rm id}_{U_{ijl}}$) 
and $h$ is a morphism over $S$. 
\end{proof} 

\begin{rema}\label{rems:lifk}
(1) 
%We suspect that 
It is doubtful whether 
%most 
arithmetic or algebraic geometers 
%(at least most Japanese arithmetic or algebraic geometers) 
%do not 
can read the proof of the very well-known result 
\cite[III (6.3)]{sga1} rigorously and 
can
%not 
give the detailed proof of it 
%as in the proof of (\ref{theo:loc}) 
because to give the precise proof of it is tiresome and hard 
as shown in the proof of (\ref{theo:loc}) 
and because it needs quite unacceptable patience.  
(We have never seen 
(\ref{ali:dgggel}), (\ref{ali:dgidrel}), (\ref{ali:delael}), (\ref{ali:delrijkel})  
and (\ref{lemm:elin}) in other references.) 
For this reason, there exist the mistakes pointed out in 
(\ref{rema:ktm}) and (2) below and no one except us 
has noticed the mistakes.  
\par 
The statements  
\cite[(2.2) (3)]{kwn} and \cite[(8.6) 3]{kaf} seems obscure because 
Kawamata-Namikawa and F.~Kato have not constructed the obstruction 
class in their article (they have only claimed that 
the construction is the same as that of \cite[III (6.3)]{sga1}) 
and because we cannot understand whether 
the obstruction classes in their article is canonical. 
%To be worse, there exist the mistakes pointed out in 
%(\ref{rema:ktm}) and (2) below and no one except us 
%has noticed the mistakes.  
\par 
(2) 
Let the notations be as in \cite[p.~338]{kaf}. 
We do not understand why there exists an isomorphism 
$I\cdot {\cal O}_{\wt{X}{}'}\os{\sim}{\lo} I\otimes_A {\cal O}_{\wt{X}{}'}$ in [loc.~cit.]. 
Indeed, there exist a lot of counter-examples for this isomorphism. 
For example, log blow ups by Fujiwara-Kato 
(\cite{fk}, \cite{nz}) give us counter-examples: 
the underlying morphisms of log blow ups are not necessarily flat. 
One of the simplest examples 
is as follows. 
\par 
Let $K$ be a field of any characteristic.  
Set $A=K[x_1,x_2]$ and $B=K[x_1,x_2,t]/(x_2-tx_1)=K[x_1,t]$.  
Endow ${\rm Spec}(A)$ (resp.~${\rm Spec}(B)$) 
with a log structure associated to a morphism 
${\mab N}^{\oplus 2}\owns e_i\lom x_i\in A$ 
(resp.~${\mab N}^{\oplus 2}\owns e_1\lom x_1, e_2\lom t\in B$), 
where $e_i$ $(i=1,2)$ is a canonical basis of ${\mab N}^{\oplus 2}$. 
Let $T$ (resp.~$Y$) be the resulting log scheme. 
Let $A\lo B$ be a morphism defined by $x_1\lom x_1$ and $x_2\lom x_2$. 
Let ${\mab N}^{\oplus 2}\lo {\mab N}^{\oplus 2}$ be a morphism defined by 
the following: $e_1\lom e_1$, $e_2\lom e_1+e_2$. 
Then we have a morphism $Y\lo T$. 
This morphism is log \'{e}tale by the criterion of K.~Kato (\cite[(3.5)]{klog1}). 
Let $S$ be an exact closed subscheme of $T$ defined by 
the ideal sheaf $(x_1^2,x_2)$. Set $\wt{X}:=Y\times_TS$. 
Then the projection $\wt{X}\lo S$ is log \'{e}tale since log \'{e}tale morphisms 
are stable under base changes. 
Let $S_0$ be an exact closed subscheme of $S$ defined by the ideal sheaf 
${\cal I}:=(x_1)$. 
Then the global sections of ${\cal I}{\cal O}_{\wt{X}}$ are equal to 
$x_1(K[x_1,t]/(x_1^2,tx_1))=Kx_1$. 
On the other hand, the global sections of 
${\cal I}\otimes_{{\cal O}_S}{\cal O}_{\wt{X}}$ 
are equal to 
\begin{align*} 
&(x_1K[x_1])/(x_1^2)\otimes_{K[x_1]/(x_1^2)}K[x_1,t]/(x_1^2,tx_1)
=Kx_1\otimes_{K[x_1]/(x_1^2)}(K[x_1]/(x_1^2))[t]/(tx_1)\\
&\simeq K\otimes_{K[x_1]/(x_1^2)}(K[x_1]/(x_1^2))[t]/(tx_1)
=(K[x_1]/(x_1^2))[t]/(tx_1,x_1)=K[t].
\end{align*}  
Hence ${\cal I}{\cal O}_{\wt{X}}$ cannot be isomorphic to 
${\cal I}\otimes_{{\cal O}_S}{\cal O}_{\wt{X}}$.  
In particular, the structural morphism $\os{\circ}{\wt{X}}\lo \os{\circ}{S}$ is 
{\it not} flat. This is a counter example of the claim after \cite[(4.1)]{kfl}: 
``underlying morphisms of log smooth liftings are flat''. 
\par 
By virtue of this remark, the title 
``Log smooth deformation theory''
of \cite{kaf} had to be replaced by 
``Log smooth integral deformation theory''.
\par 
(3) Once one proves (\ref{prop:npp}) (1) and (\ref{theo:loc}) (1), (\ref{theo:loc}) (2) 
is a formal consequence obtained by a general theory  
as in \cite[VII (1.2.2)]{gir} without using the assumption of the separatedness 
in  (\ref{theo:loc}) (2). However, in this article, 
we shall use the explicit description of 
the obstruction class in the proof of (\ref{theo:loc}) (2) 
(see the proof of (\ref{theo:ts}) (4)). 
\par 
(4) In \cite[(5.6), (8,36)]{ols} Olsson has already obtained (\ref{theo:loc}) 
by using his theory of log cotanget complexes. 
\end{rema}

As already stated, in the following 
we always assume that the log smooth morphism 
$Y\lo S_0$ is integral. 

\par 
For a log scheme $Z$, let $Z_{\rm red}$ be the log exact closed subscheme 
of $Z$ whose underlying scheme is $\os{\circ}{Z}_{\rm red}$ 
and whose log structure is the inverse image of that of $Z$. 
Assume that $\os{\circ}{Z}$ is of characteristic $p>0$. 
Let $F_Z\col Z\lo Z$ be the $p$-th power Frobenius endomorphism. 
Let $e$ be a fixed positive integer. Set $q:=p^e$. 
Set $Z^{[q]}:=Z\times_{\os{\circ}{Z},\os{\circ}{F}{}^e_{\!Z}}\os{\circ}{Z}$.  
This is different from $Z^{\{q\}}:=Z\times_{Z,F^e_Z}Z=Z$, though 
$(Z^{[q]})^{\circ}=\os{\circ}{Z}=(Z^{\{q\}})^{\circ}$. 
Then we have the following two natural morphisms 
\begin{align*} 
Z\lo Z^{[q]}\quad {\rm and} \quad Z^{[q]}\lo Z. 
\end{align*}
%The first morphism is the morphism 
%induced by the multiplication by $q$ of the log structures 
%and the second morphism is the morphism 
%induced by the $q$-th power endomorphism of the structure sheaf. 
We denote the first morphism by $F^{[e]}_{Z/\os{\circ}{Z}}$.   
Let $W$ be a fine log scheme over $Z$. 
Set $`W:=W\times_ZZ^{[q]}$. 
Then we have the following commutative diagram 
\begin{equation*} 
\begin{CD} 
W@>>> `W@>>> W\\
@VVV @VVV @VVV\\
Z@>{F^{[e]}_{Z/\os{\circ}{Z}}}>>Z^{[q]}@>>> Z. 
\end{CD}
\tag{4.9.1}\label{cd:wzf}
\end{equation*}
Here the upper (resp.~lower) horizontal composite endomorphism is 
the $q$-th power Frobenius endomorphism of $W$ (resp.~$Z$). 
We call $F^{[e]}_{Z/\os{\circ}{Z}}$ the $e$-{\it times iterated} 
{\it abrelative Frobenius morphism of a base log scheme}.   
(The adjective ``abrelative'' is a coined word which implies ``absolute and relative'' or 
``far from being relative''.)
We call the morphism $W\lo `W$ the {\it abrelative Frobenius morphism}   
of $W$ over $Z\lo Z^{[q]}$.  
Set also $W':=W\times_{Z,F_Z^e}Z$. 
If the structural morphism $W\lo Z$ is integral, 
then $\os{\circ}{W}{}'
=\os{\circ}{W}\times_{\os{\circ}{Z},\os{\circ}{F}{}^{e}_Z}\os{\circ}{Z}
=`\os{\circ}{W}$. 

\begin{rema}\label{rema:on}
In \cite[p.~197]{ollc} Ogus has already defined $Z^{[p]}$ (he has denoted it by $Z^{(1)}$). 
%See [loc.~cit.] for the definition of $Z^{[p]}$ 
%even when $\os{\circ}{Z}$ is not necessarily of characteristic $p>0$. 
%we recall this log scheme in (\ref{rema:bd}) below.) 
%In [loc.~cit.] he has used it to give a simpler proof of Hyodo-Kato isomorphism 
%than the proof of it in \cite{hk}. In \cite{nlw} the first named author in this article 
%has used $Z^{[q]}$ to prove the log convergence of the weight filtration 
%on the log crystalline cohomology of a proper SNCL scheme over a 
%family of log points in characteristic $p>0$. 
\end{rema}

\par 
Now assume that $S_{0,{\rm red}}$ is of characteristic $p>0$. 
Set $S_{00}:=S_{0,{\rm red}}$.    
Assume that there exists a lift $F^{[e]}_S\col S\lo S$ of 
the $q$-th power Frobenius endomorphism $F^e_{S_{00}} \col S_{00}\lo S_{00}$. 
We fix $F^{[e]}_S$ and assume that $F^{[e]}_S$ induces a morphism 
$F^{[e]}_{S_0}\col S_0 \lo S_0$. 
Set $\os{\circ}{F}{}^{[e]}_S:=(F^{[e]}_S)^{\circ}$ and  
$\os{\circ}{F}{}^{[e]}_{S_0}:=(F^{[e]}_{S_0})^{\circ}$. 
Set also $S^{[q]}:=S\times_{\os{\circ}{S},\os{\circ}{F}{}^{[e]}_S}\os{\circ}{S}$ 
and $S^{[q]}_0:=S_0\times_{\os{\circ}{S}_0,\os{\circ}{F}{}^{[e]}_{\!S_0}}\os{\circ}{S}_0$ 
by abuse of notation. (These log scheme may depend on the choice of 
$\os{\circ}{F}{}^{[e]}_S$.)
Because the following diagram 
\begin{equation*} 
\begin{CD} 
S@>{F^{[e]}_S}>>S\\
@VVV @VVV \\
\os{\circ}{S}@>{\os{\circ}{F}{}^{[e]}_S}>>\os{\circ}{S}
\end{CD} 
\end{equation*} 
is commutative, we have a natural morphism 
$F^{[e]}_{S/\os{\circ}{S}}\col S\lo S^{[q]}$.  
Similarly we have a natural morphism 
$F^{[e]}_{S_0/\os{\circ}{S}_0}\col S_0\lo S^{[q]}_0$.  
We also have two projections 
$S^{[q]}\lo S$ and $S^{[q]}_0\lo S_0$.  
%Then we have a natural morphism 
%$S\lo S^{[p]}$, which we denote by $F_{S/\os{\circ}{S}}$ by abuse of notation. 
%Set $S^{[p]}_0:=S^{[p]}\times_{\os{\circ}{S}}\os{\circ}{S}_0$ 
%by abuse of notation and let $F_{S_0/\os{\circ}{S}_0}\col S_0\col S^{[p]}_0$ 
%be the induced morphism by $F_{S/\os{\circ}{S}}$. 
Then we have the following commutative diagram: 
\begin{equation*} 
\begin{CD} 
S_{00}@>{F^{e}_{S_{00}/\os{\circ}{S}_{00}}}>> S_{00}^{[q]}\\
@V{\bigcap}VV @VV{\bigcap}V \\
S_0@>{F^{[e]}_{S_0/\os{\circ}{S}_0}}>> S_0^{[q]}\\
@V{\bigcap}VV @VV{\bigcap}V \\
S@>{F^{[e]}_{S/\os{\circ}{S}}}>> S^{[q]}
\end{CD}
\tag{4.10.1}\label{cd:sps}
\end{equation*} 
over 
\begin{equation*} 
\begin{CD} 
\os{\circ}{S}_{00}@= \os{\circ}{S}_{00}\\
@V{\bigcap}VV @VV{\bigcap}V \\
\os{\circ}{S}_{0}@= \os{\circ}{S}_{0}\\
@V{\bigcap}VV @VV{\bigcap}V \\
\os{\circ}{S}@= \os{\circ}{S}. 
\end{CD}
\tag{4.10.2}\label{cd:spss}
\end{equation*}   
If $e=1$, then we denote $F^{[e]}_S$ and $F^{[e]}_{S/\os{\circ}{S}}$ 
by $F_S$ and $F_{S/\os{\circ}{S}}$, respectively.

%\begin{rema}\label{rema:bd}
%Let $(M_S,\al)$ be the log structure of $S$. 
%Following \cite[p.~196]{ollc}, consider the following composite morphism 
%$M_S\os{q\times}{\lo} M_S\os{\al}{\lo}{\cal O}_S$ and 
%let $S^{(q)}$ be a log scheme whose underlying scheme is $\os{\circ}{S}$ and 
%whose log structure is associated to this composite morphism. 
%Then we have a natural morphism 
%$S\lo S^{(q)}$, which we denote by $F^{(e)}_{S/\os{\circ}{S}}$ by abuse of notation. 
%Set $S^{(q)}_0:=S^{(q)}\times_{\os{\circ}{S}}\os{\circ}{S}_0$ 
%and let $F^{(e)}_{S_0/\os{\circ}{S}_0}\col S_0\lo S^{(q)}_0$ 
%and $F^{(e)}_{S_{00}/\os{\circ}{S}_{00}}\col S_{00}\lo S^{(q)}_{00}$ 
%be the induced morphism by $F^{(e)}_{S/\os{\circ}{S}}$. 
%Then we have the following commutative diagram: 
%\begin{equation*} 
%\begin{CD} 
%S_{00}@>{F^{(e)}_{S_{00}/\os{\circ}{S}_{00}}}>> S_{00}^{(q)}\\
%@V{\bigcap}VV @VV{\bigcap}V \\
%S_0@>{F^{(e)}_{S_0/\os{\circ}{S}_0}}>> S_0^{(q)}\\
%@V{\bigcap}VV @VV{\bigcap}V \\
%S@>{F^{(e)}_{S/\os{\circ}{S}}}>> S^{(q)}
%\end{CD}
%\end{equation*} 
%over (\ref{cd:spss}). 
%However we do not use $S^{(q)}_{\star}$ $(\star=0$ or nothing) 
%and $F^{(e)}_{S_{\star}/\os{\circ}{S}_{\star}}\col S_{\star}\lo S^{(q)}_{\star}$ 
%in this article because there does not exist a natural morphism 
%$S^{(q)}_{\star}\lo S_{\star}$ in general as pointed out in \cite[p.~196]{ollc}.  
%Consequently, even if we are given a log scheme $Y_{\star}$ over $S_{\star}$, 
%we cannot have a log scheme ``$Y_{\star}\times_{S_{\star}}S^{(q)}_{\star}$''. 
%\end{rema}

\par 
Set $Y_0:=Y_{\rm red}$ for simplicity of notation. 
Set $`Y:=Y\times_{S_0}S^{[q]}_0$ and 
$`Y_0:=`Y\times_{S_0}S^{[q]}_{00}$. 
By (\ref{cd:wzf}) we have the following commutative diagram 
\begin{equation*} 
\begin{CD} 
Y_0@>{F_0}>>`Y_0\\
@VVV @VVV \\
S_{00}@>{F^{[e]}_{S_{00}/\os{\circ}{S}_{00}}}>>S^{[q]}_{00}, 
\end{CD}
\end{equation*} 
where $F_0\col Y_0\lo `Y_0$ is the $e$-times iterated 
abrelative Frobenius morphism over $S_{00}\lo S^{[q]}_{00}$. 
We assume that there exists a lift 
$F\col Y\lo `Y$ of $F_0\col Y_0\lo `Y_0$ over $S_{0}\lo S^{[q]}_{0}$. 
That is, we assume that there exists a morphism 
$F\col Y\lo `Y$ fitting into the following commutative diagram 
\begin{equation*} 
\begin{CD} 
Y_0@>{\subset}>> Y\\
@V{F_0}VV @VV{F}V \\
`Y_0@>{\subset}>> `Y  
\end{CD} 
\end{equation*}
over the commutative diagram 
\begin{equation*} 
\begin{CD} 
S_{00}@>{\subset}>> S_0\\
@V{F^{[e]}_{S_{00}/\os{\circ}{S}_{00}}}VV @VV{F^{[e]}_{S_0/\os{\circ}{S}_0}}V \\
S_{00}^{[q]}@>{\subset}>> S^{[q]}_0. 
\end{CD} 
\end{equation*}
We say that $(\wt{Y},\wt{F})/(S\lo S^{[q]})$ is a {\it log smooth integral lift}  
(or simply a {\it lift}) of $(Y,F)/(S_0\lo S^{[q]}_0)$ 
if $\wt{Y}$ is a log smooth integral scheme over $S$ such that 
$\wt{Y}\times_SS_0=Y$ and $\wt{F}$ is a morphism 
$\wt{Y}\lo `\wt{Y}:=\wt{Y}\times_{\os{\circ}{S},\os{\circ}{F}{}^{[e]}_S}\os{\circ}{S}$ 
over $S\lo S^{[q]}$ fitting into 
the following commutative diagram 
\begin{equation*} 
\begin{CD} 
Y@>{\subset}>> \wt{Y}\\
@V{F}VV @VV{\wt{F}}V \\
`Y@>{\subset}>> `\wt{Y}  
\end{CD} 
\end{equation*}
over the commutative diagram 
\begin{equation*} 
\begin{CD} 
S_0@>{\subset}>> S\\
@V{F^{[e]}_{S_0/\os{\circ}{S}_0}}VV @VV{F^{[e]}_{S/\os{\circ}{S}}}V \\
S_0^{[q]}@>{\subset}>> S^{[q]}.   
\end{CD} 
\end{equation*}
\par 
Let ${\rm Lift}_{(Y,F)/(S_0\subset S,F^{[e]}_S)}$ be the following sheaf defined by 
the following equality: 
\begin{align*} 
{\rm Lift}_{(Y,F)/(S_0\subset S,F^{[e]}_S)}(U):=
\{&{\rm isomorphism~classes~of~lifts~of}~(U,F\vert_U)/(S_0\lo S^{[q]}_0)\\
&{\rm over}~(S\lo S^{[q]})\}
\end{align*}
for each log open subscheme $U$ of $Y$. Here 
$`U:=U\times_{\os{\circ}{S}_0,\os{\circ}{F}{}^{[e]}_{\!S_0}}\os{\circ}{S}_0$, 
$F\vert_U\col U\lo `U$ is the restriction of $F$ to $U$,  
and the isomorphism classes of lifts of $(U,F\vert_U)/(S_0\lo S^{[q]}_0)$ 
over $(S\lo S^{[q]})$ are defined in an obvious way. 
%Note that, because $\os{\circ}{F}$ is an isomorphism of topological spaces, 
%$`U$ is uniquely determined by $U$. 
\par  
Let the notations be as above. 
Let $`\iota \col `Y \os{\sus}{\lo} `\wt{Y}$ be the closed immersion. 
Let $\wt{G}\col \wt{Y}\lo `\wt{Y}$ be another lift of $F$. 
Take $s$ in (\ref{prop:eq}) as the composite morphism 
$Y\os{F}{\lo} `Y\os{`\iota}{\lo} `\wt{Y}$ over 
the composite morphism $S_0\os{\subset}{\lo} S\lo  S^{[q]}$. 
Then  
$\wt{G}$ defines an element $\wt{G}{}^*-\wt{F}{}^*$ of 
\begin{align*} 
&{\rm Hom}_{{\cal O} _Y}
(F^*`\iota{}^*(\Om^1_{`\wt{Y}{}/S^{[q]}}),{\cal I}{\cal O}_{\wt{Y}})
=
{\rm Hom}_{{\cal O} _Y}(F^*(\Om^1_{`Y/S^{[q]}_0}),{\cal I}{\cal O}_{\wt{Y}})
\tag{4.10.3}\label{ali:fooi}\\
&={\rm Hom}_{{\cal O} _Y}(F^*(\Om^1_{`Y/S^{[q]}_0}),{\cal O}_Y)
\otimes_{{\cal O}_{S_0}}{\cal I}
={\rm Hom}_{{\cal O}_{`Y}}(\Om^1_{`Y/S^{[q]}_0},F_*({\cal O}_Y))
\otimes_{{\cal O}_{S_0}}{\cal I}.
\end{align*} 
This is the log version of a generalization of \cite[p.~208, iii)]{ns}.

%\end{prop}
%\begin{proof} 
%We define the element as follows. 
%\par 
%Let $a$ be a local section of ${\cal O}_Y$. 
%We consider the following morphism 
%$a\lom \wt{G}^*(a)-\wt{F}{}^*(a)$. 
%By using the identity 
%$\wt{F}{}^*(a)(\wt{G}{}^*(b)-\wt{F}{}^*(b))+
%\wt{F}{}^*(b)(\wt{G}{}^*(a)-\wt{F}{}^*(a))
%= \wt{G}{}^*(a)(\wt{G}{}^*(a)-\wt{F}{}^*(a))
%+\wt{F}{}^*(b)(\wt{G}{}^*(a)-\wt{F}{}^*(a))$, 
%we easily see that the morphism above is a derivation of ${\cal O}_Y$. 
%Hence we have an element of 
%${\rm Hom}_{{\cal O}_Y}(\Om^1_{\os{\circ}{Y}/\os{\circ}{S}_0},{\cal I}{\cal O}_{\wt{Y}})$. 
%\par 
%Let $m$ be a local section of $M_Y$. 
%Let $u\in {\cal O}_Y^*$ be a local section such that 
%$\wt{G}{}^*(m)=\wt{F}{}^*(m)u(m)$. 
%Consider the following morphism $d\log m\lom u(m)-1$.  
%Because $u(mn)=u(m)u(n)$ and ${\cal I}^2=0$,  $u(mn)-1=u(m)-1+u(n)-1$. 
%Furthermore, $\wt{G}{}^*(\al(m))=\wt{F}{}^*(\al(m))u(m)$, 
%the two morphisms define a morphism 
%$\Om^1_{Y/S_0}\lo {\cal I}{\cal O}_{\wt{Y}}$.
%\end{proof}

%\begin{prop}\label{prop:fl}
%Let the notations be as in the construction of the map 
%{\rm (\ref{prop:npp}) (1)}. 
%\end{prop}
%\begin{proof} 
%\end{proof}

The following is the log version of \cite[p.~208, iv)]{ns}. 
This is a key lemma for (\ref{theo:ts}) below.

\begin{lemm}\label{lemm:fl}
Assume that there exists a lift $(\wt{Y},\wt{F})/S$ of $(Y,F)/S_0$. 
Assume that ${\cal I}=\pi^n{\cal O}_S$ for 
a global section $\pi$ of ${\cal O}_S$ 
and that $q\pi^n=0$ in ${\cal O}_S$ for a positive integer $n$. 
Assume also that $S_{00}=S\mod \pi$ and that 
the morphism ${\cal O}_{S_{00}}\owns 1\lom \pi^n\in {\cal I}$ 
is a well-defined isomorphism. 
%Assume that ${\cal I}{\cal O}_{\wt{Y}}=\pi^n{\cal O}_{\wt{Y}}$ for 
%a global section $\pi$ of ${\cal O}_S$ with some positive integer $n$, that 
%the morphisms ${\cal O}_{Y_0}\owns 1\lom 1\in {\cal O}_{\wt{Y}}/\pi{\cal O}_{\wt{Y}}$ 
%and ${\cal O}_{Y_0}\owns 1\lom \pi^n\in {\cal I}{\cal O}_{\wt{Y}}$
%are isomorphisms of ${\cal O}_{Y_0}$-modules 
%and that $p\pi^n=0$ in ${\cal O}_{\wt{Y}}$. 
Let the notations be as in {\rm (\ref{prop:npp}) (1)} and  
denote $\del$ in {\rm (\ref{prop:npp}) (1)} by $\pi^n \del$ 
in this lemma$;$ 
the new $\del$ is an element of 
${\rm Hom}_{{\cal O}_{Y_0}}
(\Om^1_{Y_0/S_{00}},{\cal O}_{Y_0})$
since 
\begin{align*} 
{\rm Hom}_{{\cal O}_Y}(\Om^1_{Y/S_{0}},{\cal I}{\cal O}_{\wt{Y}})
={\rm Hom}_{{\cal O}_Y}(\Om^1_{Y/S_{0}},{\cal O}_Y)
\otimes_{{\cal O}_{S_0}}{\cal I}\os{\sim}{\longleftarrow} 
{\rm Hom}_{{\cal O}_{Y_0}}(\Om^1_{Y_0/S_{00}},{\cal O}_{Y_0}).
\end{align*} 
Denote by $`g\in {\rm Aut}_{S^{[q]}}(`\wt{Y}{},`Y)$ 
the induced automorphism of $`\wt{Y}{}$ by an element 
$g\in {\rm Aut}_S(\wt{Y},Y)$.  
Let $`\del$ be the element of 
${\rm Hom}_{{\cal O}_{`Y_0}}
(\Om^1_{`Y_0/S_{00}^{[q]}},{\cal O}_{`Y_0})$ obtained by 
$`g$. 
%Assume also that $\wt{F}^*$ is $\pi^n$-linear. 
Then the following hold$:$
\par 
$(1)$ $`\del(`a)=\sum_{i}\del(a_i)\otimes b_i$ for 
$`a=\sum_ia_i\otimes b_i\in {\cal O}_{`Y_0}
={\cal O}_{Y_0}\otimes_{{\cal O}_{S_{00}, F^{e*}_{S_{00}}}}{\cal O}_{S_{00}}$
$(a_i\in {\cal O}_{Y_0}, b_i\in {\cal O}_{S_{00}})$. 
\par 
$(2)$ $g^*{}^{-1}\wt{F}{}^*`g{}^*(`a)-\wt{F}{}^*(`a)
%=\pi^n \sum_i\del(a_i)^pb_i
=\pi^nF_0^*(`\del(`a))$ for $`a=\sum_ia_i\otimes b_i
\in {\cal O}_{`\wt{Y}}$. 
Here we denote the image of 
$`a\in {\cal O}_{`\wt{Y}}$ in ${\cal O}_{`Y_0}$ 
by $`a$ by abuse of notation. 
\par 
$(3)$ 
$1+\pi^n`\del(`m)= [1+\pi^n\del(m),1]$ 
for $`m=[m, u]\in M_{`Y_0}=
M_{Y_0}\oplus_{{\cal O}^*_{S_{00}}, F^{e*}_{S_{00}}}{\cal O}^*_{S_{00}}$ 
$(m\in M_{Y_0},u\in {\cal O}^*_{S_{00}})$. 
%$1+\pi^n`\del(`m)= [1+\pi^n\sum_i\del(m_i),1]$ 
%for $`m=\prod_i [m_i, u_i]\in M_{`Y_0}=
%M_{Y_0}\oplus_{{\cal O}^*_{S_{00}}, F^{e*}_{S_{00}}}{\cal O}^*_{S_{00}}$. 
\par 
$(4)$ 
$(g^*{}^{-1}\wt{F}{}^*`g{}^*)(`m)(\wt{F}{}^*(`m))^{-1}
=1+\pi^n F_0^*(`\del(`m))$
%\sum_i\del(m_i)^p$ 
for $`m\in M_{`\wt{Y}}$.  
Here we denote the image of $`m\in M_{`\wt{Y}}$ in $M_{`Y_0}$ 
by $`m$ by abuse of notation. 
\par 
$(5)$ 
%Assume that $\os{\circ}{Y}_0$ is reduced.  
Let $(\wt{Y}_1,\wt{F}_1)/S$ and 
$(\wt{Y}_2,\wt{F}_2)/S$ be two lifts of $(Y,F)/S_0$.  
Then there exists at most one element $h$ of 
${\rm Isom}_S(\wt{Y}_1,\wt{Y}_2)$ such that $h\vert_Y={\rm id}_Y$ and 
$`h\circ \wt{F}_1=\wt{F}_2\circ h$. 
Here the $`h$ on the left hand side of this equality 
is the induced isomorphism $`\wt{Y}{}_1\os{\sim}{\lo}`\wt{Y}{}_2$ by $h$.  
\par 
$(6)$ Let $F$, $(\wt{Y}_1,\wt{F}_1)$ 
and $(\wt{Y}_2,\wt{F}_2)$ be as in {\rm (5)}.  
%Assume that the natural morphism 
%${\cal O}_{`Y_0}\lo F_{0*}({\cal O}_{Y_0})$ is injective. 
Let $h$ be an element of 
${\rm Isom}_S(\wt{Y}_1,\wt{Y}_2)$ such that 
$h\vert_Y={\rm id}_Y$. 
%and $(`g\vert_{`Y})\circ F_1=F_2\circ (g\vert_{Y})$. 
Then the image of 
\begin{align*} 
(`h\wt{F}_1 h^{-1})^*-\wt{F}_2^* &\in 
{\rm Hom}_{{\cal O}_{`Y}}(\Om^1_{`Y/S_0},F_*({\cal I}{\cal O}_{\wt{Y}}))
={\rm Hom}_{{\cal O}_{`Y}}(\Om^1_{`Y/S_0},F_*({\cal O}_Y))
\otimes_{{\cal O}_{S_0}}{\cal I}
\\
&\os{\sim}{\longleftarrow}{\rm Hom}_{{\cal O}_{`Y_0}}
(\Om^1_{`Y_0/S_{00}},F_{0*}({\cal O}_{Y_0}))
\end{align*} 
in ${\rm Hom}_{{\cal O}_{`Y_0}}(\Om^1_{`Y_0/S_{00}},
F_{0*}({\cal O}_{Y_0})/{\cal O}_{`Y_0})$ 
is independent of the choice of $h$. 
%Here we have identified ${\cal I}{\cal O}_{\wt{Y}}$ with ${\cal O}_{Y_0}$. 
\end{lemm}
\begin{proof} 
%Because $\wt{Y}\lo S$ is log smooth and integral, 
%$\os{\circ}{\wt{Y}}\lo \os{\circ}{S}$ is flat (\cite[(4.5)]{klog1}).  
%Hence ${\cal I}{\cal O}_{\wt{Y}}=
%{\cal I}\otimes_{{\cal O}_S}{\cal O}_{\wt{Y}}=
%{\cal I}\otimes_{{\cal O}_{S_{0}}}{\cal O}_{Y}
%\os{\sim}{\longleftarrow}
%{\cal O}_{S_{00}}\otimes_{{\cal O}_{S_{0}}}{\cal O}_{Y}={\cal O}_{Y_{0}}$. 
%\par 
(1): Since 
\begin{align*} 
`a+\pi^n`\del(`a)=`g{}^*(`a)=`g{}^*(\sum_ia_i\otimes b_i)
=\sum_i(a_i+\pi^n\del(a_i))\otimes b_i
=`a+\sum_i\pi^n\del(a_i)\otimes b_i
\end{align*} 
and ${\cal O}_{S_{00}}\simeq {\cal I}$, 
we obtain (1).  
\par 
(2): Because $S_{00}=S\mod \pi$ and that 
the morphism ${\cal O}_{S_{00}}\owns 1\lom \pi^n\in {\cal I}$ 
is an isomorphism, $\pi^{n+1}=0$ in ${\cal O}_S$. 
Let $\star$ be nothing or $`$. 
Express ${}^{\star}g^*({}^{\star}a)={}^{\star}a+\pi^n{}^{\star}\del({}^{\star}a)$ 
$({}^{\star}a\in {\cal O}_{{}^{\star}\wt{Y}})$ with 
${}^{\star}\del({}^{\star}a)\in {\cal O}_{{}^{\star}Y_0}$. 
%We claim that ${}^{\star}\del(\pi {}^{\star}a)=0$ for ${}^{\star}a\in {\cal O}_{{}^{\star}Y}$. 
%Indeed, since ${}^{\star}g^*$ is $\pi$-linear and $\pi^{n+1}=0$, 
%$\pi {}^{\star}a=\pi({}^{\star}a+\pi^n {}^{\star}\del ({}^{\star}a))
%=\pi {}^{\star}g^*({}^{\star}a)={}^{\star}g^*(\pi {}^{\star}a)
%=\pi {}^{\star}a+\pi^n{}^{\star}\del(\pi {}^{\star}a)$. 
%Hence ${}^{\star}\del(\pi {}^{\star}a)=0$ in ${\cal O}_{{}^{\star}Y_0}$. 
Consider a local section 
$`a:=\sum_i a_i\otimes b_i\in 
{\cal O}_{\wt{Y}}\otimes_{{\cal O}_S, F^{[e]*}_S}{\cal O}_S={\cal O}_{`\wt{Y}}$. 
Then $\wt{F}{}^*(`a)=\sum (a^q_ib_i+\pi c_i)$ 
for some section $c_i\in {\cal O}_{\wt{Y}}$. 
Hence $\pi^n \del(\wt{F}{}^*(`a))=
\pi^n \del(\sum (a^q_ib_i+\pi c_i))=0$ 
because $\del$ is a derivation over ${\cal O}_S$ 
and because $q\pi^n=0=\pi^{n+1}$ in ${\cal O}_{\wt{Y}}$.  
Using these vanishings, 
we have the following equalities:  
\begin{align*} 
g^*{}^{-1}\wt{F}{}^*`g{}^*(`a)&=g^*{}^{-1}\wt{F}{}^*(`a+\pi^n`\del(`a))=
g^*{}^{-1}(\wt{F}{}^*(`a)+\pi^n\sum_i\del(a_i)^qb_i)\\
& = \wt{F}{}^*(`a)-\pi^n\del(\wt{F}{}^*(`a))
+\pi^n\sum_i(\del(a_i)^qb_i-\pi^n\del(\pi^n\del(a_i)^qb_i))\\
&=\wt{F}{}^*(`a)+\pi^n\sum_i\del(a_i)^qb_i
=\wt{F}{}^*(`a)+\pi^nF_0^*(`\del(`a)).
\end{align*}   
\par 
(3): Because  
\begin{align*} 
`m(1+\pi^n `\del(`m))&=`g^*(`m)
=[g^*(m),u]=[m(1+\pi^n\del(m)),u]=`m[1+\pi^n\del(m), 1],  
\end{align*} 
we obtain (3).  
\par 
(4): 
Consider a local section 
$`m:=[m,u]\in M_{`\wt{Y}}$ in (3). 
%Set $`m_i:=[m_i, u_i]$. 
Express $\wt{F}{}^*(`m)=(m^qu)(1+\pi \eta(`m))$ 
$(\eta(`m)\in {\cal O}_Y)$. 
We have the following equalities: 
\begin{align*} 
(g^*{}^{-1}\wt{F}{}^*`g^*)(`m)& 
=g^*{}^{-1}\wt{F}{}^*(`m(1+\pi^n`\del(`m)))
=g^*{}^{-1}(\wt{F}{}^*(`m)\wt{F}{}^*([1+ \pi^n\del(m),1])) \\
& =g^*{}^{-1}(\wt{F}{}^*(`m)(1+\pi^n\del(m)^q))\\
&=g^*{}^{-1}(\wt{F}{}^*(`m))g^*{}^{-1}(1+\pi^n\del(m)^q)\\
&=\wt{F}{}^*(`m)
(1-\pi^n\del(\wt{F}{}^*(`m)))(1+\pi^n\del(m)^q)
(1-\pi^n\del(1+\pi^n\del(m)^q))\\ 
&=\wt{F}{}^*(`m)(1-\pi^n
\del(m^qu(1+\pi \eta(`m))))(1+\pi^n \del(m)^q)\\
&=\wt{F}{}^*(`m)(1-\pi^n(\del(m^qu)+
\{1-\pi \eta(`m)\}
\del(1+\pi \eta(`m))))(1+\pi^n \del(m)^q)\\
&=\wt{F}{}^*(`m)(1+\pi^n \del(m)^q)=
\wt{F}{}^*(`m)(1+\pi^n F_0^*(`\del(`m))). 
\end{align*} 
Here we have used the formula in (\ref{coro:lys}) (3) 
for the sixth and the seventh equalities; we have also used (3) 
for the second and the last equalities. 
\par 
(5):  
Let $g_i \col \wt{Y}_1\os{\sim}{\lo}\wt{Y}_2$ $(i=1,2)$ be 
an isomorphism such that 
$g_i\vert_Y={\rm id}_Y$ 
and $`g_i \circ \wt{F}_1=\wt{F}_2 \circ g_i$ on $\wt{Y}_1$. 
Set $g:=g_1\circ g_2^{-1}\in {\rm Aut}_S(\wt{Y}_2,Y)$. 
Let $\del \in {\rm Hom}_{{\cal O}_{Y_0}}(\Om^1_{Y_0/S_{00}},{\cal O}_{Y_0})$ 
be the morphism corresponding to $g$.  
%Then  
%$$\wt{F}{}^*_2(`a)-(`g_1\wt{F}{}_1g_1^{-1})^*(`a)=0$$
%=\pi^n \eta(`a)$$ 
%with $\eta(`a)\in {\cal O}_Y$. 
%Assume that $`g_i \circ \wt{F}_1=\wt{F}_2 \circ g_i$ $(i=1,2)$. 
Then we obtain the following equalities by using (2): 
\begin{align*} 
0&=(`g_1\wt{F}{}_1g_1^{-1})^*(`a)-(`g_1\wt{F}{}_1g_1^{-1})^*(`a)
=(`g_1`g_2^{-1}\wt{F}{}_2g_2g_1^{-1})^*(`a)-(`g_1\wt{F}{}_1g_1^{-1})^*(`a)\\
&=(`g\wt{F}{}_2g^{-1})^*(`a)-(`g_1\wt{F}{}_1g_1^{-1})^*(`a)
=(`g\wt{F}{}_2g^{-1})^*(`a)-\wt{F}{}^*_2(`a)=\pi^n(F^*_0(`\del(`a))). 
\end{align*} 
Because  ${\cal O}_{Y_0}\os{\sim}{\lo} \pi^n{\cal O}_{\wt{Y}}$ 
(since $\os{\circ}{\wt{Y}}$ is flat over $\os{\circ}{S}$), 
$F^*_0(`\del(`a))=0$. 
Because $\os{\circ}{Y}_0$ is reduced, the morphism 
$F_0^*\col {\cal O}_{`Y_0}\lo F_{0*}({\cal O}_{Y_0})$ is injective
%$\os{\circ}{Y}_0$ is reduced and of characteristic $p$, 
and hence $`\del(`a)=0$. 
%(Recall that $Y'_0=Y_0\times_{S_{00},F_{S_{00}}}S_{00}$ and 
%that $\os{\circ}{Y}{}'_0=`\os{\circ}{Y}_0$.) 
Because $\os{\circ}{Y}_0$ is reduced, 
the pull-back $`{\rm pr}^*_0\col {\cal O}_{Y_0}\lo  {\rm pr}_*({\cal O}_{`Y_0})$ 
of the projection $`{\rm pr}_0\col `Y_0\lo Y_0$ is also injective. 
Hence 
\begin{align*} 
\del(a)=0 \quad (a\in {\cal O}_{\wt{Y}}).
\tag{4.11.1}\label{ali:aaa}
\end{align*}  
\par 
On the other hand, 
%express 
%$$\wt{F}{}^*_2(`m)=(`g_1\wt{F}{}_1g_1^{-1})^*(`m)(1+\pi^n \eta(`m))$$  
%with $\eta(`m)\in {\cal O}_Y$ ($`m\in M_{`Y}$). 
%Then 
we obtain the following equalities by using (4):
\begin{align*} 
1&=(`g_1\wt{F}{}_1g_1^{-1})^*(`m)\{(`g_1\wt{F}{}_1g_1^{-1})^*(`m)\}^{-1}\\
&=(`g_1`g_2^{-1}\wt{F}{}_2g_2g_1^{-1})^*(`m)
\{(`g_1\wt{F}{}_1g_1^{-1})^*(`m)\}^{-1}\\
&=(`g\wt{F}{}_2g^{-1})^*(`m)\{(`g_1\wt{F}{}_1g_1^{-1})^*(`m)\}^{-1}\\
&=[(`g\wt{F}{}_2g^{-1})^*(`m)\{\wt{F}{}^*_2(`m)\}^{-1}]
=1+\pi^n F_0^*(`\del(`m)). 
\end{align*} 
%Hence $F_0^*(`\del(`m))=0$. Hence $[1+\pi^n\del(m),1]=[1,1]$ for $m\in M_{Y}$. 
%By this formula, $(1+\pi^n\del(m))\otimes 1=1\otimes 1$ in ${\cal O}_{`Y}$. 
%Because $\os{\circ}{Y}_0$ is reduced, 
By the same argument as that in the previous paragraph, we see that 
the following equality holds: 
\begin{align*} 
\del(m)=0 \quad (m\in M_{\wt{Y}}).
\tag{4.11.2}\label{ali:mmm}
\end{align*}   
By (\ref{ali:aaa}) and (\ref{ali:mmm}), we see that 
$g={\rm id}_{\wt{Y}}$ and consequently $g_1=g_2$.  
We have completed the proof of (5). 
\par 
(6) Let $g_i \col \wt{Y}_1\os{\sim}{\lo}\wt{Y}_2$ 
$(i=1,2)$ be an isomorphism such that $g_i\vert_Y={\rm id}_Y$. 
Set $g:=g_1\circ g_2^{-1}\in {\rm Aut}_S(\wt{Y}_2,Y)$. 
%Take another $h\in {\rm Isom}_S(\wt{Y}_1,\wt{Y}_2)$ such that 
%$h\vert_Y={\rm id}_Y$.  
%Express $h^*(a)=g^*(a)+\pi^n\Del(a)$.
Then we obtain the following equalities as in (5): 
\begin{align*} 
(`g_1\wt{F}_1g^{-1}_1)^*(`a)-\wt{F}^*_2(`a)
&= 
(`g_1`g_2^{-1}`g_2\wt{F}_1g_2^{-1}g_2g^{-1}_1)^*(`a)-\wt{F}^*_2(`a)
\\
&= (`g`g_2\wt{F}_1g_2^{-1}g^{-1})^*(`a)-\wt{F}^*_2(`a)\\
&=(`g_2\wt{F}_1g_2^{-1})^*(`a)+\pi^n(`g_2F_0g_2^{-1})^*(`\del(`a))-\wt{F}^*_2(`a)\\
&=(`g_2\wt{F}_1g_2^{-1})^*(`a)-\wt{F}^*_2(`a)+\pi^nF_0^*(`\del(`a))
\end{align*} 
since $g_2\vert_Y={\rm id}_Y$. 
Analogously  we obtain the following equalities as in (5): 
\begin{align*} 
(`g_1\wt{F}_1g^{-1}_1)^*(`m)(\wt{F}^*_2(`m))^{-1}
&= 
(`g_1`g_2^{-1}`g_2\wt{F}_1g_2^{-1}g_2g^{-1}_1)^*(`m)(\wt{F}^*_2(`m))^{-1}
\\
&= (`g`g_2\wt{F}_1g_2^{-1}g^{-1})^*(`m)(\wt{F}^*_2(`m))^{-1}\\
&=\{(`g_2\wt{F}_1g_2^{-1})^*(`m)(1+\pi^n(`g_2F_0g_2^{-1})^*(`\del(`m)))\}(\wt{F}^*_2(`m))^{-1}\\
&=\{(`g_2\wt{F}_1g_2^{-1})^*(`m)\wt{F}^*_2(`m)^{-1}\}(1+\pi^nF_0^*(`\del(`m)))
\end{align*} 
since $g_2\vert_Y={\rm id}_Y$. 
%\begin{align*} 
%(`h\wt{F}_1h^{-1})^*-\wt{F}^*_2
%&=((`h\wt{F}{}^*_1h^{-1})^*-\wt{F}{}^*_1)
%-(`g\wt{F}{}_1g^{-1})^*-\wt{F}{}^*_1))+(`g\wt{F}{}_1g^{-1})^*-\wt{F}^*_2\\
%&.
%\end{align*}   
%\begin{align*} 
%(`g_1\wt{F}_1g^{-1}_1)^*(`a)-\wt{F}^*_2(`a)
%&= h^{-1}{}^*\wt{F}_1^*(`g^*(`a)+\pi^n`\Del(`a))-\wt{F}^*_2(`a)
%\\
%&= h^{-1}{}^*\wt{F}_1^*(`g^*(`a))+
%\pi^nh^{-1}{}^*(\wt{F}_0^*`\Del(`a)))-\wt{F}^*_2(`a)\\
%&=g^{-1}{}^*\wt{F}_1^*(`g^*(`a))-
%\pi^n\Del(\wt{F}_1^*(`g^*(`a)))-\wt{F}^*_2(`a).
%\end{align*} 
%\begin{align*} 
%(`h\wt{F}_1h^{-1})^*-\wt{F}^*_2
%&=((`h\wt{F}{}^*_1h^{-1})^*-\wt{F}{}^*_1)
%-(`g\wt{F}{}_1g^{-1})^*-\wt{F}{}^*_1))+(`g\wt{F}{}_1g^{-1})^*
%-\wt{F}^*_2\\
%&.
%\end{align*}   
Hence we see that the image of 
$(`h\wt{F}_1h^{-1})^*-\wt{F}^*_2$ in the quotient of the map 
$F_0^*\col {\rm Hom}_{{\cal O}_{`Y_0}}(\Om^1_{`Y_0/S_{00}},{\cal O}_{`Y_0})
\lo {\rm Hom}_{{\cal O}_{`Y_0}}(\Om^1_{`Y_0/S_{00}},F_{0*}({\cal O}_{Y_0}))$ 
is independent of $h$.  This proves (6). 
\end{proof}

The following (2) and (3) are the log versions of \cite[Proposition 1 in p.~205]{ns}; 
the following (4) is the log version of \cite[p.~104 (ii)]{sr}; 
the following (5) is an additional result. 
Roughly speaking, we follow the argument in \cite{ns} for 
the proof of (3). However to give the precise proof of it is very involved 
(cf.~(\ref{rema:tlt})). 
The most important part in (\ref{theo:ts}) is (2). Once one knows (2), 
(3) is a formal consequence of (2); however to prove (2), 
we use the argument in the proof of (3); 
the proof of (\ref{theo:ts}) is logically more complicated than 
those of (\ref{prop:eq}) and (\ref{prop:ep}).

\begin{theo}\label{theo:ts}  
Let ${\cal I}$, $\pi$ and $n$ be as in {\rm (\ref{lemm:fl})}. 
%Assume that $\os{\circ}{Y}_0$ is reduced. 
%the morphism ${\cal O}_{`Y_0}\lo F_{0*}({\cal O}_{Y_0})$ is injective 
%and the pull-back morphism ${\cal O}_{Y_0}\lo {\cal O}_{`Y_0}$ 
%by the projection $`Y_0\lo Y_0$ is injective. 
Then the following hold$:$ 
\par 
$(1)$ 
Assume that $(Y,F)/S_0$ has a lift $(\wt{Y},\wt{F})/S$. 
Set ${\rm Aut}_{S,F^{[e]}_S}(\wt{Y},Y)
:=\{g\in {\rm Aut}_{S}(\wt{Y})~\vert~g\vert_Y={\rm id}_Y, 
~\wt{F}\circ g=`g\circ \wt{F}\}$. 
Then ${\rm Aut}_{S,F^{[e]}_S}(\wt{Y},Y)=\{{\rm id}_{\wt{Y}}\}$. 
\par 
$(2)$ 
%Assume that $\os{\circ}{Y}$ is separated.  
Let $`{\rm pr}_0\col `Y_0\lo Y_0$ be the projection. 
The sheaf ${\rm Lift}_{(Y,F)/(S_0\subset S,F^{[e]}_S)}$ on 
$\os{\circ}{Y}$ is a torsor under 
$`{\rm pr}_{0*}
({\cal H}{\it om}_{{\cal O}_{`Y_0}}(\Om^1_{`Y_0/S^{[q]}_{00}}, 
F_{0*}({\cal O}_{Y_0})/{\cal O}_{`Y_0}))$.   
%if $(Y,F)/S_0$ has a lift $(\wt{Y},\wt{F})/S$, 
%then there exists the following bijection 
%of sets$:$ 
%\begin{align*} 
%{\rm Lift}_{(Y,F)/(S_0\subset S,F^{[e]}_S)}(Y)\os{\sim}{\lo}
%{\rm Hom}_{{\cal O}_{`Y_0}}(\Om^1_{`Y_0/S^{[q]}_{00}},
%F_{0*}({\cal O}_{Y_0})/{\cal O}_{`Y_0}). 
%\tag{4.14.1}\label{ali:oo0ys}
%\end{align*} 
\par 
$(3)$ Assume that $\os{\circ}{Y}$ is separated.  
%The zariski sheaf ${\rm Lift}_{(Y,F)/(S,F_S)}$ on $\os{\circ}{Y}$ is a torsor 
%under ${\cal H}{\it om}_{{\cal O}_{Y_0}}(\Om^1_{Y_0/S_{00}},B\Om^1_{Y_0/S_{00}});$ 
In 
$${\rm Ext}_{`Y_0}^1(\Om^1_{`Y_0/S^{[q]}_{00}},F_{0*}({\cal O}_{Y_0})/{\cal O}_{`Y_0}),$$ 
there exists a canonical obstruction class 
${\rm obs}_{(Y,F)/(S_0\subset S,F^{[e]}_S)}$ 
of a lift of $(Y,F)/(S_0\lo S^{[q]}_0)$ over $S\lo S^{[q]}$. 
%(B\Om^1_{Y_0/S_{00}}))$. 
\par 
$(4)$ Assume that $\os{\circ}{Y}$ is separated.  Let 
\begin{align*} 
\partial \col {\rm Ext}_{`Y_0}^1(\Om^1_{`Y_0/S^{[q]}_{00}}, 
F_{0*} ({\cal O}_{Y_0})/{\cal O}_{`Y_0})
\lo {\rm Ext}_{`Y_0}^2(\Om^1_{`Y_0/S^{[q]}_{00}},{\cal O}_{`Y_0})
\end{align*} 
be the boundary morphism obtained by 
the following exact sequence {\rm (\ref{prop:ee})}$:$
\begin{align*} 
0\lo {\cal O}_{`Y_0}\lo F_{0*}({\cal O}_{Y_0})
\lo F_{0*} ({\cal O}_{Y_0})/{\cal O}_{`Y_0}\lo 0. 
%F_*(B\Om^1_{Y_0/S_{00}})\lo 0. 
\tag{4.12.1}\label{ali:oobbs}
\end{align*} 
Then $\partial ({\rm obs}_{(Y,F)/(S_0\subset S,F^{[e]}_S)})=
{\rm obs}_{`Y/(S^{[q]}_0\subset S^{[q]})}$. 
\par 
$(5)$ Assume that $\os{\circ}{Y}$ is separated.  
Assume 
%that $S_0=S_{00}$ $($hence $Y=Y_0$ and $`Y=`Y_0)$,  
%that $\os{\circ}{S}_{0}$ is perfect and 
that there exists a lift $\wt{Y}/S$ of $Y/S_0$. 
Let $`\wt{Y}$ be the base change of $\wt{Y}$ by the morphism $S^{[q]}\lo S$. 
Then, in 
$${\rm Ext}^1_{`Y_0}(\Om^1_{`Y_0/S^{[q]}_{00}},F_{0*}({\cal O}_{Y_0})),$$  
there exists a canonical obstruction class ${\rm obs}_{\wt{Y}/S}(F)$ 
of a lift $\wt{F}\col \wt{Y}\lo `\wt{Y}$ of $F\col Y\lo `Y$  
and this is mapped to ${\rm obs}_{(Y,F)/(S_0\subset S,F^{[e]}_S)}$ 
by the following natural morphism 
\begin{align*} 
\partial \col {\rm Ext}_{`Y_0}^1(\Om^1_{`Y_0/S^{[q]}_{00}}, F_{0*} ({\cal O}_{Y_0}))
\lo {\rm Ext}_{`Y_0}^1(\Om^1_{`Y_0/S^{[q]}_{00}}, F_{0*} ({\cal O}_{Y_0})/{\cal O}_{`Y_0}). 
\end{align*} 
\end{theo}
\begin{proof} 
%As in the proof of (\ref{lemm:fl}), we denote 
%$`g_{\star}$ and $`\del_{\star}$ $(\star=$something) by $`g_{\star}$ 
%and $`\del_{\star}$ in the following. 
If there exists a lift $(\wt{Y},\wt{F})/(S\lo S^{[q]})$ of 
$(Y,F)/(S_0\lo S^{[q]}_0)$, 
identify 
%${\cal H}{\it om}_{{\cal O}_{`Y}}(\Om^1_{`Y/S_0^{[q]}},{\cal I}{\cal O}_{`\wt{Y}})$ and 
${\cal H}{\it om}_{{\cal O}_{`Y}}(\Om^1_{`Y/S_0^{[q]}},\wt{F}_*({\cal I}{\cal O}_{\wt{Y}}))$ 
with 
%${\cal H}{\it om}_{{\cal O}_{`Y_0}}(\Om^1_{`Y_0/S_{00}^{[q]}},{\cal O}_{`Y_0})$ and 
${\cal H}{\it om}_{{\cal O}_{`Y_0}}(\Om^1_{`Y_0/S_{00}^{[q]}},F_{0*}({\cal O}_{Y_0}))$.  
%respectively.  
\par 
(1): (1) immediately follows from (\ref{lemm:fl}) (5). 
\par 
(2): Assume that there exists a lift 
 $(\wt{Y},\wt{F})/S$ of $(Y,F)$ over $S$. 
Let $(\wt{Z},\wt{G})/S$ be another lift of $(Y,F)$ over $S$. 
Let $\wt{\cal U}:=\{\wt{U}_i\}_{i\in I}$ and 
(resp.~$\wt{\cal V}:=\{\wt{V}_i\}_{i\in I}$) be an open covering of $\wt{Y}$ 
(resp.~an open covering of $\wt{Z}$) 
such that there exists an isomorphism 
$g_i \col \wt{U}_i\os{\sim}{\lo} \wt{V}_i$ such that $g_i\vert_{U_i}={\rm id}_{U_i}$. 
Here $U_i:=\wt{U}_i\vert_Y$ and $\wt{V}_i$ is an open log subscheme of $\wt{Z}$ 
such that $\wt{V}_i\vert_Y=U_i$. Set $\wt{F}_i:=\wt{F}\vert_{\wt{U}_i}\col 
\wt{U}_i\lo `\wt{U}_i$ and 
$\wt{G}_i:=\wt{G}\vert_{\wt{V}_i}\col \wt{V}_i\lo `\wt{V}_i$. 
Then we have a section 
$(`g_i^{-1}\wt{G}_ig_i)^*-\wt{F}_i^*\in 
{\rm Hom}_{{\cal O}_{`Y_0}}(\Om^1_{`Y_0/S_{00}^{[q]}},F_{0*}({\cal O}_{Y_0}))(`U_i)$. 
If we change $g_i$, then this section may change. 
However the image of this section in 
${\rm Hom}_{{\cal O}_{`Y_0}}
(\Om^1_{`Y_0/S_{00}^{[q]}},F_{0*}({\cal O}_{Y_0})/{\cal O}_{`Y_0})(`U_i)$ 
does not change by (\ref{lemm:fl}) (6) 
%(cf.~\cite[p.~208]{ns}) 
and it is a well-defined section.
This well-definedness also tells us that these local sections patch together. 
Consequently we have an element of 
${\rm Hom}_{{\cal O}_{`Y_0}}(\Om^1_{`Y_0/S_{00}^{[q]}},
F_{0*}({\cal O}_{Y_0})/{\cal O}_{`Y_0})$. 
\par 
Conversely assume that we are given a global section of 
${\rm Hom}_{{\cal O}_{`Y_0}}
(\Om^1_{`Y_0/S_{00}^{[q]}},F_{0*}({\cal O}_{Y_0})/{\cal O}_{`Y_0})$. 
Take a local lift in 
${\rm Hom}_{{\cal O}_{`Y_0}}(\Om^1_{`Y_0/S_{00}^{[q]}},F_{0*}({\cal O}_{Y_0}))$ 
of this global section. 
There exists a lift $\wt{U}_i/S$ of $U_i/S_0$ 
if $U_i$ is a small log affine open subscheme of $Y$. 
Set $`\wt{U}_i:=\wt{U}_i\times_SS^{[q]}$. 
Let $\wt{F}_i\col \wt{U}_i\lo `\wt{U}_i$ be a lift 
of $F_i:=F\vert_{U_i}\col U_i\lo `U_i$. 
(This lift exists.)
%Since $\os{\circ}{\wt{F}}_i$ is a topological isomorphism of $\os{\circ}{U}_i$, 
By (\ref{prop:eq}) the local section of 
${\cal H}{\it om}_{{\cal O}_{`Y_0}}
(\Om^1_{`Y_0/S_{00}^{[q]}},F_{0*}({\cal O}_{Y_0}))$ 
 corresponds to a local lift 
$(\wt{U}_i,\wt{F}{}'_i)$ of $(U_i,F_i)$. 
Since this is obtained by the global section of 
${\cal H}{\it om}_{{\cal O}_{`Y_0}}(\Om^1_{`Y_0/S_{00}^{[q]}},
F_{0*}({\cal O}_{Y_0})/{\cal O}_{`Y_0})$, 
they patch together by the proof of (3) below: 
we have only to change 
$g_{ij}\col \wt{U}_{ij}\os{\sim}{\lo} \wt{U}_{ji}$ in the proof of (\ref{theo:loc}) 
by $g'_{ij}\col \wt{U}_{ij}\os{\sim}{\lo} \wt{U}_{ji}$, 
where $g'_{ij}$ is the isomorphism in the proof of (3) below. 
\par 
(3): 
Let the notations be as in the proof of (\ref{theo:loc}). 
On $\wt{U}_i$ there exists a lift 
$\wt{F}_i\col \wt{U}_i\lo `\wt{U}_i$ of $F_i\col U_i\lo `U_i$. 
%Since $\os{\circ}{\wt{F}}_i$ is a topological isomorphism of $\os{\circ}{U}_i$, 
This morphism  defines a morphism 
$\wt{F}_{j}\vert_{\wt{U}_{ji}} \col \wt{U}_{ji}\lo `\wt{U}{}_{ji}$.  
Then $\wt{F}_j\vert_{\wt{U}_{ji}}$ and 
$`g_{ij}(\wt{F}_i\vert_{\wt{U}_{ij}})g^{-1}_{ij}$ are two lifts of 
$F_{ij}:=F\vert_{U_{ij}}\col U_{ij}=U_{ji}\lo `U_{ji}=`U_{ij}$. 
%(This convention is different from the convention in \cite[p.~208]{ns}.)
Hence we have an element 
\begin{align*} 
\wt{\om}_{ij}:=
(\wt{F}_j\vert_{\wt{U}_{ji}})^*-(`g_{ij}(\wt{F}_i\vert_{\wt{U}_{ji}})g^{-1}_{ij})^* & \in ~
{\rm Hom}_{{\cal O}_{`U_{ij}}}(\Om^1_{`U_{ij}/S_{0}^{[q]}},F_{*}({\cal I}{\cal O}_{\wt{U}_{ji}}))\\
&\os{\sim}{\longleftarrow}
{\rm Hom}_{{\cal O}_{`U_{ij,0}}}(\Om^1_{`U_{ij,0}/S_{00}^{[q]}},F_{0*}({\cal O}_{U_{ij,0}})).
\end{align*}  
Let $\om_{ij}$ be the image of this element in  
${\rm Hom}_{{\cal O}_{`U_{ij,0}}}(\Om^1_{`U_{ij,0}/S_{00}^{[q]}},
F_{0*}({\cal O}_{U_{ij,0}})/{\cal O}_{`U_{ij,0}})$. 
Then, by (\ref{lemm:fl}) (6), $\om_{ij}$ is independent of the choice of 
$g_{ij}$. 
We claim that 
%This independence tells us that 
the following equality holds: 
\begin{align*} 
\om_{ik}=\om_{ij}+\om_{jk}. 
\end{align*}  
Indeed, by (\ref{lemm:elin}), 
\begin{align*} 
\wt{\om}_{ij}+\wt{\om}_{jk}&=
g^*{}^{-1}_{\! \! \! jk}
\{(\wt{F}_j\vert_{\wt{U}_{jik}})^*-(`g_{ij}(\wt{F}_i\vert_{\wt{U}_{ijk}})g^{-1}_{ij})^*\}`g^*_{jk}
+(\wt{F}_k\vert_{\wt{U}_{kji}})^*-(`g_{jk}(\wt{F}_j\vert_{\wt{U}_{jki}})g^{-1}_{jk})^*\\
&=(\wt{F}_k\vert_{\wt{U}_{kji}})^*-
g^*{}^{-1}_{\! \!jk}(`g_{ij}(\wt{F}_i\vert_{\wt{U}_{jik}})g^{-1}_{ij})^*`g^*_{jk}
=(\wt{F}_k\vert_{\wt{U}_{kji}})^*-
(`g_{jk}`g_{ij}(\wt{F}_i\vert_{\wt{U}_{ijk}})(g_{jk}g_{ij})^{-1})^*. 
\end{align*} 
By (\ref{lemm:fl}) (6) again, the image of the last term in 
${\rm Hom}_{{\cal O}_{`U_{ijk,0}}}(\Om^1_{`U_{ijk,0}/S_{00}^{[q]}},
F_{0*}({\cal O}_{U_{ijk,0}})/{\cal O}_{`U_{ijk,0}})$
is equal to $\om_{ik}$. 
Set ${\cal U}_0:=\{U_{i,0}\}_{i\in I}$ and $`{\cal U}_0:=\{`U_{i,0}\}_{i\in I}$. 
As a result, we obtain the class $\{\om_{ij}\}$ 
in $\check{H}{}^1(`{\cal U}_0,{\cal H}{\it om}_{{\cal O}_{`Y_0}}
(\Om^1_{`Y_0/S_{00}^{[q]}},F_{0*}({\cal O}_{Y_0})/{\cal O}_{`Y_0}))$. 
\par 
We claim that the class $\{\om_{ij}\}$,  more strongly 
the class $\{\wt{\om}_{ij}\}$ is independent of the choice of 
the lift $\wt{F}_i$. 
Indeed, let us take another lift $\wt{F}{}'_i$. 
Then $\{\wt{F}{}'_i-\wt{F}_i\}_i$ defines an element of 
$$\prod_i
{\rm Hom}_{{\cal O}_{`U_{i,0}}}(\Om^1_{`U_{i,0}/S_{00}^{[q]}},
F_{0*}({\cal O}_{U_{i,0}})).$$ 
Hence the class of $\{\wt{\om}_{ij}\}_{ij}$ is independent of 
the choice of the lift $\wt{F}_i$ by 
the following equalities obtained by (\ref{lemm:elin}): 
\begin{align*} 
&(\wt{F}{}'_j\vert_{\wt{U}_{jik}})^*-(`g_{ij}(\wt{F}{}'_i\vert_{\wt{U}_{ijk}})g^{-1}_{ij})^*
-\{(\wt{F}_j\vert_{\wt{U}_{jik}})^*-(`g_{ij}(\wt{F}_i\vert_{\wt{U}_{ijk}})g^{-1}_{ij})^*\}\\
&=(\wt{F}{}'_j\vert_{\wt{U}_{jik}})^*-(\wt{F}_j\vert_{\wt{U}_{jik}})^*
-\{(`g_{ij}(\wt{F}{}'_i\vert_{\wt{U}_{ijk}})g^{-1}_{ij})^*
-(`g_{ij}(\wt{F}_i\vert_{\wt{U}_{ijk}})g^{-1}_{ij})^*\}\\
&=
(\wt{F}{}'_j\vert_{\wt{U}_{jik}})^*-(\wt{F}_j\vert_{\wt{U}_{jik}})^*
-\{(\wt{F}{}'_i\vert_{\wt{U}_{ijk}})^*-(\wt{F}_i\vert_{\wt{U}_{ijk}})^*\}
=(\partial \{\wt{F}{}'_i-\wt{F}_i\}_i)_{ij}. 
\end{align*} 
(The calculation above is missing in \cite{ns}.)
\par 
Next we claim that the class of $\{\om_{ij}\}_{ij}$ is independent of the choice of 
the open covering ${\cal U}_0$. 
This is clear since any two open coverings have a refinement 
and $\wt{F}_i\vert_V$ is a lift of $F\vert_{V}$ for any log open subscheme 
$V$ of $U_i$. 
\par 
As in \cite{ns}, we claim that the class $\{\om_{ij}\}$ is 
the obstruction class of a lift of $(Y,F)/(S_0\lo S^{[q]}_0)$ 
over $(S\lo S^{[q]})$. 
Indeed, if there exists a lift $(\wt{Y},\wt{F})/(S\lo S^{[q]})$ of 
$(Y,F)/(S_0\lo S^{[q]}_0)$, then we can take $g_{ij}$ (resp.~$\wt{F}_i$) 
as the identity ${\rm id}_{\wt{U}_{ij}}$ (resp.~$\wt{F}\vert_{\wt{U}_i}$) 
and hence $\om_{ij}=0$. 
%Assume that $`\os{\circ}{U}_i$ is affine. 
\par 
Conversely assume that $\{\om_{ij}\}=0$ 
in $\check{H}{}^1(`{\cal U}_0,{\cal H}{\it om}_{{\cal O}_{`Y_0}}
(\Om^1_{`Y_0/S_{00}^{[q]}},F_{0*}({\cal O}_{Y_0})/{\cal O}_{`Y_0}))$. 
Then there exists a section 
$\om_i\in {\cal H}{\it om}_{{\cal O}_{`Y_0}}(\Om^1_{`Y_0/S_{00}^{[q]}},
F_{0*}({\cal O}_{Y_0})/{\cal O}_{`Y_0})(`U_{i,0})$ 
such that $\om_{ij}=\om_j-\om_i$. 
Assume that the image of $\os{\circ}{U}_{i,0}$ in $\os{\circ}{S}_0$ is contained in 
an affine open subscheme of $\os{\circ}{S}_0$.  
Since $\os{\circ}{U}_{i,0}$ is affine, so is $`\os{\circ}{U}_{i,0}$. 
Hence the following sequence 
\begin{align*} 
0\lo \Gam(`U_{i,0},{\cal O}_{`Y_0})\os{F_0^*}{\lo} \Gam(`U_{i,0},F_{0*}({\cal O}_{Y_0}))\lo 
\Gam(`U_{i,0},F_{0*}({\cal O}_{Y_0})/{\cal O}_{`Y_0})\lo 0
\tag{4.12.2}\label{ali:wflom} 
\end{align*} 
is exact. 
Because $\os{\circ}{Y}$ is separated, $\os{\circ}{U}_{ij,0}$ is affine and then 
we see that $`\os{\circ}{U}_{ij,0}$ is affine. 
Because $`\os{\circ}{U}_{ij,0}$ is affine, 
the following sequence 
\begin{align*} 
0\lo \Gam(`U_{ij,0},{\cal O}_{`Y_0})\os{F_0^*}{\lo} \Gam(`U_{ij,0},F_{0*}({\cal O}_{Y_0}))\lo 
\Gam(`U_{ij,0},F_{0*}({\cal O}_{Y_0})/{\cal O}_{`Y_0})\lo 0
\tag{4.12.3}\label{ali:wlom} 
\end{align*} 
is exact. 
By using this exact sequence, we see that 
\begin{align*} 
\wt{\om}_{ij}=\wt{\om}_j-\wt{\om}_i+F_0^*(\eta_{ij})
\tag{4.12.4}\label{ali:wom} 
\end{align*}  
in ${\cal H}{\it om}_{{\cal O}_{`Y_0}}(\Om^1_{`Y_0/S_{00}^{[q]}},
F_{0*}({\cal O}_{Y_0}))(`U_{i,0})$ 
for a section $\eta_{ij}\in 
{\cal H}{\it om}_{{\cal O}_{`Y_0}}
(\Om^1_{`Y_0/S_{00}^{[q]}},{\cal O}_{`Y_0})(`U_{ij,0})$.  
Here $\wt{\om}_i\in 
{\cal H}{\it om}_{{\cal O}_{`Y_0}}(\Om^1_{`Y_0/S_{00}^{[q]}},
F_{0*}({\cal O}_{Y_0}))(`U_{i,0})$ is a lift of $\om_i$. 
Change $\wt{F}_i$ by $\wt{F}{}'_i$ such that 
$\wt{F}{}'{}^*_{\! \!i}-\wt{F}{}^*_i=-\wt{\om}_i$ 
((\ref{prop:eq}))
and change $g_{ij}$ by $g'_{ij}$ such that  
$`g'{}^*_{\! \!ij}-`g^*_{ij}=\eta_{ij}$ ((\ref{prop:eq}), (\ref{coro:lys})).   
%Here $\del'_{ij}$ is the $\del$ for $g_{ij}$. 
In the following we denote $\wt{F}_j\vert_{\wt{U}_{ij}}$ by 
$\wt{F}_j$ for simplicity of notation. 
%Then 
%\begin{align*}
%(`(g'_{ij})\wt{F}{}'_jg'^{-1}_{ij})^*&=
%(`(g'_{ij})\wt{F}{}'_jg'^{-1}_{ij})^*-\wt{F}{}'{}^*_{\! \!j}+\wt{F}{}'{}^*_{\! \!j}
%=F_0^*(`(\del'_{ij}))+\wt{F}{}'{}^*_{\! \!j}=F_0^*(`\del_{ij}+\eta_{ij})+\wt{F}{}'{}^*_{\! \!j}\\
%&=(`g_{ij}\wt{F}_jg^{-1}_{ij})^*-\wt{F}{}^*_j-F_0^*(\eta_{ij})+\wt{F}{}'{}^*_{\! \!j}\\
%&=(`g_{ij}\wt{F}_jg^{-1}_{ij})^*-\wt{F}{}^*_i+\wt{F}{}^*_i-F_0^*(\eta_{ij})-\om_i\\
%&=\om_{ij}+\wt{F}{}^*_i-\wt{F}{}'{}^*_{\!\!i}-F_0^*(\eta_{ij})-\om_j+\wt{F}{}'{}^*_{\!\!i}
%=\wt{F}{}'{}^*_{\!\!i}.
%\end{align*}  
Then the equality (\ref{ali:wom}) is equivalent to 
the following equality: 
\begin{align*} 
\wt{F}_j^*-(`g_{ij}\wt{F}_ig^{-1}_{ij})^*
=-\wt{F}{}'{}^*_{\! \!j}+\wt{F}{}^*_j
+\wt{F}{}'{}^*_{\! \!i}-\wt{F}{}^*_i+
F^*_0(`g'{}^*_{\! \! ij}-`g^*_{ij}). 
\end{align*} 
Hence  
\begin{align*} 
\wt{F}{}'{}^*_{\!\!j}-(`(g'_{ij})\wt{F}{}'_ig'^{-1}_{ij})^*
&=
(`g_{ij}\wt{F}_ig^{-1}_{ij})^*-(`(g'_{ij})\wt{F}{}'_ig'^{-1}_{ij})^*
+\wt{F}{}'{}^*_{\! \!i}-\wt{F}{}^*_i+
F^*_0(`g'{}^*_{\! \! ij}-`g^*_{ij}). 
%&= -((`(g'_{ij})\wt{F}{}'_ig'^{-1}_{ij})^*-\wt{F}{}'{}^*_{\! \!i})+
%((`g_{ij}\wt{F}_ig^{-1}_{ij})^*-\wt{F}{}^*_i)+F^*_0(`g'{}^*_{\! \! ij}-`g^*_{ij}).
\tag{4.12.5}\label{ali:gfg}
\end{align*}   
We claim that the right hand side of (\ref{ali:gfg}) vanishes. 
To prove this vanishing, we have to 
make quite strange calculations (at least at first glance) as follows.
(These calculations are missing in \cite{ns}.) 
\par 
Let $a$ be a local section of ${\cal O}_{`\wt{U}_{ji}}$. 
Let $b\in {\cal O}_{`\wt{U}_{ij}}$ 
be a lift of the image of $a$ 
in ${\cal O}_{`U_{ji}}={\cal O}_{`U_{ij}}$. 
By (\ref{lemm:ges}) (1) below, 
we have the following equalities: 
\begin{align*} 
&(`g_{ij}\wt{F}_ig^{-1}_{ij})^*(a)-(`(g'_{ij})\wt{F}{}'_ig'^{-1}_{ij})^*(a)
\\
& 
=(g^{-1}_{ij})^*(\wt{F}_i^*(b)-\wt{F}{}'{}^*_{\! \!i}(b))+
\pi^n(F^*_0(\del_{`g_{ij}}(a,b))-F^*_0(\del_{`(g'_{ij})}(a,b)))\\
%F^*_i(\pi^n\del_{`g_{ij}}(a,b))-F'{}^*_i(\pi^n\del_{`(g'_{ij})}(a,b))\\
&=\wt{F}_i^*(b)-\wt{F}{}'{}^*_{\! \!i}(b)+\pi^nF^*_0(\del_{`g_{ij}}(a,b)))-
\pi^nF^*_0(\del_{`(g'_{ij})}(a,b))).
\end{align*}  
Hence 
\begin{align*} 
&(`g_{ij}\wt{F}_ig^{-1}_{ij})^*(a)-(`(g'_{ij})\wt{F}{}'_ig'^{-1}_{ij})^*(a)
+\wt{F}{}'{}^*_{\! \!i}(b)-\wt{F}{}^*_i(b)\\
&=\pi^n(F^*_0(\del_{`g_{ij}}(a,b)))-F^*_0(\del_{`(g'_{ij})}(a,b)))). 
\end{align*} 
The last term is equal to 
$F^*_0(\del_{`g_{ij}}^*(a,b)-\del_{`(g'_{ij})}^*(a,b))$ 
via the identification $\pi^n{\cal O}_{\wt{Y}_1}\simeq {\cal O}_{Y_0}$. 
This is equal to $F^*_0((`g^*_{ij}-`g'{}^*_{\! \! ij})(a))$ by 
(\ref{lemm:ges}) (2) below. 
Consequently the value of the right hand side of (\ref{ali:gfg}) for 
any $a\in {\cal O}_{`\wt{U}_{ji}}$ and any lift $b\in {\cal O}_{`\wt{U}_{ij}}$ of 
the image of $a$ in ${\cal O}_{`U_{ji}}$ 
is equal to $0$. 
Similarly, by using (\ref{lemm:ges}) (3) and (4) below, 
the value of the right hand side of (\ref{ali:gfg}) for 
any $m\in M_{`\wt{U}_{ji}}$ and any lift $l\in M_{`\wt{U}_{ij}}$ of 
the image of $m$ in $M_{`U_{ji}}$ is equal to $0$. 
In conclusion, the right hand side of (\ref{ali:gfg}) is $0$. 
Hence 
\begin{align*}
\wt{F}{}'_j=
(`g'_{ij}\wt{F}{}'_ig'^{-1}_{ij})^*.
\tag{4.12.6}\label{ali:gfij}
\end{align*} 
Consequently 
\begin{align*}
(`(g'_{jk})`(g'_{ij})\wt{F}{}'_i(g'_{jk}g'_{ij})^{-1})^*=
\wt{F}{}'_k=(`(g'_{ik})\wt{F}{}'_ig'{}^{-1}_{\! \!ki})^*.
\tag{4.12.7}\label{ali:gffj}
\end{align*} 
By (\ref{lemm:fl}) (5), $g'{}^*_{\! \! ki}=(g'_{jk}g'_{ij})^*$. 
Obviously this implies that $g'_{ik}=g'_{jk}g'_{ij}$. 
In this way, we see that $\wt{U}_i$ and $\wt{F}{}'_i$ patch together. 
\par 
(4): Because $\om_{ik}=\om_{ij}+\om_{jk}$ in 
${\rm Hom}_{{\cal O}_{`U_{ij,0}}}(\Om^1_{`U_{ij,0}/S_{00}^{[q]}},
F_{0*}({\cal O}_{U_{ij,0}})/{\cal O}_{`U_{ij,0}})$, 
there exists an element 
$\om_{ijk}\in 
{\rm Hom}_{{\cal O}_{`U_{ijk,0}}}(\Om^1_{`U_{ijk,0}/S_{00}^{[q]}},{\cal O}_{`U_{ijk,0}})$ 
such that $\wt{\om}_{ij}+\wt{\om}_{jk}-\wt{\om}_{ik}=F_0^*(\om_{ijk})$. 
By the definition of $\partial$, 
$\{\om_{ijk}\}\in \wh{H}{}^2(`{\cal U}_0,
{\rm Hom}_{{\cal O}_{`Y_0}}(\Om^1_{`Y_0/S_{00}^{[q]}},{\cal O}_{`Y_0}))$ 
is the element 
$\partial({\rm obs}_{(Y,F)/(S_0\subset S,F^{[e]}_S)})$.  
On the other hand, by the definition of $\wt{\om}_{ij}$, 
%$$\wt{F}_j^*\vert_{\wt{U}_{ij}}-
%(g^{-1}_{ij})^*(\wt{F}_i\vert_{\wt{U}_{ij}})^*`g_{ij}^*=
%\wt{F}_j^*\vert_{\wt{U}_{ij}}-(`g_{ij}(\wt{F}_i\vert_{\wt{U}_{ij}})g^{-1}_{ij})^*=\wt{\om}_{ij}.$$ 
$$\wt{F}_j^*-
(g^{-1}_{ij})^*(\wt{F}_i)^*`g_{ij}^*=
\wt{F}_j^*-(`g_{ij}(\wt{F}_i)g^{-1}_{ij})^*=\wt{\om}_{ij}.$$ 
By this equality, we also have the following equality by (\ref{lemm:elin}): 
$$g_{ij}^*\wt{F}_j^*\vert_{\wt{U}_{ij}}
(`g^{-1}_{ij})^*-\wt{F}{}^*_i\vert_{\wt{U}_{ij}}=\wt{\om} _{ij}.$$ 
Hence 
\begin{align*}
((`g_{ik}^{-1}`g_{jk}`g_{ij})^{-1}
\wt{F}_i(g_{ik}^{-1}g_{jk}g_{ij}))^*-\wt{F}{}^*_i
&=g_{ij}^*g_{jk}^*(g^{-1}_{ik})^*\wt{F}_i^*`g_{ik}^*
(`g^{-1}_{jk})^*(`g^{-1}_{ij})^*-\wt{F}{}^*_i\\
&=g_{ij}^*g_{jk}^*(\wt{F}_k^*-\wt{\om}_{ik})(`g^{-1}_{jk})^*(`g^{-1}_{ij})^*-\wt{F}{}^*_i\\
&=g_{ij}^*(g_{jk}^*\wt{F}_k^*(`g^{-1}_{jk})^*-\wt{\om}_{ik})(`g^{-1}_{ij})^*-\wt{F}{}^*_i\\
&=-\wt{\om}_{ik}+\wt{\om}_{jk}+\wt{\om}_{ij}=F_0^*(\om_{ijk}).
\end{align*} 
Let $\del_{ijk}$ be the $\del$ for $g_{ijk}=g_{ik}^{-1}g_{jk}g_{ij}$. 
Then 
\begin{align*}
((`g_{ik}^{-1}`g_{jk}`g_{ij})^{-1}
\wt{F}_i(g_{ik}^{-1}g_{jk}g_{ij}))^*-\wt{F}{}^*_i\
%((`(g'_{ij}g'_{jk}g'{}^{-1}_{ik})\wt{F}_i((g_{ij}g_{jk}g_{ik}^{-1})^{-1})^*-\wt{F}{}^*_i
=(`g_{ijk}\wt{F}_kg_{ijk}^{-1})^*-\wt{F}{}^*_i
=F_0^*(\del_{ijk})
\end{align*} 
by (\ref{lemm:fl}) (2) and (4). 
%Here we have identified 
%$Z^2(`{\cal U},{\cal H}{\it om}_{{\cal O}_{`Y}}
%(\Om^1_{`Y/S_0},{\cal I}{\cal O}_{`\wt{Y}}))$ with 
%$Z^2(`{\cal U}_0,{\cal H}{\it om}_{{\cal O}_{`Y_0}}(\Om^1_{`Y_0/S_{00}},{\cal O}_{`Y_0}))$.
Hence $F_0^*(\del_{ijk})=F_0^*(\om_{ijk})$. 
Since $F_0^*$ is injective,
$\del_{ijk}=\om_{ijk}$. This implies the desired equality 
${\rm obs}_{`Y/(S^{[q]}_0\subset S^{[q]})}
=\partial ({\rm obs}_{(Y,F)/(S_0\subset S,F^{[e]}_S)})$. 
\par 
(5): (5) follows from (\ref{prop:eq}) and the argument in the proof of (4). 
Indeed, we have only to set $g_{ij}={\rm id}_{\wt{U}_{ij}}$, 
where $\wt{U}_{ij}$ is an open log subscheme of $\wt{Y}$ corresponding to 
$U_{ij}$ in $Y$.  
\end{proof} 
 
We have to give analogues of (\ref{lemm:fl}) (2) and (4) 
for the proof of (\ref{theo:ts}) (3). This is a non-trivial lemma:

\begin{lemm}\label{lemm:ges}
Let the notations be before {\rm (\ref{lemm:fl})}. 
Let $\wt{Y}_i$ $(i=1,2)$ be a lift of $Y$ over $S$. 
Assume that there exists an isomorphism 
$g\col \wt{Y}_1\os{\sim}{\lo} \wt{Y}_2$ over $S$ 
such that $g\vert_Y={\rm id}_Y$. 
Let $`g\col `\wt{Y}_1\os{\sim}{\lo} `\wt{Y}_2$ 
be the induced isomorphism by $g$. 
Let $(\wt{Y}_1,\wt{F})$ be a lift of $(Y,F)$ over $S$. 
Then the following hold$:$ 
\par 
$(1)$ For a local section $a$ of ${\cal O}_{`\wt{Y}_2}$,  
let $\ol{a}$ be the image of $a$ in ${\cal O}_{`Y}$ 
and let $b$ be a lift of $\ol{a}$ in ${\cal O}_{`\wt{Y}_1}$.
%Let $a\in {\cal O}_{`Y_2}$ be a lift of $a$. 
%Set $\del_{12}(a):=g_{2}^*(\wt{a}_2)-\wt{a}_1$. 
Let $\del_{`g}(a,b)$ be 
a unique local section of ${\cal O}_{`Y_0}$ 
such that $\pi^n\del_{`g}(a,b)=`g^*(a)-b$. 
Here we have considered $\pi^n\del_{`g}(a,b)$ as a local section of 
${\cal O}_{`\wt{Y}_1}$. 
Then 
\begin{align*} 
(g^{-1})^*\wt{F}{}^*(`g)^*(a)=
(g^{-1})^*\wt{F}{}^*(b)+\pi^nF_0^*(\del_{`g}(a,b)).
\tag{4.13.1}\label{ali:gfa} 
\end{align*}  
\par 
$(2)$ Let the notations be as in {\rm (1)}. 
Let $h$ be another isomorphism 
$h\col \wt{Y}_1\os{\sim}{\lo} \wt{Y}_2$ over $S$ 
such that $h\vert_Y={\rm id}_Y$. 
Then 
$\pi^n(\del_{`h}(a,b)-\del_{`g}(a,b))=`h^*(a)-`g^*(a)$. 
In particular this is independent of the choice of $b$. 
\par 
$(3)$ For a local section $m$ of $M_{`\wt{Y}_2}$,  
let $\ol{m}$ be the image of $m$ in $M_{`Y}$ 
and let $l$ be a lift of $\ol{m}$ in  $M_{{}^{\star}\wt{Y}_1}$.
%Let $a\in {\cal O}_{`Y_2}$ be a lift of $a$. 
%Set $\del_{12}(a):=g_{2}^*(\wt{a}_2)-\wt{a}_1$. 
Let $\del_{`g}(m,l)$ be a unique local section of ${\cal O}_{`Y_0}$ 
such that 
$`g^*(m)=l(1+\pi^n\del_{`g}(m,l))$. 
Then 
\begin{align*}  
(g^{-1})^*\wt{F}{}^*(`g)^*(m)=
(g^{-1})^*(\wt{F}{}^*(l))(1+\pi^nF_0^*(\del_{`g}(m,l))).
\tag{4.13.2}\label{ali:gfmm}  
\end{align*}  
\par 
$(4)$ Let the notations be as in {\rm (2)} and {\rm (4)}. 
Then 
$\del_{`h}(m,l)-\del_{`g}(m,l)=`h^*(m)(`g^*(m))^{-1}$. 
In particular this is independent of the choice of $l$. 
\end{lemm}
\begin{proof} 
(1): We have the following equalities since $g\vert_Y={\rm id}_Y$: 
\begin{align*} 
(g^{-1})^*\wt{F}{}^*(`g)^*(a)&=
(g^{-1})^*\wt{F}{}^*(b+\pi^n\del_{`g}(a,b))\\
&=(g^{-1})^*(\wt{F}{}^*(b)+\wt{F}{}^*(\pi^n\del_{`g}(a,b)))\\
&=(g^{-1})^*\wt{F}{}^*(b)+\pi^n F^*_0(\del_{`g}(a,b)).\\
\end{align*} 
\par 
(2): Obvious. 
\par 
(3): We have the following equalities since $g\vert_Y={\rm id}_Y$: 
\begin{align*} 
(g^{-1})^*\wt{F}{}^*(`g)^*(m)&=
(g^{-1})^*\wt{F}{}^*(l(1+\pi^n\del_{`g}(m,l)))\\
&=(g^{-1})^*(\wt{F}{}^*(l)(1+\wt{F}{}^*(\pi^n\del_{`g}(m,l))))\\
& = (g^{-1})^*(\wt{F}{}^*(l))(1+\pi^nF^*_0(\del_{`g}(m,l))). 
\end{align*} 
\par 
(4): Obvious. 
\end{proof}

\begin{rema}\label{rema:tlt} 
The lemma (\ref{lemm:ges}) in the trivial logarithmic case is missing in \cite{ns}. 
The strange calculation to prove the equality (\ref{ali:gfij})  
in the trivial logarithmic case is also missing in [loc.~cit.]. 
These are indispensable in [loc.~cit.] for the proof of the equality 
(\ref{ali:gfij}); the complicatedness for the proof 
arises because we have to make calculations for local sections in 
$F_{0*}({\cal O}_{Y_0})$ not in 
$F_{0*}({\cal O}_{Y_0})/{\cal O}_{`Y_0}$.  
\end{rema}  
 
\begin{coro}\label{coro:nc}
Assume that $e=1$. 
Assume also that $Y_0/S_{00}$ is of Cartier type. 
%Assume also that the pull-back $\wt{F}^*$ of 
%any local lift $\wt{F}$ of $F$ is $\pi$-linear. 
Then the following hold$:$ 
\par 
$(1)$ The sheaf ${\rm Lift}_{(Y,F)/(S_0\subset S,F_S)}$ on 
$\os{\circ}{Y}$ is a torsor under 
$`{\rm pr}_{0*}({\cal H}{\it om}_{{\cal O}_{`Y_0}}(\Om^1_{`Y_0/S^{[p]}_{00}},
F_{0*}(B\Om^1_{Y_0/S_{00}})))$. 
%if $(Y,F)/S_0$ 
%has a lift $(\wt{Y},\wt{F})/S$, then there exists the following bijection 
%\begin{align*} 
%{\rm Lift}_{(Y,F)/(S_0\subset S,F_S)}(Y)\os{\sim}{\lo}
%{\rm Hom}_{{\cal O}_{`Y_0}}(\Om^1_{`Y_0/S^{[p]}_{00}},
%F_{0*}(B\Om^1_{Y_0/S_{00}})). 
%\tag{4.17.1}\label{ali:obo0ys}
%\end{align*} 
\par 
$(2)$ 
%The zariski sheaf ${\rm Lift}_{(Y,F)/(S,F_S)}$ on $\os{\circ}{Y}$ is a torsor 
%under ${\cal H}{\it om}_{{\cal O}_{Y_0}}(\Om^1_{Y_0/S_{00}},B\Om^1_{Y_0/S_{00}});$ 
Assume that $\os{\circ}{Y}$ is separated.  
In 
$${\rm Ext}_{`Y_0}^1(\Om^1_{`Y_0/S^{[p]}_{00}},
F_{0*}(B\Om^1_{Y_0/S_{00}})),$$
there exists a canonical obstruction class 
${\rm obs}_{(Y,F)/(S_0\subset S,F_S)}$ 
of a lift of $(Y,F)/(S_0\lo S^{[p]}_0)$ over $S\lo S^{[p]}$. 
\par 
$(3)$ 
Assume that $\os{\circ}{Y}$ is separated. Let 
\begin{align*} 
\partial \col {\rm Ext}_{`Y_0}^1(\Om^1_{`Y_0/S^{[p]}_{00}}, F_{0*}(B\Om^1_{Y_0/S_{00}}))
\lo {\rm Ext}_{`Y_0}^2(\Om^1_{`Y_0/S^{[p]}_{00}},{\cal O}_{`Y_0})
\end{align*} 
be the boundary morphism obtained by 
the following exact sequence {\rm (\ref{prop:ee})}$:$
\begin{align*} 
0\lo {\cal O}_{`Y_0}\lo F_{0*}({\cal O}_{Y_0})\os{F_{0*}(d)}\lo 
F_{0*}(B\Om^1_{Y_0/S_{00}})\lo 0. 
\tag{4.15.1}\label{ali:obobbs}
\end{align*} 
Then $\partial ({\rm obs}_{(Y,F)/(S_0\subset S,F_S)})=
{\rm obs}_{`Y/(S^{[p]}_0\subset S^{[p]})}$. 
\end{coro} 
\begin{proof}
Recall that $Y'_0:=Y_0\times_{S_{00},F_{S_{00}}}S_{00}$ 
%Because $Y_0/S_{00}$ is of Cartier type, it is integral. 
%Hence 
and that $\os{\circ}{Y}{}'_0=`\os{\circ}{Y}_0$. 
%Furthermore $\os{\circ}{Y}_0$ is reduced by (\ref{rema:rds}). 
By (\ref{ali:cifes}) 
\begin{align*} 
F_{0*}({\cal O}_{Y_0})/{\cal O}_{Y'_0}=
F_{0*}(B\Om^1_{Y_0/S_{00}}).
\tag{4.15.2}\label{ali:yy00}
\end{align*}  
Now (\ref{coro:nc}) follows from (\ref{theo:ts}) and (\ref{ali:yy00}). 
\end{proof}

\par 
Next we develop a log deformation theory 
with (standard) relative Frobenius morphisms. 
Because the proof of the main result (\ref{theo:trs}) below 
are very similar to that of (\ref{theo:ts}), we omit it.  
The log deformation theory with abrelative Frobenius morphisms 
and the theory with relative Frobenius morphisms turns out 
equivalent theories by (\ref{coro:nofs}) below.  
To obtain this equivalence, we can also use  
W.~Zheng's proof in \cite{ss}. 
See (\ref{rema:ss}) below for this. 
\par 
Let the notations be as before.  
Set 
%$F_{S_0}:=F_S\vert _{S_0}$, 
%$F_{S_{00}}:=F_S\vert _{S_{00}}$, 
$Y':=Y\times_{S_0,F^{[e]}_{S_0}}S_0$ 
and $Y'_0:=Y\times_{S_{00},F^{e}_{S_{00}}}S_{00}$. 
\par 
Let $F_0 \col Y_0\lo Y'_0$ be the $e$-iterated relative Frobenius morphism of $Y_0/S_{00}$. 
Assume that there exists a lift $F\col Y\lo Y'$ 
of  $F_0\col Y_0\lo Y_0'$ over $S$. 
We say that $(\wt{Y},\wt{F})/S$ 
is a {\it log smooth integral lift} (or simply a {\it lift}) of $(Y,F)/S_0$ 
if $\wt{Y}$ is a log smooth integral scheme over $S$ such that 
$\wt{Y}\times_SS_0=Y$ and $\wt{F}$ is a morphism 
$\wt{Y}\lo \wt{Y}{}':=\wt{Y}\times_{S,F^{[e]}_S}S$ over $S$ fitting into 
the following commutative diagram 
\begin{equation*} 
\begin{CD} 
Y@>{\subset}>> \wt{Y}\\
@V{F}VV @VV{\wt{F}}V \\
Y'@>{\subset}>> \wt{Y}{}'  
\end{CD} 
\end{equation*}
over the morphism $S_0\os{\subset}{\lo}S$. 
%\begin{equation*} 
%\begin{CD} 
%S_0@>{\subset}>> S\\
%@V{F_{S_0}}VV @VV{F_S}V \\
%S_0@>{\subset}>> S.  
%\end{CD} 
%\end{equation*}
\par 
Let ${\rm Lift}_{(Y,F)/(S_0\subset S,F^{[e]}_S)}$ be the following sheaf 
\begin{align*} 
{\rm Lift}_{(Y,F)/(S_0\subset S)}(U):=
\{{\rm isomorphism~classes~of~lifts~of}~(U,F\vert_U)/S_0~{\rm over}~S\}
\end{align*}
for each log open subscheme $U$ of $Y$. 
The isomorphism class of a lift of $(U,F\vert_U)/S_0$ over $S$ is 
defined in an obvious way. 
%Let the notations be as above. 
Let $\iota' \col Y' \os{\sus}{\lo} \wt{Y}{}'$ be the closed immersion. 
Let $\wt{G}\col \wt{Y}\lo \wt{Y}{}'$ be another lift of $F$. 
Take $s$ in (\ref{prop:eq}) as the composite morphism 
$Y\os{F}{\lo} Y'\os{\iota'}{\lo} \wt{Y}{}'$ over 
the composite morphism $S_0\subset S$. 
Then  
$\wt{G}$ defines an element $\wt{G}{}^*-\wt{F}{}^*$ of 
\begin{align*} 
{\rm Hom}_{{\cal O} _Y}
(F^*\iota'{}^*(\Om^1_{\wt{Y}{}'/S}),{\cal I}{\cal O}_{\wt{Y}})
=
{\rm Hom}_{{\cal O} _Y}(F^*(\Om^1_{Y'/S_0}),{\cal I}{\cal O}_{\wt{Y}})
={\rm Hom}_{{\cal O}_{Y'}}(\Om^1_{Y'/S_0},F_*({\cal I}{\cal O}_{\wt{Y}})).
\tag{4.15.3}\label{ali:forroi}
\end{align*}

\begin{theo}\label{theo:trs}  
Let ${\cal I}$, $\pi$ and $n$ be as in {\rm (\ref{lemm:fl})}. 
Assume that $\os{\circ}{Y}_0$ is reduced.   
Then the following hold$:$ 
\par 
$(1)$ 
Assume that $(Y,F)/S_0$ has a lift 
$(\wt{Y},\wt{F})/S$. 
Set ${\rm Aut}_{S,F^{[e]}_S}(\wt{Y},Y)
:=\{g\in {\rm Aut}_{S}(\wt{Y})~\vert~g\vert_Y={\rm id}_Y, 
\wt{F}\circ g=g'\circ \wt{F}\}$. 
Then ${\rm Aut}_{S,F^{[e]}_S}(\wt{Y},Y)=\{{\rm id}_Y\}$. 
\par 
$(2)$ Let ${\rm pr}_0\col Y'_0\lo Y_0$ be the projection. 
Then the sheaf ${\rm Lift}_{(Y,F)/(S_0\subset S,F^{[e]}_S)}$ on 
$\os{\circ}{Y}$ is a torsor under 
${\rm pr}_{0*}({\cal H}{\it om}_{{\cal O}_{Y'_0}}(\Om^1_{Y'_0/S_{00}}, 
F_{0*}({\cal O}_{Y_0})/{\cal O}_{Y'_0}))$. 
%if $(Y,F)/S_0$ has a lift $(\wt{Y},\wt{F})/S$, then there exists the following bijection of sets$:$  
%\begin{align*} 
%{\rm Lift}_{(Y,F)/S}(Y)\os{\sim}{\lo}
%{\rm Hom}_{{\cal O}_{Y'_0}}(\Om^1_{Y'_0/S_{00}},
%F_{0*}({\cal O}_{Y_0})/{\cal O}_{Y'_0}). 
%\tag{5.2.1}\label{ali:oror0ys}
%\end{align*} 
\par 
$(3)$ 
%The zariski sheaf ${\rm Lift}_{(Y,F)/(S,F_S)}$ on $\os{\circ}{Y}$ is a torsor 
%under ${\cal H}{\it om}_{{\cal O}_{Y_0}}(\Om^1_{Y_0/S_{00}},B\Om^1_{Y_0/S_{00}});$ 
Assume that $\os{\circ}{Y}$ is separated.  
In 
$${\rm Ext}_{Y'_0}^1(\Om^1_{Y'_0/S_{00}},F_{0*}({\cal O}_{Y_0})/{\cal O}_{Y'_0}),$$ 
there exists a canonical obstruction class 
${\rm obs}_{(Y,F)/(S_0\subset S,F^{[e]}_S)}$ of a lift of $(Y,F)/S_0$ over $S$. 
\par 
$(4)$ Assume that $\os{\circ}{Y}$ is separated. 
Let 
\begin{align*} 
\partial \col {\rm Ext}_{Y'_0}^1(\Om^1_{Y'_0/S_{00}},
F_{0*}({\cal O}_{Y_0})/{\cal O}_{Y'_0})
\lo {\rm Ext}_{Y'_0}^2(\Om^1_{Y'_0/S_{00}},{\cal O}_{Y'_0})
\end{align*} 
be the boundary morphism obtained by 
the following exact sequence {\rm (\ref{prop:ee})}$:$
\begin{align*} 
0\lo {\cal O}_{Y'_0}\lo F_{0*}({\cal O}_{Y_0})\lo 
F_{0*}({\cal O}_{Y_0})/{\cal O}_{Y'_0}\lo 0. 
\tag{4.16.1}\label{ali:orobrbs}
\end{align*} 
Then $\partial ({\rm obs}_{(Y,F)/(S_0\subset S,F^{[e]}_S)})=
{\rm obs}_{Y'/(S_0\subset S)}$. 
\par 
\par 
$(5)$ Assume that $\os{\circ}{Y}$ is separated. 
Assume that there exists a lift $\wt{Y}/S$ of $Y/S_0$. 
Let $\wt{Y}{}'$ be the base change of $\wt{Y}$ by 
the morphism $F^{[e]}_S\col S\lo S$. 
Then the obstruction class ${\rm obs}_{F}(\wt{Y})$ 
of a lift $\wt{F}\col \wt{Y}\lo \wt{Y}{}'$ of $F\col Y\lo Y'$  
is an element of ${\rm Ext}^1_{Y'_0}(\Om^1_{Y'_0/S_{00}},F_{*}({\cal O}_{Y_0}))$ 
and this is mapped to ${\rm obs}_{(Y,F)/(S_{0}\subset S)}$ 
by the natural morphism 
\begin{align*} 
\partial \col {\rm Ext}_{Y'_0}^1(\Om^1_{Y'_0/S_{00}}, F_{0*} ({\cal O}_{Y_0}))
\lo {\rm Ext}_{Y'_0}^1(\Om^1_{Y'_0/S_{00}}, F_{0*} ({\cal O}_{Y_0})/{\cal O}_{Y'_0}). 
\end{align*} 
\end{theo}
\begin{proof} 
We omit the proof because it is the same as that of (\ref{theo:ts}). 
\end{proof} 
 
\begin{coro}\label{coro:nofs}
Let $\bet_0 \col Y'_0\lo `Y_0$ and 
$\bet \col Y'\lo `Y$ be the natural morphisms.  
Set $`F:=\bet \circ F$.  Then the following hold$:$
\par 
$(1)$ 
Let 
\begin{align*} 
{\rm Lift}_{(Y,`F)/(S_0\subset S,F^{[e]}_S)}
\lo
{\rm Lift}_{(Y,F)/(S_0\subset S,F^{[e]}_S)} 
\tag{4.17.1}\label{ali:bcfss} 
\end{align*}  
be the natural morphism of sheaves in $Y_{\rm zar}$ 
obtained by the base changes of the lifts of the open log subschemes of 
$`Y$ by the morphism $F^{[e]}_{S/\os{\circ}{S}}\col S\lo S^{[q]}$. 
Assume that ${\rm Lift}_{(Y,`F)/(S_0\subset S,F^{[e]}_S)}(Y)$ is not empty. 
Then the following diagram is commutative$:$
\begin{equation*} 
\begin{CD}
{\rm Lift}_{(Y,F)/(S_0\subset S,F^{[e]}_S)}(Y)@>{\sim}>>
{\rm Hom}_{{\cal O}_{Y'_0}}(\Om^1_{Y'_0/S_{00}},
F_{0*}({\cal O}_{Y_0})/{\cal O}_{Y'_0})\\  
@AAA @A{\simeq}A{\bet^*_0}A \\
{\rm Lift}_{(Y,`F)/(S_0\subset S,F^{[e]}_S)}(Y)@>{\sim}>>
{\rm Hom}_{{\cal O}_{`Y_0}}
(\Om^1_{`Y_0/S^{[p]}_{00}},`F_{0*}({\cal O}_{Y_0})/{\cal O}_{`Y_0})
\end{CD} 
\tag{4.17.2}\label{ali:bcfcdss} 
\end{equation*} 
%In particular, ${\rm Lift}_{(Y,F)/(S_0\subset S,F_S)}(Y)=
%{\rm Lift}_{(Y,F)/(S_0\subset S,F_S)}(Y)$. 
\par 
$(2)$ Assume that $\os{\circ}{Y}$ is separated.  
Let 
\begin{align*} 
\bet^*_0\col 
{\rm Ext}_{`Y_0}^1(\Om^1_{`Y_0/S^{[p]}_{00}},
F_{0*}({\cal O}_{Y_0})/{\cal O}_{`Y_0})\os{\sim}{\lo} 
{\rm Ext}_{Y'_0}^1(\Om^1_{Y'_0/S_{00}},
F_{0*}({\cal O}_{Y_0})/{\cal O}_{Y'_0})
\end{align*}
be the natural isomorphism. 
Then 
\begin{align*} 
\bet^*_0({\rm obs}_{(Y,`F)/(S_0\subset S,F_S^{[e]})})=
{\rm obs}_{(Y,F)/(S_0\subset S,F_S^{[e]})}.
\end{align*}  
\par 
$(3)$ The following diagram is commutative for $q\in {\mab Z}:$
\begin{equation*} 
\begin{CD} 
{\rm Ext}_{Y'_0}^q(\Om^1_{Y'_0/S_{00}}, F_{0*}({\cal O}_{Y_0})/{\cal O}_{Y'_0})
@>{\partial}>> {\rm Ext}_{Y'_0}^{q+1}(\Om^1_{Y'_0/S_{00}},{\cal O}_{Y'_0})\\
@A{\bet^*_0}A{\simeq}A @A{\simeq}A{\bet^*_0}A \\
{\rm Ext}_{`Y_0}^q(\Om^1_{`Y_0/S^{[p]}_{00}},
`F_{0*}({\cal O}_{Y_0})/{\cal O}_{`Y_0})
@>{\partial}>> {\rm Ext}_{`Y_0}^{q+1}
(\Om^1_{`Y_0/S^{[p]}_{00}},{\cal O}_{`Y_0}). 
\end{CD} 
\tag{4.17.3}\label{ali:bcf0oss} 
\end{equation*}  
\end{coro} 
\begin{proof} 
(1), (2), (3): 
If $f\col V\lo W$ is a morphism of fine log schemes over $S^{[q]}$, 
then we obtain the base change 
$V\times_{S^{[q]},F_{S/\os{\circ}{S}}^{[e]}}S\lo W\times_{S,F_{S/\os{\circ}{S}}^{[e]}}S$ of $f$ over $S$. 
This base change defines the left vertical morphism in 
(\ref{ali:bcfcdss}).  
Recall that $Y'_0:=Y_0\times_{S_{00},F^e_{S_{00}}}S_{00}$. 
Because $Y_0/S_{00}$ is integral,  
$\os{\circ}{Y}{}'_0=`\os{\circ}{Y}_0$. 
By the isomorphism before \cite[(1.8)]{klog1}, 
we also obtain the equality 
$\Om^1_{`Y_0/S^{[p]}_{00}}=\Om^1_{Y'_0/S_{00}}$; 
$\bet^*_0$ is nothing but the identity.  
Hence we obtain (1), (2) and (3).  
\end{proof} 

\begin{coro}\label{coro:nrc}
Assume that $e=1$. 
Assume that $Y_0/S_{00}$ is of Cartier type. 
%Assume also that the pull-back $\wt{F}^*$ of 
%any local lift $\wt{F}$ of $F$ is $\pi$-linear. 
Then the following hold$:$ 
\par 
$(1)$ 
Let ${\rm pr}_0\col Y'_0\lo Y_0$ be the projection. 
Then the sheaf ${\rm Lift}_{(Y,F)/(S_0\subset S,F^{[e]}_S)}$ on 
$\os{\circ}{Y}$ is a torsor under 
${\rm pr}_{0*}({\cal H}{\it om}_{{\cal O}_{Y'_0}}
(\Om^1_{Y'_0/S_{00}},F_{0*}(B\Om^1_{Y_0/S_{00}}))$. 
%If $(Y,F)/S_0$ 
%has a lift $(\wt{Y},\wt{F})/S$, then there exists the following bijection of sets$:$
%\begin{align*} 
%{\rm Lift}_{(Y,F)/(S_0\subset S,F_S)}(Y)\os{\sim}{\lo}
%{\rm Hom}_{{\cal O}_{Y'_0}}(\Om^1_{Y'_0/S_{00}},F_{0*}(B\Om^1_{Y_0/S_{00}})). 
%\tag{5.4.1}\label{ali:orbo0ys}
%\end{align*} 
\par 
$(2)$  Assume that $\os{\circ}{Y}$ is separated.
In 
$${\rm Ext}_{Y'_0}^1(\Om^1_{Y'_0/S_{00}},
F_{0*}(B\Om^1_{Y_0/S_{00}})),$$
there exists a canonical obstruction class 
${\rm obs}_{(Y,F)/(S_0\subset S,F^{[e]}_S)}$ of a lift of $(Y,F)/S_0$ over $S$. 
\par 
$(3)$ Assume that $\os{\circ}{Y}$ is separated. Let 
\begin{align*} 
\partial \col {\rm Ext}_{Y'_0}^1
(\Om^1_{Y'_0/S_{00}}, F_{0*}(B\Om^1_{Y'_0/S_{00}}))
\lo {\rm Ext}_{Y'_0}^2(\Om^1_{Y'_0/S_{00}},{\cal O}_{Y'_0})
\end{align*} 
be the boundary morphism obtained by 
the following exact sequence {\rm (\ref{prop:ee})}$:$
\begin{align*} 
0\lo {\cal O}_{Y'_0}\lo F_{0*}({\cal O}_{Y_0})\os{F_{0*}(d)}\lo 
F_{0*}(B\Om^1_{Y_0/S_{00}})\lo 0. 
\tag{4.18.1}\label{ali:orbobbs}
\end{align*} 
Then $\partial ({\rm obs}_{(Y,F)/(S_0\subset S,F^{[e]}_S)})=
{\rm obs}_{Y'/(S_0\subset S)}$. 
\end{coro} 
\begin{proof}
By (\ref{ali:cifes}) 
\begin{align*} 
F_{0*}({\cal O}_{Y_0})/{\cal O}_{Y'_0}=
F_{0*}(B\Om^1_{Y_0/S_{00}}).
\tag{4.18.2}\label{ali:yzy00}
\end{align*}  
Now (\ref{coro:nrc}) immediately follows from (\ref{coro:nofs}) and (\ref{ali:yzy00}). 
\end{proof}

\begin{rema}\label{rema:ss} 
Let the notations be as in (\ref{coro:nofs}). 
More directly, we also obtain the following isomorphism of sheaves on $\os{\circ}{Y}$ 
by W.~Zheng's proof in \cite[(2.5)]{ss}: 
\begin{align*} 
{\rm Lift}_{(Y,`F)/(S_0\subset S,F^{[e]}_S)}\os{\sim}{\lo} 
{\rm Lift}_{(Y,F)/(S_0\subset S,F^{[e]}_S)}. 
\tag{4.19.1}\label{ali:bcdrfss} 
\end{align*} 
Indeed, we have only to construct the inverse of the natural morphism 
(\ref{ali:bcfss}). 
Assume that we are given a representable of an element 
$(\wt{U},\wt{U}{}',\wt{F})$ of ${\rm Lift}_{(Y,F)/(S_0\subset S,F^{[e]}_S)}$. 
Then, following [loc.~cit.], consider 
the sum $\os{\circ}{\wt{U}{}'}\coprod_{\os{\circ}{U'}}\os{\circ}{`U}$ of schemes, 
where $\os{\circ}{U'}\os{\sus}{\lo} \os{\circ}{\wt{U}{}'}$ 
is the natural exact closed immersion and 
the morphism $\os{\circ}{U'}\lo `\os{\circ}{U}$ is the identity. 
Hence this sum of the schemes is isomorphic to 
$\os{\circ}{\wt{U}{}'}$. 
Endow this scheme with the log structure 
$$M_{\wt{U}{}'}\times_{M_{U'}}M_{`U}
=M_{\wt{U}{}'}\times_{M_{`U}\oplus_{M_{S^{[q]}}}M_S}M_{`U}$$ 
with natural composite structural morphism 
$M_{\wt{U}{}'}\times_{M_{U'}}M_{`U}\subset M_{\wt{U}{}'}\lo {\cal O}_{\wt{U}{}'}$. 
Let $`\wt{U}$ be the resulting log scheme. 
Then $\wt{F}\col \wt{U}\lo \wt{U}{}'$ induces a morphism 
$`\wt{F}\col \wt{U}\lo `\wt{U}{}$. 
The triple $(\wt{U},`\wt{U}{},`\wt{F})$ is the desired object of 
${\rm Lift}_{(Y,`F)/(S_0\subset S,F^{[e]}_S)}$ 
since $(M_{\wt{U}{}'}\times_{M_{`U}\oplus_{M_{S^{[q]}}}M_S}M_{`U})
\oplus_{M_{`\wt{U}}}M_{\wt{U}{}'}=M_{\wt{U}{}'}$. 
\end{rema}

%Assume that $S_0$ is of characteristic $p>0$ and that $Y/S_0$ is of Cartier type.  
%Let ${\rm Sec}_C$ be the following sheaf 
%\begin{align*} 
%{\rm Sec}_{C}(U):=\{{\cal O}_{Y'}{\textrm -}{\rm linear~sections~of}~
%C\col \Gam(U,F_*(Z\Om^1_{Y/S_0}))\lo \Gam(U,\Om^1_{Y'/S_0})\}
%\end{align*}
%for each log open subscheme $U$ of $Y$. 
%\par 

%The following is a log version of \cite[(2.2) {\bf 1}]{y} (cf.~\cite[Theorem 3.5]{di}). 

Assume that $S_0$ is of characteristic $p>0$ and 
that $Y/S_{0}$ is of Cartier type.  
Let ${\rm Sec}_C$ be the following sheaf 
\begin{align*} 
{\rm Sec}_{C}(U):=\{\Gam(U',{\cal O}_{U'}){\textrm -}{\rm linear~sections~of}~
C\col F_*(Z\Om^1_{U/S_{0}})\lo \Om^1_{U'/S_{0}}\}
\end{align*}
for each log open subscheme $U$ of $Y_0$. 
Here recall the following exact sequence 
\begin{align*} 
0\lo F_*(B\Om^1_{U/S_{0}})\lo F_*(Z\Om^1_{U/S_{0}})\os{C}{\lo} 
\Om^1_{U'/S_{0}}\lo 0. 
\end{align*} 

\par

The following is the log version of 
a generalization of \cite[(2.2) {\bf 1}]{y} (cf.~\cite[Theorem 3.5]{di}). 

\begin{theo}\label{theo:uryn}
Let the assumptions be as in {\rm (\ref{theo:trs})}. 
Assume that $e=1$ and $n=1$ and that $Y=Y_0$ and $S_0=S_{00}$.   
Assume also that $\pi=p$ and that 
$\os{\circ}{S}$ is flat over ${\rm Spec}({\mab Z}/p^2)$. 
Then there exists the following canonical isomorphism of sheaves on $\os{\circ}{Y}:$
\begin{align*} 
{\rm Lift}_{(Y,F)/(S_0\subset S,F_S)}\os{\sim}{\lo} 
{\rm Sec}_{C}. 
\tag{4.20.1}\label{ali:erapb} 
\end{align*} 
\end{theo}
\begin{proof} 
Let $(\wt{U},\wt{F})$ be a representative of an element of 
${\rm Lift}_{(Y,F)/(S_0\subset S,F_S)}(U)$. 
We have to construct a morphism 
${\rm Lift}_{(Y,F)/(S_0\subset S,F_S)}(U)\lo 
{\rm Sec}_{C}(U)$ of sets. 
Let $t \col {\cal F}\lo {\cal G}$ be a morphism of abelian sheaves on 
$\os{\circ}{\wt{U}}$. If ${\cal G}$ is a sheaf of flat ${\mab Z}/p^2$-modules 
in $\wt{U}_{\rm zar}$  
and if ${\rm Im}(t)\subset p{\cal G}$, then 
we can define a unique morphism 
$p^{-1}t\col {\cal F}/p\lo {\cal G}/p$ fitting into the following commutative diagram:  
\begin{equation*} 
\begin{CD} 
{\cal F}@>{t}>>{\cal G}\\
@V{{\rm proj}.}VV @AA{p\times}A \\
{\cal F}/p@>{p^{-1}t}>> {\cal G}/p. 
\end{CD}
\end{equation*} 
Since $\wt{F}$ is a lift of 
the relative Frobenius morphism of $X\lo X'$ over $S$, 
the image of the pull-back morphism 
$\wt{F}{}^*\col \Om^1_{\os{\circ}{\wt{U}}{}'/\os{\circ}{S}}
\lo \wt{F}_*(\Om^1_{\os{\circ}{\wt{U}}/\os{\circ}{S}})$ 
is contained in $p\wt{F}_*(\Om^1_{\os{\circ}{\wt{U}}/\os{\circ}{S}})$. 
Similarly, because of the expression $\wt{F}{}^*(m')=\prod_im_i^pn_i
(1+p \eta (m'))$ for $m'=\prod_i[m_i,n_i]\in M_{\wt{U}{}'}$ 
$(m_i\in M_{\wt{U}}, n_i\in {\cal O}_{M_S})$ with $\eta(m')\in {\cal O}_U$, 
we see that the image of the morphism 
$d\log \wt{F}^* \col M_{\wt{U}{}'}\lo \wt{F}_*(\Om^1_{\wt{U}/S})$
is contained in $p\wt{F}_*(\Om^1_{\wt{U}/S})$. 
The morphism 
$\wt{F}{}^* \col \Om^1_{\wt{U}{}'/S} \lo \wt{F}_*(\Om^1_{\wt{U}/S})$ 
induces the following morphism:  
%\begin{align*}   
$\wt{F}{}^* \col \Om^1_{\wt{U}{}'/S} \lo p\wt{F}_*(\Om^1_{\wt{U}/S})$. 
%\tag{5.5.2}\label{ali:utsttu}
%\end{align*} 
This morphism induces the following morphism 
\begin{align*}  
\wt{F}_{\Om^1}:=p^{-1}\wt{F}{}^* \col \Om^1_{U'/S_0} \lo F_*(\Om^1_{U/S_0}). 
\tag{4.20.2}\label{ali:utsttu}
\end{align*} 
In fact, this morphism induces the following morphism 
\begin{align*}  
\wt{F}_{\Om^1} \col \Om^1_{U'/S_0} \lo F_*(Z\Om^1_{U/S_0})
\tag{4.20.3}\label{ali:utsfu}
\end{align*}  
(cf.~the formulas (\ref{ali:ulssfu}) and (\ref{ali:ulssgfu}) below). 
Express $\wt{F}{}^*(a')=\sum_i a_i^pb_i+p\eta(a')$ for $a'\in {\cal O}_{U'}$ 
with $a'=\sum_ia_i\otimes b_i$ $(a_i\in {\cal O}_{\wt{U}}$, $b_i\in {\cal O}_S)$ 
and $\eta(a')\in {\cal O}_{\wt{U}}$.  
We can easily check that the following equalities hold: 
\begin{align*}
\wt{F}_{\Om^1}(d(a\otimes b))= b(a^{p-1}da)+d\eta(a\otimes b), \quad 
(a\in {\cal O}_U, b\in {\cal O}_S)
\tag{4.20.4}\label{ali:ulssfu}
\end{align*} 
and 
\begin{align*} 
\wt{F}_{\Om^1}(d\log ([m,n]))=d\log m+d\eta([m,n])\quad (m\in M_U, n\in M_S)
\tag{4.20.5}\label{ali:ulssgfu}
\end{align*} 
(cf.~\cite[p.~106]{sr}).  
The morphism $\wt{F}_{\Om^1}$ is compatible with 
the restrictions of log open subschemes of $Y$. 
Hence the morphism (\ref{ali:utsfu}) 
is a section of $C\col F_*(Z\Om^1_{U/S_{0}})\lo \Om^1_{U'/S_{0}}$. 
Because 
${\rm Lift}_{(Y,F)/(S_0\subset S,F_S)}$ and ${\rm Sec}_{C}$ is a torsor under 
${\rm pr}_{0*}{\cal H}{\it om}_{{\cal O}_{Y'}}
(\Om^1_{Y'/S_{0}},F_{*}(B\Om^1_{Y/S_{0}}))$ 
on $Y'$, these are isomorphic. 
\end{proof} 

The following statement is the log version of \cite[p.~103 (i)]{sr}): 

\begin{theo}\label{theo:uaryn}
Let the assumptions be as in 
{\rm (\ref{theo:uryn})}. 
Assume that $\os{\circ}{Y}$ is separated.
The obstruction class ${\rm obs}_{(Y,F)/(S_0\subset S)}$ in 
${\rm Ext}_{Y'}^1(\Om^1_{Y'/S_0},F_{*}(B\Om^1_{Y/S_0}))$ is equal to 
the extension class of the following exact sequence 
\begin{align*}
0\lo F_{*}(B\Om^1_{Y/S_{0}})
\lo F_{*}(Z\Om^1_{Y/S_{0}})\os{C}{\lo} \Om^1_{Y'/S_{0}}\lo 0. 
\tag{4.21.1}\label{ali:orys}
\end{align*} 
\end{theo}
\begin{proof} 
Let the notations be as in the proof of (\ref{theo:ts}). 
Let $\wt{F}_i\col \wt{U}_i\lo \wt{U}{}'_i$ 
be a lift of $F_i\col U_i\lo U'_i$ over $S$. 
Let $\eta_i$ be the $\eta$ in the proof of (\ref{theo:uryn}) for $\wt{F}_i$. 
%For simplicity of notation, denote $\eta_{i}\vert_{U_{ij}}$ and 
%$\eta_{j}\vert_{U_{ij}}$ by $\eta_i$ and $\eta_j$, respectively. 
Let $\wt{m}{}'_{ij}=[\wt{m}_{ij},\wt{n}_{ij}]\in M_{\wt{U}{}'_{ij}}$ and 
$\wt{a}{}'_{ij}=\wt{a}_{ij}\otimes \wt{b}_{ij}\in {\cal O}_{\wt{U}{}'_{ij}}$ 
are lifts of local sections 
$m'_{ij}=[m_{ij},n_{ij}]\in M_{U'_{ij}}$ and 
$a'_{ij}=a_{ij}\otimes b_{ij}\in {\cal O}_{U'_{ij}}$, respectively. 
Set $\wt{m}{}''_{ij}:=(g_{ij}')^*(\wt{m}{}'_{ij})=[g_{ij}^*(m_{ij}),n_{ij}]
\in M_{\wt{U}{}'_{ji}}$ and 
$\wt{a}{}''_{ji}:=(g_{ij}')^*(\wt{a}{}''_{ij})=
g_{ij}^*(\wt{a}_{ij})\otimes \wt{b}_{ij}\in {\cal O}_{\wt{U}{}'_{ji}}$. 
%Set $\ol{m}{}'_{ij}:=[\ol{m}_{ij},1]$ and $\ol{a}{}'_{ij}:=\ol{a}_{ij}\otimes 1$. 
%Let $d\log \ol{m}{}'_i$ (resp.~$d\ol{a}{}'_i$) 
%be an element of $\Gam(U'_{i},\Om^1_{Y'/S_{0}})$ such that 
%$d\log (\ol{m}{}'_j\vert_{U'_{ij}})=d\log (\ol{m}{}'_i\vert_{U'_{ij}})$ 
%(resp.~$d(\ol{a}{}'_j\vert_{U'_{ij}})=d(\ol{a}{}'_i\vert_{U'_{ij}}$)).  
Let $ d\log m_{ij}+d(\eta_i\vert_{U_{ij}})(m'_{ij})$ 
and $b_{ij}a^{p-1}_{ij}da_{ij}+d(\eta_i\vert_{U_{ij}})(a'_{ij})$) 
be elements of $\Gam(U_{ij},Z\Om^1_{Y/S_{0}})$.
Then we have an element  
$$\{d\log m_{ij}+d(\eta_j\vert_{U_{ij}})(m'_{ij})\}- 
\{(d\log m_{ij}+d(\eta_i\vert_{U_{ij}})(m'_{ij}))\}=
d((\eta_j\vert_{U_{ij}})(m'_{ij})-(\eta_i\vert_{U_{ij}})(m'_{ij}))$$ 
of 
$\Gam(U'_{ij},F_{*}(B\Om^1_{Y/S_{0}}))$. 
We also have an element 
$d((\eta_j\vert_{U_{ij}})(a'_{ij})-(\eta_i\vert_{U_{ij}})(a'_{ij}))$
of $\Gam(U'_{ij},F_{*}(B\Om^1_{Y/S_{0}}))$. 
Hence we have an element of 
$\Gam(U'_{ij},{\cal H}{\it om}_{{\cal O}_{Y'}}(\Om^1_{Y'/S_{0}},
F_{*}(B\Om^1_{Y/S_{0}})))$.
Via the identification 
$(F_*({\cal O}_Y)/{\cal O}_{Y'})(U'_{ij})
\os{d,~\simeq}{\lo}F_{*}(B\Om^1_{Y/S_{0}})(U'_{ij})$,  
this is nothing but a 1-cocycle arising from 
\begin{align*} 
&(g'_{ij}\wt{F}_j\vert_{\wt{U}_{ij}}g^{-1}_{ij})^*(\wt{m}{}'_{ij})
((\wt{F}_i\vert_{\wt{U}_{ij}})^*(\wt{m}{}'_{ij}))^{-1}\\
&=
(g^{-1}_{ij})^*(\wt{F}_j\vert_{U_{ij}})^*(\wt{m}{}''_{ij})
\vert_{U_{ij}}(1+p F^*(\del_{`g_{ij}}(\wt{m}{}'_{ij},\wt{m}{}''_{ij})))
((\wt{F}_i\vert_{U_{ij}})^*(\wt{m}{}'_{ij})){}^{-1}\\
&\equiv
(g^{-1}_{ij})^*(g_{ij}^*(\wt{m}{}^p_{ij})\wt{n}_{ij})
(1+p (\eta_j\vert_{U_{ij}})(m'_{ij}))(\wt{m}{}^p_{ij}\wt{n}_{ij}
(1+p (\eta_i\vert_{U_{ij}})(m'_{ij})))^{-1} \\
&=1+p ((\eta_j\vert_{U_{ij}})(m'_{ij})-p (\eta_i\vert_{U_{ij}})(m'_{ij}))
%&=1+p (\eta_j(m'_j)\vert_{U_{ij}}-p \eta_i(m'_i)\vert_{U_{ij}})
\end{align*} 
and 
\begin{align*} 
& (g'_{ij}\wt{F}_j\vert_{\wt{U}_{ij}}g^{-1}_{ij})^*(\wt{a}{}'_{ij})-
(\wt{F}_i\vert_{\wt{U}_{ij}})^*(\wt{a}{}'_{ij})
=(g^{-1}_{ij})^*(\wt{F}_j\vert_{\wt{U}_{ij}})^*(\wt{a}{}''_{ij})+pF^*
(\del_{`g_{ij}}(\wt{a}{}'_{ij},\wt{a}{}''_{ij}))
-(\wt{F}_i\vert_{\wt{U}_{ij}})^*(\wt{a}{}'_{ij})\\
&\equiv
(g^{-1}_{ij})^*(g_{ij})^*(\wt{a}{}_{ij}^p)\wt{b}_{ij}
+p(\eta_j\vert_{U_{ij}})(a'_{ij})-
(\wt{a}{}_{ij}^p\wt{b}_{ij}+p(\eta_i\vert_{U_{ij}})(a'_{ij}))\\
&=p(\eta_j\vert_{U_{ij}})(a'_{ij})-(\eta_i\vert_{U_{ij}})(a'_{ij}).
\end{align*}   
Here $\equiv$ means the equality in the quotient 
$(F_*({\cal O}_Y)/{\cal O}_{Y'})(U'_{ij})$ 
and we have used (\ref{lemm:ges}). 
\end{proof} 

\begin{rema}
In \cite{sr} there is no proof of the trivial log version of 
(\ref{theo:uaryn}). In particular, 
(\ref{lemm:ges}) is missing in [loc.~cit.]. 
\end{rema}
 
\par
In the rest of this section, we consider the log deformation theory 
with {\it absolute} Frobenius endomorphism when $\os{\circ}{S}$ is 
perfect. 
\par 
Let $F_0\col Y_0\lo Y_0$ be the $e$-times iterated 
absolute Frobenius endomorphism over 
$F^e_{S_{00}}\col S_{00}\lo S_{00}$. 
We assume that there exists a lift 
$F\col Y\lo Y$ of $F_0\col Y_0\lo Y_0$ over $F^{[e]}_{S_{0}}$. 
We say that $(\wt{Y},\wt{F})/F_S^{[e]}$ is a {\it log smooth integral lift}  
(or simply a {\it lift}) of $(Y,F)/F^{[e]}_{S_0}$ 
if $\wt{Y}$ is a log smooth integral scheme over $S$ such that 
$\wt{Y}\times_SS_0=Y$ and $\wt{F}$ is a morphism 
$\wt{Y}\lo \wt{Y}$ 
over $F^{[e]}_{S}$ fitting into 
the following commutative diagram 
\begin{equation*} 
\begin{CD} 
Y@>{\subset}>> \wt{Y}\\
@V{F}VV @VV{\wt{F}}V \\
Y@>{\subset}>> \wt{Y}  
\end{CD} 
\end{equation*}
over the commutative diagram 
\begin{equation*} 
\begin{CD} 
S_0@>{\subset}>> S\\
@V{F^{[e]}_{S_0}}VV @VV{F^{[e]}_{S}}V \\
S_0@>{\subset}>> S.   
\end{CD} 
\end{equation*}
\par 
Let ${\rm Lift}_{(Y,F)/(S_0\subset S,F^{[e]}_S)}$ be the following sheaf defined by 
the following equality: 
\begin{align*} 
{\rm Lift}_{(Y,F)/(S_0\subset S,F^{[e]}_S)}(U):=
\{&{\rm isomorphism~classes~of~lifts~of}~(U,F\vert_U)/F^{[e]}_{S_0}\\
&{\rm over}~F^{[e]}_{S}\}
\end{align*}
for each log open subscheme $U$ of $Y$. 
Here the isomorphism classes of lifts of $(U,F\vert_U)/F^{[e]}_{S_0}$ 
over $F^{[e]}_{S}$ are defined in an obvious way. 
%Note that, because $\os{\circ}{F}$ is an isomorphism of topological spaces, 
%$`U$ is uniquely determined by $U$. 
\par  
Then the following hold by the same proof as that of (\ref{theo:ts}): 

\begin{theo}\label{theo:taas}  
Let ${\cal I}$, $\pi$ and $n$ be as in {\rm (\ref{lemm:fl})}. 
%Assume that $\os{\circ}{Y}_0$ is reduced. 
%the morphism ${\cal O}_{`Y_0}\lo F_{0*}({\cal O}_{Y_0})$ is injective 
%and the pull-back morphism ${\cal O}_{Y_0}\lo {\cal O}_{`Y_0}$ 
%by the projection $`Y_0\lo Y_0$ is injective. 
Then the following hold$:$ 
\par 
$(1)$ 
Assume that $(Y,F)/S_0$ has a lift $(\wt{Y},\wt{F})/S$. 
Set ${\rm Aut}_{S,F^{[e]}_S}(\wt{Y},Y)
:=\{g\in {\rm Aut}_{S}(\wt{Y})~\vert~g\vert_Y={\rm id}_Y, 
\wt{F}\circ g=g\circ \wt{F}\}$. 
Then ${\rm Aut}_{S,F^{[e]}_S}(\wt{Y},Y)=\{{\rm id}_Y\}$. 
\par 
$(2)$ 
%Let $`{\rm pr}_0\col `Y_0\lo Y_0$ be the projection. 
The sheaf ${\rm Lift}_{(Y,F)/(S_0\subset S,F^{[e]}_S)}$ on 
$\os{\circ}{Y}$ is a torsor under 
${\cal H}{\it om}_{{\cal O}_{Y_0}}(\Om^1_{Y_0/S_{00}}, 
F_{0*}({\cal O}_{Y_0})/{\cal O}_{Y_0})$.   
%if $(Y,F)/S_0$ has a lift $(\wt{Y},\wt{F})/S$, 
%then there exists the following bijection 
%of sets$:$ 
%\begin{align*} 
%{\rm Lift}_{(Y,F)/(S_0\subset S,F^{[e]}_S)}(Y)\os{\sim}{\lo}
%{\rm Hom}_{{\cal O}_{`Y_0}}(\Om^1_{`Y_0/S^{[q]}_{00}},
%F_{0*}({\cal O}_{Y_0})/{\cal O}_{`Y_0}). 
%\tag{4.14.1}\label{ali:oo0ys}
%\end{align*} 
\par 
$(3)$ 
%The zariski sheaf ${\rm Lift}_{(Y,F)/(S,F_S)}$ on $\os{\circ}{Y}$ is a torsor 
%under ${\cal H}{\it om}_{{\cal O}_{Y_0}}(\Om^1_{Y_0/S_{00}},B\Om^1_{Y_0/S_{00}});$ 
Assume that $\os{\circ}{Y}$ is separated.
In 
$${\rm Ext}_{Y_0}^1(\Om^1_{Y_0/S_{00}},
F_{0*}({\cal O}_{Y_0})/{\cal O}_{Y_0}),$$ 
there exists a canonical obstruction class 
${\rm obs}_{(Y,F)/(S_0\subset S,F^{[e]}_S)}$ 
of a lift of $(Y,F)/F^{[e]}_{S_0}$ over $F^{[e]}_S$. 
%(B\Om^1_{Y_0/S_{00}}))$. 
\par 
$(4)$ Assume that $\os{\circ}{Y}$ is separated. 
Let 
\begin{align*} 
\partial \col {\rm Ext}_{Y_0}^1(\Om^1_{Y_0/S_{00}},
F_{0*} ({\cal O}_{Y_0})/{\cal O}_{Y_0})
\lo {\rm Ext}_{Y_0}^2(\Om^1_{Y_0/S_{00}},{\cal O}_{Y_0})
\end{align*} 
be the boundary morphism obtained by 
the following exact sequence {\rm (\ref{prop:ee})}$:$
\begin{align*} 
0\lo {\cal O}_{Y_0}\lo F_{0*}({\cal O}_{Y_0})
\lo F_{0*} ({\cal O}_{Y_0})/{\cal O}_{Y_0}\lo 0. 
%F_*(B\Om^1_{Y_0/S_{00}})\lo 0. 
\tag{4.23.1}\label{ali:ooabs}
\end{align*} 
Then $\partial ({\rm obs}_{(Y,F)/(S_0\subset S,F^{[e]}_S)})=
{\rm obs}_{Y/(S_0\subset S)}$. 
\par 
$(5)$ Assume that $\os{\circ}{Y}$ is separated. Assume 
%that $S_0=S_{00}$ $($hence $Y=Y_0$ and $`Y=`Y_0)$,  
%that $\os{\circ}{S}_{0}$ is perfect and 
that there exists a lift $\wt{Y}/S$ of $Y/S_0$. 
Then, in 
$${\rm Ext}^1_{Y_0}(\Om^1_{Y_0/S_{00}},F_{0*}({\cal O}_{Y_0})),$$  
there exists a canonical obstruction class ${\rm obs}_{\wt{Y}/S}(F)$ 
of a lift $\wt{F}\col \wt{Y}\lo \wt{Y}$ of $F\col Y\lo Y$  
and this is mapped to ${\rm obs}_{(Y,F)/(S_0\subset S,F^{[e]}_S)}$ 
by the following natural morphism 
\begin{align*} 
\partial \col {\rm Ext}_{Y_0}^1(\Om^1_{Y_0/S_{00}}, F_{0*} ({\cal O}_{Y_0}))
\lo {\rm Ext}_{Y_0}^1(\Om^1_{Y_0/S_{00}}, F_{0*}({\cal O}_{Y_0})/{\cal O}_{Y_0}). 
\end{align*} 
\end{theo}

\begin{coro}\label{coro:nac}
Assume that $e=1$. 
Assume also that $Y_0/S_{00}$ is of Cartier type 
and that $\os{\circ}{S}_{00}$ is perfect.  
Then the following hold$:$ 
\par 
$(1)$ The sheaf ${\rm Lift}_{(Y,F)/(S_0\subset S,F_S)}$ on 
$\os{\circ}{Y}$ is a torsor under 
${\cal H}{\it om}_{{\cal O}_{Y_0}}(\Om^1_{Y_0/S_{00}},
F_{0*}(B\Om^1_{Y_0/S_{00}}))$. 
\par 
$(2)$ Assume that $\os{\circ}{Y}$ is separated. 
In 
$${\rm Ext}_{Y_0}^1(\Om^1_{Y_0/S_{00}},
F_{0*}(B\Om^1_{Y_0/S_{00}})),$$
there exists a canonical obstruction class 
${\rm obs}_{(Y,F)/(S_0\subset S,F_S)}$ 
of a lift of $(Y,F)/F_{S_0}$ over $F_S$. 
\par 
$(3)$ Assume that $\os{\circ}{Y}$ is separated. 
Let 
\begin{align*} 
\partial \col {\rm Ext}_{Y_0}^1(\Om^1_{Y_0/S_{00}},
F_{0*}(B\Om^1_{Y_0/S_{00}}))
\lo {\rm Ext}_{Y_0}^2(\Om^1_{Y_0/S_{00}},{\cal O}_{Y_0})
\end{align*} 
be the boundary morphism obtained by 
the following exact sequence {\rm (\ref{prop:ee})}$:$
\begin{align*} 
0\lo {\cal O}_{Y_0}\lo F_{0*}({\cal O}_{Y_0})\os{F_{0*}(d)}\lo 
F_{0*}(B\Om^1_{Y_0/S_{00}})\lo 0. 
\tag{4.24.1}\label{ali:obaabs}
\end{align*} 
Then $\partial ({\rm obs}_{(Y,F)/(S_0\subset S,F_S)})=
{\rm obs}_{Y/(S_0\subset S)}$. 
\end{coro} 
\begin{proof}
Recall that $Y'_0:=Y_0\times_{S_{00},F_{S_{00}}}S_{00}$.  
Let $F^{\rm rel}_0\col Y_0\lo Y'_0$ be the relative Frobenius morphism. 
Because $\os{\circ}{S}_{00}$ is perfect,  
the projection $\os{\circ}{Y}{}'_0\lo \os{\circ}{Y}_0$ is an isomorphism. 
%Furthermore $\os{\circ}{Y}_0$ is reduced by (\ref{rema:rds}). 
By (\ref{ali:yy00}) we obtain the following composite isomorphism 
\begin{align*} 
F_{0*}({\cal O}_{Y_0})/{\cal O}_{Y_0}
\os{\sim}{\lo}
F_{0*}({\cal O}_{Y_0})/{\cal O}_{Y'_0}=
F^{\rm rel}_{0*}(B\Om^1_{Y_0/S_{00}})\os{\sim}{\longleftarrow}
F_{0*}(B\Om^1_{Y_0/S_{00}}).
\tag{4.24.2}\label{ali:yrf00}
\end{align*}  
Now (\ref{coro:nac}) follows from (\ref{theo:taas}) and (\ref{ali:yrf00}). 
\end{proof}
%\begin{rema}\label{rema:fsl}
%Because ${\rm obs}_{(Y,F)/(S_0\subset S)}$ also depends on 
%the lift $F_S$, it is more precise to denote it by 
%${\rm obs}_{(Y,F)/(S_0\subset S,F_S)}$. 
%However, because we would like to strengthen the consideration 
%of the relative Frobenius morphism $F$, we do not use this more 
%precise notation. 
%\end{rema}  

%\section{Recall on Tsuji's results on log schemes}\label{sec:lst}

\section{Applications of log deformation theory with 
relative Frobenius morphisms}
\label{sec:app} 
In this section we give another short proof of Kato's theorem  
on the $E_1$-degeneration of the log Hodge 
de Rham spectral sequence (\cite{klog1})
by following the method of Srinivas (\cite{sr}). 
We also give the log versions of vanishing theorems of 
Kodaira-Akizuki-Nakano in characteristic $p$ and $0$ 
following the method of Raynaud (\cite{di}). 
\par 
In \cite[p.~104--105]{sr} Srinivas has given another short  
proof of the $E_1$-degeneration of the Hodge de Rham spectral sequence due to 
Deligne and Illusie (\cite{di}) by using the deformation theory in \cite{ns}. 
(Strictly speaking, he has proved this only in the case 
where the base scheme is the spectrum of a perfect field of characteristic $p>0$.)  
By using the theory in \S\ref{sec:latv} and his idea, 
we can also give another short proof of 
the degeneration at $E_1$ of 
the log Hodge de Rham spectral sequence 
due to Kato in \cite[(4.12) (3)]{klog1} in the case 
where there exists a lift of the Frobenius endomorphism of the base log scheme: 

\begin{theo}[{\bf A special case of 
\cite[(4.12) (2)]{klog1}}]\label{theo:kdg} 
Let the notations and the assumptions be as in {\rm (\ref{theo:uryn})}. 
Assume that $\os{\circ}{S}$ is flat over ${\rm Spec}({\mab Z}/p^2)$. 
Set $S_0:=S\mod p$. 
Let $Y$ be a log smooth separated scheme of Cartier type over $S_0$. 
%Assume that $\os{\circ}{Y}$ is flat over  $\os{\circ}{S}$. 
Let $F\col Y\lo Y'$ be the relative Frobenius morphism over $S$. 
If $Y'$ has a log smooth integral lift ${\cal Z}$ over a fine log scheme $S$,  
%and if $\dim(\os{\circ}{Y}/\os{\circ}{S})< p$, 
then there exists an isomorphism 
\begin{align*} 
\bigoplus_{i< p}\Om^i_{Y'/S_0}[-i] \os{\sim}{\lo} \tau_{<p}F_*(\Om^{\bul}_{Y/S_0})
\tag{5.1.1}\label{ali:oeomxs} 
\end{align*}
in the derived category $D^+(Y'_{\rm zar})$ of bounded above complexes of 
${\cal O}_{Y'}$-modules. 
%then the following spectral sequence 
%\begin{align*} 
%E_1^{ij}=H^j(X,\Om^i_{Y/s})\Lo H^{i+j}_{\rm dR}(Y/s)
%\tag{4.9.1}\label{ali:oedrxs} 
%\end{align*}  
%degenerates at $E_1$. 
\end{theo} 
\begin{proof} 
%(The proof is the same as that of \cite[p.~104--105]{sr}.)
%Set $d:=\dim \os{\circ}{Y}$. 
By (\ref{ali:cifes}) we have the following exact sequence 
\begin{align*} 
{\rm Ext}_{Y'}^1(\Om^1_{Y'/S_0},F_*({\cal O}_Y))
\lo 
{\rm Ext}_{Y'}^1(\Om^1_{Y'/S_0},F_*(B\Om^1_{Y/S_0}))
\os{\partial}{\lo}  
{\rm Ext}_{Y'}^2(\Om^1_{Y'/S_0},{\cal O}_{Y'}).  
\tag{5.1.2}\label{ali:oorexs}
\end{align*} 
By (\ref{theo:trs}) (4) and the assumption, 
there exists the extension class 
of the following exact sequence
\begin{align*} 
0\lo F_*({\cal O}_Y)\lo {\cal V}\lo \Om^1_{Y'/S_0}\lo 0
\end{align*}  
whose image in 
${\rm Ext}_{Y'}^1(\Om^1_{Y'/S_0},F_*(B\Om^1_{Y/S_0}))$
is equal to 
${\rm obs}_{(Y,F)/(S_0\subset S,F_S)}$. 
By (\ref{theo:uaryn}) 
this exact sequence fits into the following commutative diagram: 
\begin{equation*} 
\begin{CD} 
0@>>> F_*({\cal O}_Y)@>>>  {\cal V}@>>> \Om^1_{Y'/S_0}@>>> 0\\
@. @V{F_*(d)}VV @VVV @| \\
0@>>> F_*(B\Om^1_{Y/S_0})@>>> F_*(Z\Om^1_{Y/S_0})@>{C}>> 
\Om^1_{Y'/S_0}@>>> 0. 
\end{CD}
\tag{5.1.3}\label{cd:oro}
\end{equation*}  
Set ${\cal C}_1:=(F_*({\cal O}_Y)\lo  {\cal V})$. 
This is quasi-isomorphic to $\Om^1_{Y'/S_0}[-1]$.   
The diagram (\ref{cd:oro}) induces the following morphism 
\begin{equation*} 
\begin{CD} 
0@>>> F_*({\cal O}_Y)@>>>  {\cal V}@>>> 0
@>>>\cdots \\
%@>>>  0\\
@. @| @VVV @VVV @. \\
%@VVV\\
0@>>>F_*({\cal O}_Y)@>{F_*(d)}>> 
F_*(\Om^1_{Y/S_0}) @>{F_*(d)}>> F_*(\Om^2_{Y/S_0})
@>>>\cdots 
%@>>>  F_*(Z\Om^{p-1}_{Y/S_0}) 
\end{CD}
\tag{5.1.4}\label{cd:oio}
\end{equation*}  
of complexes since $F_*(Z\Om^1_{Y/S_0})\subset F_*(\Om^1_{Y/S_0})$. 
We denote this morphism by 
$\varphi_1 \col 
{\cal C}_1\lo F_*(\Om^{\bul}_{Y/S_0})$.  
Because the following diagram 
\begin{equation*} 
\begin{CD} 
{\cal H}^1({\cal C}_1) @>{{\cal H}^1(\varphi_1)}>> 
{\cal H}^1(F_*(\Om^{\bul}_{Y/S_0}))\\
@| @A{\simeq}A{C^{-1}}A\\
\Om^1_{Y'/S_0} @=\Om^1_{Y'/S_0}
\end{CD}
\tag{5.1.5}\label{cd:yys}
\end{equation*} 
is commutative, ${\cal H}^1(\varphi_1)$ 
is an isomorphism. 
Let $\varphi^1$ be the following morphism in the derived category $D(X')$: 
\begin{align*} 
\varphi^1\col \Om^1_{Y'/S_0}[-1]
\os{\sim}{\longleftarrow} {\cal C}_1\os{\varphi_1}{\lo} 
F_*(\Om^{\bul}_{Y/S_0}).
\end{align*} 
Then ${\cal H}^1(\varphi^1)=C^{-1}\col \Om^1_{Y'/S_0}
\os{\sim}{\lo} {\cal H}^1(F_*(\Om^{\bul}_{Y/S_0}))$. 
\par 
The rest of the proof is the same as that in 
the first step of the proof of \cite[(2.1)]{di}. 
For the completeness of this article, we recall it here. 
\par 
Let $\varphi^0\col {\cal O}_{Y'}\lo F_*(\Om^{\bul}_{Y/S_0})$ be 
the following composite morphism 
\begin{align*}
{\cal O}_{Y'}\os{C^{-1}}{\lo} {\cal H}^0(F_*(\Om^{\bul}_{Y/S_0}))\os{\subset}{\lo} 
F_*(\Om^{\bul}_{Y/S_0}). 
\end{align*} 
\par 
Let $i$ be a positive integer less than $p$. 
Consider the following splitting  
$$\Om^i_{Y'/S_0}\lo (\Om^1_{Y'/S_0})^{\otimes i}$$ 
of a natural surjection 
$(\Om^1_{Y'/S_0})^{\otimes i}\lo \Om^i_{Y'/S_0}$ defined by the morphism 
$$\om_1\wedge \cdots \wedge \om_i \lom 
(i!)^{-1}\sum_{\sig \in {\mathfrak S}_i}{\rm sgn}(\sig)
\om_{\sig(1)}\otimes \cdots \otimes \om_{\sig(i)}$$
as in \cite[p.~251]{di}. 
Then we have the following composite morphism 
\begin{align*} 
\varphi^i \col \Om^i_{Y'/S_0}[-i]\lo 
(\Om^1_{Y'/S_0})^{\otimes i}[-i]
\os{(\varphi^1)^{\otimes i}}{\lo}  
(F_*(\Om^{\bul}_{Y/S_0}))^{\otimes i}
\os{{\rm product}}{\lo} F_*(\Om^{\bul}_{Y/S_0}). 
\end{align*} 
By the multiplicative property of $C^{-1}$, 
${\cal H}^i(\varphi^i)$ is equal to the Cartier isomorphism 
$C^{-1}\col \Om^i_{Y'/S_0}\os{\sim}{\lo} 
{\cal H}^i(F_*(\Om^{\bul}_{Y/S_0}))$. 
Hence $\sum_{i=0}^{p-1}\varphi^i$ is 
the desired isomorphism (\ref{ali:oeomxs}).  
\end{proof}

\begin{rema}
%(1) For $i<p$, as in \cite[p.~104--105]{sr}, we can construct two explicit morphisms 
%\begin{align*} 
%\Om^i_{Y'/S_0}[-i]\longleftarrow {\cal C}_i \lo \tau_{<p}F_*(\Om^{\bul}_{Y/S_0}). 
%\end{align*} 
%of complexes such that the following two morphisms 
%\begin{align*}
%\bigoplus_{i=0}^{p-1}\Om^i_{Y'/S_0}[-i]{\longleftarrow} 
%\bigoplus_{i=0}^{p-1}{\cal C}_i {\lo} \tau_{<p}F_*(\Om^{\bul}_{Y/S_0}). 
%\tag{5.2.1}\label{ali:cdhp}
%\end{align*}  
%are quasi-isomorphisms, though 
%the complex ${\cal C}_i$ $(i\not=1)$
%is unnecessary for the proof of (\ref{theo:kdg}). 
%\par 
%(2) 
Assume that there exists a lift $\wt{Y}/S$ of $Y/S_0$. 
Then, by (\ref{theo:trs}) (4), 
we can take the element ${\rm obs}_{\wt{Y}/S}(F)$  
as an element in ${\rm Ext}_{Y'}^1(\Om^1_{Y'/S_0},F_*({\cal O}_{Y}))$
in the proof of (\ref{theo:kdg}). 
(In \cite{sr} this has not been mentioned.)
\end{rema}

\begin{coro}[{\bf A special case of \cite[(4.12) (1)]{klog1}}]\label{coro:cfd}
Let the notations and the assumptions be as in {\rm (\ref{theo:kdg})}. 
Then the following hold$:$
\par 
$(1)$ Let $f\col Y\lo S_0$ and $f'\col Y'\lo S_0$ be the structural morphisms. 
Then there exists the following decomposition 
\begin{align*} 
R^qf_*(\Om^{\bul}_{Y/S_0})=\bigoplus_{i+j=q}R^jf'_*(\Om^i_{Y'/S_0})
\tag{5.3.1}\label{ali:obybxs} 
\end{align*}
for $q<p$.  
%Here the contravariant functoriality means the 
%contravariant functoriality for the morphism 
%${\cal Z}\lo {\cal Z}_1$ over $S\lo S_1$, where ${\cal Z}_1/S_1$ 
%is a similar log scheme to ${\cal Z}/S$.   
Moreover, if $Y/S_0$ is proper, 
then $E_1^{ij}=E_{\infty}^{ij}$ for $i+j<p$, 
where $E^{ij}_{\star}$ $(\star=1, \infty)$ is 
the $E^{ij}_{\star}$-term of 
the following spectral sequence 
\begin{align*} 
E_1^{ij}=R^jf_*(\Om^i_{Y/S_0})\Lo R^{i+j}f_*(\Om^{\bul}_{Y/S_0}). 
\tag{5.3.2}\label{ali:oedegrs} 
\end{align*} 
Furthermore, in this case, 
$R^qf_*(\Om^{\bul}_{Y/S_0})$ $(0\leq q<p)$ is locally free and 
commutes with any base change of fine log schemes. 
\par 
$(2)$ Assume that the structural morphism 
$\os{\circ}{Y}\lo \os{\circ}{S}$ of schemes is flat,  that 
$\dim (\os{\circ}{Y}/\os{\circ}{S})\leq p$,  that 
$F_*({\cal O}_Y)$ is a locally free ${\cal O}_{Y'}$-modules $($of finite rank$)$ and that 
$$H^{p+1}(Y',
{\cal H}{\it om}_{{\cal O}_{Y'}}(\Om^p_{Y'/S},{\cal O}_{Y'}))=0.$$ 
Then there exists a decomposition 
\begin{align*} 
\bigoplus_{i\leq p}\Om^i_{Y'/S_0}[-i] \os{\sim}{\lo} F_*(\Om^{\bul}_{Y/S_0})
\tag{5.3.3}\label{ali:oebxs} 
\end{align*}
in the derived category $D^+(Y_{\rm zar})$. 
Consequently there exists a decomposition 
\begin{align*} 
R^qf_*(\Om^{\bul}_{Y/S_0})=\bigoplus_{i+j=q}R^jf'_*(\Om^i_{Y'/S_0}) 
\quad (q\in {\mab Z})
\tag{5.3.4}\label{ali:oybxs} 
\end{align*}
and the following spectral sequence 
\begin{align*} 
E_1^{ij}=R^jf_*(\Om^i_{Y/S_0})\Lo R^{i+j}f_*(\Om^{\bul}_{Y/S_0})
\tag{5.3.5}\label{ali:oedrxs} 
\end{align*}  
degenerates at $E_1$.   
Furthermore, 
$R^qf_*(\Om^{\bul}_{Y/S_0})$ $(0\leq q\leq p)$ is locally free and 
commutes with any base change of fine log schemes. 
\end{coro}
\begin{proof} 
(1): This immediately follows from (\ref{ali:oeomxs}) and 
the log version of the argument of \cite[(4.1.2), (4.1.4)]{di}. 
\par 
(2) (The proof is the same as that of \cite[(2.3)]{di}.) 
In the case $\dim (\os{\circ}{Y}/\os{\circ}{S})<p$, 
(\ref{ali:oebxs}) follows from (\ref{ali:oeomxs}). 
Consider the case $\dim (\os{\circ}{Y}/\os{\circ}{S})=p$. 
We may assume that $\os{\circ}{Y}$ is connected. 
Then the wedge product 
\begin{align*} 
\Om^i_{Y/S_0}\times \Om^{p-i}_{Y/S_0}\lo \Om^p_{Y/S_0}
\end{align*} 
is a perfect pairing. 
Because $F_*({\cal O}_Y)$ is a locally free ${\cal O}_{Y'}$-modules, 
we can check that 
\begin{align*} 
F_*(\Om^i_{Y/S_0})\times F_*(\Om^{p-i}_{Y/S_0})\lo 
F_*(\Om^p_{Y/S_0})\os{\rm proj}{\lo} 
{\cal H}^p(F_*(\Om^{\bul}_{Y/S_0}))\os{C,~\simeq}{\lo} \Om^p_{Y'/S_0}
\end{align*} 
is also a perfect pairing of locally free ${\cal O}_{Y'}$-modules of finite rank. 
The rest of the proof of (\ref{ali:oebxs}) 
is completely the same as that of \cite[(2.3), (3.7)]{di}. 
\par 
Now (\ref{ali:oybxs}) follows from the equality 
$R^qf'_*(F_*(\Om^{\bul}_{Y/S_0}))=R^qf_*(\Om^{\bul}_{Y/S_0})$ 
(since $\os{\circ}{F}$ is finite).  
\end{proof}

The following has not been stated in literatures: 

\begin{coro}\label{coro:e1dg} 
Let the notations be as in {\rm (\ref{coro:cfd}) (1)}. 
Assume that $S_0$ is the log point $s$ of a perfect field of characteristic $p>0$ 
and that $Y/s$ is of vertical type. 
Assume that $\os{\circ}{Y}$ is of pure dimension $d$. 
Then $E_1^{ij}=E_{\infty}^{ij}$ for $i+j>2d-p$. 
\end{coro}
\begin{proof} 
The equality in the statement follows from (\ref{coro:cfd}) (1) 
and Tsuji's duality for log de Rham cohomologies and 
his log Serre duality (see (\ref{theo:ico}) below).
\end{proof} 

%\parno  
%The following is a log version of the degeneration at $E_1$ of 
%Hodge de Rham spectral sequence proved algebraically in \cite[(2.7)]{di}: 

\begin{rema}\label{rema:de1}
(1) 
%Let the notations be as in {\rm (\ref{coro:e1dg})}. 
%In \cite{ss} Sheng and Shentu have proved that 
%an existence of a log smooth lift of $Y$ over ${\cal W}_2(s)$ is 
%equivalent to an existence of a log smooth lift of $Y'$ over ${\cal W}_2(s)$. 
%More generally, the following holds by the proof in \cite{ss}: 
%\par 
%Let $Y/S_0$ be as in (\ref{theo:kdg}) without 
%assuming the existence of the lift of $F_{S_0}$. 
%Then an existence of a log smooth {\it integral} lift of $Y$ over $S_0$ 
%is equivalent to an existence of a log smooth integral lift of $Y'$ over $S'_0$. 
%\par 
%(2) 
In (\ref{coro:e1dg}) it is not necessary to assume that 
$\Gam(s,{\cal O}_s)$ is perfect.  In fact, one has only to take the perfection 
of $\Gam(s,{\cal O}_s)$. 
\par 
(2) Let $K$ be a field of characteristic $0$. 
Let $T$ be an fs log scheme whose underlying scheme is ${\rm Spec}(K)$. 
Let $g\col Z\lo T$ be a proper log smooth integral morphism of fs log schemes. 
Assume that $g$ is saturated. 
Then, in \cite[p.~37]{ikn}, by using \cite[(4.12) (1)]{klog1}, 
Illusie, Kato and Nakayama have proved that 
the following spectral sequence 
\begin{align*} 
E_1^{ij}=
R^jg_*(\Om^i_{Z/T})\Lo R^{i+j}g_*(\Om^{\bul}_{Z/T})
\tag{5.5.1}\label{ali:oedr0xs} 
\end{align*}  
degenerates at $E_1$. 
\end{rema} 
More strongly, in \cite[(7.2)]{ikn}, 
they have proved the $E_1$-degeneration of 
(\ref{ali:oedr0xs}) if $g$ is proper log smooth and exact. 
They have also proved that 
$E_1^{ij}$ is locally free if any stalk of $M_Y/{\cal O}_Y^*$ is 
a free monoid. 
See also \cite[(9.15)]{nlpi} and \cite{illlel} 
for the log Hodge symmetry. 
%\begin{proof} 
%By a standard argument, 
%there exists an integral domain $A$ of finite type 
%over ${\mab Z}$ and an fs log scheme ${\cal T}$ 
%whose underlying scheme is ${\rm Spec}(A)$ 
%and a proper log smooth integral scheme ${\cal Z}$ over ${\cal T}$ such that 
%the structural morphism ${\mathfrak g} \col {\cal Z}\lo {\cal T}$ 
%is saturated (cf.~\cite{nd}). 
%By Kato-Tsuji's result (\ref{prop:ktj}), 
%every fiber of ${\mathfrak g}$ is of Cartier type 
%if the fiber is in characteristic $p>0$. 
%Shrinking ${\cal T}$ if necessary, we may assume that  
%$R^jg_*(\Om^i_{{\cal Z}/{\cal T}})$ and $R^{j}g_*(\Om^{\bul}_{{\cal Z}/{\cal T}})$ 
%are locally free ${\cal O}_{\cal T}$-modules.  
%The rest of the proof is the same as that of \cite[(2.7)]{di}. 
%\end{proof} 

Next we give the log version of Raynaud's result in \cite[(2.8)]{di}. 
To give it, we need to recall Tsuji's ideal sheaf. 
\par 
Let $g\col Y\lo Z$ be a morphism of fs log schemes. 
Secondly let us recall Tsuji's ideal sheaf 
${\cal I}_{Y/Z}$ of the log structure $M_Y$ 
denoted by $I_g$ in \cite{tsp}
for the review of Tsuji's log Serre duality.  
%Here $f$ is the structural morphism $X\lo Y$.   
\par 
For a commutative monoid $P$ with unit element, 
an ideal is, by definition, a subset $I$ of $P$ such that $PI\subset I$. 
An ideal ${\mathfrak p}$ of $P$ is called a prime ideal if 
$P\setminus {\mathfrak p}$ is a submonoid of $P$ (\cite[(5.1)]{klog2}). 
For a prime ideal ${\mathfrak p}$ of $P$, the height ${\rm ht}({\mathfrak p})$ 
is the maximal length of sequence's ${\mathfrak p}\supsetneq  
{\mathfrak p}_1\supsetneq \cdots \supsetneq {\mathfrak p}_r$ of prime ideals of $P$. 
Let $h\col Q\lo P$ be a morphism of monoids.  
A prime ideal ${\mathfrak p}$ of $P$ is said to be horizontal with respect to $h$ 
if $h(Q)\subset P\setminus {\mathfrak p}$ (\cite[(2.4)]{tsp}). 
\par 
Let $Y\lo Z$ be a morphism of fs log schemes.  
Let $h\col Q\lo P$ be a local chart of $g$ such that $P$ and $Q$ are saturated. 
Set 
$$I:=\{a\in P~\vert~a\in {\mathfrak p}~\text{for any 
horizontal prime ideal of $P$ of height 1 with respect to}~h\}.$$  
%an ideal sheaf of $P$ defined by the following equality 
%\begin{align*} 
%I:=\bigcap {\mathfrak p},  
%\end{align*} 
%where ${\mathfrak p}$ is a horizontal prime ideal of $P$ of height 1 
Let ${\cal I}_{Y/Z}$ be the ideal sheaf of $M_Y$ generated by ${\rm Im}(I\lo M_Y)$. 
In  \cite[(2.6)]{tsp} 
Tsuji has proved that ${\cal I}_{Y/Z}$ is independent of the choice of the local chart $h$. 
Let ${\cal I}_{Y/Z}{\cal O}_Y$ be the ideal sheaf of ${\cal O}_Y$ 
generated by the image of ${\cal I}_{Y/Z}$. 
For a quasi-coherent sheaf ${\cal F}$ of ${\cal O}_Y$-modules, 
denote $({\cal I}_{Y/Z}{\cal O}_Y){\cal F}$ by ${\cal I}_{Y/Z}{\cal F}$. 

\begin{theo}[{\bf \cite[(2.21)]{tsp}}]\label{theo:ico} 
Let $A$ be a discrete valuation ring with uniformizer $\pi$. 
Let $Z$ be an fs log scheme whose underlying scheme 
is ${\rm Spec}(A/\pi^m)$ for some $m\geq 1$ 
and whose log structure is associated to 
the morphism ${\mab N}\owns 1\lom a\in A/\pi^m$ 
for some $a\in A/\pi^m$. 
Let $g\col Y\lo Z$ be a saturated morphism of fs log schemes such that $\os{\circ}{g}$ is of finite type.  
Assume that $\Om^1_{Y/Z}$ is a locally free ${\cal O}_Y$-modules of constant rank $d$. 
Then $g^{!}({\cal O}_Z)={\cal I}_{Y/Z}\Om^d_{Y/Z}[d]$. 
\end{theo}

\begin{defi}\label{defi:vt} 
We say that $Y/Z$ is {\it of vertical type} 
if ${\cal I}_{Y/Z}{\cal O}_Y={\cal O}_Y$.  
\end{defi}

\begin{exem}\label{exem:vt}
If $X/s$ is an SNCL scheme (\cite{nlk3}, \cite{nlw}), then 
$X/s$ is of vertical type. 
%More generally, assume that 
%$X/s$ has a local chart around any point of 
%$\os{\circ}{X}$ etale locally 
%given by the following commutative diagram 
%\begin{equation*}
%\begin{CD} 
%\bigoplus_{i=1}^b{\mab N}^{{\oplus}(a_i+1)}
%@>>> \kap[x_{10}, \ldots, x_{1d_1},\ldots, x_{b0}, \ldots, x_{bd_b}]/
%(\prod_{j=0}^{a_1}x_{1j}, \ldots,\prod_{j=0}^{a_b}x_{bj}) 
%\\
%@A{h}AA @AAA\\
%{\mab N}@>>> \kap,  
%\end{CD} 
%\end{equation*} 
%where $a_i\leq d_i$ $(1\leq i\leq b)$, 
%the vertical map $h$ is the diagonal map
%${\mab N}\owns 1\lom (1,\ldots,1)\in 
%\bigoplus_{i=1}^b{\mab N}^{{\oplus}(a_i+1)}$ 
%and the horizontal maps are defined by the following maps  
%\begin{align*} 
%{\mab N}^{{\oplus}(a_i+1)}&\owns 
%(0, \ldots,0,\os{j}{1},0,\ldots, 0)\lom  \\
%&x_{i,j-1} \in 
%\kap[x_{10}, \ldots, x_{1d_1},\ldots, x_{b0}, \ldots, x_{bd_b}]
%/(\prod_{j=0}^{a_1}x_{1j}, \ldots,\prod_{j=0}^{a_b}x_{bj}) 
%\end{align*} 
%and ${\mab N}\owns 1\lom 0\in \kap$ and $a_i\leq d_i$ $(i=1,\ldots,b)$. 
%Because there does not exist a horizontal prime ideal of 
%$\bigoplus_{i=1}^b{\mab N}^{{\oplus}(a_i+1)}$ 
%of height 1 with respect to $h$, 
%${\cal I}_{X/s}=M_X$. 
%Hence ${\cal I}_{X/s}{\cal O}_X={\cal O}_X$. 
\end{exem}

\begin{coro}[{\bf The log version of the 
vanishing theorem of Kodaira-Akizuki-Nakano in 
characteristic $p$}]\label{coro:lkan}
Let $\kap$ be a perfect field of characteristic $p>0$. 
Let $s$ be the log point of $\kap$ or $({\rm Spec}(\kap),\kap^*)$. 
Let $Y\lo s$ be 
a projective log smooth morphism of Cartier type of fs log schemes 
which has a log smooth integral lift over ${\cal W}_2(s)$. 
%Assume that $\os{\circ}{Y}$ is projective over $\kap$. 
Assume that $\os{\circ}{Y}$ is of pure dimension $d$. 
Let ${\cal I}_{Y/s}$ be Tsuji's ideal sheaf of $M_Y$. 
%which will be recalled
%after {\rm (\ref{coro:bp})} below.  
Let ${\cal L}$ be an ample invertible ${\cal O}_Y$-module. 
Then the following hold$:$
%$R^jf_*(\Om^i_{Y/s}\otimes{\cal L}^{-1})=0$ for $i+j<\min \{d,p\}$. 
\par 
$(1)$ $H^j(Y,\Om^i_{Y/s}\otimes{\cal L}^{-1})=0$ for $i+j<\min \{d,p\}$.  
\par 
$(2)$ $H^j(Y,{\cal I}_{Y/s}\Om^i_{Y/s}\otimes{\cal L})=0$ for $i+j>\max \{d,2d-p\}$.  
%\par 
%$(2)$ 
%$R^jf_*(\Om^i_{Y/s}\otimes{\cal L}^{-1})=0$ $(i+j>\max \{d,2d-p\}$ 
\end{coro}
\begin{proof} 
(1): The proof is completely the same as that of \cite[(2.8), (2.9)]{di} by using 
%Serre's theorem (\cite[(2.2.1)]{ega31}) and 
Tsuji's log Serre duality (\ref{theo:ico}). 
\par 
Indeed, set ${\cal F}(m):={\cal F}\otimes_{{\cal O}_Y}{\cal L}^{\otimes m}$ 
$(m\in {\mab Z})$ for a coherent ${\cal O}_Y$-module ${\cal F}$ 
and ${\cal M}:={\cal L}^{-1}$ and $b:={\rm min}\{d,p\}$. 
If $m$ is large enough, then 
$H^q(Y,{\cal I}_{Y/s}\Om^{d-i}_{Y/s}(m))=0$ for any $i$ and any $q>0$ 
by Serre's theorem \cite[(2.2.1)]{ega31}. 
By Tsuji's log Serre duality, $H^j(Y,\Om^{i}_{Y/s}(-m))=0$ for 
any $i\in {\mab N}$ and $j<d$. 
In particular, $H^j(Y,\Om^{i}_{Y/s}(-m))=0$ for 
any $i+j<d$ $(i,j\in {\mab N})$ and hence 
$H^j(Y,\Om^{i}_{Y/s}(-m))=0$ for 
any $i+j<b$ $(i,j\in {\mab N})$. 
Assume that $H^j(Y,\Om^{i}_{Y/s}(-p^n))=0$ 
for all $i+j< b$ and  a positive integer $n$. 
Then we claim that $H^j(Y,\Om^{i}_{Y/s}(-p^{n-1}))=0$. 
Indeed, let $W\col Y'\lo Y$ be the projection.  
Because the differential 
$d\col F_*(\Om^i_{Y/s})\lo F_*(\Om^{i+1}_{Y/s})$ is ${\cal O}_{Y'}$-linear, 
we can consider the complex 
$W^*({\cal M}^{\otimes p^{n-1}})
\otimes_{{\cal O}_{Y'}}F_*(\Om^{\bul}_{Y/s})$. 
Take the tensorization with $W^*({\cal M}^{\otimes p^{n-1}})$ 
for the isomorphism 
$\bigoplus_{i<b}\Om^i_{Y'/s}[-i]\os{\sim}{\lo} F_*(\Om^{\bul}_{Y/s})$ 
in $D^+(Y'_{\rm zar})$: 
\begin{equation*} 
\begin{CD} 
\bigoplus_{i<b}
W^*({\cal M}^{\otimes p^{n-1}})
\otimes^L_{{\cal O}_{Y'}}\Om^i_{Y'/s}[-i]
@>{\sim}>>W^*({\cal M}^{\otimes p^{n-1}})
\otimes^L_{{\cal O}_{Y'}}F_*(\Om^{\bul}_{Y/s})\\
@| @|\\
\bigoplus_{i<b}W^*({\cal M}^{\otimes p^{n-1}})
\otimes_{{\cal O}_{Y'}}\Om^i_{Y'/s}[-i]
@.W^*({\cal M}^{\otimes p^{n-1}})
\otimes_{{\cal O}_{Y'}}F_*(\Om^{\bul}_{Y/s}).  
\end{CD}
\tag{5.9.1}\label{ali:ys}
\end{equation*} 
(Note that $W^*({\cal M}^{\otimes p^{n-1}})$ 
is a flat ${\cal O}_{Y'}$-module.)
We have the following spectral sequence:
\begin{align*} 
E_1^{ij}=H^j(Y',W^*({\cal M}^{\otimes p^{n-1}})\otimes_{{\cal O}_{Y'}}
F_*(\Om^i_{Y/s}))\Lo 
H^{i+j}(Y',W^*({\cal M}^{\otimes p^{n-1}})\otimes_{{\cal O}_{Y'}}
F_*(\Om^{\bul}_{Y/s})). 
\end{align*} 
By the projection formula and the assumption, $E_1^{ij}=
R^jf'_*F_*(F^*W^*({\cal M}^{\otimes p^{n-1}})\otimes_{{\cal O}_{Y}}\Om^i_{Y/s})
=H^j(Y,{\cal M}^{\otimes p^n}\otimes_{{\cal O}_{Y}}\Om^i_{Y/s})=H^j(Y,\Om^i_{Y/s}(-p^n))=0$. 
Hence $H^{i+j}(Y',W^*({\cal M}^{\otimes p^{n-1}})\otimes_{{\cal O}_{Y'}}
F_*(\Om^{\bul}_{Y/s}))=0$ for $i+j<b$. 
By (\ref{ali:ys}),  
$H^j(Y',W^*({\cal M}^{\otimes p^{n-1}})
\otimes_{{\cal O}_{Y'}}\Om^i_{Y'/s})=0$. 
Since $Y/s$ is integral and $\os{\circ}{s}$ is perfect, 
$\os{\circ}{Y}{}'=`\os{\circ}{Y}\simeq \os{\circ}{Y}$ and 
$\Om^i_{Y'/s}=\Om^i_{`Y/s^{[p]}}\simeq \Om^i_{Y/s}$. 
Hence $H^j(Y,\Om^i_{Y/s}(-p^{n-1}))=
H^j(Y',W^*({\cal M}^{\otimes p^{n-1}})
\otimes_{{\cal O}_{Y'}}\Om^i_{Y'/s})=0$. 
\par 
(2): (2) follows from (1) and Tsuji's log Serre duality. 
\end{proof}

The following is a generalization of Norimatsu's vanishing theorem 
(\cite[Theorem 1]{no}). 
The following vanishing theorem is not a special case of 
Ambro-Fujino's vanishing theorem 
(\cite[Theorem 3.2]{amq}, \cite[Theorem 5.7]{fuv}) 
and Fujino's vanishing theorem \cite[Theorem 1.1]{fuk}.

\begin{coro}[{\bf A log version of 
vanishing theorem of Kodaira-Akizuki-Nakano in 
characteristic $0$}]\label{coro:0lkan}
Let $K$, $T$ and $Z$ be as in {\rm (\ref{rema:de1}) (2)}. 
Assume that the log structure of $T$ is associated to a morphism 
${\mab N}\owns 1\lom a\in K$ for some $a\in K$. 
Assume also that $\os{\circ}{Z}$ is projective over $K$.  
%Let ${\cal J}_{Y/S_0}$ be ${\cal I}_{Y/S_0}$ or ${\cal O}_Y$. 
Let ${\cal L}$ be an ample invertible ${\cal O}_Z$-module. 
Then the following hold$:$
%$R^jf_*(\Om^i_{Y/S_0}\otimes{\cal L}^{-1})=0$ for $i+j<\min \{d,p\}$. 
\par 
$(1)$ $H^j(Z,\Om^i_{Z/T}\otimes{\cal L}^{-1})=0$ for $i+j<d$.  
\par 
$(2)$ $H^j(Z,{\cal I}_{Z/T}\Om^i_{Z/T}\otimes{\cal L})=0$ for $i+j>d$.  
%\par 
%$(2)$ 
%$R^jf_*(\Om^i_{Y/S_0}\otimes{\cal L}^{-1})=0$ $(i+j>\max \{d,2d-p\}$ 
\end{coro}
\begin{proof} 
The proof is the same as that of \cite[(7.1.2)]{ikn} by using 
Kato-Tsuji's result (\ref{prop:ktj}). 
%By a standard argument, 
%there exists an integral domain $A$ of finite type 
%over ${\mab Z}$ and an fs log scheme ${\cal T}$ 
%whose underlying scheme is ${\rm Spec}(A)$ 
%and a proper log smooth integral scheme ${\cal Z}$ over ${\cal T}$ such that 
%the structural morphism ${\mathfrak g} \col {\cal Z}\lo {\cal T}$ is saturated 
%(cf.~\cite{nd}). 
%By Kato-Tsuji's result (\ref{prop:ktj}), 
%every fiber of ${\mathfrak g}$ is of Cartier type 
%if the fiber is in characteristic $p>0$. 
%Shrinking ${\cal T}$ if necessary, we may assume that  
%$R^j{\mathfrak g}_*(\Om^i_{{\cal Z}/{\cal T}}\otimes{\cal L}^{-1})$ 
%and $R^{j}{\mathfrak g}_*
%({\cal I}_{{\cal Z}/{\cal T}}\Om^i_{{\cal Z}/{\cal T}}\otimes{\cal L})$ 
%are locally free ${\cal O}_{\cal T}$-modules.  
%The rest of the proof is the same as that of \cite[(2.7)]{di}. 
\end{proof}

\section{Log weak Lefschetz conjecture}\label{sec:lw} 
In this section we give the precise definition of 
the horizontal divisor appearing in the log weak Lefschetz conjecture 
 (\ref{conj:lwl}) 
and we prove the log weak Lefschetz conjecture in characteristic $0$ 
and we prove this conjecture in characteristic $p>0$ in certain cases. 
\par 
First we give the following definitions: 

\begin{defi}\label{defi:hdd} 
(1) Let $S_0$ be a family of log points 
(\cite[(1.1)]{nlw}) and let $X/S_0$ be an SNCL scheme ([loc.~cit., (1.1.16)]). 
Let ${\mab A}_{S_0}(a,d+e)$ $(a\leq d)$ 
be a log scheme whose underlying scheme is 
$\ul{\rm Spec}_{S_0}({\cal O}_{S_0}
[x_0,\ldots, x_d,y_1, \ldots, y_e]/(x_0\cdots x_a))$ and whose log structure is 
associated to the morphism 
$${\mab N}^{\oplus a+1}\owns e_i\lom x_{i-1}\in {\cal O}_{S_0}
[x_0,\ldots, x_d,y_1, \ldots, y_e]/(x_0\cdots x_a).$$ 
\par 
Let $\os{\circ}{D}$ be an effective Cartier divisor on $\os{\circ}{X}/\os{\circ}{S}_0$. 
Endow $\os{\circ}{D}$ with the inverse image of the log structure of $X$ and 
let $D$ be the resulting log scheme. 
We call $D$ a {\it relative simple normal crossing divisor 
$(=:$relative SNCD$)$} on $X/S_0$ if there exists a family 
$\Del:=\{\os{\circ}{D}_{\lam}\}_{\lam \in \Lam}$ of 
non-zero effective Cartier divisors on $X/S_0$ 
of locally finite intersection which are 
SNC(=simple normal crossing) schemes over $S_0$ (\cite[(1.1.9)]{nlw}) such that 
\begin{equation*}  
\os{\circ}{D} = \sum_{\lam \in \Lam}\os{\circ}{D}_{\lam}
\quad \text{in} \quad {\rm Div}(\os{\circ}{X}/S_0)_{\geq 0} 
\tag{6.1.1}\label{eqn:dcsncd}
\end{equation*} 
and,  
for any point $z$ of $\os{\circ}{D}$, there exist 
a Zariski open neighborhood $\os{\circ}{V}$ of $z$ in $\os{\circ}{X}$ and 
the following cartesian diagram
\begin{equation*}
\begin{CD}
D\vert_V @>>> (y_1\cdots y_b=0) \\ 
@V{\bigcap}VV  @VVV \\
V @>{g}>> {\mab A}_{T_0}(a,d+e) \\ 
@VVV  @VVV \\
T_0@=T_0 
\end{CD}
\tag{6.1.2}\label{cd:1b}
\end{equation*}
for some positive integers $a$, $b$, $d$ and $e$ 
such that $a\leq d$ and $b\leq e$. 
Here $T_0$ is an open log subscheme of $S_0$ 
whose log structure is associated to 
the morphism ${\mab N}\owns 1\lom 0\in {\cal O}_{T_0}$, 
$(y_1\cdots y_b=0)$ is an exact closed log subscheme of 
${\mab A}_{T_0}(a,d+e)$ defined by an ideal sheaf $(y_1\cdots y_b)$,  
$g$ is strictly \'{e}tale and ${\mab A}_{T_0}(a,d+e)\lo T_0$ is obtained by 
the diagonal embedding ${\mab N} \os{\subset}{\lo} {\mab N}^{\oplus a+1}$. 
Endow $\os{\circ}{D}_{\lam}$ with the inverse image of 
the log structure of $X$ and let $D_{\lam}$ be the resulting log scheme. 
We call $D_{\lam}$ an {\it SNCL component} of $D$ and 
the equality (\ref{eqn:dcsncd}) a {\it decomposition} of $D$ 
by SNCL components of $D$. 
\par 
(2) Let the notations be as in (1). 
Let $E$ be another SNCD on $X/S_0$. 
Let $D\cup E$ be a log scheme 
whose underlying scheme is $\os{\circ}{D}\cup \os{\circ}{E}$ 
and whose log structure is the inverse image of the log structure of $X$. 
Then we say that $D\cup E$ is an SNCD on $X/S_0$ 
if, in the diagram (\ref{cd:1b}) for any point $z\in \os{\circ}{D}\cup \os{\circ}{E}$, 
$(D\cup E)\vert_V=(y_1\cdots  y_c=0)$ for some 
$b\leq c\leq e$. In this case, we denote $D\cup E$ by $D+E$. 
\end{defi} 

\par
The following construction of $M(D)$ is 
the log version of the construction in \cite[p.~61]{nh2}. 
\par 
Let ${\rm Div}_D(\os{\circ}{X}/\os{\circ}{S}_0)_{\geq 0}$ be a submonoid of 
${\rm Div}(\os{\circ}{X}/\os{\circ}{S}_0)_{\geq 0}$ consisting 
of effective Cartier divisors $E$'s on $\os{\circ}{X}/\os{\circ}{S}_0$
such that there exists an open covering 
$X = \bigcup_{i \in I}V_i$ (depending on $E$) 
of $X$ such that $E \vert_{V_i}$ is contained 
in the submonoid of ${\rm Div}(\os{\circ}{V}_i/\os{\circ}{S}_0)_{\geq 0}$ 
generated by $\os{\circ}{D}_{\lam}\vert_{\os{\circ}{V}_i}$ $(\lam \in \Lam)$. 
By \cite[A.0.1]{nh2} the definition of ${\rm Div}_{\os{\circ}{D}}(\os{\circ}{X}/\os{\circ}{S}_0)_{\geq 0}$ 
is independent of the choice of $\Del$. 
(We have only to set  
$S:={\rm Spec}_{\os{\circ}{T}_0}({\cal O}_{\os{\circ}{T}_0}
[x_0,\ldots,x_d]/(x_0\cdots x_a))$ 
in [loc.~cit.] and to consider the projection 
$X_{\os{\circ}{T}_0}\times_{\os{\circ}{T}_0}S
\lo X_{\os{\circ}{T}_0}$.)
\par   
The pair $(X,D)$ gives the following fs log structure $M(D)$ 
in the zariski topos $\os{\circ}{X}_{\rm zar}$ as in \cite[p.~61]{nh2}. 
\par 
Let $M(D)'$ be a presheaf of monoids in $\os{\circ}{X}_{\rm zar}$ defined as follows: 
for an open subscheme $\os{\circ}{V}$ of $\os{\circ}{X}$, 
\begin{align*}
\Gamma(\os{\circ}{V}, M(D)'):=
\{(E,a)\in & {\rm Div}_{\os{\circ}{D}\vert_{\os{\circ}{V}}}(\os{\circ}{V}/\os{\circ}{S}_0)_{\geq 0} 
\times \Gamma(\os{\circ}{V}, {\cal O}_X)\vert ~a\text{ is a generator of }
 \Gamma(\os{\circ}{V}, {\cal O}_{\os{\circ}{X}}(-E))\} 
\end{align*} 
with a monoid structure defined by an equation 
$(E,a) \cdot (E',a') := (E + E', aa')$. 
The  natural  morphism 
$M(D)' \lo {\cal O}_X$ defined by 
the second projection 
$(E,a) \mapsto a$
induces a morphism 
$M(D)' \lo ({\cal O}_X,*)$
of presheaves of monoids in 
$X_{\rm zar}$.
The log structure $M(D)$ is, 
by definition, the associated log 
structure to the sheafification of $M(D)'$. 
Because ${\rm Div}_{D\vert_V}(V/S_0)_{\geq 0}$ 
is independent of the choice of the decomposition 
of $D\vert_V$ by smooth components, 
$M(D)$ is independent of 
the choice of the decomposition of $D$ 
by SNCL components of $D$. 
\par

\begin{prop}\label{prop:logst}
Let the notations be as above. 
Let $z$ be a point of $D$ and let 
$V$ be an open neighborhood of $z$ in $X$ 
which admits the diagram {\rm (\ref{cd:1b})}. 
Assume that $z \in \bigcap_{i=1}^b\{y_i=0\}$. 
If $V$ is small, then the log structure 
$M(D) \vert_V \lo {\cal O}_V$ is 
isomorphic to ${\cal O}_V^*y_1^{{\mab N}} 
\cdots y_b^{{\mab N}} 
\os{\subset}{\lo} {\cal O}_V$. 
Consequently $M(D) \vert_V$ is associated to 
the homomorphism 
${\mab N}^b_V \owns e_i \lom  y_i \in M(D) \vert_V$ 
$(1 \leq i \leq b)$ of sheaves of monoids on $V$, 
where $\{e_i\}_{i=1}^b$ is the canonical 
basis of ${\mab N}^b$.  
In particular, $M(D)$ is fs. 
\end{prop}
\begin{proof}
We claim that, by shrinking $V$ in (\ref{cd:1b}), 
for any $1 \leq i \leq b$, 
there exists a unique element 
$\lam_i \in \Lam$ satisfying  
\begin{align*}
\os{\circ}{D}_{\lam_i}\vert_{\os{\circ}{V}} 
= {\rm div}(y_i) \quad {\rm in} \quad {\rm Div}(\os{\circ}{V}/\os{\circ}{S}_0)_{\geq 0}.
\tag{6.2.1}\label{ali:idv}
\end{align*}
This follows from \cite[Proposition A.~0.1]{nh2} 
by setting $S:=({\mab A}_{T_0}(a,d))^{\circ}$, 
$X:=\os{\circ}{V}$ and $D:=\os{\circ}{D}$ in [loc.~cit]. 
The rest of the proof is the same as that of \cite[(2.1.9)]{nh2}. 
\end{proof} 

\parno 
Set 
\begin{align*} 
X(D):=(X,M_X\oplus_{{\cal O}^*_X}M(D)\lo {\cal O}_X). 
\end{align*} 
Then $X(D)/S_0$ is log smooth, integral and saturated by (\ref{prop:ktj}) (4).

\begin{rema}\label{rema:ldc} 
As in the classical case (e.~g., \cite{dh2}), 
we can consider the log de Rham complex 
$\Om^{\bul}_{X/S_0}(\log D)$ with logarithmic poles along $D$.  
It is clear that 
the complex $\Om^{\bul}_{X/S_0}(\log D)$ is equal to the log de Rham complex 
$\Om^{\bul}_{X(D)/S_0}$. 
Set $\Om^i_{X/S_0}(\log D)(-D):=
{\cal O}_{\os{\circ}{X}}(-\os{\circ}{D})\otimes_{{\cal O}_X}\Om^i_{X/S_0}(\log D)$ 
$(i\in {\mab N})$. 
It is easy to check that 
the family $\{\Om^i_{X/S_0}(\log D)(-D)\}_{i\in \mab N}$ gives a complex 
$\Om^{\bul}_{X/S_0}(\log D)(-D)$: 
$d\Om^i_{X/S_0}(\log D)(-D)\subset \Om^{i+1}_{X/S_0}(\log D)(-D)$. 
%Consequently we have a complex $\Om^{\bul}_{X/S_0}(\log D)(-D)$. 
\end{rema}

\par
Set 
\begin{equation*}
D_{\{\lam_1, \lam_2,\ldots \lam_k\}} 
:=D_{\lam_1}\cap D_{\lam_2} \cap 
\cdots \cap D_{\lam_k} \quad 
(\lam_i \not= \lam_j~{\rm if}~i\not= j) 
%%%\tag{9.13.1}
\label{eqn:parlm}
\end{equation*}
for a positive integer $k$, and set
\begin{equation*}
D^{(k)} = 
\begin{cases} \quad \quad  \quad \quad X & (k=0), \\
\us{\{\lam_1, \ldots,  \lam_{k}~\vert~\lam_i 
\not= \lam_j~(i\not=j)\}}{\coprod}
D_{\{\lam_1, \lam_2, \ldots, \lam_k\}} & (k\geq 1)
\end{cases}
%%%\tag{9.13.2}
\label{eqn:kfdintd}
\end{equation*}
for a nonnegative integer $k$.
%Set 
%\begin{equation}
%D_{\emptyset}:=X  
%%%\tag{9.13.3}
%\label{eqn:dphix}
%\end{equation}
%for later convenience.

\begin{prop}
\label{prop:dkwdef}
$D^{(k)}$ is independent of the choice of 
the decomposition of $D$ by 
smooth components of $D$.
\end{prop}
\begin{proof}
The proof is the same as those of \cite[(2.2.14), (2.2.15)]{nh2}. 
\end{proof}

\parno 
The following is the log version of a generalization of \cite[(2.12)]{di}. 
\begin{coro}\label{coro:lwl}
Let $X$ be a projective SNCL scheme over the log point $s$ of 
a perfect field of characteristic $p>0$. 
Let $D$ be a $($relative$)$ SNCD on $X/s$. 
%Assume that $b=1$ in {\rm (\ref{cd:1b})} for 
%any point $z\in \os{\circ}{D}$ 
Let $E$ be a $($relative$)$ SNCD on $X/s$ such that 
$D+E$ is also a $($relative$)$ SNCD on $X/s$. 
Assume that 
${\cal O}_{\os{\circ}{X}}(\os{\circ}{E})$ is an ample invertible ${\cal O}_X$-module. 
Assume that $X(D)/s$ and $E(D\cap E)/s$ lift to ${\cal W}_2(s)$. 
For simplicity of notation, denote $E(D\cap E)$ and 
$E^{(k)}(D\cap E^{(k)})$ by $E(D)$ and 
$E^{(k)}(D)$, respectively. 
Let $a\col E^{(1)}(D)\lo E^{(2)}(D)$ be the natural morphism. 
Set 
$K(E(D))^{\bul}:={\rm Ker}(\Om^{\bul}_{E^{(1)}(D)/s}\lo a_*(\Om^{\bul}_{E^{(2)}(D)/s}))$. 
%Here we omit to write the direct image of the morphism $E^{(2)}(D)\lo E^{(1)}(D)$. 
Then the following hold$:$
\par 
$(1)$ The restriction morphism 
\begin{align*} 
H^q_{\rm dR}(X(D)/s)\lo  H^q(E^{(1)},K(E(D))^{\bul})
\end{align*} 
is an isomorphism for $q< \min \{d,p\}-1$ and injective for $q= \min \{d,p\}-1$. 
\par 
$(2)$ The restriction morphism 
\begin{align*} 
H^j(X(D),\Om^i_{X(D)/s})\lo  H^j(E^{(1)},K(E(D))^i)
\end{align*} 
is an isomorphism for $i+j< \min \{d,p\}-1$ and injective for $i+j= \min \{d,p\}-1$. 
\end{coro}
\begin{proof} 
(1): (The following proof includes a correction of the proof of \cite[(2.12)]{di} 
(see (\ref{rema:mdi}) below).) 
As in \cite[(4.2.2) (c)]{di}, the following sequence 
\begin{align*} 
0\lo \Om^{\bul}_{X(D+E)/s}(-E)\lo \Om^{\bul}_{X(D)/s}\lo \Om^{\bul}_{E^{(1)}(D)/s}\lo 
\Om^{\bul}_{E^{(2)}(D)/s}\lo \cdots
\end{align*}  
is exact. 
Hence we have the following exact sequence 
\begin{align*} 
0\lo \Om^{\bul}_{X(D+E)/s}(-E)\lo \Om^{\bul}_{X(D)/s}\lo K(E(D))^{\bul}\lo 0. 
\end{align*} 
It suffices to prove that 
$H^q(X,\Om^{\bul}_{X(D+E)/s}(-E))=0$ for $q< \min \{d,p\}$. 
By the following spectral sequence 
\begin{align*} 
E^{ij}_1=H^j(X,\Om^i_{X(D+E)/s}(-E))\Lo 
H^{i+j}(X,\Om^{\bul}_{X(D+E)/s}(-E)), 
\tag{6.5.1}\label{ali:spv} 
\end{align*} 
it suffices to prove that 
$H^j(X,\Om^i_{X(D+E)/s}(-E))=0$ for $i+j< \min \{d,p\}$. 
This is a special case of (\ref{coro:lkan}) (1). 
\par 
$(2)$ As in the proof of (1), it suffices to prove that 
$H^j(X,\Om^i_{X(D+E)/s}(-E))=0$ for $i+j< \min \{d,p\}$. 
We have already proved this in the proof of (1). 
\end{proof}

\begin{rema}\label{rema:mdi}
(cf.~the proof of \cite[Theorem 1]{no})
(1) Let the notations be as in \cite[(2.12)]{di}. 
There is an elementary error in the proof of [loc.~cit.]
because there does not exist 
complexes $\Om^{\bul}_X(-D)$ and $\Om^{\bul}_D(-D)$ in [loc.~cit.]. 
Consequently we do not have an exact sequence 
\begin{align*} 
0\lo \Om^{\bul}_{X}(-D)\lo \Om^{\bul}_{X}(\log D)(-D)
\lo \Om^{\bul-1}_{D}(-D)\lo 0. 
\tag{6.6.1}\label{ali:exbxd} 
\end{align*} 
of complexes in [loc.~cit.]. 
\par 
The correction of the proof is easy.   
We have only to use the following spectral sequence (\ref{ali:spdv})  
and the following exact sequence (\ref{ali:exdxd}) and the 
following vanishing (\ref{ali:exfrxd}) (which follows from \cite[(2.8)]{di}):   
%It suffices to prove that 
%$H^q(X,\Om^{\bul}_{X/\kap}(\log D)(-D))=0$ for $q< \min \{d,p\}$. 
%By the following spectral sequence 
\begin{align*} 
E^{ij}_1=
H^j(X,\Om^i_{X/\kap}(\log D)(-D)))\Lo H^{i+j}(X,\Om^{\bul}_{X/\kap}(\log D)(-D)), 
\tag{6.6.2}\label{ali:spdv} 
\end{align*} 
%it suffices to prove that 
%$H^j(X,\Om^i_{X/\kap}(\log D)(-D)))=0$ for $i+j< \min \{d,p\}$. 
%We have the following exact sequence: 
\begin{align*} 
0\lo \Om^i_{X/\kap}(-D)\lo \Om^i_{X/\kap}(\log D)(-D)
\os{{\rm Res}\otimes {\rm id}_{{\cal O}_X(-D)}}{\lo} 
\Om^{i-1}_{D/\kap}(-D\vert_{D})\lo 0. 
\tag{6.6.3}\label{ali:exdxd} 
\end{align*} 
%By \cite[(2.8)]{di} 
\begin{align*} 
H^j(X,\Om^i_{X/\kap}(-D))=0=
H^j(D,\Om^{i-1}_{D/\kap}(-D\vert_{D}))=0\quad {\rm for}  \quad
i+j< \min \{d,p\}.
\tag{6.6.4}\label{ali:exfrxd} 
\end{align*}  
%Hence $H^j(X,\Om^i_{X/\kap}(\log D)(-D))=0$ for $i+j<\min \{d,p\}$. 
%we can correct the proof immediately as follows:  
\par 
(2) As in the proof of (\ref{coro:lwl}), to prove \cite[(2.8)]{di}, one can also use 
theory for log de Rham complex in \cite[4.2]{di}.  
\end{rema}

\par 
The following is a generalization of 
Norimatsu's results \cite[Theorem 2, Corollary]{no}.

\begin{coro}\label{coro:0lrkan}
Let the notations be as in {\rm (\ref{coro:0lkan})}. 
Assume that $a$ in {\rm (\ref{coro:0lkan})} is equal to $0$ 
and denote $T$ by $s$.  
Let $Z/s$ be a projective SNCL scheme.  
Let $D$ and $E$ be SNCD's on $Z/s$ such 
that $D+E$ is also an SNCD on $Z/s$. 
Assume that ${\cal O}_{\os{\circ}{Z}}(\os{\circ}{E})$ 
is an ample invertible ${\cal O}_Z$-module. 
%Set $K(Z(E))^{\bul}:={\rm Ker}(\Om^{\bul}_{E^{(1)}/s}\lo \Om^{\bul}_{E^{(2)}/s})$. 
Let $a\col E^{(1)}(D)\lo E^{(2)}(D)$ be the natural morphism. 
Let $K(E(D))^{\bul}:={\rm Ker}(\Om^{\bul}_{E^{(1)}(D)/s}\lo 
a_*(\Om^{\bul}_{E^{(2)}(D)/s}))$ be a complex defined similarly as in {\rm (\ref{coro:lwl})}. 
Then the following hold$:$ 
\par 
$(1)$ The restriction morphism 
\begin{align*} 
H^q_{\rm dR}(Z(D)/s)\lo  H^q(E^{(1)},K(E(D))^{\bul})
\end{align*} 
is an isomorphism for $q< d-1$ and injective for $q= d-1$. 
\par 
$(2)$ 
The restriction morphism 
\begin{align*} 
H^j(Z(D),\Om^i_{Z(D)/s})\lo  H^j(E^{(1)},K(E(D))^i)
\end{align*} 
is an isomorphism for $i+j< d-1$ and injective for $i+j= d-1$. 
\end{coro} 
\begin{proof} 
The proof is an analogue of the proof of \cite[(7.1.2)]{ikn}. 
\end{proof} 

\begin{coro}[{\bf Log weak Lefschetz theorem in characteristic $0$}]\label{coro:lwlc} 
Let the notations be as in {\rm (\ref{coro:0lrkan})}.  
Assume that $D=\emptyset$ and $E^{(2)}=\emptyset$. 
Then the following hold$:$
\par 
$(1)$ 
The following pull-back morphism by the inclusion 
$\iota \col E\os{\sus}{\lo} Z$ 
\begin{align*} 
\iota^*\col H^q_{\rm dR}(Z/s)\lo H^q_{\rm dR}(E/s)
\end{align*} 
is an isomorphism for $q< d-1$ and injective for $q= d-1$. 
\par 
$(2)$ 
Assume furthermore that $K={\mab C}$. 
Let $Z^{\log}$ be the Kato-Nakayama space of $Z$ 
with natural morphism $Z\lo {\mab S}^1$ 
{\rm (\cite[(1.2)]{ktn})}.  
Let ${\mab R}\owns x\lom 
\exp(2\pi\sqrt{-1}x) \in {\mab S}^1$ be 
the universal cover of 
${\mab S}^{1}$ and set $Z_{\infty}:=
Z^{\log}\times_{{\mab S}^1}{\mab R}$ {\rm (\cite{us})}. 
The following pull-back morphism by the inclusion 
$\iota \col E\os{\sus}{\lo} Z$ 
\begin{align*} 
\iota^*\col H^q(Z_{\infty},{\mab Q})\lo H^q(E_{\infty},{\mab Q})
\end{align*} 
is an isomorphism of mixed Hodge structures for $q< d-1$ 
and a strictly injective morphism of mixed Hodge structures for $q= d-1$. 
\end{coro}
\begin{proof} 
(1): (1) is a special case of (\ref{coro:0lrkan}). 
\par 
(2): By \cite{fn} the morphism 
$\iota^*$ is a morphism of mixed Hodge structures. 
Hence (2) follows (1) and theory of mixed Hodge structures in \cite{dh2}. 
\end{proof}

\begin{rema}
(1) In \cite[(9.14)]{nlpi} we have proved the 
log hard Lefschetz theorem over ${\mab C}$. 
The result is as follows. 
\par 
Assume that $\os{\circ}{s}={\rm Spec}({\mab C})$. 
Let $Z/s$ be a projective SNCL variety. 
Let $\lam_{\infty}:=c_{1,\infty}({\cal L})\in H^2(Z_{\infty},{\mab Q})$ 
be the log cohomology class of 
an ample invertible ${\cal O}_Z$-module ${\cal L}$.
Then the left cup product of $\lam^j_{\infty}$ $(j\geq 0)$ 
\begin{equation*}
\lam^j_{\infty}  
\col H^{d-j}(Z_{\infty},{\mab Q}) 
\lo H^{d+j}(Z_{\infty},{\mab Q})(j) 
\tag{6.9.1}\label{eqn:dmlhl}
\end{equation*}
\parno
is an isomorphism of mixed Hodge structures. 
We have proved this theorem by using a result of M.~Saito (\cite[(4.2.2)]{sm}). 
Let 
\begin{equation*}
\del_E \col H^0_{\rm dR}(E/s) \lo H^2_{\rm dR}(X/s)(1).
\tag{6.9.2}\label{eqn:deldhlm}
\end{equation*}  
be the morphism defined in \cite[(10.1.2)]{nlpi}. 
This morphism induces the following morphism 
\begin{equation*}
\iota_* \col H^q_{\rm dR}(E/s) \lo H^{q+2}_{\rm dR}(X/s)(1).
\tag{6.9.3}\label{eqn:itota}
\end{equation*}  
By \cite[(10.1.3)]{nlpi}, the composite morphism 
\begin{equation*}
-\iota_*\iota^* \col H^q_{\rm dR}(X/s) \lo 
H^q_{\rm dR}(E/s) \lo H^{q+2}_{\rm dR}(X/s)(1).
\tag{6.9.4}\label{eqn:itocta}
\end{equation*} 
is the cup product with $\lam_{\infty}\cup (?)$. 
Hence we obtain (\ref{coro:lwlc}) (2) and (1) by 
the hard Lefschetz theorem above 
as in \cite[II Corollary]{kme}. 
\par 
(2) We would like to lay emphasis on the algebraic nature of the proof of 
(\ref{coro:lwlc}) as in \cite{di}.  
\end{rema}

%\begin{rema}\label{rema:ntr}
%Let $N$ be a nonnegative integer or $\infty$. 
%Then the $N$-truncated versions of the results in this section except 
%(\ref{coro:tsud}) hold. For example,  the following holds: 
%\par 
%Let $K$ be a field of characteristic $0$. 
%Let $T$ be an fs log scheme whose underlying scheme is ${\rm Spec}(K)$. 
%Let $g\col Z_{\bul \leq N} \lo T$ be an $N$-truncated  
%proper log smooth integral morphism of fs log schemes. 
%Assume that $g$ is saturated. 
%Then the following spectral sequence 
%\begin{align*} 
%E_1^{ij}=
%R^jg_*(\Om^i_{Z_{\bul \leq N}/T})\Lo R^{i+j}g_*(\Om^{\bul}_{Z_{\bul \leq N}/T})
%\tag{7.16.1}\label{ali:oezbxs} 
%\end{align*}  
%degenerates at $E_1$. 
%\end{rema}

\par 
Let us go back to the case ${\rm ch}(\kap)=p>0$.  
Let $E$ be an SNCD on $X/s$ such that $E^{(2)}=\emptyset$. 
Let $q$ be a nonnegative integer. 
For a proper log smooth scheme $Y/s$, 
let $H^q_{\rm crys}(Y/{\cal W}(s))$ 
be the log crystalline cohomology of $Y/{\cal W}(s)$ (\cite{klog1}). 
By the works in \cite{msemi}, \cite{ndw} and \cite{nlw}, 
$H^q_{\rm crys}(X/{\cal W}(s))$ and 
$H^q_{\rm crys}(E/{\cal W}(s))$ have the weight filtrations $P$'s.
Set $K_0:={\rm Frac}({\cal W})$. 
For a module $M$ over ${\cal W}$, 
set $M_{K_0}:=M\otimes_{\cal W}K_0$. 
Let $\iota \col E\os{\sus}{\lo} X$ be the closed immersion. 
By a general theorem about the strict compatibility of 
the pull-back of a morphism of proper SNCL schemes 
in \cite[(5.4.7)]{nlw} (see (\ref{theo:imp}) below for the statement), 
the pull-back  of $\iota$  
\begin{align*} 
\iota^*_{\rm crys}  \col H^q_{\rm crys}(X/{\cal W}(s))_{K_0}\lo 
H^q_{\rm crys}(E/{\cal W}(s))_{K_0} \quad (q\in {\mab Z})
\tag{6.9.5}\label{ali:crxsw} 
\end{align*} 
is a strict filtered morphism with respect to the $P$'s. 
%In this paper we conjecture the following: 
%\begin{conj}[{\bf 
%Log weak Lefschetz conjecture in log crystalline cohomologies}]\label{conj:lwl}
%Assume that ${\cal O}_X(E)$ is ample. Then the morphism 
%{\rm (\ref{ali:crxw})} is a filtered isomorphism with respect to $P$'s 
%if $q\leq d-2$ and strictly injective for $q=d-1$. 
%\end{conj} 
\par
As to the log weak Lefschetz conjecture (\ref{conj:lwl}),  
we prove the following stimulated by 
the work of Berthelot (\cite{bwl}) in this article: 

\begin{theo}[{\bf Log weak Lefschetz theorem 
in log crystalline cohomologies}]\label{theo:qfls} 
Let the notations be as in {\rm (\ref{coro:lwl})}. 
Assume that $E^{(2)}=\emptyset$. 
Let $\iota \col E(D)\os{\sus}{\lo} X(D)$ be the closed immersion. 
Then the following hold$:$ 
\par 
$(1)$ The pull-back 
\begin{align*} 
\iota^*_{\rm crys}  \col H^q_{\rm crys}(X(D)/{\cal W}(s))\lo 
H^q_{\rm crys}(E(D)/{\cal W}(s)) \quad (q\in {\mab Z})
\tag{6.10.1}\label{ali:crxw} 
\end{align*} 
is an isomorphism if $q< \min \{d,p\}-1$ 
and injective for $q= \min \{d,p\}-1$ with torsion free cokernel. 
\par 
$(2)$ Assume that $D=\emptyset$. 
Then the morphism {\rm (\ref{ali:crxw})} modulo torsion 
is a filtered isomorphism for $q< \min \{d,p\}-1$ 
and strictly injective for $q= \min \{d,p\}-1$. 
\par 
$(3)$ Assume that $\os{\circ}{E}$ is a hypersurface section of 
$\os{\circ}{X}$ with respect to a closed immersion 
$\os{\circ}{X}\os{\sus}{\lo} {\mab P}^n_{\kap}$
into a projective space over $\kap$ and that 
the degree of $\os{\circ}{E}$ is sufficiently large. 
Then the morphism 
{\rm (\ref{ali:crxw})} is an isomorphism for $q< d-1$ 
and injective for $q= d-1$. 
\par 
$(4)$ Let the assumption be as in {\rm (3)}. 
Assume that $D=\emptyset$.  
Then the morphism {\rm (\ref{ali:crxw})} modulo torsion 
is a filtered isomorphism for $q< d-1$ 
and strictly injective for $q= d-1$. 
\end{theo} 
\begin{proof} 
The proof of (3) is slightly simpler than that in \cite{bwl}; 
the proof of (3) corrects the proof in \cite{bwl}. 
\par 
(1): Set $m:=\min \{d,p\}-1$. 
Let $K^{\bul}$ be the mapping cone of the following morphism 
\begin{align*} 
\iota^*_{\rm crys}  \col 
R\Gam_{\rm crys}(X(D)/{\cal W}(s))\lo 
R\Gam_{\rm crys}(E(D)/{\cal W}(s))). 
\end{align*}
Then it suffices to prove that $H^q(K^{\bul})=0$ for $q<m$ and 
that $H^m(K^{\bul})$ has no nontrivial torsion.  
By the universal coefficient theorem 
\begin{align*} 
0\lo H^q(K^{\bul})\otimes_{\cal W}\kap \lo 
H^q(K^{\bul}\otimes^L_{\cal W}\kap)\lo 
{\rm Tor}_1^{\cal W}(H^{q+1}(K^{\bul}),\kap)\lo 0, 
\end{align*}
it suffices to prove that $H^q(K^{\bul}\otimes^L_{\cal W}\kap)=0$ 
for $q<m$. 
Since $R\Gam((X(D)/{\cal W}(s))_{\rm crys})\otimes^L_{\cal W}\kap=
R\Gam_{\rm dR}(X(D)/\kap)$ and 
$R\Gam_{\rm crys}(E(D)/{\cal W}(s))\otimes^L_{\cal W}\kap=
R\Gam_{\rm dR}(E(D)/\kap)$ by the log base change theorem 
of K.~Kato (\cite[(6.10)]{klog1}),  
$K^{\bul}\otimes^L_{\cal W}\kap$ is the mapping cone of 
the following morphism 
\begin{align*} 
\iota^*_{\rm crys}  \col 
R\Gam_{\rm dR}(X(D)/\kap)\lo 
R\Gam_{\rm dR}(E(D)/\kap). 
\end{align*}
Hence (1) follows from the following exact sequence 
\begin{align*} 
\cdots \lo H^q_{\rm dR}(X(D)/\kap) \os{\iota^*_{\rm crys}}{\lo} 
H^q_{\rm dR}(E(D)/\kap)
\lo H^q_{\rm dR}(K^{\bul}\otimes^L_{\cal W}\kap)\lo \cdots 
\end{align*} 
and 
(\ref{coro:lwl}) (1). 
\par 
(2): (2) follows from (1) and (\ref{theo:imp}) below. 
\par 
(3): Set $e:=\deg \os{\circ}{E}$. 
Since 
${\rm Ker}(\Om^{\bul}_{X/s}\lo \Om^{\bul}_{E/s})=
\Om^{\bul}_{X/s}(\log E)(-E)$,  
it suffices to prove that 
$H^q(X,\Om^{\bul}_{X/s}(\log E)(-E))=0$ for $q<d$ by the proof of (1). 
Let $i$ and $j$ be nonnegative integers. 
By (\ref{ali:spv}) 
it suffices to prove that 
$H^j(X,\Om^i_{X/s}(\log E)(-E))=0$ for $i+j< d$. 
By the following sequence 
\begin{align*} 
0\lo \Om^i_{X/s}(-E)\lo \Om^i_{X/s}(\log E)(-E)
\os{{\rm Res}}{\lo} 
\Om^{i-1}_{E/s}(-E)\lo 0, 
\end{align*} 
it suffices to prove that 
$H^j(X,\Om^i_{X/s}(-E))=0$ for $i+j< d$ 
and $H^j(E,\Om^i_{E/s}(-E))=0$ for $i+j< d-1$. 
By Serre's theorem \cite[(2.2.1)]{ega31}, 
$H^{d-j}(X,\Om^{d-i}_{X/s}(E))=0$ for $i+j< d$ if $e$ is large enough. 
Hence $H^j(X,\Om^i_{X/s}(-E))=0$ for $i+j< d$ 
by the log Serre duality of Tsuji. 
The rest is to prove that  $H^j(E,\Om^i_{E/s}(-E))=0$ for $i+j< d-1$ 
if $e$ is large enough.  
Though this is the dual of 
$H^{d-1-j}(E,\Om^{d-1-i}_{E/s}(E))=0$ for $i+j< d-1$,  
the vanishing of this cohomology is nontrivial since 
$\Om^{d-1-i}_{E/s}$ depends on $E/s$.
\par 
Set ${\cal J}:={\cal O}_X(-E)$. 
Because $E/s$ is log smooth,  
the following second fundamental exact sequence in \cite[(2.1.3)]{nh2}
\begin{align*} 
{\cal J}/{\cal J}^2\lo \Om^1_{X/s}\otimes_{{\cal O}_X}{\cal O}_E
\lo \Om^1_{E/s}\lo 0
\end{align*} 
becomes the following exact sequence 
\begin{align*} 
0\lo ({\cal J}/{\cal J}^2)\lo \Om^1_{X/s}\otimes_{{\cal O}_X}{\cal O}_E
\lo \Om^1_{E/s}\lo 0. 
\end{align*} 
Because ${\cal J}/{\cal J}^2={\cal O}_E(-E)$, 
this sequence is equal to the following exact sequence 
\begin{align*} 
0\lo {\cal O}_E(-E)\lo \Om^1_{X/s}\otimes_{{\cal O}_X}{\cal O}_E
\lo \Om^1_{E/s}\lo 0. 
\end{align*} 
Hence we have the following exact sequences 
\begin{align*} 
0\lo \Om^{i-1}_{X/s}\otimes_{{\cal O}_X}
{\cal O}_E(-E)\lo 
\Om^i_{X/s}\otimes_{{\cal O}_X}{\cal O}_E
\lo \Om^i_{E/s}\lo 0. 
\end{align*} 
and 
\begin{align*} 
0\lo \Om^{i-1}_{X/s}\otimes_{{\cal O}_X}
{\cal O}_E(-2E)\lo 
\Om^i_{X/s}\otimes_{{\cal O}_X}{\cal O}_E(-E)
\lo \Om^i_{E/s}(-E)\lo 0. 
\end{align*} 
Hence we have the following exact sequence 
\begin{align*} 
\cdots \lo H^j(E,\Om^{i-1}_{X/s}\otimes_{{\cal O}_X}
{\cal O}_E(-2E))\lo 
H^j(E,\Om^i_{X/s}\otimes_{{\cal O}_X}{\cal O}_E(-E))
\lo H^j(E,\Om^i_{E/s}(-E))\lo \cdots . 
\end{align*} 
Thus it suffices to prove that 
$H^j(E,\Om^i_{X/s}\otimes_{{\cal O}_X}{\cal O}_E(-mE))=0$ for $i+j<d-1$ 
for $m\in {\mab Z}_{\geq 1}$.   
%First we prove $H^j(E,{\cal O}_E(-mE))=0$ for $j<d-1$. 
By the following exact sequence 
\begin{align*} 
0\lo \Om^i_{X/s}(-(m+1)E) \lo \Om^i_{X/s}(-mE)
\lo \Om^i_{X/s}\otimes_{{\cal O}_X}
{\cal O}_E(-mE)\lo 0, 
\tag{6.10.2}\label{ali:meme}
\end{align*} 
it suffices to prove that 
\begin{align*} 
H^j(X,\Om^i_{X/s}(-mE))=
0=H^{j+1}(X,\Om^i_{X/s}(-(m+1)E)).
\end{align*} 
Because 
$H^j(X,\Om^i_{X/s}(-mE))$ 
and $H^{j+1}(X,\Om^i_{X/s}(-(m+1)E))$ are 
the duals of 
$H^{d-j}(X,\Om^{d-i}_{X/s}(mE))$ and $H^{d-j-1}(X,\Om^{d-i}_{X/s}(mE))$ 
by the log Serre duality of Tsuji, the vanishing of the latter cohomologies  
follows from Serre's theorem if $e$ is large enough. 
%If $j=d-2$, then $H^{j+2}(X,{\cal O}_X(-3E))=
%H^0(X,{\cal O}_X(-3E))=0$ because $\deg (-3E)<0$.
%If $j=d-2$, then the vanishing of $H^{j+2}(X,{\cal O}_X(-3E))$ 
%follows from the vanishing of  
%$H^{d-j-2}(X,\Om^d_{X/s}(2E))$ if $e$ is large enough. 
\par 
(4): (4) follows from (3) and (\ref{theo:imp}) below. 
\end{proof} 

\begin{rema} 
(1) Let the notations be as in \cite{bwl}. 
There is a gap in 
the proof of $\ll$th\'{e}or\`{e}me de 
Lefschetz faible$\gg$ in [loc.~cit.]
because the sheaf $\Om^j_{Y/k}$ in [loc.~cit.] 
depends on $Y$: 
it is not clear that $H^q(Y,\Om^j_{Y/k}(Y))=0$ for $q>0$ 
even if the degree of $Y$ is large enough. 
%we can use Serre's vanishing theorem for 
%$\Om^j_{Y/k}$ since this sheaf is not a fixed sheaf with respect to 
%the inverse of the invertible sheaf defined by the ideal sheaf of $Y$. 
Strictly speaking, the proof of 
the weak and the hard Lefschetz theorems in \cite{kme} 
for the crystalline cohomology 
is also incomplete because  
it depends on Berthelot's result.  
\par 
(2) In the Appendix we give an easy proof of 
%Using rigid cohomology, one can easily prove 
the weak Lefschetz theorem in \cite{kme} by using 
a theory of rigid cohomology of Berthelot. 
\end{rema} 

The following theorem is a very special theorem of \cite[(5.4.7)]{nlw}: 

\begin{theo}\label{theo:imp}
Let $s$ and $s'$ be log points of perfect fields of characteristic $p>0$. 
Let 
\begin{equation*} 
\begin{CD} 
Y@>{g}>> Y' \\
@VVV @VVV \\
s@>>> s'
\end{CD} 
\end{equation*} 
be a commutative diagram of proper SNCL schemes.  
Then the pull-back morphism 
\begin{align*}  
g^*_{\rm crys} \col H^q_{\rm crys}(Y'/{\cal W}(s'))_{K_0} 
\lo H^q_{\rm crys}(Y/{\cal W}(s))_{K_0} 
\end{align*} 
is strictly compatible with $P$'s. 
\end{theo}

%\begin{rema}\label{rema:tp} 
%(1) Let the notations be as in (\ref{theo:imp}). 
%We do not know whether 
%the morphism 
%\begin{align*}  
%g^*_{\rm crys} \col H^q_{\rm crys}(Y'/{\cal W}(s'))
%\lo H^q_{\rm crys}(Y/{\cal W}(s)) 
%\end{align*} 
%is strictly compatible with $P$'s in general. 
%It seems to us that it is a very interesting problem.  
%\end{rema} 

\section{Quasi-$F$-split log schemes}\label{sec:qfs}
In this section we give a generalization 
of the definition of quasi-$F$-split varieties 
due to the second named author (\cite{y}). 
We give two types of log Kodaira vanishing theorems  
for a quasi-$F$-split projective log smooth scheme. 
These are generalizations of Mehta and Ramanathan's vanishing theorems  
for $F$-split varieties in \cite{mr}.  
One of our log Kodaira vanishing theorems for the log scheme is a generalization of 
the Kodaira vanishing theorem in \cite{y}; 
the other of them 
is much stronger than the log Kodaira vanishing theorem in \S\ref{sec:app} for the log scheme.  
The proof of our log Kodaira vanishing theorems are harder than those of 
Mehta and Ramanathan's vanishing theorems.   
In this section we also give a generalization of the lifting theorem of  
quasi-$F$-split varieties in \cite{y}. 
%Let $s$ be the log point of 
%a perfect field $\kap$ of characteristic $p>0$. 
%Let us consider the following diagram 
%\begin{equation*} \input{Complc.tex}

%\begin{CD} 
%{\cal W}_n({\cal O}_X)@>>>
%\end{CD} 
%\end{equation*} 
The following definition (1) (resp.~(2)) is  a relative version of 
the definition due to the second named author of this article 
(resp.~Mehta and Ramanathan). 

\begin{defi}\label{defi:yt0} 
Let $Y\lo T_0$ be a morphism of schemes of characteristic $p>0$. 
Let $F_{T_0}\col T_0\lo T_0$ be the 
$p$-th power Frobenius endomorphism of $T_0$. 
Set $Y':=Y\times_{T_0,F_{T_0}}T_0$. 
%Assume that $Y$ is reduced. 
\par 
(1) {\bf (cf.~\cite[(4.1)]{y})} 
Let $F:=F^*_n\col {\cal W}_n({\cal O}_{Y'})\lo 
F_{n*}({\cal W}_n({\cal O}_Y))=F_*({\cal W}_n({\cal O}_Y))$ be 
the pull-back of the relative Frobenius morphism 
$F_n\col {\cal W}_n(Y)\lo {\cal W}_n(Y')$ of ${\cal W}_n(Y)/T$. 
(This is a morphism of ${\cal W}_n({\cal O}_{Y'})$-modules.) 
Let $n_0$ be the minimum of positive integers $n$'s such that 
there exists a morphism 
$\rho \col F_*({\cal W}_n({\cal O}_Y))\lo {\cal O}_{Y'}$ of 
${\cal W}_n({\cal O}_{Y'})$-modules such that 
$\rho \circ F^*\col {\cal W}_n({\cal O}_{Y'})\lo {\cal O}_{Y'}$ 
is the natural projection. In this case we say that $Y$ is 
{\it quasi-}$F${\it -split}. If $Y/T_0=\os{\circ}{X}/\os{\circ}{S}_0$ 
for a relative log scheme $X/S_0$, 
then we say that $X$ is {\it quasi-}$F${\it -split} by abuse of terminology. 
(If there does not exist $n$, then set $n_0:=\infty$.) 
Note that $Y'=\os{\circ}{`X}$ (not necessarily equal to $\os{\circ}{X}{}'$) in this case. 
We call $n_0$ the {\it quasi-$F$-split height}
%{\it Frobenius-splitting height} 
and denote it by $h_F(Y/T_0)$. 
%(In \cite{y} $h_F(Y/T_0)$ in the absolute case 
%$T_0={\rm Spec}(\kap)$ with the replacement of 
%the relative Frobenius morphism 
%by the absolute Frobenius endomoprhism 
%has been called the quasi-Frobenius splitting height 
%and denoted by ${\rm sht}(Y)$.)
If $Y/T_0=\os{\circ}{X}/\os{\circ}{S}_0$ 
for a relative log scheme $X/S_0$ of characteristic $p>0$, 
then we denote $h_F(Y/T_0)$ by $h_F(X/S_0)$ by abuse of notation. 
\par 
(2) {\bf (cf.~\cite[Definition 2]{mr})} 
If $n_0=1$ in (1), then we say that $Y/T_0$ is $F$-{\it split}.
If $Y/T_0=\os{\circ}{X}/\os{\circ}{S}_0$ for a relative log scheme 
$X/S_0$, 
we say that $X/S_0$ is $F${\it -split} by abuse of terminology. 
\end{defi}

\begin{rema}\label{rema:absh}
(1) Assume that $T_0$ is perfect. 
Let $F\col {\cal W}_n({\cal O}_{Y})\lo F_*({\cal W}_n({\cal O}_Y))$ be 
the pull-back of the absolute Frobenius endomorphism.
Because $Y'\os{\simeq}{\lo} Y$, 
$h_F(Y/T_0)$ is equal to the minimum of positive integers $n$'s such that 
there exists a morphism $\rho \col F_*({\cal W}_n({\cal O}_Y))\lo {\cal O}_Y$ 
of ${\cal W}_n({\cal O}_{Y})$-modules
such that 
$\rho \circ F\col {\cal W}_n({\cal O}_Y)\lo {\cal O}_{Y}$ 
is the natural projection. This is the original definition in \cite[(4.1)]{y} 
in the case $T_0={\rm Spec}(\kap)$. 
\par 
(2) Let the notations be as in (\ref{defi:yt0}). 
If there exists a morphism 
$\rho \col F_{n*}({\cal W}_n({\cal O}_Y))\lo {\cal O}_{Y'}$ for $n\in {\mab Z}_{\geq 1}$ 
such that $\rho \circ F^*_n\col {\cal W}_n({\cal O}_{Y'})\lo {\cal O}_{Y'}$ 
is the natural projection, then, for all $m\geq n$, 
there exists a morphism 
$\rho' \col F_{m*}({\cal W}_m({\cal O}_Y))\lo {\cal O}_{Y'}$ 
such that 
$\rho' \circ F^*_m\col {\cal W}_m({\cal O}_Y)\lo {\cal O}_{Y'}$ 
is the natural projection. Indeed, we have only to set 
$\rho':=\rho \circ R^{n-m}$, where $R\col 
F_{l*}({\cal W}_l({\cal O}_Y))\lo F_{l-1*}({\cal W}_{l-1}({\cal O}_Y))$ 
$(m+1\leq l \leq n)$ is the projection. 
\end{rema}

The following easy lemma is necessary 
for the theorem (\ref{theo:vsss}) below. 

\begin{lemm}\label{lemm:nle} 
Let $q$ be a nonnegative integer. 
Let $X/S_0$ be as in {\rm (\ref{defi:yt0})}. 
Assume that $\os{\circ}{S}_0$ is perfect 
and that $X$ is a log smooth scheme of Cartier type over $S_0$.  
Let ${\cal M}$ be an invertible ${\cal O}_X$-module. 
Let $i$ be a positive integer and let $q$ be a nonnegative integer. 
Let $e_0$ be a nonnegative integer. 
Let $g\col \os{\circ}{X}\lo Y$ be a morphism of schemes over $\os{\circ}{S}_0$. 
%Then the following hold$:$
%\par
%$(1)$ 
Assume that the Frobenius endomorphism $F_Y\col Y\lo Y$ of $Y$ is finite. 
If 
$R^qg_*(B_1\Om^i_{X/S_0}\otimes_{{\cal O}_X}{\cal M}^{\otimes p^e})=0$ 
for $\forall e\geq e_0$, 
then $R^qg_*(B_n\Om^i_{X/S_0}\otimes_{{\cal O}_X}{\cal M}^{\otimes p^e})=0$ 
for $\forall e\geq e_0$ and $\forall n\geq 1$. 
%\par 
%$(2)$ Assume that 
%$R^qf_*(Z_1\Om^i_{X/S_0}\otimes_{{\cal O}_X}{\cal M}^{\otimes p^e})=0$ 
%for $\forall e\geq e_0$. 
%Then $R^qf_*(Z_n\Om^i_{X/S_0}\otimes_{{\cal O}_X}{\cal M}^{\otimes p^e})=0$ 
%for $\forall e\geq e_0$ and $\forall n\geq 1$.  
\end{lemm} 
\begin{proof} 
Let $F\col X\lo X$ be the $p$-th power Frobenius endomorphism. 
Consider the following exact sequence of ${\cal O}_X$-modules: 
\begin{align*} 
0\lo F_*(B_{n-1}\Om^i_{X/S_0})\lo 
B_n\Om^i_{X/S_0}\os{C^{n-1}}{\lo} B_1\Om^i_{X/S_0}\lo 0 
\quad (n\geq 1). 
\tag{7.3.1}\label{ali:blcb}
\end{align*} 
%(1): 
By the projection formula and noting that $\os{\circ}{F}$ and $F_Y$ are finite morphisms, 
%and $\os{\circ}{F}_{S_0}$ are finite morphisms  
%(note that $\os{\circ}{F}_{S_0}$ is an isomorphism), 
we have the following formula for 
a quasi-coherent ${\cal O}_X$-module ${\cal F}$ and 
an invertible ${\cal O}_X$-module ${\cal N}$: 
\begin{align*} 
R^qg_*(F_*({\cal F})\otimes_{{\cal O}_X}{\cal N})
&=R^qg_*(F_*({\cal F}\otimes_{{\cal O}_X}F^*({\cal N})))
=R^qg_*(F_*({\cal F}\otimes_{{\cal O}_X}{\cal N}^{\otimes p})) 
\tag{7.3.2}\label{ali:booxm} \\
& =F_{Y*}R^qg_*({\cal F}\otimes_{{\cal O}_X}{\cal N}^{\otimes p}). 
\end{align*}  
Hence we have the following exact sequence 
\begin{align*} 
F_{Y*}R^qg_*(B_{n-1}\Om^i_{X/S_0}\otimes_{{\cal O}_X}{\cal M}^{p^{e+1}})
\lo 
R^qg_*(B_n\Om^i_{X/S_0}\otimes_{{\cal O}_X}{\cal M}^{p^{e}})
\lo 
R^qg_*(B_1\Om^i_{X/S_0}\otimes_{{\cal O}_X}{\cal M}^{p^{e}}).
\end{align*}
Induction on $n$ tells us that  
$R^qg_*(B_n\Om^i_{X/S_0}\otimes_{{\cal O}_X}{\cal M}^{p^{e}})
=0$ for $\forall e\geq e_0$ and $\forall n\geq 1$.  
%\par 
%(2): The proof is the same as that of (1) if one uses the following exact sequence 
%\begin{align*} 
%0\lo F_*(Z_{n-1}\Om^i_{X/S_0})\lo 
%Z_n\Om^i_{X/S_0}\os{C^{n-1}}{\lo} Z_1\Om^i_{X/S_0}\lo 0 
%\quad (n\geq 1). 
%\tag{7.3.3}\label{ali:blczsb}
%\end{align*} 
%instead of (\ref{ali:blcb}). 
\end{proof} 

\par 
Next we construct key exact sequences as in the proof in \cite[(3.1)]{yoh}. 
\par 
Let the notations be as in (\ref{lemm:nle}). 
Assume that $F_Y$ is finite. 
(We do not assume that there exists a nonnegative integer $e_0$ such that 
$R^qg_*(B_1\Om^i_{X/S_0}\otimes_{{\cal O}_X}{\cal M}^{\otimes p^e})=0$ 
for $\forall e\geq e_0$.)
Push out the exact sequence (\ref{ali:wnox}) by the morphism 
$R^{n-1}\col {\cal W}_n({\cal O}_X)\lo {\cal O}_X$. 
Then we have the following exact sequence 
\begin{align*} 
0\lo {\cal O}_X \lo {\cal E}_n \lo B_n\Om^1_{X/S_{0}}\lo 0,
\tag{7.3.3}\label{ali:wrnyex} 
\end{align*}
where 
${\cal E}_n
:={\cal O}_X\oplus_{{\cal W}_n({\cal O}_X),F}F_*({\cal W}_n({\cal O}_X))$. 
Consider the following diagram 
\begin{equation*} 
\begin{CD} 
@. @. @.0\\
@. @. @. @VVV \\
@.0@. @.F_*(B_{n-1}\Om^1_{X/S_0})@. @.\\
@. @VVV @. @VVV \\
0@>>>{\cal O}_{X} @>>> {\cal E}_n@>>> B_n\Om^1_{X/S_0}@>>> 0\\
@. @| @VVV @VV{C^{n-1}}V \\
0@>>> {\cal O}_{X}
@>>> {\cal E}_1=F_*({\cal O}_X) 
@>{d}>>B_1\Om^1_{X/S_0}@>>> 0\\
@. @VVV @VVV @VVV \\
@.0@. 0@. 0
\end{CD}
\tag{7.3.4}\label{cd:dnbx}
\end{equation*} 
of ${\cal O}_X$-modules with exact rows and exact columns. 
The snake lemma tells us that ${\rm Ker}({\cal E}_n\lo {\cal E}_1)
=F_*(B_{n-1}\Om^1_{X/S_0})$. 
Hence we have the following exact sequence 
\begin{align*} 
0\lo F_*(B_{n-1}\Om^1_{X/S_0})
\lo {\cal E}_n\lo F_*({\cal O}_X)\lo 0.
\tag{7.3.5}\label{ali:obox}  
\end{align*}

\begin{defi}\label{defi:dye}
We call the exact sequence (\ref{ali:wrnyex}) (resp.~(\ref{ali:obox})) 
of ${\cal O}_X$-modules the 
{\it fundamental exact sequence of Type I} 
(resp.~{\it fundamental exact sequence of Type II}) of $X/S_0$. 
 (It maybe better to call  (\ref{ali:wrnyex}) the {\it modified log Serre exact sequence}.) 
\end{defi} 

Let ${\cal M}$ be an invertible ${\cal O}_X$-module.  
Then we have the following exact sequences: 
\begin{align*} 
0\lo {\cal M}\lo {\cal E}_n\otimes_{{\cal O}_X}{\cal M} \lo 
B_n\Om^1_{X/S_0}\otimes_{{\cal O}_X}{\cal M}\lo 0,  
\tag{7.4.1}\label{ali:obmx}  
\end{align*}
\begin{align*} 
0\lo F_*(B_{n-1}\Om^1_{X/S_0})\otimes_{{\cal O}_X}{\cal M}
\lo {\cal E}_n\otimes_{{\cal O}_X}{\cal M} \lo F_*({\cal O}_X)\otimes_{{\cal O}_X}{\cal M}\lo 0 
\tag{7.4.2}\label{ali:obamx}  
\end{align*}
and 
\begin{align*} 
0\lo {\cal M}\lo F_*({\cal O}_X)\otimes_{{\cal O}_X}{\cal M} \lo 
B_1\Om^1_{X/S_0}\otimes_{{\cal O}_X}{\cal M}\lo 0.
\tag{7.4.3}\label{ali:obmmx}  
\end{align*}
By (\ref{ali:obamx}), (\ref{ali:obmmx}) and (\ref{ali:booxm}), 
we have the following exact sequences:  
\begin{align*} 
\cdots &\lo F_{Y*}R^{q-1}g_*({\cal M}^{\otimes p})\lo 
F_{Y*}R^qg_*(B_{n-1}\Om^1_{X/S_0}
\otimes_{{\cal O}_{X}}{\cal M}^{\otimes p})\lo 
R^qg_*({\cal E}_n\otimes_{{\cal O}_{X}}{\cal M})
\tag{7.4.4}\label{ali:dnebx}\\
&\lo F_{Y*}R^qg_*({\cal M}^{\otimes p})\lo \cdots  \quad (q\in {\mab N})
\end{align*}
and  
\begin{align*} 
\cdots &\lo R^qg_*({\cal M})\lo
F_{Y*}R^qg_*({\cal M}^{\otimes p})
\lo R^qg_*(B_1\Om^1_{X/S_0}
\otimes_{{\cal O}_{X}}{\cal M})\tag{7.4.5}\label{ali:dnepbx}\\
&\lo R^{q+1}g_*({\cal M})\lo \cdots \quad (q\in {\mab N}). 
\end{align*}
\par 
Now set $h:=h_F(X/S_0)$ and assume that $h<\infty$ 
Then we have the following decomposition by (\ref{rema:absh}) (2) and (\ref{ali:obmx}): 
\begin{align*} 
R^qg_*({\cal E}_n\otimes_{{\cal O}_{X}}{\cal M})
=R^qg_*({\cal M}) \oplus R^qg_*(B_n\Om^1_{X/S_0}\otimes_{{\cal O}_{X}}{\cal M}) 
\quad 
(n\geq h, q\in {\mab N}). 
\tag{7.4.6}\label{ali:dnexbx}
\end{align*}

%\begin{lemm}\label{lemm:zos}
%Assume that $R^qf_*({\cal M})=0=R^qf_*({\cal M}^{\otimes p})$.  
%and $R^{q-1}f_*(B_1\Om^{i+1}_{X/S_0}\otimes_{{\cal O}_X}{\cal M})=0$. 
%Then $R^qf_*(Z_1\Om^{i}_{X/S_0}\otimes_{{\cal O}_X}{\cal M})=0$.
%\end{lemm}
%\begin{proof} 
%Consider the following exact sequence: 
%\begin{align*} 
%0\lo Z_1\Om^{i}_{X/S_0}\lo F_*(\Om^i_{X/S_0})\lo B_1\Om^{i+1}_{X/S_0}\lo 0. 
%\end{align*} 
%\end{proof} 

The following lemma is a key one for (\ref{theo:vsss}) and (\ref{theo:lkdv}) below: 
(1) (resp.~(2)) in this lemma is necessary for the proof of (\ref{theo:vsss})
(resp.~(\ref{theo:lkdv})).

\begin{lemm}\label{lemm:vp}
Let the notations be as in {\rm (\ref{lemm:nle})}. 
Assume also that $h_F(X/S_0)<\infty$. 
Let ${\cal L}$ be an invertible ${\cal O}_X$-module. 
Let $e_0$ be a fixed positive integer.  
Then the following hold$:$
\par 
$(1)$ Let $q_0$ be a fixed nonnegative integer. 
If $R^{q}g_*({\cal L}^{\otimes p^e})=0$ 
and 
$R^{q}g_*(B_1\Om^1_{X/S_0}
\otimes_{{\cal O}_{X}}{\cal L}^{\otimes p^e})=0$ 
for $\forall e\geq e_0$ and $\forall q\geq q_0$,   
then $R^{q}g_*({\cal L})=0$ and 
$R^{q}g_*(B_n\Om^1_{X/S_0}\otimes_{{\cal O}_{X}}{\cal L})=0$ 
for $\forall n\geq 1$ and $\forall q\geq q_0$.   
\par 
$(2)$ Let $q$ be a fixed nonnegative integer. 
If, for $\forall e\geq e_0$, $R^qg_*({\cal L}^{\otimes p^e})=0$ and 
if,  for $\forall e\geq e_0$,  there exists an integer $n(e)\geq h-1$ such that  
$R^qg_*(B_{n(e)}\Om^1_{X/S_0}
\otimes_{{\cal O}_{X}}{\cal L}^{\otimes p^e})=0$, 
then  
$R^qg_*({\cal L})=0$ and $R^qg_*(B_{n(e)+e}\Om^1_{X/S_0}\otimes_{{\cal O}_{X}}{\cal L})=0$.   
\end{lemm}
\begin{proof}
(1): (Though the statement of (1) is different from \cite[(4.1), (4.2)]{yoh} 
and the following proof of (1) is a simplification of 
of [loc.~cit.], the following proof is essentially the same as that of [loc.~cit.]: 
the simplification is to focus on the vanishing of  
$R^{q}g_*(B_1\Om^1_{X/S_0}
\otimes_{{\cal O}_{X}}{\cal L}^{\otimes p^e})$ and not to consider 
the vanishing of  
$R^{q}g_*(B_l\Om^1_{X/S_0}
\otimes_{{\cal O}_{X}}{\cal L}^{\otimes p^e})=0$ 
for other $l$'s as an assumption; this focus is possible by (\ref{lemm:nle}).)  
\par 
Set $h:=h_F(X/S_0)$.  
For a fixed positive integer $e_1$ and for $\forall e\geq e_1$, 
consider the following two conditions:  
\begin{align*}
R^{q}g_*({\cal L}^{\otimes p^e})=0 \quad (\forall q\geq q_0)
\tag{${\rm Hyp}_1(e_1)$}
\end{align*}
and 
\begin{align*}
R^{q}g_*(B_1\Om^1_{X/S_0}
\otimes_{{\cal O}_{X}}{\cal L}^{\otimes p^{e}})=0 \quad (\forall q\geq q_0). 
\tag{${\rm Hyp}_2(e_1)$}
\end{align*}
By the assumption ${\rm Hyp}_i(e_1)$ $(i=1,2)$ 
is satisfied for the case $e_1=e_0$. 
\par 
Now assume that ${\rm Hyp}_i(e_1)$ $(i=1,2)$ holds. 
By (\ref{lemm:nle}) 
\begin{align*} 
R^qg_*(B_n\Om^1_{X/S_0}
\otimes_{{\cal O}_{X}}{\cal L}^{\otimes p^{e}})=0
\tag{7.5.1}\label{ali:bpelb}
\end{align*} 
for $\forall e\geq e_1$, $\forall n\geq 1$ and $\forall q\geq q_0$. 
In particular, $R^qg_*(B_{h-1}\Om^1_{X/S_0}
\otimes_{{\cal O}_{X}}{\cal L}^{\otimes p^{e}})=0$ 
for $\forall e\geq e_1$ and $\forall q\geq q_0$. 
By (\ref{ali:dnebx}) in the case ${\cal M}={\cal L}^{\otimes p^{e-1}}$, 
we see that 
\begin{align*} 
R^qg_*({\cal E}_h\otimes_{{\cal O}_{X}}{\cal L}^{\otimes p^{e-1}})=0
\quad (\forall e\geq e_1, \forall q\geq q_0).
\tag{7.5.2}\label{ali:bpeelb}
\end{align*}  
Hence, by (\ref{ali:dnexbx}) in the case ${\cal M}={\cal L}^{\otimes p^{e-1}}$, 
\begin{align*} 
R^qg_*(B_h\Om^1_{X/S_0}
\otimes_{{\cal O}_{X}}{\cal L}^{\otimes p^{e-1}})
=0=R^qg_*({\cal L}^{\otimes p^{e-1}}) \quad (\forall e\geq e_1, \forall q\geq q_0).
\tag{7.5.3}\label{ali:bpqlb}
\end{align*}  
By (\ref{ali:dnepbx}) for the case ${\cal M}={\cal L}^{\otimes p^{e-1}}$, 
we see that 
\begin{align*} 
R^qg_*(B_1\Om^1_{X/S_0}
\otimes_{{\cal O}_{X}}{\cal L}^{\otimes p^{e-1}})=0
\quad (\forall e\geq e_1, \forall q\geq q_0). 
\tag{7.5.4}\label{ali:bbxlb}
\end{align*}  
By (\ref{ali:bpqlb}) and (\ref{ali:bbxlb}),  
we have proved that ${\rm Hyp}_i(e_1-1)$ $(i=1,2)$ holds. 
Descending induction on $e_1$ shows (\ref{lemm:vp}). 
\par 
(2): By (\ref{ali:dnebx}) in the case ${\cal M}={\cal L}^{\otimes p^{e-1}}$ and $n=n(e)+1$,  
$R^qg_*({\cal E}_{n(e)+1}\otimes_{{\cal O}_{X}}{\cal L}^{\otimes p^{e-1}})=0$. 
Since $n(e)+1\geq h$, 
$R^qg_*({\cal L}^{\otimes p^{e-1}})=0=
R^qg_*(B_{n(e)+1}\Om^1_{X/S_0}
\otimes_{{\cal O}_{X}}{\cal L}^{\otimes p^{e-1}})=0$ by (\ref{ali:dnexbx}). 
Continuing this process, 
we see that $R^qg_*({\cal L})=0=
R^qg_*(B_{n(e)+e}\Om^1_{X/S_0}\otimes_{{\cal O}_{X}}{\cal L})$. 
\end{proof}

The following is the relative log version of \cite[Theorem 4.1]{yoh}, which 
is a nontrivial generalization of \cite[Proposition 1]{mr}.  

\begin{theo}\label{theo:vsss} 
Let the notations be as in {\rm (\ref{lemm:nle})}.  
Assume that the structural morphism $g\col \os{\circ}{X}\lo Y$ is projective. 
Let ${\cal L}$ be a relatively ample line bundle on $\os{\circ}{X}$ 
with respect to $g$.  
Assume also that $h_F(X/S_0)<\infty$.  
Then $R^{q}g_*({\cal L})=0$ and 
$R^{q}g_*(B_n\Om^1_{X/S_0}\otimes_{{\cal O}_{X}}{\cal L})=0$ 
for any $q\geq 1$ and any $n\geq 1$.   
\end{theo}
\begin{proof} 
By Serre's theorem (\cite[(2.2.1)]{ega31}), 
there exists a positive integer $m_0$ such that  
$R^{q}g_*({\cal L}^{\otimes m})=0$ 
and $R^{q}g_*(B_1\Om^1_{X/S_0}
\otimes_{{\cal O}_{X}}{\cal L}^{\otimes m})=0$
for $\forall m\geq m_0$  and $\forall q\geq 1$. 
By considering the case where $q_0=1$ and $e_0$  in (\ref{lemm:vp}) 
is a large integer, we immediately obtain (\ref{theo:vsss}). 
\end{proof} 

%The following is a new statement even in the $F$-split case. 
%\begin{coro}\label{coro:nxlm}
%Let the notations and the assumptions be as in {\rm (\ref{theo:vsss})}. 
%Then $R^{q}g_*(\Om^1_{X/S_0}\otimes_{{\cal O}_{X}}{\cal L})=0$ and 
%$R^{q}g_*(Z_n\Om^1_{X/S_0}\otimes_{{\cal O}_{X}}{\cal L})=0$ 
%for any $q\geq 1$ and any $n\geq 1$.   
%\end{coro}

\par 
The following vanishing theorem is much stronger than (\ref{coro:lkan}) 
in the case $i=0$ and $i=d$ for a projective log smooth variety 
a quasi-$F$-split height in characteristic $p>0$. 
The following (1) is also a nontrivial generalization of 
Kodaira vanishing theorem in \cite[Proposition 2]{mr}.

\begin{theo}\label{theo:lkdv}
Let the notations be as in {\rm (\ref{lemm:nle})}. 
Assume moreover that $S_0$ is equal to 
the log point $s$ of a perfect field of characteristic $p>0$.  
Assume that $\os{\circ}{X}$ is 
%projective over $\kap$ 
of pure dimension $d$.  
Assume also that $h_F(X/s)<\infty$. 
Let ${\cal L}$ be an ample invertible 
${\cal O}_X$-module. 
%Let $f\col X\lo S_0$ be the structural morphism. 
Set $h:=h_F(X/s)$.
Then the following hold$:$ 
\par 
$(1)$ $H^q(X,{\cal L}^{\otimes (-1)})=0$ for $\forall q<d$. 
%\par 
%$(1)'$ $\forall j<d$, $\exists n\geq h$. 
%$H^q(X,B_n\Om^1_{X/s}
%\otimes_{{\cal O}_{X}}{\cal L}^{\otimes (-1)})=0$.
\par 
$(2)$ $H^q(X,{\cal I}_{X/s}\Om^d_{X/s}\otimes_{{\cal O}_X}{\cal L})=0$ for $\forall q>0$. 
%\par 
%$(2)'$ $\forall j>0$, $\exists n\geq h$, 
%$H^q(X,{\cal H}{\it om}_{{\cal O}_X}(B_n\Om^1_{X/s},
%{\cal I}_{X/s}\Om^{d}_{X/s})
%\otimes_{{\cal O}_{X}}{\cal L})=0$. 
\end{theo}
\begin{proof} 
By (\ref{lemm:locf}) and the log Serre duality of Tsuji, 
we have only to prove (1).
\par 
Let $m$ be a positive integer. 
By (\ref{theo:ico}), 
$$H^q(X,{\cal L}^{\otimes (-m)})~{\rm and}~  
H^q(X,B_1\Om^1_{X/s}
\otimes_{{\cal O}_{X}}{\cal L}^{\otimes (-m)})$$
are the duals of 
$$H^{d-q}(X,{\cal I}_{X/s}\Om^{d}_{X/s}
\otimes_{{\cal O}_X}{\cal L}^{\otimes m})
~{\rm and}~H^{d-q}(X,{\cal H}{\it om}_{{\cal O}_X}(B_1\Om^1_{X/s},
{\cal I}_{X/s}\Om^{d}_{X/s})
\otimes_{{\cal O}_{X}}{\cal L}^{\otimes m}),$$
respectively. 
By Serre's theorem (\cite[(2.2.1)]{ega31}), 
there exists a positive integer $m_0$ such that, 
for $\forall m\geq m_0$ and  $\forall q<d$   
the latter cohomologies vanish. 
Hence there exists a positive integer $e_0$ such that, 
for $\forall e\geq e_0$ and $\forall q<d$,    
\begin{align*} 
H^q(X,{\cal L}^{\otimes(-p^e)})=0
\tag{7.7.1}\label{ali:xlm}
\end{align*}  
and 
\begin{align*} 
H^q(X,B_1\Om^1_{X/s}
\otimes_{{\cal O}_{X}}{\cal L}^{\otimes (-p^e)})=0.
\tag{7.7.2}\label{ali:xblam} 
\end{align*} 
By (\ref{lemm:nle})
\begin{align*} 
H^q(X,B_l\Om^1_{X/s}
\otimes_{{\cal O}_{X}}{\cal L}^{\otimes (-p^e)})=0 \quad (\forall l\geq 1).
\tag{7.7.3}\label{ali:xblm} 
\end{align*} 
By (\ref{lemm:vp}) (2), 
$H^q(X,{\cal L}^{\otimes(-1)})=0$ and 
$H^q(X,B_l\Om^1_{X/s}
\otimes_{{\cal O}_{X}}{\cal L}^{\otimes (-1)})=0$ for $\forall l\geq h-1+e_0$. 
%\par 
%Let $f\col X\lo s$ be the structural morphism. 
%Fix a nonnegative integer $j<d$. 
%Consider the following conditions:  
%\begin{align*}
%R^qg_*({\cal L}^{\otimes (-p^e)})=0 \quad 
%\tag{${\rm Hyp}^*_1(e)$} 
%\end{align*}
%and 
%\begin{align*}
%R^qg_*(B_n\Om^1_{X/S_0}
%\otimes_{{\cal O}_{X}}{\cal L}^{\otimes (-p^e)})=0. 
%\tag{${\rm Hyp}^*(e,n)$}
%\end{align*}  
%(To consider the ${\rm Hyp}^*(e,n)$ is the point of this proof.) 
%By (\ref{ali:xlm}) and (\ref{ali:xblm}),  
%${\rm Hyp}^*(e)$ and ${\rm Hyp}^*(e,n)$ holds for a fixed $e>>0$. 
%\par 
%Now assume that ${\rm Hyp}^*(e)$ and ${\rm Hyp}^*(e,n-1)$ hold. 
%Then, by (\ref{ali:dnebx}) in the case ${\cal M}={\cal L}^{\otimes (-p^{e-1})}$, 
%$R^qg_*({\cal E}_n\otimes_{{\cal O}_{X}}{\cal L}^{\otimes (-p^{e-1})})=0$. 
%By (\ref{ali:dnexbx}) in the case ${\cal M}={\cal L}^{\otimes (-p^{e-1})}$,  
%we see that 
%\begin{align*}
%R^qg_*({\cal L}^{(-\otimes p^{e-1})})=0
%\tag{7.7.3}\label{ali:kjnhbx}
%\end{align*} 
%and 
%\begin{align*}
%R^qg_*(B_n\Om^1_{X/S_0}\otimes_{{\cal O}_{X}}{\cal L}^{(-\otimes p^{e-1})}) 
%=0. 
%\tag{7.7.4}\label{ali:dnjebx}
%\end{align*} 
%Hence ${\rm Hyp}^*(e-1)$ and ${\rm Hyp}^*(e-1,n)$ hold. 
%Replacing $n$ by $n+1$, we see that 
%${\rm Hyp}^*(e-2)$ and ${\rm Hyp}^*(e-2,n+1)$ hold.
%The descending induction on $e$ tells us that 
%$H^q(X,{\cal L}^{-1})=\Gam(s,R^qf_*({\cal L}^{-1}))=0$. 
\end{proof}

The following problem seems very interesting (cf.~\cite[Conjecture 1.1]{mus}): 

\begin{prob}\label{prob:wfs}
Let $Y$ be a projective log smooth integral scheme over a fine log scheme $s$ 
whose underlying scheme is a field $K$ of characteristic zero. 
Assume that $\os{\circ}{Y}$ is of pure dimension $d$. 
Set 
$p_g(Y,r):=
{\rm dim}_KH^0(Y,\os{r}{\us{{\cal O}_Y}{\otimes}}\Om^d_Y)$ and 
$\kap(Y/s):=\us{r\lo \infty}{\varlimsup}\dfrac{\log p_g(Y,r)}{{\log r}}$.
Assume that $\kap(Y/s)\leq 0$. 
Let $S$ be a fine log scheme whose underlying scheme is the spectrum of 
an algebra of finite type over ${\mab Z}$ and let  
$s\lo S$ be a morphism of fine log schemes. 
Let ${\cal Y}$ be a projective log smooth integral scheme over $S$ 
such that $Y={\cal Y}\times_Ss$. 
Then does there exist a dense set of exact closed points $T$ of $S$ 
such that $({\cal Y}_t)^{\circ}$ is quasi-$F$-split (or more strongly $F$-split) for every $t\in T$?
\end{prob}

%with canonical log structure, 
%that is, the associated log structure to the morphism
%${\mab N}\owns 1\lom p\in {\cal W}$. 
%\par

The following is the log version of a generalization of \cite[(4.4)]{y}. 

\begin{theo}\label{theo:fshl}
%Let $s$ be a fine log scheme 
%whose underlying scheme is the spectrum of $\kap$. 
Let $S_0$ be a fine log scheme of characteristic $p>0$. 
Assume that $\os{\circ}{S}_0$ is perfect. 
Let $S$ be a fine log scheme with exact closed immersion 
$S_0\os{\sus}{\lo} S$.  
Let ${\cal I}$ be the ideal sheaf of this exact closed immersion. 
Assume that ${\cal I}=\pi{\cal O}_S$ for 
a global section $\pi$ of ${\cal O}_S$ 
and that $p\pi=0$ in ${\cal O}_S$. 
Assume also that the morphism ${\cal O}_{S_0}\owns 1\lom 
\pi \in {\cal I}$ is a well-defined isomorphism. 
Assume that there exists a lift $F_{S}\col S\lo S$ of 
the Frobenius endomorphism $F_{S_{0}}\col S_{0}\lo S_{0}$ of $S_{0}$.  
Let $Y$ be a $($not necessarily proper$)$ 
log smooth integral separated scheme over $S_0$. 
Assume that $Y/S_{0}$ is of Cartier type and that 
$h_F(Y/S_{0})< \infty$. 
%Set $`Y:=Y\times_{\os{\circ}{S}_0,\os{\circ}{F}_{S_0}}\os{\circ}{S}_0$ 
%and $Y':=Y\times_{S_0,F_{S_0}}S_0$.  
%Let $F\col Y\lo `Y$ $($resp.~$F'\col Y\lo Y')$ 
%be the abrelative Frobenius morphism 
%$($resp.~the relative Frobenius morphism$)$   
%of $Y/S_{0}$ over $S_0\lo S^{[p]}_0$ $($resp.~$S_0)$. 
Let $F\col Y\lo Y$ be the absolute Frobenius endomorphism of $Y$ 
over $F_{S_{0}}$. 
Then 
%the following hold$:$
%\par 
%$(1)$ 
there exists a log smooth integral scheme ${\cal Y}$ over $S$ 
such that ${\cal Y}\times_{S}S_0=Y$. 
%\par 
%$(2)$ There exists a log smooth integral scheme ${\cal Z}$ over $S$ 
%such that ${\cal Z}\times_{S}S_0=Y'$. 
\end{theo} 
\begin{proof} 
(The following proof is the log version of the proof of \cite[(4.4)]{y}.)
%Since (2) immediately follows from (1), we have only to prove (1). 
\par 
%Because $Y\lo S_{0}$ is integral, $\os{\circ}{Y}{}'=`\os{\circ}{Y}{}$. 
For simplicity of notation, we denote the 
$p$-th power Frobenius endomorphism 
${\cal W}_n(Y)\lo {\cal W}_n(Y)$ by $F$ for any $n\geq 1$. 
Push out the exact sequence (\ref{ali:wnox})  by the morphism 
$R^{n-1}\col {\cal W}_n({\cal O}_{Y})\lo {\cal O}_{Y}$. 
Then we have the following exact sequence 
\begin{align*} 
0\lo {\cal O}_{Y} \lo {\cal E}_n \lo B_n\Om^1_{Y/S_{0}}\lo 0. 
\tag{7.9.1;$n$}\label{ali:wrnex} 
\end{align*} 
Let $e_n\in {\rm Ext}_{Y}^1(B_n\Om^1_{Y/S_{0}},{\cal O}_{Y})$ 
be the extension class of (\ref{ali:wrnex}).   
By the definition of $h_F(Y/S_{0})$ and (\ref{defi:yt0}) (1), 
the invariant $h_F(Y/S_{0})$ 
is the minimum of positive integers $n$'s such that the exact sequence 
(\ref{ali:wrnex}) is split. 
By (\ref{coro:wsta}) we have the following commutative diagram 
\begin{equation*} 
\begin{CD} 
F_*({\cal W}_n({\cal O}_{Y}))@>{d_n}>>B_n\Om^1_{Y/S_{0}}\\
@V{R^{n-1}}VV @VV{C^{n-1}}V  \\
F_*({\cal O}_{Y})@>{d}>>B_1\Om^1_{Y_0/S_{0}}.\\
\end{CD}
\tag{7.9.2}\label{cd:edebx}
\end{equation*} 
%Hence the exact sequence (\ref{ali:wrnex}) is the pull-back of 
%the exact sequence 
%(\ref{ali:wnox}) for the case $n=1$ by the morphism 
%$F_{*}(C^{n-1})\col F_{*}(B_n\Om^1_{Y/S_{0}})\lo F_{*}(B\Om^1_{Y/S_{0}})$. 
Because ${\cal E}_n
=F_*({\cal W}_n({\cal O}_Y))\oplus_{{\cal W}_n({\cal O}_{Y})}{\cal O}_{Y}$, 
we have the following commutative diagram of exact sequences by using 
(\ref{cd:edebx}):
\begin{equation*} 
\begin{CD} 
0@>>> {\cal O}_{Y} 
@>>> {\cal E}_n @>>> B_n\Om^1_{Y/S_{0}}@>>> 0\\
@. @| @VVV @VV{C^{n-1}}V  \\
0@>>> {\cal O}_{Y} @>>> F_{*}({\cal O}_{Y}) @>>> 
B_1\Om^1_{Y/S_{0}}@>>> 0.\\
\end{CD}
\tag{7.9.3}\label{cd:edfbx}
\end{equation*} 
Hence we have the following commutative diagram
\begin{equation*} 
\begin{CD} 
{\rm Ext}_{Y}^1(\Om^1_{Y/S_0},B_n\Om^1_{Y/S_{0}})
@>{{\partial}_n}>>
{\rm Ext}_{Y}^2(\Om^1_{Y/S_0},{\cal O}_{Y})\\
@V{C^{n-1}}VV @|  \\
{\rm Ext}_{Y}^1(\Om^1_{Y/S_0},B_1\Om^1_{Y/S_{0}})
@>{\partial_1}>>
{\rm Ext}_{Y}^2(\Om^1_{Y/S_0},{\cal O}_{Y}), \\
\end{CD}
\tag{7.9.4}\label{cd:debx}
\end{equation*} 
where $\partial_n$ is the boundary morphism obtained by  
the exact sequence (\ref{ali:wrnex}). 
%Let $q \col `Y\lo Y$ be the projection. 
%Because $\os{\circ}{S}_0$ is perfect, the following functor 
%\begin{align*} 
%q^*\col \{{\rm the}~{\rm category}~{\rm of}~ 
%{\cal O}_Y{\textrm -}{\rm modules}\} \lo
%\{{\rm the}~{\rm category}~{\rm of}~ 
%{\cal O}_Y{\textrm -}{\rm modules}\} 
%\end{align*} 
%gives an equivalence of categories.  
%Furthermore, by \cite[(1.8)]{klog1}, 
%we also obtain the equality 
%$\Om^1_{`Y/S^{[p]}_{0}}=q^*(\Om^1_{Y/S_{0}})$. 
%\par 
%\begin{equation*} 
%\begin{CD} 
%{\rm Ext}_{Y}^1(\Om^1_{Y/S_0},B_1\Om^1_{Y/S_{0}})
%@>{{\partial}}>>
%{\rm Ext}_{Y}^2(\Om^1_{Y/S},{\cal O}_{Y})\\
%@V{q^*}VV @VV{q^*}V  \\
%{\rm Ext}_{`Y}^1(\Om^1_{`Y/S^{[p]}_0},F^{\rm abr}_*(B\Om^1_{Y/S_{0}}))
%@>{\partial}>>
%{\rm Ext}_{`Y}^2(\Om^1_{`Y/S^{[p]}_0},{\cal O}_{`Y}), \\
%\end{CD}
%\tag{6.8.6}\label{cd:debbx}
%\end{equation*} 
\par 
Now assume that (\ref{ali:wrnex}) is split for a positive integer $n$. 
(Since $h_F(Y/S_{0})<\infty$, the $n$ exists.)
Because the sequence (\ref{ali:wrnex}) is split, 
$\partial_n$ is the zero morphism. Because ${\rm obs}_{Y/(S_0\subset S)}
=\partial_1({\rm obs}_{(Y,F)/(S_0\subset S,F_S)})$ by (\ref{coro:nac}) (3), 
it suffices to prove that 
${\rm obs}_{(Y,F)/(S_0\subset S,F_S)}\in {\rm Im}(C^{n-1})$. 
%the natural morphism. 
%Then, by the proof of (\ref{coro:nofs}), 
%$\os{\circ}{\bet}={\rm id}_{\os{\circ}{Y}{}'}$ and 
%we have an isomorphism 
%$\Om^1_{`Y_0/S^{[p]}_{0}}\os{\sim}{\lo} \Om^1_{Y'/S_0}$. 
%Hence 
%Then we have the following commutative diagram 
%\begin{equation*} 
%\begin{CD} 
%{\rm Ext}_{Y}^1(\Om^1_{Y/S_0},B_n\Om^1_{Y/S_{0}})
%@>{\sim}>>
%{\rm Ext}_{Y}^1(\Om^1_{Y/S_0},B_n\Om^1_{Y/S_{0}})\\
%@V{C^{n-1}}VV @VV{C^{n-1}}V \\
%{\rm Ext}_{Y}^1(\Om^1_{Y/S_0},B\Om^1_{Y/S_{0}})
%@>{\sim}>>
%{\rm Ext}_{Y}^1(\Om^1_{Y/S_{0}},B\Om^1_{Y/S_{0}}), \\
%\end{CD} 
%\tag{6.8.6}\label{ali:wnbfx} 
%\end{equation*} 
%Because ${\rm obs}_{(Y,F)/(S_0\subset S,F_S)}$ 
%is mapped to ${\rm obs}_{(Y,F')/(S_0\subset S,F_S)}$ by 
%the lower horizontal morphism in the commutative diagram 
%(\ref{ali:wnbfx}) ((\ref{coro:nofs}) (2)), we have only to prove that 
%${\rm obs}_{(Y,F')/(S_0\subset S,F_S)}\in {\rm Im}(C^{n-1})$.
Because ${\rm obs}_{(Y,F)/(S_0\subset S,F_S)}$ is equal to 
the extension class of the following exact sequence 
\begin{align*} 
0\lo B_1\Om^1_{Y/S_{0}} \lo 
Z_1\Om^1_{Y/S_{0}}
\os{C}{\lo} \Om^1_{Y/S_0}\lo 0
\tag{7.9.5}\label{ali:wnbzox} 
\end{align*} 
by (\ref{theo:uaryn}) 
and because we have the following commutative diagram 
\begin{equation*} 
\begin{CD} 
0@>>> B_n\Om^1_{Y/S_{0}} @>>> Z_n\Om^1_{Y/S_{0}} 
@>{C^n}>> \Om^1_{Y/S_{0}}@>>> 0\\
@. @V{C^{n-1}}VV @V{C^{n-1}}VV @|  \\
0@>>> B_1\Om^1_{Y/S_{0}} 
@>>> Z_1\Om^1_{Y/S_{0}} @>{C}>> 
\Om^1_{Y/S_{0}}@>>> 0,\\
\end{CD}
\tag{7.9.6}\label{cd:edgbx}
\end{equation*} 
we see that ${\rm obs}_{(Y,F)/(S_0\subset S)}\in {\rm Im}(C^{n-1})$. 
\par 
We complete the proof. 
\end{proof}  

The following is one of results what we want to obtain: 

\begin{coro}\label{coro:dhg} 
The conclusions of $(\ref{coro:cfd})$ $(1)$ and $(2)$ hold for $Y/S_0$. 
\end{coro}

\section{Lifts of certain log schemes over ${\cal W}_2$}\label{sec:lif}
In this section we give the log version of the main result in \cite{y}. 
\par
The following is the log version of a generalization of \cite[(4.5)]{y}. 

\begin{theo}\label{theo:cyh}
Let $s$ be as in {\rm (\ref{coro:e1dg})}.  
%Assume that $M_s$ is saturated.  
Let $X$ be a proper log smooth, integral
and saturated log scheme over $s$ of 
pure dimension $d$. Assume that $X/s$ is of Cartier type and of vertical type. 
%Let $F\col X\lo X$ be the Frobenius endomorphism of $X$. 
%${\mab N}\owns 1\lom a\in \kap$ for some $a\in \kap$. 
Assume also that the following three conditions hold$:$
\par 
$({\rm a})$ $H^{d-1}(X,{\cal O}_X)=0$ if $d\geq 2$, 
\par 
$({\rm b})$ $H^{d-2}(X,{\cal O}_X)=0$ if $d\geq 3$, 
\par 
$({\rm c})$ $\Omega^d_{X/s}\simeq {\cal O}_X$. 
\parno 
Then $h_F(\os{\circ}{X}/\kap)=h^d(\os{\circ}{X}/\kap)$. 
\end{theo}
\begin{proof} 
(The following proof is the log version of the proof of \cite[(4.5)]{y}.)
Set $h=h^d(\os{\circ}{X}/\kap)$. 
Let $F\col X\lo X$ be the Frobenius endomorphism of $X$. 
Consider the following exact sequence of ${\cal O}_X$-modules:  
\begin{align*} 
0\lo F_{*}(B_{n-1}\Om^1_{X/s}) \lo B_n\Om^1_{X/s} 
\os{C^{n-1}}{\lo} B_1\Om^1_{X/s}\lo 0. 
\tag{8.1.1}\label{ali:wnaex} 
\end{align*} 
Here note that the direct image $F_{*}$ is necessary for 
$B_{n-1}\Om^1_{X/s}$ as in [loc.~cit.]. 
Taking ${\rm Ext}_X^{\bul}(*,{\cal O}_X)$ of 
the exact sequence (\ref{ali:wnaex}), 
we have the following exact sequence 
\begin{align*} 
&{\rm Ext}_X^1(B_1\Om^1_{X/s},{\cal O}_X)
\os{C^{{n-1}*}}{\lo} {\rm Ext}_X^1(B_n\Om^1_{X/s},{\cal O}_X) 
\lo {\rm Ext}_X^1(F_*(B_{n-1}\Om^1_{X/s}),{\cal O}_X)
\tag{8.1.2}\label{ali:wbbnex} \\
& \lo {\rm Ext}_X^2(B_1\Om^1_{X/s},{\cal O}_X). 
\end{align*} 
By Tsuji's log Serre duality ((\ref{theo:ico})), 
we have the following isomorphism 
\begin{align*} 
{\rm Ext}_X^q(B_n\Om^1_{X/s},{\cal O}_X)\simeq  
{\rm Ext}_X^q(B_n\Om^1_{X/s},\Om^d_{X/s})
\simeq H^{d-q}(X,B_n\Om^1_{X/s})^*,
\tag{8.1.3}\label{ali:wboex}
\end{align*}  
where $*$ means the dual of a finite dimensional $\kap$-vector space. 
Hence ${\rm Ext}_X^2(B_n\Om^1_{X/s},{\cal O}_X)=
H^{d-2}(X,B_n\Om^1_{X/s})^*=0$ by (\ref{coro:chob}) (3) 
and we have the following exact sequence 
\begin{align*} 
&{\rm Ext}_X^1(B_1\Om^1_{X/s},{\cal O}_X)\os{C^{{n-1}*}}{\lo} 
{\rm Ext}_X^1(B_n\Om^1_{X/s},{\cal O}_X) 
\lo {\rm Ext}_X^1(F_*(B_{n-1}\Om^1_{X/s}),{\cal O}_X)\lo 0
\tag{8.1.4}\label{ali:wbbneox} 
\end{align*} 
of $\kap$-modules. 
%Here note that the $\kap$-module structure of 
%${\rm Ext}_X^1(B_{n-1}\Om^1_{X/s},{\cal O}_X)$ 
%is defined by the inverse of the $p$-th power Frobenius endomorphism. 
By (\ref{eqn:pkdppw}) and (\ref{ali:wboex}), 
%$H^{d-1}(X,B_1\Om^1_{X/s})$ 
\begin{align*} 
\dim_{\kap}{\rm Ext}_X^1(B_1\Om^1_{X/s},{\cal O}_X)\leq 1
\tag{8.1.5}\label{ali:wbx1eox} 
\end{align*}  
and 
$$\dim_{\kap}{\rm Ext}_X^1(B_n\Om^1_{X/s},{\cal O}_X)=
{\rm min}\{n,h-1\}.$$ 
Since $F\col \os{\circ}{X}\lo \os{\circ}{X}$ is a finite morphism, 
we also have the following isomorphism 
\begin{align*} 
{\rm Ext}_X^q(F_*(B_{n-1}\Om^1_{X/s}),{\cal O}_X)&\simeq  
{\rm Ext}_X^q(F_*(B_{n-1}\Om^1_{X/s}),\Om^d_{X/s})
=H^{d-q}(X,F_*(B_{n-1}\Om^1_{X/s}))^*\\
&\simeq H^{d-q}(X,B_{n-1}\Om^1_{X/s})^*. 
%={\rm Ext}_X^q(B_{n-1}\Om^1_{X/s},{\cal O}_X)
\end{align*}  
Hence 
$$\dim_{\kap}{\rm Ext}_X^q(F_*(B_{n-1}\Om^1_{X/s}),{\cal O}_X)
={\rm min}\{n-1,h-1\}.$$ 
\par  
First consider the case $e_1=0$. 
Then the following exact sequence 
\begin{align*}
0\lo {\cal O}_X\os{F}{\lo} F_*({\cal O}_X)\lo B_1\Om^1_{X/s}\lo 0
\end{align*} 
is split. 
Hence 
$H^q(X,F_*({\cal O}_X))=H^q(X,{\cal O}_X)
\oplus H^q(X,B_1\Om^1_{X/s})$ 
$(q\in {\mab N})$. 
%(cf.~(\ref{ali:cifes})). 
Since $\os{\circ}{F}$ is finite, 
$H^q(X,F_*({\cal O}_X))=H^q(X,{\cal O}_X)$. 
Hence  $H^q(X,B_1\Om^1_{X/s})=0$. 
(We can find this argument in \cite[(2.4.1)]{jr} 
in the trivial log case.)
%the exact sequences (8.8.1.1) and 
%hence (\ref{ali:orobrbs}) are split. 
%Since (\ref{ali:orobrbs}) is split, the morphism 
%${\rm Hom}_{{\cal O}_X}(F_*({\cal O}_X),{\cal O}_X)\lo 
%{\rm Hom}_{{\cal O}_X}({\cal O}_X,{\cal O}_X)$ 
%is surjective and the morphism 
%$F\col H^1(X,{\cal O}_X)\lo H^1(X,{\cal O}_X)$ is injective. 
%Hence ${\rm Ext}_X^1(B_1\Om^1_{X/s},{\cal O}_X)=0$ by (\ref{ali:oo1x}). 
In particular, $H^{d-1}(X,B_1\Om^1_{X/s})=0$.  
By (\ref{eqn:pkdppw}) we see that $h=1$. 
\par 
Next consider the case $e_1\not=0$. 
Then ${\rm Ext}_X^1(B_1\Om^1_X,{\cal O}_X)=\kap e_1$ 
by (\ref{ali:wbx1eox}). 
Because $C^{{n-1}*}(e_1)=e_n$ by (\ref{cd:edfbx}), 
we see that $e_n=0$ if and only if 
the morphism 
${\rm Ext}_X^1(B_n\Om^1_{X/s},{\cal O}_X) 
\lo {\rm Ext}_X^1(B_{n-1}\Om^1_{X/s},{\cal O}_X)$ is an isomorphism by 
(\ref{ali:wbbneox}). 
Hence $e_n=0$ if and only if ${\rm min}\{n-1,h-1\}={\rm min}\{n,h-1\}$. 
This is equivalent to $h\leq n$. 
\end{proof}

\begin{coro}\label{coro:cb}
{\rm (\ref{theo:ny1})} holds. 
\end{coro}
\begin{proof} 
(\ref{coro:cb}) immediately follows from 
(\ref{theo:fshl}) and (\ref{theo:cyh}). 
\end{proof} 

%\begin{rema}\label{rema:iff}
%By \cite[(4.4), (4.5)]{klog1}, 
%$\os{\circ}{\cal X}$ in (\ref{theo:ny1}) is flat over ${\cal W}_2$. 
%\end{rema}

By using the degeneration at $E_1$ of the log Hodge spectral sequence 
due to Kato (\cite[(4.12) (1)]{klog1}) or (\ref{theo:kdg}) and (\ref{coro:cfd}), 
we obtain the following: 

\begin{coro}\label{coro:nip}
Let the assumptions be as in {\rm (\ref{theo:cyh})}. 
Then {\rm (\ref{theo:np})} holds. 
%Set $H^q_{\rm dR}(X/s):=H^q(X,\Om^{\bul}_{X/s})$ $(q\in {\mab N})$. 
%Assume that $\dim \os{\circ}{X}< p$. 
%Then the following spectral sequence 
%\begin{align*} 
%E_1^{ij}=H^j(X,\Om^i_{X/s})\Lo H^{i+j}_{\rm dR}(X/s)
%\tag{6.8.1}\label{ali:oxs} 
%\end{align*}  
%degenerates at $E_1$. 
\end{coro}

\begin{coro}\label{coro:indm}
$(1)$ {\rm (\ref{theo:stk})} holds. 
\par 
$(2)$ Let the notations and the assumptions be as in {\rm (\ref{theo:ray})}. 
Then $H^j(Y,{\cal L}^{-1})=0$ for $j<d$. 
\end{coro}
\begin{proof}
This follows from (\ref{theo:lkdv}) and (\ref{theo:cyh}). 
\end{proof}

The following is of independent interest: 
\begin{coro}\label{coro:no} 
Let the notations and the assumptions be as in 
{\rm (\ref{theo:ny1})}. 
%Assume that there exists a morphism $F_S\col S\lo S$ 
%which is a lift of the Frobenius endomorphism of $s$. 
Assume that $h^d(\os{\circ}{X}/\kap)\geq 2$. Then there does not exist a lift 
$\wt{F}\col {\cal X}\lo {\cal X}$ over the Frobenius endomorphism 
$F_{{\cal W}_2(s)}\col {\cal W}_2(s)\lo {\cal W}_2(s)$ 
which is a lift of the Frobenius endomorphism of $X$. 
\end{coro}
\begin{proof} 
If there exists the $\wt{F}$ in (\ref{coro:no}), then 
\cite[(3.2)]{nil} and \cite[(8.6), (8.8)]{idf} (cf.~\cite[(4.3)]{cl}) 
tell us that $X/s$ is log ordinary, 
that is, $H^q(X,B\Om^i_{X/s})=0$ for all $q$'s and $i$'s. 
By (\ref{eqn:pkdefpw}) in the case $n=1$,  $h^d(\os{\circ}{X}/\kap)=1$. 
This contradicts the assumption 
$h^d(\os{\circ}{X}/\kap)\geq 2$.  
\end{proof}

\begin{exem}
Let $X/s$ be a log $K3$ surface of type II (\cite[\S3]{nlk3}). 
By \cite[Theorem 1]{rs}, we have the following spectral sequence 
\begin{align*} 
E^{ij}=H^j(\os{\circ}{X}{}^{(i)},{\cal W}({\cal O}_{\os{\circ}{X}{}^{(i)}}))
\Lo H^{i+j}(X,{\cal W}({\cal O}_X))
\tag{8.6.1}\label{ali:xag}
\end{align*}
obtained by the following exact sequence 
\begin{align*} 
0\lo {\cal W}({\cal O}_X)\lo {\cal W}({\cal O}_{\os{\circ}{X}{}^{(0)}})\lo 
{\cal W}({\cal O}_{\os{\circ}{X}{}^{(1)}})\lo 0.  
\end{align*}
By using this spectral sequence, 
it is easy to prove that 
$H^2(X,{\cal W}({\cal O}_X))\simeq H^1(E,{\cal W}({\cal O}_E))$. 
Assume that the double elliptic curve $E$ is supersingular. 
Then $h_F(\os{\circ}{X}/\kap)=h^2(\os{\circ}{X}/\kap)=2$.  
Let ${\cal X}$ be a log smooth lift of $X$ over ${\cal W}_2(s)$. 
Then there does not exist a lift 
$\wt{F}\col {\cal X}\lo {\cal X}$ over $F_{{\cal W}_2(s)}$ 
which is a lift of the Frobenius endomorphism of $X$. 
\end{exem}

\begin{rema} 
Let $X/s$ be as in (\ref{theo:cyh}). 
Assume that $\dim \os{\circ}{X}\geq 2$. 
In the case where $h^d(\os{\circ}{X}/\kap)=1$, 
we do not know whether there does not exist a lift 
$\wt{F}\col {\cal X}\lo {\cal X}$ over $F_{{\cal W}_2(s)}$ which is a lift of 
the Frobenius endomorphism of $X$ in general. 
\par 
In the case where $\dim \os{\circ}{X}=2$, 
there does not exist the lift $\wt{F}$ above 
if the log structures of ${\cal X}$ and $s$ are trivial 
(\cite[(3.3)]{x}). 
(If $h^2(\os{\circ}{X}/\kap)\geq 2$, the proof of (\ref{coro:no}) gives us 
another proof of this fact.) 
\par 
In the case  $\dim \os{\circ}{X}=1$ and the log structure of $X$ is nontrivial, 
one can prove that there exists a lift 
$\wt{F}\col {\cal X}\lo {\cal X}$ over $F_{{\cal W}_2(s)}$ which is a lift of 
the Frobenius endomorphism of $X$ (\cite{nclw}).
\end{rema} 

By using (\ref{theo:np}), we obtain the following as in \cite{ln}.

\begin{coro}\label{coro:lh}
Assume that the log structures of $s$ and $X$ are trivial and that 
${\rm NS}(X)$ is $p$-torsion-free.  
Then $H^0(X,\Om^1_{X/\kap})=0$. 
\end{coro}

\newpage
\parno
\begin{center}
{{\rm \Large{\bf Appendix}}}
\end{center}

\parno
\begin{center}
{{\rm \large{Yukiyoshi~Nakkajima}}}
\end{center}
%$${{\rm \Large }}$$

\section{Weak Lefschetz theorem for isocrystalline cohomologies}\label{sec:wlt} 
In this section we prove the weak Lefschetz theorem in \cite{kme}
(cf.~\cite{bwl}) for crystalline cohomologies of proper smooth schemes 
over $\kap$ by using rigid cohomologies. 
To prove this, we prove the following: 

\begin{theo}\label{theo:be}
Let $K$ be the fraction field of 
a complete discrete valuation ring ${\cal V}$ 
of mixed characteristic with residue field $\kap$. 
Let $X$ be a projective scheme over $\kap$ 
with a closed immersion $X\os{\sus}{\lo} {\mab P}^n_{\kap}$. 
Set $d:=\dim X$. 
Let $H$ be a hypersurface of ${\mab P}^n_{\kap}$. 
Set $Y:=X\cap H$ and $U:=X\setminus Y$. 
Assume that $U$ is smooth over $\kap$. 
Let $\iota \col Y \os{\subset}{\lo} X$ be 
the inclusion morphism. 
Then the pull-back of $\iota$ 
\begin{align*} 
\iota^*\col H^q_{\rm rig}(X/K)\lo H^q_{\rm rig}(Y/K)
\tag{9.1.1}\label{eqn:cseis}
\end{align*} 
is an isomorphism for $q<d-1$ and injective for $q= d-1$. 
\end{theo}
\begin{proof} 
By \cite[(3.1) (iii)]{brc} we have the following exact sequence 
\begin{align*} 
\cdots \lo H^i_{\rm rig,c}(U/K)\lo H^i_{\rm rig}(X/K)\lo H^i_{\rm rig}(Y/K) \lo \cdots.  
\tag{9.1.2}\label{eqn:csexs}
\end{align*} 
Hence it suffices to prove that 
\begin{align*} 
H^i_{\rm rig,c}(U/K)=0 \quad (i< d)
\tag{9.1.3}\label{eqn:csds}
\end{align*} 
Let $H^i_{\rm MW}(U/K)$ be 
the $i$-th Monsky-Washnitzer cohomology of $U/K$. 
We have the following equalities by Berthelot's duality (\cite[(2.4)]{bd}), 
Berthelot's comparison theorem (\cite[(1.10.1)]{bfi}): 
\begin{align*} 
H^i_{\rm rig,c}(U/K)={\rm Hom}_K(H^{2d-i}_{\rm rig}(U/K),K(-d))
={\rm Hom}_K(H^{2d-i}_{\rm MW}(U/K),K(-d)). 
\tag{9.1.4}\label{eqn:csdds}
\end{align*}  
Hence it suffices to prove that 
\begin{align*} 
H^i_{\rm MW}(U/K)=0 \quad (i>d). 
\tag{9.1.5}\label{eqn:csers}
\end{align*} 
Since $U$ is affine, express $U={\rm Spec}(A_0)$. 
Let ${\cal U}$ be a lift of $U$ over ${\cal V}$ (\cite[Th\'{e}or\`{e}me 6]{elk}). 
Express ${\cal U}={\rm Spec}(A)$. 
Then, by the definition of Monsky-Washnitzer cohomology, 
$H^i_{\rm MW}(U/K)=H^i(K\otimes_{\cal V}A^{\dagger}
\otimes_{A}\Om^{\bul}_{A/{\cal V}})$. 
Now the vanishing (\ref{eqn:csers}) is obvious.
\end{proof} 

\begin{rema} 
In \cite{ca} Caro has proved the hard Lefschetz Theorem in $p$-adic cohomologies. 
In particular, he has reproved the hard Lefschetz theorem proved in \cite{kme}. 
However it seems that (\ref{theo:be}) cannot be obtained by the hard Lefschetz theorem 
in \cite{kme} and \cite{ca} because $X$ nor $Y$ is not necessarily smooth. 
\end{rema}

\begin{coro}\label{coro:c}
Let the notations be as in {\rm (\ref{theo:be})}. 
Let ${\cal W}$ be a Cohen ring of $\kap$. Let $K_0$ 
be the fraction field of ${\cal W}$. 
Assume that $X$ and $Y$ are smooth over $\kap$. 
Then 
\begin{align*} 
\iota^*\col H^q_{\rm crys}(X/{\cal W})_{K_0}\lo H^q_{\rm crys}(Y/{\cal W})_{K_0}
\end{align*} 
is an isomorphism for $q<d-1$ and injective for $q= d-1$. 
\end{coro} 
\begin{proof} 
This follows from the comparison theorem of Berthelot 
(\cite[(1.9)]{bfi}: 
\begin{align*} 
H^q_{\rm crys}(Z/{\cal W})_{K_0}=H^q_{\rm rig}(Z/K_0)
\end{align*}
for a proper smooth scheme $Z/\kap$. 
\end{proof} 

\begin{rema}\label{rema:wle}
Because of the development of theory of rigid cohomology by Berthelot, 
we have been able to 
give a very short proof of the weak Lefshetz theorem for crystalline cohomologies 
of proper smooth schemes over $\kap$ without using the Weil conjecture nor 
the hard Lefschetz theorem for crystalline cohomologies 
(as in the $l$-adic case). 
\end{rema}

\bigskip
\bigskip
\parno
Yukiyoshi Nakkajima 
\parno
Department of Mathematics,
Tokyo Denki University,
5 Asahi-cho Senju Adachi-ku,
Tokyo 120--8551, Japan. 
\parno
{\it E-mail address\/}: 
nakayuki@cck.dendai.ac.jp

\bigskip
\bigskip
\parno
Fuetaro Yobuko
\parno 
Graduate School of Mathematics, Nagoya University, 
Furocho, Chikusaku, Nagoya, 464-8602, Japan. 
%Mathematical Institute, 
%Tohoku University,
%6-3 Aoba  Aramaki Aoba-Ku Sendai,  
%Miyagi 980--8578, Japan. 
\parno
{\it E-mail address\/}: 
soratobumusasabidesu@gmail.com 

\end{document}